\documentclass[12pt,a4paper]{amsart}

\usepackage{amssymb}
\usepackage{amsfonts}
\usepackage{amsmath}
\usepackage[mathscr]{eucal}
\usepackage{mathrsfs}
\usepackage{color}

\usepackage{setspace} %行間調整 
\usepackage{mathtools, bbm}
\usepackage{mathdots}
\usepackage[poorman]{cleveref}

\usepackage[lite,abbrev,alphabetic]{amsrefs}

\usepackage{comment}

\usepackage[top=30truemm,bottom=30truemm,left=30truemm,right=30truemm]{geometry}

\usepackage{amsthm}

\newtheorem{thm}{Theorem}[section]
\newtheorem{prop}[thm]{Proposition}
\newtheorem{lem}[thm]{Lemma}
\newtheorem{cor}[thm]{Corollary}
\newtheorem{defn}{Definition}

\numberwithin{equation}{section}

\def\JJ{{\mathbb J}}

\def\N{{\mathbb N}}
\def\Z{{\mathbb Z}}
\def\Q{{\mathbb Q}}
\def\R{{\mathbb R}}
\def\C{{\mathbb C}}

\def\EE{{\mathbb E}}
\def\II{{\mathbb I}}
\def\A{{\mathbb A}}

\def\emp{\varnothing}
\def\fa{{\mathfrak a}}

\def\fc{{\mathfrak c}}

\def\ff{{\mathfrak f}}
\def\fg{{\mathfrak g}}
\def\fh{{\mathfrak h}}

\def\fo{{\mathfrak o}}
\def\fp{{\mathfrak p}}
\def\fq{{\mathfrak q}}

\def\ft{{\mathfrak t}}

\def\fy{{\mathfrak y}}

\def\fA{{\mathfrak A}}

\def\fF{{\mathfrak F}}

\def\fc{{\mathfrak c}}

\def\sG{\mathsf G}

\def\sM{\mathsf M}
\def\sN{\mathsf N}
\def\sP{\mathsf P}

\def\sm{\mathsf m}
\def\sn{\mathsf n}

\def\sf{\mathsf f}

\def\su{\mathsf u}
\def\sr{\mathsf r}

\def\sT{{\mathsf T}}
\def\cO{\mathfrak o}

\def\cB{{\mathscr B}}

\def\cK{{\mathcal K}}

\def\cH{{\mathscr H}}
\def\cM{{\mathscr M}}
\def\cN{{\mathcal N}}

\def\cS{{\mathcal S}}

\def\cV{{\mathcal V}}

\def\cT{{\mathcal T}}

\def\cW{{\mathscr W}}
\def\cU{{\mathcal U}}
\def\cJ{{\mathscr J}}
\def\Re{{\operatorname {Re}}}
\def\Im{{\operatorname {Im}}}
\def\tr{{\operatorname{tr}}}
\def\nr{{\operatorname{N}}}
\def\GL{{\operatorname {GL}}}

\def\Mat{{\operatorname{M}}}
\def\End{{\operatorname{End}}}

\def\Hom{{\rm{Hom}}}
\def\diag{{\operatorname {diag}}}

\def\Ad{{\operatorname{Ad}}} 
\def\vol{{\operatorname{vol}}}

\def\leq{\leqslant}
\def\geq{\geqslant}
\def\bsl{\backslash}
\def\ch{{\cosh\,}}
\def\sh{{\sinh\,}}

\def\bZ{{\bar Z}}

\def\e{\varepsilon}

\def\cQ{{\mathcal Q}}
\def\fD{{\mathfrak D}}

\def\d{{\rm{d}}}

\def\fX{{\mathfrak X}}

\def\fin{{\rm{\bf {f}}}}
\def\bK{{\bf K}}

\def\b1{{\bold 1}}

\def\B{{\mathbb B}}

\def\bZ{{\mathbb Z}}

\def\F{{\mathbb F}}

\def\SS{{\mathbf R}}

\def\sB{{\mathsf B}}

\def\MM{{\mathbb M}}
\def\calW{{\mathscr W}}

\def\cJ{{\mathscr J}}

\def\Ical{{\mathcal I}}

\def\ss{{\mathsf s}}
\newcommand{\sslash}{\mathbin{/\mkern-6mu/}}

\def\fP{{\mathfrak P}}

\def\sZ{{\mathsf Z}}

\newcommand{\cchi}{\operatorname{\mbox{{\rm 1}}\hspace{-0.25em}\mbox{{\rm l}}}}
\def\sV{{\mathsf V}}
\def\tt{{\mathtt t}}
\def\ttN{{\mathtt N}}
\def\sX{{\mathsf X}}

\def\ud{{\underline d}}
\def\mes{{\mathfrak m}}
\title{An asymptotic formula of spectral average of central $L$-values on ${\bf GSp}(2)$ for square free levels} 

\author{Seiji Kuga and Masao Tsuzuki}
\address{Faculty of Science and Technology, Sophia University, Kioi-cho 7-1 Chiyoda-ku Tokyo, 102-8554, Japan}
\email{s-kuga-2g7@sophia.ac.jp}
\address{Faculty of Science and Technology, Sophia University, Kioi-cho 7-1 Chiyoda-ku Tokyo, 102-8554, Japan}
\email{m-tsuduk@sophia.ac.jp}
\begin{document}
\date{}
\begin{abstract} We develop a new kind of relative trace formulas on ${\bf PGSp}_2$ involving the Bessel periods and the Rankin-Selberg type integral a la Piatetski-Shapiro for Siegel cusp forms on its spectral side. As an application, a version of weighted equidistribution theorems for the Satake parameters of Siegel cusp forms of square-free level and of scalar weights is proved. 
\end{abstract}
\maketitle
\section{Introduction}
The periods of automorphic forms is of pivotal impotance in number theory. For elliptic modular forms, the toric periods and the Fourier coefficients are two of such; they have a deep connection with the special values of automorphic $L$-functions (\cite{Waldspurger1985}, \cite{Waldspurger1991}). The Gan-Gross-Prasad conjecture (the GGP conjecture for short), originally posed by \cite{GrossPrasad} and its refined form by Ichino-Ikeda (\cite{IchinoIkeda}) and by Liu (\cite{Liu}), is a far reaching generalization of Waldspurger's results. Recently, Furusawa-Morimoto completely solved an important case of the refined GGP conjecture for the Bessel periods on ${\bf PGSp}_2$ (\cite{FurusawaMorimoto1}, \cite{FurusawaMorimoto2}, \cite{FurusawaMorimoto3}), giving an affirmative answer to B\"{o}cherer's conjecture on the average of the Fourier coefficients of Siegel cusp forms of degree $2$. In this paper, we are concerned ourselves with the statistical nature of the family of spinor $L$-functions for Siegel modular forms. For this family, Kowalski-Saha-Tsimerman (\cite{KST}) introduced an analogue of the harmonic average for the elliptic cusp forms in terms of the Bessel periods. They studied an asymptotic formula of the ``harmonic average" of spinor $L$-values on the convergent range of the Euler products in the weight aspect. Later, asymptotic formulas are proved for the second moment of the central $L$-values by Blomer (\cite{Blomer}) in the weight aspect, and by Waibel (\cite{Waibel}) in the level aspect for square free levels. In these works, the basic tool is Kitaoka's formula, which is a Sielgel modular forms counterpart of Petersson's trace formula for elliptic modular forms. In this paper, continuing our study \cite{KugaTsuzuki}, we develop a new kind of relative trace formula involving the Bessel periods and the Rankin-Selberg integrals of Andrianov type on the spectral side, and prove a weighted density theorem of Satake parameters for the family of Siegel cusp forms of varying weights and square-free levels (Theorems \ref{MainT-2} and \ref{APPL-Thm}). Note that, for Siegel modular forms with principal congruence levels, an automorphic density theorem for the natural average over a family is proved by \cite{KWY}. As an application, we prove a new version of simultaneous non-vanishing results of central critical values of $3$ twisted spinor $L$-functions in a family of Siegel cusp forms of general type; this is the first such result with a good control of Satake parameters (Corollary \ref{MainT-3}). In\S\ref{sec:APPL}, we deduce the existence of an infinite family of regular algebraic self-dual cuspidal automorphic representations of ${\bf GL}_4$ with large Hecke fields (Corollary \ref{APPL-C1}), invoking Arthurs's classification of discrete spectrum of ${\bf PGSp}_2$ and Furusawa-Morimoto's resolution of refined GGP conjecture on Bessel periods.  

\subsection{Main results}\label{sec:Intro2}
Set $\sG={\bf GSp}_2:=\{g\in {\bf GL}_4 \mid {}^t g \left[\begin{smallmatrix} 0_2 & 1_2 \\ -1_2 & 0_2 \end{smallmatrix} \right] g=\nu(g)\,\left[\begin{smallmatrix} 0_2 & 1_2 \\ -1_2 & 0_2 \end{smallmatrix}\right]\}$ with $\nu:\sG\longrightarrow {\bf GL}_1$ being the similitude character; ${\bf Sp}_2$ is defined to be the kernel of $\nu$. Let $\fh_2:=\{Z=X+i Y\in {\bf M}_2(\C)\mid {}^t Z=Z,\,Y\gg 0\}$ be the Siegel upper-half space. For a square-free positive integer $N$, the space $S_l(N)$ of all the holomorphic Siegel cusp forms on $\Gamma_0^{(2)}(N):=\{\left[\begin{smallmatrix} A & B \\ C & D \end{smallmatrix}\right]\in {\bf Sp}_2(\Z)\mid C\equiv 0 \pmod{N}\}$ of weight $l\in \Z_{>0}$ is a finite dimensional $\C$-vector space. We endow $S_l(N)$ with the hermitian inner product associated with the norm  
$$
\langle \Phi\mid \Phi \rangle_{L^2}=\frac{1}{[{\bf Sp}_2(\Z):\Gamma_0^{(2)}(N)]}\int_{\Gamma_0^{(2)}(N)\bsl \fh_2}|\Phi(Z)|^2
\,(\det \Im(Z))^{l}\d\mu_{\fh_2}(Z), \quad \Phi\in S_l(N),
$$
where $\d\mu_{\fh_2}(Z)=\det(Y)^{-3}\,{\d X\,\d Y}$ is the invariant volume element on $\fh_2$. Let $\Phi(Z)$ be a non-zero element of $S_l(N)$ and 
$$
\Phi(Z)=\sum_{T\in \cQ^+}a_{\Phi}(T)\,\exp(2\pi i \tr(TZ)), \quad Z\in \fh_2
$$
its Fourier expansion, where $\{a_{\Phi}(T)\mid T\in \cQ^+\}$ is the set of Fourier coefficients of $\Phi$, which is indexed by $\cQ^{+}$, the set of positive definite half-integral symmetric matrices of degree $2$. By the modularity of $\Phi(Z)$, we have $a_{\Phi}(\delta T{}^t \delta)=a_{\Phi}(T)$ for all $\delta \in {\bf SL}_2(\Z)$ and $T\in \cQ^{+}$. Let $E=\Q(\sqrt{D})$ be an imaginary quadratic field of discriminant $D<0$. Then, the set
$$\cQ_{\rm prim}^{+}(D):=\Bigl\{\left[\begin{smallmatrix} a & \frac{b}{2} \\ \frac{b}{2} & c\end{smallmatrix} \right]\in \cQ^{+}\Bigm|\,b^2-4ac=D,\,(a,b,c)=1\,\Bigr\}
$$
is ${\bf SL}_2(\Z)$-invariant, and the quotient ${\bf SL}_2(\Z)\bsl \cQ_{\rm prim}^{+}(D)$ is in a natural bijective correspondence with the ideal class group ${\rm Cl}(E)$ of $E$; the ideal class corresponding to the ${\bf SL}_2(\Z)$-orbit of $T \in \cQ_{\rm prim}^{+}(D)$ is denoted by $[T]$. Following \cite{KST}, for a character $\Lambda$ of ${\rm Cl}(E)$, define
$$
R(\Phi,D,\Lambda):=\sum_{T\in {\bf SL}_2(\Z)\bsl \cQ_{\rm prim}^{+}(D)}a_{\Phi}(T)\,\Lambda([T])^{-1},
$$ 
and 
\begin{align}
\omega^{\Phi}_{l,D,\Lambda}:=
\tfrac{\sqrt{\pi}}{4}(4\pi)^{3-2l}\Gamma\left(l-\tfrac{3}{2}\right)\Gamma(l-2)
\times \left(\tfrac{|D|}{4}\right)^{\frac{3}{2}-l}\frac{4(2-\delta_{\Lambda^2,{\bf 1}})}{w_D\,h_D}\frac{|R(\Phi,D,\Lambda)|^2}{\langle \Phi\mid \Phi \rangle_{L^2}} 
\label{Intro-defC}
\end{align}
with $h_D:=\# {\rm Cl}(E)$ and $w_{D}$ being the number of the roots of unity in $E$. 

To describe our results, we use the notion of automorphic representations of $\sG$, which we refer to \cite[\S3]{DPSS}. Let $\A$ denote the ring of adeles of $\Q$ and $\A_\fin$ the subring of finite adeles. Since $\sG(\A)=\sG(\Q)\sG(\R)^0
\bK_0(N)$ with 
$$
\bK_0(N):=\{\left[\begin{smallmatrix} A & B  \\ C & D\end{smallmatrix}\right]\in \sG(\widehat \Z)\mid C\in N\,{\bf M}_2(\widehat \Z)\},
$$
from $\Phi(Z)$, we can define a function $\tilde \Phi(g)$ on $\sG(\A)$ in such a way that
\begin{align}
\tilde \Phi(\gamma g_\infty g_\fin)&=
\nu(g_\infty)^{l}\det(Ci+D)^{-l}\Phi((Ai +D)(Ci +D)^{-1}), 
\label{AdeleLift}
\\
&\quad \gamma \in \sG(\Q),\, g_\infty=\left[\begin{smallmatrix} A & B \\ C & D
\end{smallmatrix}\right]\in \sG(\R)^{0},\,g_\fin \in \bK_0(N).
 \notag
\end{align}
The image of the linear injective map $\Phi\mapsto \widetilde \Phi$, say $\Tilde S_l(N)$, is contained in the space of cusp forms on $\sG(\A)$ that are $\sZ(\A)\bK_0(N)$-invariant. Let $\Pi_{\rm cusp}$ be a maximal set of mutually orthogonal irreducible subrepresentations $(\pi,V_\pi)$ of the automorphic forms in $L_{\rm cusp}^2(\sZ(\A)\sG(\Q)\bsl \sG(\A))$. Let $\Pi_{\rm cusp}(l,N)$ be the set of $\pi \in \Pi_{\rm cusp}$ such that $V_\pi(l,N):=V_\pi \cap \widetilde S_l(N)$ is non-zero, so that $\widetilde S_l(N)$ is a direct sum of $V_\pi(l,N)$ for $\pi \in \Pi_{\rm cusp}(l,N)$. Let $N_\pi$ for $\pi \in \Pi_{\rm cusp}(l,N)$ be the least positive divisor of $N$ such that $\pi \in \Pi_{\rm cusp}(l,N_\pi)$. Any $\pi\in \Pi_{\rm cusp}(l,N)$ is decomposed to a restricted tensor product of irreducible smooth representations $\pi_{p}$ of $\sG(\Q_p)$ for $p<\infty$ and a holomorphic discrete series representation $D_l$ of weight $l$ of $\sG(\R)$ and of trivial central character. As in \cite{KugaTsuzuki}, throughout this paper, we suppose that 
\begin{align}
\text{all the prime divisors of $N$ are inert in $E/\Q$.}
\label{AssmpInert}
\end{align}
Let $\Pi_{\rm cusp}^{E,\Lambda}(l,N)$ be a set of $\pi\in \Pi_{\rm cusp}(l,N)$ such that $R(D,\Phi,\Lambda^{-1})\not=0$ for some $\widetilde \Phi \in V_\pi(l,N)$. Recall that, under assumption \eqref{AssmpInert}, the label of the local representation $\pi_p$ for $\pi \in \Pi_{\rm cusp}^{E,\Lambda}(l,N)$ is either I or IIb if $p\mid \frac{N}{N_\pi}$ and is either IIIa or VIb if $p\mid N_\pi$ (\cite[\S3]{DPSS}). Given a character $\mu=\otimes_{p\leq \infty} \mu_p$ of $\A^\times/\Q^\times \R_{>0}$, define $\ft(\pi,\mu)\in \R_{\geq 0}$ to be the product $\prod_{p\mid N}\ft(\pi_p,\mu_p)$ with $\ft(\pi_p,\mu_p)$ being defined as  
\begin{align*}
\begin{cases}
1 \quad (\text{$p\mid N_\pi$ and $\pi_p$ is of type VIb}),  \\
2 \quad (\text{$p\mid N_\pi$ and $\pi_p$ is of type IIIa}), \\
2(p-1)p^{-5}L(1,\pi_p,{\rm Std})\,\{1+\mu_p^2(p)-\tfrac{\mu_p(p)}{p+1}\tr(p^{-1}T_p+\eta_p|\pi_p^{\bK_0(p\Z_p)})\} \quad (\text{$p\mid \frac{N}{N_\pi}$}),
\end{cases}
\end{align*}
where $T_p$ is the Hecke operator defined by the double coset $\bK_0(p\Z_p)\left[\begin{smallmatrix} p 1_2 & {} \\ {} & 1_2\end{smallmatrix}\right]
\bK_{0}(p\Z_p)$ on $\bK_0(p\Z_p)$-fixed vectors in an admissible $\sG(\Q_p)$-module, $\eta_p\in \sG(\Q_p)$ is the Atkin-Lehner element at $p$ (see \S\ref{sec:ErrTEST}(ii)), and $L(s,\pi_p,{\rm Std})$ is the local standard $L$-factor of $\pi_p$. Note that $\ft(\pi,\mu)=1$ if $N=1$, and $\ft(\pi,\mu)>0$ if $\mu^2={\bf 1}$ and $N_\pi=N$. As in \cite[\S4.4]{KugaTsuzuki}, for each $\pi \in \Pi_{\rm cusp}^{E,\Lambda}(l,N)$, we, once and for all, fix an element $\Phi_\pi^0\in  S_l(N_\pi)-\{0\}$, which is a newform in the sense of \cite[\S3.3.5]{Schmidt}, in such a way that $\widetilde {\Phi_\pi^0}\in V_\pi$ corresponds to a pure tensor $(\otimes_{p<\infty} \phi_{\pi_p}^0) \otimes \phi_\infty$ with $\phi_{\pi_p}^0$ being a $\bK_0(N_\pi\Z_p)$-invariant vector of $\pi_p$ which is an eigenvector of the operator $T_p$ if $\pi_p$ is of type IIIa, and $\phi_\infty$ being the fixed lowest weight vector of $D_l$. 
%For $p\nmid N$, the representation $\pi_{p}$ is $\sG(\Z_p)$-spherical, so that it is determined up to equivalence by its Satake parameter, which will be recalled next briefly.
 Let ${\bf B}$ be the Borel subgroup of $\sG$, which consists of all the matrices of the form $\sm(A,c)\sn(B)$ (for this notation we refer to \S\ref{sec:P2}) with $A$ being an upper-triangular matrix of degree $2$. The set of elements $\sm(A,1)\sn(B)$ with $A$ being upper-triangular unipotent (resp. $\sm(A,c)$ with $A$ being diagonal) is denoted by ${\bf U}$ (resp. by ${\bf T}$). Then ${\bf U}$ is the unipotent radical of ${\bf B}$ and ${\bf B}={\bf T}{\bf U}$ is a Levi decomposition. The involutions on the complex torus $(\C^\times)^2$, $(a,b)\mapsto (b,a)$ and $(a,b)\mapsto (a^{-1},b)$, generate a subgroup $W \subset {\rm Aut}((\C^\times)^2)$ isomorphic to the $C_2$-Weyl group. For $y=(a,b)\in (\C^\times)^2/W$, define a quasi-character $\chi_{\alpha,\beta}: {\bf B}(\Q_p)\longrightarrow \C^\times$ by$$
\chi_{\alpha,\beta}(\diag(t_1,t_2,\lambda t_1^{-1},\lambda t_2^{-1})n)=|t_1|_p^{-\alpha-\beta}|t_2|_p^{-\alpha+\beta}|\lambda|_p^{\alpha}, \quad (t_1,t_2,\lambda)\in (\Q_p^\times)^3,\,u\in {\bf U}(\Q_p),
$$
where $\alpha:={\rm ord}_{p}(a)$ and $\beta:={\rm ord}_p(b)$. 
%Note that $\Lambda_{-3,1}$ is the modulus character of ${\bf B}(\Q_p)$. 
%Then, $I_p(y)$ is of finite length and contains a unique unramified irreducible subquotient to be denoted by $\pi_p^{\rm ur}(y)$. 
Let $\pi_p^{\rm ur}(y)$ be the smallest subspace of smooth functions on $\sG(\Q_p)$ that is invariant by the right translations by $\sG(\Q_p)$ and contains the spherical function $\omega_p(y):\sZ(\Q_p)\bsl \sG(\Q_p)\rightarrow \C$ (a la Harish-Chandra and Satake) defined by
\begin{align}
\omega_{p}(y;g)=\int_{\sG(\Z_p)}\chi_{\alpha-3/2,\beta+1/2}(t(kg))\d k, \quad g\in \sG(\Q_p), 
\label{Def-SphFtn}
\end{align}
where $t(g)\in {\bf T}(\Q_p)/{\bf T}(\Z_p)$ is uniquely defined by $g\in t(g){\bf U}(\Q_p)\sG(\Z_p)$. The map $y\mapsto \pi_p^{\rm ur}(y)$ yields a bijection from $(\C^\times)^2/W$ onto the set of all the equivalence classes of smooth irreducible unramified representations of $\sZ(\Q_p)\bsl \sG(\Q_p)$. Let $Y_p$ denote the set of $y=(a,b)\in \C^2$ such that $\pi_p^{\rm ur}(y)$ is unitarizable. Set $[Y_p]:=Y_p/W$. For a prime $p\nmid N$, let $y_{p}(\pi):=(a_{p},b_p)\in [Y_p]$ be the Satake parameter of $\pi$ at $p$, so that $\pi_{p}\cong \pi_p^{\rm ur}(a_p,b_p)$. The local spinor $L$-factor of $\pi_p$ is defined as
$$
L(s,\pi_p):=(1-a_pp^{-s})^{-1}(1-b_p p^{-s})^{-1}(1-a_p^{-1}p^{-s})^{-1}(1-b_p^{-1}p^{-s})^{-1}.
$$ The representation $\pi_q$ with $q\mid N$ is one of the Iwahori spherical representations listed in \cite[Table 2]{Schmidt}; we use the local $L$-factor $L(s,\pi_q)$ as proposed there. The twisted local $L$-factor $L(s,\pi_p,\mu_p)$ by a character $\mu_p:\Q_p^\times \rightarrow \C^1$ is defined as $L(s,\pi_p,\mu_p):=L(s+c_p,\pi_p)$ if $\mu_p$ is unramified with $c_p\in \C$ being defined by $\mu_p(p)=p^{-c_p}$, and $L(s,\pi_p,\mu_p):=1$ if $\mu_p$ is ramified and $\pi_p$ is $\sG(\Z_p)$-spherical. Let $\mu:\A^\times/\Q^\times \R_{>0} \rightarrow \C^1$ be a character of conductor $M$ prime to $N_\pi$. The spinor $L$-function of $\pi$ twisted by $\mu$ is initially defined by the Euler product $L(s,\pi,\mu):=\prod_{p<\infty}L(s,\pi_p,\mu_p)$ absolutely convergent on $\Re(s)>5/2$. The completed $L$-function for $L(s,\pi, \mu)$ is defined as
\begin{align}
\widehat L (s,\pi,\mu):=\Gamma_{\C}\left(s+\tfrac{1}{2}\right)\Gamma_{\C}\left(s+l-\tfrac{3}{2}\right)\times L(s,\pi,\mu). 
\label{Intro-defLambda}
\end{align}
%When $\mu$ is trivial, $L(s,\pi,\mu)$ and $\Lambda(s,\pi,\mu)$ are abbreviated to $L(s,\pi)$ and $\Lambda(s,\pi)$, respectively. 
In \cite[\S4.6]{KugaTsuzuki}, for $\pi\in \Pi_{\rm cusp}^{E,\Lambda}(l,N)$ with $N$ being as in \eqref{AssmpInert} and for $\mu$ as above with the conductor being a product of odd primes inert in $E/\Q$, we have the explicit functional equation as expected (\cite[\S3.1.4]{Schmidt}) by identifying the local $L$-factor and the local $\varepsilon$-factor defined in \cite{PS97} (in terms of the Bessel model from $E$ and $\Lambda$) with the ones proposed in \cite[Table 2,3]{Schmidt}. In particular, the functional equation of $\widehat L(s,\pi,\mu)$ is self-dual with the sign being $(-1)^l$ for all $\pi \in \Pi_{\rm cusp}^{E,\Lambda}(l,N)$ if $\mu^2={\bf 1}$. 

\begin{thm}\label{MainT-1} Let $E=\Q(\sqrt{D})$ be an imaginary quadratic field of discriminant $D<0$, and $\Lambda$ a character of ${\rm Cl}(E)$. Let $\mu:\A^\times/\Q^\times\R_{>0} \rightarrow \C^1$ be a character whose conductor $M$ is a product of odd primes inert in $E$; as such, $\mu$ induces a primitive Dirichlet character $\tilde \mu$ mod $M$. Then, there exists a constant $C>0$ such  
\begin{align*}
[{\bf Sp}_2(\Z):\Gamma_0^{(2)}(N)]^{-1}
&\sum_{\pi \in \Pi_{\rm cusp}^{E,\Lambda}(l,N)}\omega_{l,D,\Lambda^{-1}}^{\Phi_\pi^0}\, \,L\left(\tfrac{1}{2},\pi,\mu\right)\,\ft(\pi,\mu) \\
&=\prod_{q\in S(N)}(1+q^{-2})\times 2P(l,N,\Lambda,\mu)+O(2^{-l+\# S(N)}N^{-\frac{l-6}{4}})
\end{align*}
uniformly for $l\in 2\Z_{\geq 8}$ and all square free positive integer $N$ relatively prime to $M$ such that the set $S(N)$ of prime divisors of $N$ consists of inert primes for $E/\Q$, where \begin{align*}
P(l,N,D,\Lambda,\mu):=\begin{cases}
L(1,\kappa_D)\,(\log N+\psi(l-1)-\log(4\pi^2))+L'(1,\kappa_D) \quad &(\Lambda\mu_{E}={\bf 1}), \\ 
 \tilde \mu(DN)^{-2}L(1,\Lambda\mu_{E})+\tilde \mu(4)^{-1}
 \left(\tfrac{G(\tilde \mu^{-1})}{\sqrt{M}}\right)^{4}\,L(1,\Lambda^{-1}\mu_{E}^{-1})\quad &(\Lambda\mu_{E}\not={\bf 1})
\end{cases}
\end{align*}
with $\psi(s):=\Gamma'(s)/\Gamma(s)$ the di-gamma function, $\kappa_{D}$ the Kronecker character, $L(s,\Lambda\mu_{E})$ the Hecke $L$-function of $\Lambda\mu_{E}$ with $\mu_{E}:=\mu \circ \nr_{E/\Q}$, and $G(\tilde \mu):=\sum_{a\in (\Z/M\Z)^\times}\tilde \mu(a)e^{2\pi i a/M}.$
\end{thm}
Note that $\Lambda\mu_{E}={\bf 1}$ is equivalent to both $\Lambda$ and $\mu$ being trivial. We remark that this result is consistent with \cite[Theorem 1]{Blomer}, which deals with the case $\mu={\bf 1}$, $N=1$ and $D=-4$, so that $\Lambda={\bf 1}$ and $M=1$. Due to the factor $\ft(\pi,\mu)$, our average is a bit different from the one in \cite[Theorem 6]{Blomer}, which corresponds to $\mu={\bf 1}$, $N\equiv 3\pmod{4}$ being a large prime and $D=-4$ in our setting. Discarding the contribution from the old forms and the Saito-Kurokawa forms, the two formulas should be consistent with each other. It should be noted that $P(l,N,D,\Lambda,\mu)/\log(lN) \sim L(1,\kappa_D)>0$ if $\Lambda\mu_{E}={\bf 1}$, and $P(l,N,\Lambda,\mu)=L(1,\Lambda\mu_{E})>0$ if $\Lambda\mu_{E}\not={\bf 1}$ and $\mu$ is quadratic. To describe our second theorem, which is a refinement of Theorem \ref{MainT-1}, we need additional notation and definitions. Let $\cH_p$ denote the Hecke algebra for $(\sG(\Q_p),\sG(\Z_p))$. The spherical Fourier transform of $f\in \cH_p$ (or rather its $\sZ(\Q_p)$-average $f^{Z}(g):=\int_{\Q_p^\times} f(zg)\,\d^\times z$) is defined by $\widehat {f}(y):=\int_{\sG(\Q_p)}f(g)\omega_p(a^{-1},b^{-1};g)\,\d g$ for $y=(a,b)\in [Y_p]$ with $\d g$ being the Haar measure on $\sG(\Q_p)$ such that $\vol(\sG(\Z_p))=1$. The Fourier inversion formula is known to be described as 
\begin{align}
f^{Z}(g)=\int_{Y_p/W}{\widehat {f}}(y)\,\omega_p(y;g)\,\d \mu_p^{\rm Pl}(y), \quad f\in \cH_p,\,g\in {\sG}(\Q_p), 
 \label{Def-FInvF}
\end{align}
where $\d\mu_p^{\rm Pl}(y)$ is the spherical  Plancherel measure of $\sZ(\Q_p)\bsl \sG(\Q_p)$, which is a Radon measure on $[Y_p]$ supported on the tempered locus $[Y_p^0]:={\bf U}(1)^2/W$. For $y=(a,b)\in (\C^\times)^2/W$ and an irreducible smooth unramified representation $\sigma_p$ of ${\bf PGL}_2(\Q_p)$ of Satake parameter $B_p=\diag(c,c^{-1})\in {\bf SL}_2(\C)$, set $A_p:=\diag(a,b,a^{-1},b^{-1})\in {\bf Sp}_4(\C)\,(={}^L{\bf PGSp}_4))$ and
\begin{align*}
L(s, \pi_p^{\rm ur}(y) \times \sigma_p)&:=\det(1_{8}-(A_p\otimes B_p)\,p^{-s})^{-1}, \\
L(s, \pi_p^{\rm ur}(y);{\Ad})&:=\det(1-\Ad(A_p)p^{-s})^{-1},
\end{align*}
where $\Ad$ is the adjoint representation of ${\bf Sp}_4(\C)$ on its Lie algebra. 
For a class group character $\Lambda\in \widehat {{\rm Cl}(E)}$, viewing it as a character of the idele class group of $E^\times$, we form its automorphic induction ${\mathcal {AI}}(\Lambda)=\bigotimes_{p\leq \infty}{\mathcal {AI}}_p(\Lambda)$ to ${\bf GL}_2(\A)$.
Let $\Pi_{\rm cusp}^{(E,\Lambda)}(l,N){}^{\rm  G,new}$ be the set of $\pi \in \Pi_{\rm cusp}^{(E,\Lambda)}(l,N)$ whose Arthur parameter is of the form $\pi^{{\rm GL}}\boxtimes {\bf 1}_{{\bf SL}_2(\C)}$ with $\pi^{{\rm GL}}$ an irreducible cuspidal automorphic representation of ${\bf GL}_4(\A)$ such that $L(s,\pi^{{\rm GL}}, \wedge^2)$ has a pole at $s=1$ (\cite{Arthur} and \cite{GeeTaibi}, see also \cite{Schmidt2008}).
\begin{thm} \label{MainT-2} 
Let $D<0$ be fundamental discriminant. Let $\mu:\A^\times/\Q^\times\R_{>0} \rightarrow \{\pm 1\}$ be a character whose conductor $M$ is a product of odd primes inert in $E:=\Q(\sqrt{D})$. Let $N$ be either $1$ or a prime number inert in $E$ and $(N,M)=1$. Let $S$ be a finite set of prime numbers prime to $NMD$. Let $l\in 2\Z_{\geq 7}$. Then, as $N \rightarrow \infty$, or $N=1$ and $l\rightarrow \infty$, 
{\allowdisplaybreaks\begin{align}
&\frac{\prod_{q\in S(N)}(1+q^{-2})^{-1}}{(\log Nl)^{\delta(\Lambda\mu_{E}={\bf 1})}
[{\bf Sp}_2(\Z):\Gamma_0^{(2)}(N)]}
\sum_{\pi \in \Pi_{\rm cusp}^{(E,\Lambda)}(l,N){}^{\rm G, new}}{\alpha}(y_{S}(\pi))\,{L\left(\tfrac{1}{2},\pi,\mu\right)}\,\omega^{\Phi_\pi^0}_{l,D,\Lambda^{-1}} \ft(\pi,\mu) 
 \label{MainT-2-f0}
\\
&\quad \longrightarrow 
2{\mes}^\Lambda_S(\alpha)\begin{cases} 
L(1,\kappa_D), \quad &(\Lambda\mu_{E}={\bf 1}), \\ 
L(1,{\mathcal {AI}}(\Lambda)\times \mu)
, \quad &(\Lambda\mu_{E}\not={\bf 1}), 
\end{cases}, \qquad {\alpha} \in C([Y_S^0]),    
\notag
\end{align}}where $y_{S}(\pi)=\{y_p(\pi)\}_{p\in S}\in [Y_S^0]$ is the set of Satake parameter of $\pi$ over $S$, and ${\mes}_S^{\Lambda,\mu}$ is a Radon measure on $[Y^0_S]$ such that ${\mes}_S^{\Lambda,\mu}(\alpha)$ for $\alpha=\otimes_{p\in S}\alpha_p\,(\alpha_p\in C([Y_p^0])$ equals $\prod_{p\in S}\int_{[Y_p^0]}\fD_p(y_p)\alpha_p(y_p)\,\d\mu_p^{\rm Pl}(y_p)$ for $y=(y_p)_{p\in S}$ with 
{\begin{align*}
&\fD_p(y_p):=
\frac{\zeta_p(1)^{-1}\zeta_p(2)\zeta_p(4)}{L(1,{\mathcal {AI}}(\Lambda)_p\times \mu_p)}\,\frac{L\left(\frac{1}{2},\pi_p^{\rm ur}(y) \times {\mathcal {AI}}(\Lambda)_p\right)L\left(\frac{1}{2},\pi_p^{\rm ur}(y),\mu_p\right)}{L(1,\pi_p^{\rm ur}(y),{\rm Ad})}. \end{align*}}
\end{thm} Combining Theorem \ref{MainT-2} with a version of the refined GGP conjecture for the Bessel periods proved in \cite{FurusawaMorimoto2}, we obtain the following corollary. 
\begin{cor} \label{MainT-3}
Let $D<0$ be a fundamental discriminant, and $N$ either $1$ or a prime number inert in $E=\Q(\sqrt{D})$. Let $\mu:\A^\times/\Q^\times \R_{>0}\rightarrow \C^1$ be a character whose conductor $M>0$ is an odd fundamental discriminant such that $(DN,M)=1$. Let $S$ be a finite set of prime numbers prime to $NMD$. Let $\Lambda$ be a character of ${\rm Cl}(E)$. Let $U$ be a measurable subset of $[Y_S^0]$ such that $\mu_S^{\rm Pl}(U)>0$ and $\mu_{S}^{\rm Pl}(\partial U)=0$. Define ${\fP}(l,N,\Lambda,\mu,U)$ to be the set of $\pi\in \Pi_{\rm cusp}(l,N)^{\rm G,new}$ such that
 \begin{align*}
&L\left(\tfrac{1}{2},\pi,\mu\right)\,L\left(\tfrac{1}{2},\pi \times {\mathcal AI}(\Lambda)\right)>0, \qquad \quad y_{S}(\pi)\in U.
\end{align*} 
\begin{itemize}
\item[(i)] Let $l\in 2\Z_{\geq 7}$ be fixed. Then, there exists $q_0\in \Z_{>0}$ such that ${\fP}(l,q,\Lambda,\mu,U)\not=\emp$ for any prime number $q>q_0$ such that $q\nmid DM$ and inert in $\Q(\sqrt{D})$. 
\item[(ii)] Let $N=1$. Then, there exists $l_0\in 2\Z_{\geq 7}$ such that ${\fP}(l,1,\Lambda,\mu,U)\not=\emp$ for any even integer $l>l_0$. 
\end{itemize}
\end{cor}
 Theorem \ref{MainT-2} is deduced from Theorem \ref{MainT-0} with an error term whose dependence on $f_S\in \otimes_{p\in S}\cH_p$ is explicit: this version is more important in potential applications. As one of such applications, in \S\ref{sec:APPL}, we prove the following.   
\begin{cor}\label{APPL-C1} Let $l\in 2\Z_{\geq 7}$. Let $M$ and $D<0$ be fundamental discriminants prime to each other. Let $p$ be a prime such that $p\nmid MD$. Then, there exists constants $N_{p}>p|DM|$ and $C_p>0$ with the following properties: For any prime number $N>N_p$, there exists an irreducible cuspidal automorphic representation $\Pi$ of ${\bf GL}_4(\A)$ such that 
\begin{itemize}
\item[(i)] The JPSS-conductor of $\Pi$ is $N^2$, 
\smallskip
\item[(ii)] $\Pi$ is regular algebraic of infinite type $(l,2,1,3-l)$ $($\cite[Definition 3.6]{Clozel}$)$, so that its field of rationality $\Q(\Pi)$ is a number field, 
\smallskip
\item[(iii)] $\Pi$ is of symplectic type, i.e., $L(s,\Pi,\wedge^2)$ has a pole at $s=1$,  
\item[(iv)] $L\left(\tfrac{1}{2},\Pi\right)\, L\left(\tfrac{1}{2}, \Pi\times \kappa_{M} \right)\,  L\left(\tfrac{1}{2}, \Pi \times \kappa_{D}\right)\not=0$, 
\smallskip
\item[(v)] $\dim_{\Q} \Q(\Pi)\geq C_p\sqrt{\log\log N}$. 
\end{itemize}  
\end{cor}
Note that the quadratic twists $L(s,\Pi\times \kappa)$ as in the corollary are expected to be ``motivic" and $s=1/2$ is the (unique) critical point in the sense of Deligne. 

\subsection{Structure of paper}
In \S2, after introducing basic notation, the notion of the Bessel periods of automorphic forms on $\sG(\A)$ is recalled and the construction of the Eisenstein series on the subgroup $\sG^\#(\A)$ is reviewed. The Shintani function on the $p$-adic orthogonal groups plays a pivotal role in the formation of the refined GGP conjecture (\cite{IchinoIkeda}, \cite{Liu}). An archimedean counterpart was first introduced in \cite{Tsud2011-1} as a kernel of a Poincar\'{e} series; the latter was closely studied in \cite{Tsud-2} to prove a weighted equidistribution theorem of the Satake parameters in the weight aspect for full level cusp forms on ${\rm SO}(2+m)$. Actually, in \cite{Tsud-3} the second author already deduce a part of Theorem \ref{MainT-2} from a general result in \cite{Tsud-2} by the exceptional isomorphism ${\bf PGSp}_2\cong {\rm SO}(5)$. In this paper, we write some of the archimedean computations in \cite{Tsud-2} directly on ${\bf GSp}_2$. In \S\ref{sec:ShinFtn}, the Shintani function on $\sG(\R)$ is introduced; then the construction of a Poincar\`{e} series given in \cite{Tsud2011-1} is modified so that the Poincar\'{e} series depend on two test functions $\phi$ and $f$; its spectral expansion is given in \S\ref{subsecSPECEXP}. In \S\ref{sec:DoubleCoset}, we describe the double coset space $\sB^\#(\Q)\bsl \sG(\Q)/\sN(\Q)$ explicitly, which allows us to write the Bessel period of the Poincar\'{e} series as a sum of 5 terms $\JJ_{1_4}$, $\JJ_{w_2({\rm s})}$, $\JJ_{w_2({\rm r})}$, $\JJ_{w_1}$ and $\JJ_{\sn(T_\theta^\dagger)w_2}$ in \eqref{BesselPer-f4} according to the double cosets. In \S\ref{sec:PrfMT}, we specify the test functions $\phi$ and $f$ and state our results on explicit formula of the main term $\JJ_{1_4}+\JJ_{w_2({\rm s})}$ (Proposition \ref{MainT-P1}) and the estimate of other terms (Proposition \ref{ErrorT-P1}), postponing their proof to later sections. The proof of Theorem \ref{MainT-1} and \ref{MainT-2} are completed in \S\ref{sec:PrfMTHM}. In \S\ref{sec:APPL}, we describe a proof of Corollary \ref{APPL-C1}. The remaining sections are devoted to addressing highly technical issues.
%\footnote{ In this paper, we rely on explicit determination of archimedean orbital integrals, which is strong in that it yields an (super-)exponential decay of the error terms in weight aspect. The authors are now diveloping a softer method involving only a majorant of the Shintani function, which works on error estimate in level aspect even for vector-valued Siegel modular forms. }. 
In \S\ref{sec:MainT}, we prove Proposition \ref{MainT-P1} by using the explicit determination of local orbital integrals given in \S\ref{sec:pLOIMTM}. Upper estimates of $\JJ_{w_2({\rm r})}$, $\JJ_{w_1}$ and $\JJ_{\sn(T_\theta^\dagger)w_2}$ are obtained by constructing gauge functions.

\section{Preliminaries}
For sets $X\subset Y$, let $\cchi_{X}$ denote the characteristic function of $X$ on $Y$; if $Y$ is a topological space, $C(Y)$ denotes the space of all $\C$-valued continuous functions on $Y$. The identity matrix of degree $n\in \Z_{>1}$ is denoted by $1_n$. For $I\subset \R$ and $\varepsilon>0$, set $\cT_{I,\varepsilon}:=\{s\in \C\mid \Re(s)\in I,\,|\Im(s)|\geq \varepsilon\}$. A meromorphic function $f(s)$, holomorphic on $\C-\R$, is said to be vertically of polynomial growth (resp. vertically bounded) if for any compact interval $I\subset \R$ there exist $\varepsilon>0$ and $N\geq 0$ such that $f(s)\ll (1+|\Im(s)|)^{N}$ (resp. $f(s)\ll 1$) for $s\in \cT_{I,\varepsilon}$.

For any totally disconnected unimodular group $L$ with a Haar measure, let $\cH(L)$ denote the Hecke algebra of $L$, i.e., the $\C$-algebra of all those compactly supported locally constant $\C$-valued functions with the product being defined by the convolution of two such functions. 
For a number field $F$, the adele ring of $F$, $\A\otimes_{\Q}F$, is denoted by $\A_F$; the subring of finite adele is denoted by $\A_{F,\fin}$. The subscript $\fin$ is used to indicate that the object is related to the finite adeles. For $a=(a_p)_{p<\infty} \in \A_{\fin}^\times$, set $|a|_\fin=\prod_{p<\infty}|a_p|_{p}$. The Kronecker character of fundamental discriminant $D$ is denoted by $\kappa_{D}$. 

\subsection{} \label{sec:P2}
Let $\sV$ denotes the space of symmetric matrices of degree $2$; we consider $\sV$ as a $\Z$-scheme by defining $
\sV(R)=\left\{ \left[\begin{smallmatrix} a & b \\ b & c \end{smallmatrix} \right]\mid a,b,c\in R\right \}
$ for any commutative ring $R$. The group ${\bf GL}_2(R)$ acts on this space from the right as 
\begin{align}
\sV(R) \times {\bf GL}_2(R)\ni (T,h)\quad \longmapsto \quad  T\ss(h):=\tfrac{1}{\det h
}{}^t h T h \in \sV(R).
\label{GLaction}
\end{align}
For a $2\times 2$-matrix $A=\left[\begin{smallmatrix} a & b \\ c & d \end{smallmatrix} \right]$, set $
A^\dagger:=\left[\begin{smallmatrix} d & -b \\ -c & a \end{smallmatrix} \right]$, so that $AA^\dagger=(\det A)\,1_2$. We define a symmetric $R$-bilinear form on $\sV(R)$ as 
\begin{align}
\langle X,Y \rangle:=\tr(X Y^\dagger), \quad X,Y\in \sV(R). 
\label{P2-f0}
\end{align}
As in \S\ref{sec:Intro2}, let $\sG$ denote the symplectic similitude group and $\sZ$ its center. 
Let $\sP=\sM\sN$ be the Siegel parabolic subgroup of $\sG$, where 
\begin{align*}
\sM&:=\{\sm(A,c):=\left[\begin{smallmatrix} A & O \\ O &{}^t A^{-1}c \end{smallmatrix} \right] \mid A\in {\bf GL}_2\,,c \in {\bf GL}_1\}, \\
\sN&:=\{\sn(B):=\left[\begin{smallmatrix} 1_2 & B \\ O & 1_2 \end{smallmatrix} \right] \mid B \in \sV\}.
\end{align*}
Note that $\nu(\sm(A,c))=c$ and $\nu(\sn(B))=1$.

\subsection{Distinguished subgroups} \label{sec:GO} 
Let $E=\Q(\sqrt{D})\,(\subset \C)$ be an imaginary quadratic field of discriminant $D<0$ and $\cO_{E}$ the ring of integers in $E$. Set $\theta:=\frac{\tt-\sqrt{D}}{2}$ with $\tt\in \{0,1\}$ having the same parity as $D$. Then, $\{1,\theta\}$ is a $\Z$-basis of $\cO_{E}$ such that $\bar\theta-\theta=\sqrt{D}$ and $\tr_{E/\Q}(\theta)=\tt$. Set ${\ttN}:=\nr_{E/\Q}(\theta)$, so that $\tt^2-4\ttN=D$. 
Throughout this article, $E$ and $\theta$ are fixed. Let $\cQ:=\{T\in \sV(\Q)\mid \tr(TX)\in \Z\,(\forall X\in \sV(\Z))\}$, which is also equal to the dual lattice of $\sV(\Z)$ with respect to \eqref{P2-f0}; explicitly, $\cQ=\sum_{j=1}^{3}\e_j\Z$ with 
\begin{align}
\varepsilon_1=\left[\begin{smallmatrix} 1 & 0 \\ 0 & 0 \end{smallmatrix}\right], \quad \varepsilon_2=\left[\begin{smallmatrix} 0 & 1/2 \\ 1/2 & 0 \end{smallmatrix}\right], \quad \varepsilon_3=\left[\begin{smallmatrix} 0 & 0 \\ 0 & 1 \end{smallmatrix}\right]. 
\label{BasisV}
\end{align}
We have $\nr_{E/\Q}(a+b\theta)=[a,b]T_\theta \left[\begin{smallmatrix} a\\ b \end{smallmatrix}\right]$ for $a,b\in \Q$ with 
$$T_{\theta}:=\left[\begin{smallmatrix} 1 & 2^{-1}\tt \\ 2^{-1}\tt & \ttN \end{smallmatrix}\right]\in \cQ. 
$$
We record the following easily confirmed formulas for later use: 
\begin{align}
\tr(T_\theta^{\dagger} T_\theta)&={-D}/{2}, \qquad  \det(T_\theta)={-D}/{4}.  
 \label{Ttheta-f1}
\end{align}
For $\tau \in E$, let $I_{\theta}(\tau)\in {\bf M}_2(\Q)$ be the matrix representing the regular representation of $E$, i.e., $[\tau,\tau \theta]=[1,\theta]I_{\theta}(\tau)$, or explicitly
$I_{\theta}(a+b\theta)=\left[\begin{smallmatrix} a & -b \ttN \\ b & a+\tt b \end{smallmatrix} \right]$ for $a,b\in \Q$.
Set 
\begin{align}
\sigma:=\left[\begin{smallmatrix} 1 & \tt \\ 0 & -1 \end{smallmatrix} \right].
\label{Def-sigma}
\end{align}
Then, $\sigma^2=1_2$ and
\begin{align}
{}^t I_{\theta}(\tau)T_\theta I_\theta(\tau)=\nr_{E/\Q}(\tau)\, T_\theta, \quad \sigma I_\theta(\tau)\sigma=I_\theta(\bar \tau)=I_{\theta}(\tau)^\dagger. 
\label{SigmaItau}
\end{align}
Let ${\bf GO}_{T_\theta}$ be the closed $\Q$-subgroup of ${\bf GL}_2$ defined by the relation ${}^t h  T_\theta h=\pm \det(h)\,T_\theta$. The Zariski identity connected component ${\bf GO}_{T_\theta}^{\circ}=\{h \in {\bf GL}_2\mid {}^t h T_\theta h=\det (h)T_\theta\}$ of ${\bf GO}_{T_\theta}$ is isomorphic to ${\rm Res}_{E/\Q}{\bf GL}_1$ by the map $I_{\theta}$. Since $\det(\sigma)=-1$ and ${}^t \sigma T_\theta \sigma=T_\theta$, we have 
\begin{align}
{\bf GO}_{T_\theta}^\circ(\Q)=I_{\theta}(E^\times), \quad {\bf GO}_{T_\theta}(\Q)=I_{\theta}(E^\times)\cup \sigma I_{\theta}(E^\times).
 \label{GOpoints}
\end{align}

Let $\sG^\#$ be a $\Z$-group scheme, whose $R$-points for any commutative ring $R$ is given as 
$$
\sG^\#(R)=\{h \in {\bf GL}_2(\cO_{E}\otimes_\Z R)\mid \det(h)\in R^\times\},
$$
and define an embedding $\iota_{\theta}: \sG^\# \longrightarrow \sG$ of group schemes as in \cite[\S2.1]{KugaTsuzuki} by the symplectic form $
\langle x,y\rangle_{D}:=\tr_{E/\Q}(\frac{-1}{\sqrt{D}} \det\left[\begin{smallmatrix} x_1 & y_1 \\ x_2 & y_2 \end{smallmatrix}\right])$ on the $\Z$-lattice $W:=\cO_{E}^{\, 2}$ with the $\Z$-basis 
\begin{align}
v_{1}^{+}:=\left[\begin{smallmatrix} 1 \\ 0 \end{smallmatrix} \right], \quad 
v_{2}^{+}:=\left[\begin{smallmatrix} \theta \\ 0 \end{smallmatrix} \right], \quad 
v_{1}^{-}:=\left[\begin{smallmatrix} 0 \\ -\bar \theta  \end{smallmatrix} \right], \quad 
v_{2}^{-}:=\left[\begin{smallmatrix} 0  \\ 1 \end{smallmatrix} \right],
 \label{def-sympBasis}
\end{align} 
so that $\nu(\iota_\theta(h))=\det h\,(h\in \sG^\#)$. Note that $\langle v_i^{+},v_j^{-}\rangle_D=\delta_{ij}$ and $\langle v_{i}^{+},v_j^{+}\rangle_D=\langle v_i^{-},v_j^{-}\rangle_D=0$ for all $i,j$. Set $\sZ^\#:=\{a 1_2\mid a\in {\bf GL}_1 \}$, which is a central subgroup of $\sG^\#$. Let $\sB^\#=\{\left[\begin{smallmatrix} * & * \\ 0 & * \end{smallmatrix}\right]\}\cap \sG^\#$ denote the Borel subgroup of $\sG^\#$ formed by the upper-triangular matrices in $\sG^\#$. We have the following formulas, which, combined with the Bruhat decomposition $\sG^\#=\sB^\#\cup \sB^\# \left[\begin{smallmatrix} 0 & 1 \\ 1 & 0\end{smallmatrix}\right] \sB^\#$, yields a practically useful description of the map $\iota_{\theta}$: 
\begin{align}
\iota_{\theta}\left(\left[\begin{smallmatrix} \tau & 0 \\ 0 & \tau^{-1}a\end{smallmatrix}\right] \right)& =\sm(I_{\theta}(\tau),a), \quad \tau \in E^\times,\,a\in \Q^\times, 
 \label{Def-iota1}
\\
\iota_{\theta} \left(\left[\begin{smallmatrix} 1 & \beta \\ 0 & 1 \end{smallmatrix}\right] \right)&=\sn(X_\beta), \quad \beta \in E, 
 \label{Def-iotaf2}
\\
\iota_\theta \left(\left[\begin{smallmatrix} 0 & 1 \\ 1 & 0\end{smallmatrix}\right] \right)&=\left[\begin{smallmatrix}
0 & 0 & -\tt & 1 \\
0 & 0 & 1 & 0 \\
0 &1& 0& 0 \\
1 & \tt& 0 &0 \end{smallmatrix} \right],
 \label{Def-iotaf3}
\end{align}
where 
\begin{align}
X_\beta:= \left[\begin{smallmatrix} b_1 & b_2 \\ b_2 & b_3 \end{smallmatrix} \right]\in \sV(\Q) \quad \text{for $\beta=b_2+b_3\,\theta\,(b_2,b_3\in \Q)$ with $b_1:=-b_2\tt-b_3 \ttN$}. 
 \label{Def-Xbeta}
\end{align} 
Note that $\beta \mapsto X_\beta$ is a $\Q$-linear isomorphism from $E$ onto the space 
$$\sV^{T_\theta}(\Q):=\{X\in \sV(\Q)\mid \tr(T_\theta X)=0\},$$ 
which is the orthogonal space of $T_\theta^\dagger$ with respect to \eqref{P2-f0}. We also have the relation 
$ -\tr(X_\alpha X_\beta^\dagger)=\tr_{E/\Q}(\alpha \bar \beta)$ for $\alpha,\beta \in E$.

\subsection{Exceptional isomorphism} \label{sec: ExIso}
For use in \S\ref{sec:JJw1} and \S\ref{sec:JJTtw2}, let us recall the exceptional isomorphism between $\sZ\bsl \sG\,(={\bf PGSp}_2)$ and a special orthogonal group of degree $5$. Recall $W=\cO_E^2\cong \Z^4$. Since $W \ni x \mapsto \langle x,\bullet \rangle_{D} \in W^\vee$ is a $\Z$-linear isomorphism, we obtain a symplectic form on $W^\vee$ by transport of structure, which yields an element $\omega$ of $\wedge^2 W$. 
Set $\Omega=2^{-1}\omega\wedge \omega=v_1^{+}\wedge v_2^{+}\wedge v_1^{-}\wedge v_2^{-}$. Since $\wedge^4 W=\Z \Omega$, a symmetric $\Z$-bilinear form $B:\wedge ^2 W \times \wedge^2 W\rightarrow \Z$ is defined by the relation $
\xi\wedge \xi'=B(\xi,\xi')\,\Omega$ ($\xi,\xi'\in \wedge^2 W$). Note that $\omega=v_1^{+}\wedge v_{2}^{+}+v_{2}^{+}\wedge v_2^{-}$ and $\Omega=v_{1}^{+}\wedge v_{2}^{+}\wedge v_{1}^{-}\wedge v_{2}^{-}$. If we set
\begin{align*}
\xi_1&=v_1^{+}\wedge v_2^{+}, \quad \xi_2=v_1^{+}\wedge v_2^{-}, \quad \xi_3=v_1^{+}\wedge v_1^{-}-v_2^{+}\wedge v_2^{-}, \\
\xi_4&=v_2^{+}\wedge v_1^{-}, \quad \xi_5=v_1^{-}\wedge v_2^{-},
\end{align*}
then $\{\xi_1,\xi_2,\xi_3,\xi_4,\xi_5,\omega\}$ is a $\Z$-basis of $\wedge^{2}W$ such that 
$$
B(\xi_i,\xi_{j})=\delta_{i,6-j}\,(i\not=3), \quad B(\xi_3,\xi_3)=2, \quad B(\omega,\omega)=-2,$$
and $B(\xi_i,\omega)=0\,(\forall i)$, which means that ${\rm SO}({\mathsf X})$ is identified with the matrix group
$${\bf SO}(Q):=\{g\in {\bf SL}_4\mid {}^t gQ g=Q\} \quad \text{with} \,  
Q=\left[\begin{smallmatrix} {} & {} & {} & {} & 1 \\{} & {} & {} & {1} & {} \\
{} & {} & {2} & {} & {} \\
{} & {1} & {} & {} & {} \\ 
{1} & {} & {} & {} & {}
\end{smallmatrix} \right].
$$
For any $g\in \sG(R)$, define $\sr(g)\in \End(\wedge^2 W \otimes_\Z R)$ by $\sr(g)=\nu(g)^{-1}\,(\wedge^{2}g)$.
\begin{comment}
Then, it is easy to check that $\sr(g)$ for $g\in \sG(R)$ preserves ${\mathsf X}(R):={\mathsf X}\otimes_\Z R$, that $\sr(g)$ belongs to the $R$-points of the special orthogonal group ${\rm SO}({\mathsf X})$ of $({\mathsf X},B)$, and that the kernel of the homomorphism $\sr:\sG(R)\rightarrow {\rm SO}({\mathsf X})$ is $\sZ(R)$. 
\end{comment}
Then, the map $\sr$ yields an isomorphism of $\Q$-groups from $\sZ\bsl \sG$ to ${\rm SO}({\mathsf X})$. Define a $\Q$-linear injection $\sV \hookrightarrow {\mathsf X}$ by 
\begin{align}
\sV \ni T=\left[\begin{smallmatrix} x & y \\ y & z \end{smallmatrix} \right] \mapsto [T]:=x\xi_{2}-y\xi_3-z \xi_4\,\in {\mathsf X}.
 \label{def-xiT}
\end{align}
A direct computation shows $B([S],[T])=-\tr(ST^\dagger)$ for $S,T\in \sV(R)$, and the formulas
\begin{align}
\begin{cases}
\sr(\sm(A,\lambda))\,\xi_1&=\lambda^{-1}\det(A)\,\xi_1, \\
\sr(\sn(S))\,\xi_1&=\xi_1,  
\end{cases}
 \qquad 
\begin{cases}
\sr(\sm(A,\lambda))\,[T]&=[T\ss({}^t A)],  \\
\sr(\sn(S))\,[T]&=\tr(ST^\dagger)\,\xi_1+[T]
\end{cases}
 \label{ExIso-f0}
\end{align}
for $\sm(A,\lambda)\in \sM(R)$, $S,T\in \sV(R)$, and that $h_0=
\sr(\iota_\theta(\left[\begin{smallmatrix} 0 &  1 \\ 1 & 0 \end{smallmatrix} \right]))$ satisfies
\begin{align}h_0 \xi_i=\xi_{6-i}\,(i=1,5), \quad  
[h_0\xi_2,h_0\xi_3,h_0\xi_4]=[\xi_2,\xi_3,\xi_4] 
\left[\begin{smallmatrix} 1 & 2\tt & -\tt^2\\ 0 & 1 & -\tt \\ 0 & 0 & 1 \end{smallmatrix} \right].
\label{ExIso-f1}
\end{align}
\begin{lem} \label{ExIso-L1}
 For any field extension $K/\Q$, 
$$\iota_\theta(\sG^\#(K))=\{g\in \sG(K)\mid \sr(g)\,[T_\theta^\dagger]=[T_\theta^\dagger]\}.$$ 
\end{lem}
\begin{proof} Let $H$ denote the stabilizer of $[T_\theta^\dagger]$ in ${\rm SO}(\mathsf X)$. The second set of formulas in \eqref{ExIso-f0} combined with \eqref{Def-iota1} and \eqref{Def-iotaf2} shows that $\sr(\iota_\theta(\sB^\#))\subset H$. By \eqref{ExIso-f1}, we get $h_0 [T_\theta^\dagger]=[T_\theta^\dagger]$. By $\sG^\#=\sB^\# \cup \sB^\#\left[\begin{smallmatrix} 0 &  1 \\ 1 & 0 \end{smallmatrix} \right]\sB^\#$, we obtain the inclusion $\sr(\iota(\sG^\#))\subset H$. By the Bruhat decomposition of $H \cong {\rm SO}([T_\theta^\dagger]^\bot)$, the converse inclusion also holds.   \end{proof}

\subsubsection{Height functions} \label{sec:Ht} 
For $p<\infty$, let $\|\cdot\|_{p}$ be the height function on $\mathsf X(\Q_p)$, defined by $\|\sum_{j=1}^{5}x_j \xi_j\|_p=\max_{1\leq j \leq 5}(|x_j|_p)$. Let ${\mathsf X}^{\star}(\A_\fin)$ denote the finite adelization of the $\Q$-scheme $\sX^\star:={\mathsf X}-\{0\}$. Then, $\|\xi_p\|_p=1$ for almost all $p$ for any $\xi=(\xi_p)_{p<\infty}\in {\mathsf X}^\star(\A_\fin)$, which allows us to set $\|\xi\|_\fin:=\prod_{p<\infty}\|\xi_p\|_{p}$. For $T\in \sV^{\star}(\A_\fin)$ with $\sV^\star:=\sV-\{0\}$, set $\|T\|_\fin:=\|[T]\|_\fin$ with $[T]\in {\mathsf X}(\A_\fin)$ being as in \eqref{def-xiT}. %Fix a norm $\|\cdot\|_\infty$ on $\sV(\R)$ once and for all. 

\subsection{maximal compact subgroups}
The identity connected component of $\sG(\R)$ is $\sG(\R)^0=\{g\in \sG(\R)\mid \nu(g)>0\}$. Set $\bK_\infty:=\sG(\R)^0\cap {\bf O}(4)$, which is a maximal compact subgroup of $\sG(\R)^0$ given as 
\begin{align*}
\bK_\infty=\{\left[\begin{smallmatrix} A & B\\ -B & A \end{smallmatrix} \right]\mid A,B\in {\bf M}_2(\R),\,A+iB\in {\bf U}(2)\}. 
\end{align*}
Note that $\bK_\infty \subset {\bf Sp}_2(\R)$. The action of $\sG(\R)^0$ on the Siegel upper-half space $\fh_2$, denoted by $(g,Z)\mapsto g\langle Z\rangle$, and the canonical automorphy factor $J(g,Z)$ are defined by  
$$
g\langle Z\rangle:=(AZ+B)(CZ+D)^{-1}, \quad J(g,Z)=(\nu(g))^{-1}\det(CZ+D)
$$
for $g=\left[\begin{smallmatrix} A & B \\ C & D \end{smallmatrix}\right] \in \sG(\R)^0$ and $Z\in \fh_2$. The stabilizer of $i 1_2\in \fh_2$ in $\sG(\R)^0$ coincides with $\sZ(\R)\bK_\infty$. Define $h_0:\C^\times \rightarrow \sG(\R)$ as $h_0(z):=\left[\begin{smallmatrix} a 1_2 & b 1_2 \\-b 1_2 & a 1_2\end{smallmatrix}\right]$ for $z=a+ib\in \C^\times$; for $(p,q)\in \Z^2$, let $\fg^{p,q}$ denote the $z^p\bar z^q$-eigenspace of the representation ${\rm Ad}(h_0(z))$ of $\C^\times$ on $\fg_\C:={\rm Lie}(\sG(\R)^0)\otimes_\R\C$. Then, $\fg^{0,0}={\rm Lie}(\sZ(\R)\bK_\infty)$ and $\fg=\fg^{0,0}\oplus \fg^{1,0}\oplus \fg^{0,1}$. Set $\fp^{+}:=\fg^{1,0}$ and $\fp^{-}:=\fg^{0,1}$. Then, $\fp^{+}$ correpsonds to the holomorphic tangent space of $\fh_2$ by the map $\sG(\R)^0/\bZ(\R)\bK_\infty \cong \fh_2$. By \cite[Lemma 2]{KugaTsuzuki}, the image of ${\bf SU}(2)$, which is a maximal compact subgroup of $\sG^\#(\R)^0$, by the map $\iota_\theta$ is described as 
\begin{align}
\iota_{\theta}\left(\left[\begin{smallmatrix} a & -\bar b \\ b & \bar a \end{smallmatrix} \right] \right)=b_{\R}^\theta \left[\begin{smallmatrix} A & B\\ -B & A \end{smallmatrix}\right] (b_\R^\theta)^{-1} \quad \text{with $A:=\left[\begin{smallmatrix} \Re(a)& -\Im(a) \\ \Im(a) & \Re(a)\end{smallmatrix} \right], \,B:=\left[\begin{smallmatrix} \Im(b) & \Re(b) \\ \Re(b) & -\Im(b)\end{smallmatrix} \right]$}, 
\label{kAB-def}
\end{align}
where
\begin{align}
b_{\R}^\theta&:=\sm(A_\theta^{-1}, 2/\sqrt{|D|}), \quad A_\theta:=\left[\begin{smallmatrix}1 & \tt/2 \\ 0 &-\sqrt{|D|}/2\end{smallmatrix} \right] \in {\bf GL}_2(\R). \label{b-def}
\end{align} 
Note that ${}^t A_{\theta}A_\theta=T_\theta$ and $\nu(b_\R^\theta)>0$. For $p<\infty$, set $\bK_p:=\sG(\Z_p)$.

\subsection{Haar measures} 
For any locally compact unimodular group $G$ with a Haar measure $\eta_{G}$ and its unimodular closed subgroup $H$ with a Haar measure $\eta_{H}$, the quotient $H\bsl G$ (or $G/H$) is always endowed with the quotient measure of $\eta_{G}$ by $\eta_{H}$. As in \cite[\S2.4]{KugaTsuzuki}, we fix $\eta_{G}$ as follows for several $G$ relevant to us. 

Let $\eta_{\R}$ and $\eta_{\C}$ be the Lebesgue measures on $\R$ and on $\C\cong \R^2$, respectively. We define $\d\eta_{\R^\times}(x)=|x|_\R^{-1}\d\eta_{\R}(x)$ and $\d \eta_{\C^\times}(\tau)=|\tau|_{\C}^{-1}\d\eta_{\C}(\tau)$. For a place $p\leq \infty$ of $\Q$, set $E_p:=E\otimes_\Q \Q_p$, and $\cO_{E,p}:=\cO_E\otimes_\Z \Z_p$ if $p<\infty$. For each $p<\infty$, $\eta_{E_p^\times}$ (resp. $\eta_{\Q_p^\times}$) is fixed by demanding that $\eta_{E_p^\times}(\cO_{E,p}^\times)=1$ (resp. $\eta_{\Q_p^\times}(\Z_p^\times)=1$). Note that our measures $\eta_{E_\infty}$ and $\eta_{E_\infty^\times}$ are half of the ones in \cite{Weil}. For idele groups, we define $\eta_{\A^\times}=\prod_{p\leq \infty} \eta_{\Q_p^\times}$, and $\eta_{\A_{E}^\times}=\prod_{p\leq \infty}\eta_{E_p^\times}$ viewing $\A_E^\times$ as the restricted product of $E_p^\times$. We have that $\A_{E}^\times/E^\times \A^\times$ is a compact group, whose volume is computed by \cite[Chap. V \S4 Proposition 9]{Weil}; indeed, we have $\vol(\A_E^\times/E^\times {\underline {\R_{>0}}})=\pi h_D/w_D$ with ${\underline {\R_{>0}}}=\{t^{1/2}\in E_\infty^\times\mid t>0\}$, $\vol(\A^\times/\Q^\times \R_{>0})=1$ and   
\begin{align}
\vol(\A_{E}^\times /E^\times \A^\times)=\pi h_{D}/2w_D,
\label{HaarEQ-f0}
\end{align} 
where $h_{D}$ is the class number of $E$, which is the order of the finite abelian group 
$$ {\rm Cl}(E):=\A_{E}^\times/ E^\times E_\infty^\times \widehat \cO_{E}^\times,$$
and $w_{D}$ the number of roots of unity in $E$. By $\sZ\cong {\bf GL}_1$, the measures $\eta_{\A^\times}$ and $\eta_{\Q_p^\times}$ are transferred to groups $\sZ(\A)$ and $\sZ(\Q_p)\,(p\leq \infty)$, respectively. Similar remark is applied to $\sZ^\#\cong {\bf GL}_1$. On adele spaces $\A$, $\A_{E}$ and $\sV(\A)$, we fix $\eta_{\A}$, $\eta_{\A_{E}}$ and $\eta_{\sV(\A)}$ to be the self-dual Haar measure with respect to the bi-characters $(x,y)\mapsto \psi(xy)$, $(\alpha,\beta)\mapsto \psi(\tr_{E/\Q}(\alpha \bar \beta))$ and $(X,Y)\mapsto \psi(\tr(XY^\dagger))$, respectively. According to $\sV(\A)=\A T_\theta^\dagger\oplus \sV^{T_\theta}(\A)$, we have $\d\eta_{\sV(\A)}(X)=\d\eta_\A(a)\otimes \d \eta_{\A_{E}}(\beta)$ for $X=a T_\theta^\dagger +X_\beta\,(a \in \A,\,\beta\in \A_{E})$. Through the identification $\sV(\A)\ni X\mapsto \sn(X)\in \sN(\A)$ and $\beta\ni \A_{E}\mapsto \left[\begin{smallmatrix}1 & \beta \\ 0 & 1 \end{smallmatrix} \right]\in \sN^\#(\A)$, the groups $\sN(\A)$ and $\sN^\#(\A)$ acquire Haar measures. Thus, 
\begin{align}
\int_{\sN(\A)}f(n)\, \d n=\int_{\A} \int_{\sN^\#(\A)}f(\sn(x T_\theta)\,\iota_{\theta}(n^\#))\,\d x\,\d n^\#
 \label{UnptInt}
\end{align}
for $f\in L^1(\sN(\A))$, and $\vol(\sN(\Q)\bsl \sN(\A))=\vol(\sN^\#(\Q)\bsl \sN^\#(\A))=1$. Note that $\eta_{\A_E}$ is $(\frac{\sqrt{|D|}}{2})^{-1} \times \prod_{p\leq \infty} \eta_{E_{p}}$ due to the formulas $(\prod_{p}\eta_{E_p})(\A_E/E)=|D|^{1/2}/2$ (\cite[Chap V \S4 Proposition 7]{Weil}) and $\eta_{\A_E}(\A_E/E)=1$.
\subsubsection{} \label{sec:HaarDGS}
Let $p\leq \infty$. Any element $g^\#\in \sG^\#(\Q_p)$ is written as 
\begin{align}
g^\#=\left[\begin{smallmatrix} 1 & \beta \\ 0 & 1\end{smallmatrix}\right]
\left[\begin{smallmatrix} \tau & 0 \\ 0 & \bar \tau a \end{smallmatrix}\right] \, k^\#
\label{Def-HaarMeasSGP-f0}
\end{align}
with $a \in \Q_{p}^\times$, $\tau \in E_p^\times$, $\beta\in E_p$ and $k^\# \in \bK_p^\#$. Then, our Haar measure on $\sG^\#(\Q_p)$ is symbolically defined as 
\begin{align}
\d\eta_{\sG^\#(\Q_p)}(g_p^\#)=|a|_p^{2} \, \d\eta_{E_p}(\beta)\, \d\eta_{\Q_p^\times}(a)
\, \d\eta_{E_p^\times}(\tau)\,\d\eta_{\bK_p^\#}(k^\#), 
 \label{Def-HaarMeasSGP}
\end{align}
where $\d\eta_{\bK_p^\#}(k^\#)$ is the Haar measure on $\bK_p^\#$ with volume $1$. For $p<\infty$, the Haar measure $\eta_{\sG^\#(\Q_p)}$ is the one with $\vol(\bK_p^\#)=1$. On the adele group $\sG^\#(\A)$, we use the product measure of $\eta_{\sG^\#(\Q_p)}$ ($p\leq \infty$). Then, 
\begin{align}
\int_{\sG^\#(\A)}f(g^\#)\,\d g^\#=
\tfrac{\sqrt{|D|}}{2} \int_{\sN^\#(\A)\times \A_E^\times \times \A^\times \times \bK^\#}
f&\left(n^\# \left[\begin{smallmatrix} \tau & 0 \\ 0 & a \bar \tau \end{smallmatrix}\right] \, k^\# \right)|a|_\A^{2} \,\d n^\#\,\d\eta_{\A^\times}(a)\,\d\eta_{\A_E^\times}(\tau)\,\d k^\#,
\label{HaarMeasSGP-Adele}
\end{align}
where $\d k^\#$ is the normalized Haar measure on $\bK^\#:=\prod_{p\leq \infty} \bK_p^\#$. 
\subsubsection{} 
We normalize $\eta_{\bK_\infty}$ so that $\vol(\bK_\infty)=1$. Then, we normalize $\eta_{{\bf Sp}_2(\R))}$ in such a way that the quotient $\eta_{{\bf Sp}_2(\R)}/\eta_{\bK_\infty}$ corresponds to the measure $(\det Z)^{-3}\d X\d Y$ on the Siegel upper-half space $\fh_2\cong {\bf Sp}_2(\R)/\bK_\infty$. By $\sG(\R)^0=\sZ(\R)^0\,{\bf Sp}_2(\R) \cong \R_{>0}\times{\bf Sp}_2(\R)$, $\eta_{\sG(\R)}$ is defined so that $\eta_{\sG(\R)}|\sG(\R)^0$ is $\eta_{\sG(\R)^0}:=\eta_{\R_{>0}}\otimes \eta_{{\bf Sp}_2(\R)}$. For $p<\infty$, we fix $\eta_{\sG(\Q_p)}$ by $\vol(\bK_p)=1$. Then, $\eta_{\sG(\A)}$ is defined to be the restricted product of $\eta_{\sG(\Q_p)}\,(p\leq \infty)$. By this Haar measure, we construct the $L^2$-space of $\sG(\Q)\sZ(\A)\bsl \sG(\A)$, which contains the space $\tilde S_l(\bK_0(N))$ in such a way that for a classical $\Phi\in S_l(N)$, the $L^2$-norm $\langle \tilde \Phi \mid \tilde \Phi \rangle_{L^2}$ of its adelic lift $\tilde \Phi \in \Tilde S_l(\bK_0(N))$ (see \eqref{AdeleLift}) is the same as the Petersson norm $\langle \Phi \mid \Phi \rangle$. 

For $t\in \R$, set 
\begin{align*}
a_{\infty}^{(t)}:= b_{\R}^{\theta} \left[\begin{smallmatrix} (\ch t)1_2 & (\sh t)\,1_2 \\ (\sh t)\,1_2 & (\ch t)\,1_2 \end{smallmatrix} \right] (b_{\R}^{\theta})^{-1}, 
\end{align*}
which is an element of $\sG(\R)$ representing the $\R$-linear symplectic automorphism of $W_\R:=W\otimes_\Z\R$ (see \S\ref{sec:GO}) given by 
$$
\left[\begin{smallmatrix} x_1 \\ x_2 \end{smallmatrix} \right]\longmapsto \left[\begin{smallmatrix} (\ch t)x_1+(i\sh t)\bar x_2 \\ (i\sh t)\bar x_1+(\ch t)x_2 \end{smallmatrix} \right]
$$
with respect to the basis \eqref{def-sympBasis}. Set $A:=\{a_\infty^{(t)}\mid t\in \R\}$ and $A^{+}=\{a_\infty^{(t)}\mid t\in \R_{\geq 0}\}$. 
\begin{lem}\label{RSInt-L1}
 We have 
$\sG(\R)=\iota_\theta(\sG^\#(\R))A^{+}b_\R^{\theta}\bK_\infty$. The Haar measure on $\sG(\R)$ is decomposed as
$$
c_{\sG}\,\d\eta_{\sG^\#(\R)}(h)\,(\sh 2t)^{2} \ch 2t\, \d t\, \d\eta_{\bK_\infty}(k)
$$ 
with $c_{\sG}=2^5 \pi^{1/2}\Gamma(3/2)^{-1}$, where $\eta_{\sG^\#}$ and $\eta_{\bK_\infty}$ are Haar measures on $\sG^\#(\R)$ and on $\bK_\infty$ fixed in \S\ref{sec:HaarDGS}, and $\d t$ is the Lebesgue measure on $\R$. 
\end{lem}
\begin{proof} Let $J$ and $C$ be the $\R$-linear automorhism of $W_\R$ defined as
$$
J\left(\left[\begin{smallmatrix} x_1 \\ x_2 \end{smallmatrix}\right] \right)=
\left[\begin{smallmatrix} ix_1 \\ ix_2 \end{smallmatrix} \right]
, \quad C\left(\left[\begin{smallmatrix} x_1 \\ x_2 \end{smallmatrix} \right]\right)=\left[\begin{smallmatrix} -i\bar x_2 \\ ix_1 \end{smallmatrix} \right].
$$ 
Then, $J,C \in \sG(\R)$, $J^2=C^2=-{\rm Id}_{W_\R}$ and $JC=-CJ$. Thus, ${\rm Int}(J): g\longmapsto JgJ^{-1}$ and ${\rm Int}(C)g\longmapsto CgC^{-1}$ are mutually commuting involutions of $\sG(\R)$. The group $\iota_\theta(\sG^{\#}(\R))$ coincides with the fixed point set of ${\rm Int}(J)$, and ${\rm Int}(C)$ is the Cartan involution corresponding to $b_\R^\theta \bK_\infty(b_\R^\theta)^{-1}$. Indeed, $\iota_\theta(\sG^{\#}(\R))\bsl \sG(\R)$ is a reductive symmetric space of split rank $1$. Since ${\rm Int}(C)(a_\infty^{(t)})={\rm Int}(J)(a_\infty^{(t)})=a_\infty^{(-t)}$ for $t\in \R$, a general result (see \cite[Theorem 2.4]{HeckmanSchlichtkrull}) can be applied to get $\sG(\R)=\iota_\theta(\sG^\#(\R))Ab_{\R}^\theta\bK_\infty (b_\R^\theta)^{-1}$. The assertion on Haar measures follows from \cite[Theorem 2.5]{HeckmanSchlichtkrull}; to determine $c_{\sG}$ explicitly, an argument is in order. We use the $\bK_\infty$-invariant inner product $2^{-1}\tr(X\,{}^t Y)$ on ${\mathfrak s}:={\mathfrak {sp}}_2(\R)$. For any connected closed subgroup $L\subset \sG(\R)^0$, we identify ${\mathfrak l}:={\rm Lie}(L)$ with its dual ${\mathfrak l}^{*}$ by this inner-product. We choose an invariant volume form $\Omega_L$ on $L$ so that the value $\Omega_L(1_4)\in \wedge^{\dim {\mathfrak l}}{\mathfrak l}^*\cong\wedge ^{\dim(\mathfrak l)}{\mathfrak l}$ is $\pm X_1\wedge\cdots \wedge X_{\dim(\mathfrak l)}$ for any orthonormal basis $\{X_i\}_i$ of ${\mathfrak l}$. Let $|\Omega_L|$ denote the Haar measure on $L$ defined by $\Omega_L$. By using an orthonormal basis from the Iwasawa decomposition, we have that the measure $|\Omega_{{\bf Sp}_2(\R)}|/\eta_{\bK_\infty}$ on ${\bf Sp}_2(\R)/\bK_\infty\cong \fh_2$ coincides with $(\det Z)^{-3}\d X\d Y$. Set $G_1^\#:=\{g^\#\in \sG^\#(\R) \mid \det g^\#=1\}$, which coincides with all the elements \eqref{Def-HaarMeasSGP-f0} with $a>0$, $k^\#\in (\bK^\#_\infty)^0$ and $\tau=a^{-1/2}$. Then, $\sG^\#(\R)^0=\sZ^\#(\R)^0\times G_1^\#$, and $(\bK_\infty^\#)^0$ is a maximal compact subgroup of $G_1^\#$. Set $H:={\rm Int}((b_\R^\theta)^{-1})\circ \iota_{\theta}(G_1^\#)$, and let $M$ be the centralizer of $A$ in $K_H:=H\cap \bK_\infty$; it turns out that $M={\rm Int}((b_{\R}^\theta)^{-1})\circ \iota_{\theta}(T)$ with $T:=\{\left[\begin{smallmatrix} \tau & 0 \\ 0 & \bar \tau\end{smallmatrix} \right])\mid \tau \in \C^{1}\}$. A computation reveals the equality 
\begin{align}
|\Omega_{H}|/|\Omega_{K_H}|=\eta_{G_1^\#}/\eta_{\bK_\infty^\#}
\label{consRSInt-L1-f1}
\end{align}
between quotient measures on $H/K_H\cong G_1^\#/(\bK_\infty^\#)^0$, where 
$$\d\eta_{G_1^\#}(g^\#):=2\pi\,|a|^2\,\d\eta_{\C}(\beta)\d\eta_{\R_{>0}}(a)\d\eta_{\bK_\infty^\#}(k^\#)$$
 for $g^\#=\left[\begin{smallmatrix}1 & \beta \\ 0 & 1 \end{smallmatrix}\right]\left[\begin{smallmatrix}a^{-1/2} & 0 \\ 0 & a^{1/2} \end{smallmatrix}\right]k^\#$. Note that the factor $2\pi$ arises as $\vol(T)=\int_{T}|\Omega_T|$ ($T$ is identified with its image by ${\rm Int}((b_\R^\theta)^{-1})\circ \iota_\theta$). By $(\bK_\infty^\#)^\circ/T \cong {\bf S}^\#\subset W$ (see \S\ref{sec:SecArch}), we compare $\Omega_{K_H}$ and the Euclidean volume form on ${\bf S}^\#$ to obtain the relation 
\begin{align}
|\Omega_{K_H}|/|\Omega_{T}|=2\pi^{3/2}\Gamma(3/2)^{-1}\times 2\eta_{\bK_\infty^\#}/\eta_{T},
\label{RSInt-L1-f2}
\end{align}
where $\eta_T$ is the normalized Haar measure on $T$. 
Note that $\vol({\bf S}^\#)=2\pi^{3/2}\Gamma(3/2)^{-1}$, whereas $\int_{(\bK_\infty^\#)^0/T} \eta_{\bK_\infty^\#}/\eta_T=1/2$. Let $f:G_1^\#/T \times \R\to {\bf Sp}_2(\R)/\bK_{\infty}$ be the map $f(g^\#,t)=(b_{\R}^\theta)^{-1} \iota_{\theta}(g^\#)b_\R^\theta a(t)\bK_\infty$ with $a(t):=\left[\begin{smallmatrix} (\ch t)1_2 & (\sh t)\,1_2 \\ (\sh t)\,1_2 & (\ch t)\,1_2 \end{smallmatrix} \right]$, and $L_t:g\mapsto a(t)^{-1}g$ on ${\bf Sp}_2(\R)/\bK_\infty$. Since $L_t\circ f(1,t,1)=1_4\bK_\infty$, the tangent map $(dL_t)_{a(t)}\circ (d f)_{(1,t)}$ is viewed as a linear map from ${\mathfrak g}^{\#}_1\times \R \times {\mathfrak 
 k}$ to ${\mathfrak{sp}}_2(\R)/{\mathfrak k}$. By using the root space decomposition by ${\rm Ad}(a(t))$, we see that, when $t\not=0$, the tangent map is an isomorphism and 
\begin{align}
(L_t\circ f)^{*}(\Omega_{{\bf Sp}_2(\R)/\bK_\infty}(e))=2^{4}(\sh 2t)^2\ch 2t\,(\Omega_{G_1^\#/T}(e)\oplus dt,
\label{RSInt-L1-f3}
\end{align}
where $\Omega_{{\bf Sp}_2(\R)/\bK_\infty}$ is the volume form on ${\bf Sp}_2(\R)/\bK_\infty$ such that $\Omega_{{\bf Sp}_2(\R)}$ equals the wedge product of the pull-back of $\Omega_{{\bf Sp}_2(\R)/\bK_\infty}$ and $\Omega_{\bK_\infty}$. We are done by \eqref{consRSInt-L1-f1}, \eqref{RSInt-L1-f2} and \eqref{RSInt-L1-f3}, noting $\sG(\R)^0=\sZ(\R)^0\times {\bf Sp}_2(\R)$ and $\iota_{\theta}(\sZ^\#(\R)^0)=\sZ(\R)^0$. We also note $\vol(T)=2\pi$ so that $c_{\sG}=4\pi^{3/2}\Gamma(3/2)^{-1}\times 2^{4}\times \vol(T)^{-1}=2^5\pi^{1/2}\Gamma(3/2)^{-1}$.
\end{proof}
\subsection{Hecke algebra} \label{sec:Def-HeckeAlg}
Let $p$ be a prime number. Let $\cH_p:=\cH(\sG(\Q_p)\sslash \bK_p)$ the subspace of bi-$\bK_p$-invariant functions in $\cH(\sG(\Q_p))$. 
%For $c\in \Z$, let $\cH_{p}[c]$ denote the subspace of $f\in \cH_p$ supported in $\{g\in \sG(\Q_p)\mid \nu(g)\in p^{c}\Z_p^\times\}$. 
Set $D:=\{{\ud}=(d_1,d_2;d_0)\in \Z^3\mid d_1\geq d_2 \geq d_0/2\}$. For ${\underline d}\in D$, set $\kappa(\ud):=\max(d_1-d_2,2d_2-d_0)\in \Z_{\geq 0}$ and $\varpi(\ud):=\diag(p^{d_1}, p^{d_2}, p^{d_0-d_1}, p^{d_0-d_2})\in \sG(\Q_p)$. Due to the Cartan decomposition $\sG(\Q_p)=\sqcup_{\ud \in D}\bK_p \varpi(\ud)\bK_p$, the characteristic function of $\bK_p \varpi(\ud)\bK_p\,(\ud \in D)
$ form a $\C$-basis of $\cH_p$. 
%For $\kappa \in \Z_{>0}$, let $\cH_p^{\leq \kappa}$ denote the $\C$-subspace of $\cH_p$ generated by $f_{p,\ud}$ for $\ud\in D$ with $\kappa(\ud)\leq \kappa$. 
%Set $\cH_p^{\leq \kappa}[c]:=\cH_p^{\leq \kappa}\cap \cH_p[c]$. 
Let $\sr:\sZ\bsl \sG \rightarrow {\bf SO}(Q)$ be as in \S\ref{sec: ExIso}. The following lemma is used in \S\ref{sec:JJw1} and \S\ref{sec:JJTtw2}. 
\begin{lem} \label{Def-HeckeAlg-L}
Let $\ud \in D$. Then, for $k_1,k_2\in \bK_p$, 
$$
\sr(k_1\varpi(\ud)k_2)^{-1}\sX(\Z_p)\subset p^{-(d_1+d_2-d_0)}\sX(\Z_p)\subset p^{-2\kappa(\ud)}\sX(\Z_p).
$$ 
\end{lem}
\begin{proof} This follows from the fact that $$\sr(\varpi(\ud))=\diag(p^{d_1+d_2-d_0}, p^{d_1-d_2},1, p^{d_2-d_1}, p^{-d_1-d_2+d_0})\in {\bf SO}(Q)$$ and that $\sr(k)\,(k\in \bK_p)$ preserves the lattice $\sX(\Z_p)$.   
\end{proof}
Let $S$ be a finite set of prime numbers. Set $\cH_S:=\otimes_{p\in S} \cH_p$ identified with a subspace of $\cH(G_S)$, where $G_S:=\prod_{p\in S}\sG(\Q_p)$. For $f_S=\otimes_{p\in S}f_p$ with $f_p\in \cH_p$, we define 
\begin{align}
\fa(f_S)&:=\prod_{p\in S} p^{2\kappa_p}
\label{Def-HeckeAlg-f0}
\end{align}
with $\kappa_p\in \Z_{>0}$ being the smallest integer such that ${\rm supp}(f_p) \subset \cup_{\kappa(\ud)\leq \kappa} \bK_p \varpi(\ud)\bK_p$. We use the $L^1$-norm $\|f_S\|_1:=\int_{G_S}|f_S(h_S)|\d h_S$ on $\cH_S$. The double coset $\bK_p\varpi(\ud)\bK_p$ is written 
as a disjoint union of right $\bK_p$-cosets as
\begin{align}
\bK_{p}\varpi(\ud)\bK_p=\bigsqcup_{i=1}^{n}
\bK_p \sm(\gamma_i;\lambda_i\det \gamma_i)\sn(X_i), \quad n:=\#(\bK_p\bsl \bK_p \varpi(\ud)\bK_p)
\label{Def-HeckeAlg-f1}
\end{align}
with $\gamma_i \in {\bf GL}_2(\Q_p)$, $\lambda_i \in \Q_p^\times$ and $X_i \in \sV(\Q_p)$. Set $\rho(\ud):=2d_1+d_2-\tfrac{3}{2}d_0$ for $\ud\in D$. Then,  
\begin{align}
&p^{2\rho(\ud)-13}\leq n=\#(\bK_p\cap \varpi(\ud)^{-1}\bK_p \varpi(\ud)\bsl \bK_p)\leq p^{2\rho(\ud)+13}, 
 \label{Def-HeckeAlg-f2}\\
&\sV(\Z_p)\ss(\gamma_i)\subset  p^{-2\kappa(\ud)}\sV(\Z_p),
 \label{Def-HeckeAlg-f3}\\
&p^{-2\kappa(\ud)}\leq |\lambda_i|_p^{-1}\leq p^{2\kappa(\ud)}.
 \label{Def-HeckeAlg-f5}
\end{align}
(For the second inequality in \eqref{Def-HeckeAlg-f2}, we refer to \cite[Lemma 2.13]{ShinTemplier}. We include a proof of the first inequality. Let $I$ denote the Iwahori subgroup of $\bK_p$. By the Iwahori factorization, $\varpi(\ud)^{-1}I \varpi(\ud) \cap \bK_p=\varpi(\ud)^{-1}I\varpi(\ud)\cap I$ and $\#(I\cap \varpi(\ud)^{-1} I\varpi(\ud)\bsl  I)=p^{2\rho(\ud)}$. Hence $p^{2\rho(\ud)}=\#(\varpi(\ud)^{-1}I\varpi(\ud)\cap \bK_p\bsl I_p)\leq \#(\varpi(\ud)^{-1}I\varpi(\ud)\cap \bK_p\bsl \bK_p)\leq [\bK_p:I]\,n\leq p^{13}n$. As for \eqref{Def-HeckeAlg-f5}, first by looking at similitude norms, $|\lambda_i \det \gamma_i|_p=p^{-d_0}$. By considering the $2\times 2$ minors, we have $|\det \gamma_i|_p\leq p^{d_1+d_2-2d_0}$ and $|\lambda_i^2 \det \gamma_i|_p\leq p^{d_1+d_2-2d_0}$. From these, we get \eqref{Def-HeckeAlg-f5} easily. ) Due to the inequality $3\kappa(\ud)\leq 2\rho(\ud)$ for $\ud \in D$ and \eqref{Def-HeckeAlg-f2}, if $f_S$ is the characteristic function of a single $\bK_S$-double coset in $\sG(\Q_S)$, we get \begin{align}
\fa(f_S)^{3/2}\prod_{p\in S} p^{-13}
\leq \|f_S\|_1
%\leq\|f_S\|_\infty \fa(f_S)^{7/2} \prod_{p\in S}p^{13}, \quad f_S\in \cH_S. 
\label{Def-HeckeAlg-f4}
\end{align}

\subsection{Bessel periods of automorphic forms} 
Let $\psi=\otimes_{p\leq \infty}\psi_p$ be a unique character of $\A$ such that $\psi|\Q$ is trivial and $\psi_\infty(x)=e^{2\pi i x}$ for $x\in \R$. For $T\in \sV(\Q)$, define a character $\psi_{T}:\sN(\A) \longrightarrow \C^1$ by 
\begin{align}
\psi_{T}(\sn(X)=\psi(\tr(TX)), \qquad X\in \sV(\A).
 \label{def-psiT}
\end{align}
The group $\sM$ acts rationally on $\sN$ by conjugation, hence on $\Hom(\sN,{\mathbb G}_a)$ by transport of structure. Let $\sM^{T_\theta}$ denotes the stabilizer of the map $\sn(X)\mapsto \tr(T_\theta X)$. Then, $(\sM^{T_\theta})^\circ(\Q)=\{\sm_\theta(\tau)\mid \tau \in E^\times\}$, where we set\begin{align*}
\sm_{\theta}(\tau)&:=\sm(I_{\theta}(\tau),\nr_{E/\Q}(\tau)), \quad \tau \in E^\times. 
\end{align*}
The group $(\sM^{T_\theta})^\circ \sN$ is called the Bessel subgroup for $T_\theta$. Any idele class character $\Lambda:\A_{E}^\times /E^\times\rightarrow \C^{1}$ gives rise to an automorphic character of the Bessel subgroup as 
$$
(\sM^{T_\theta})^\circ(\A)\sN(\A) \ni \sm_\theta(\tau)\sn(X) \longmapsto \Lambda(\tau)\psi_{T_\theta}(\sn(X))\in \C^1. 
$$
Note that $\sZ(\A) \subset (\sM^{T_\theta})^{\circ}(\A)$. Let $F$ be an automorphic form on $\sG(\A)$ with trivial central character, i.e., $F(zg)=F(g)$ for all $(z,g)\in \sZ(\A)\times \sG(\A)$. Suppose that $\Lambda| \A^\times$ is trivial. Then, the global $(T_\theta,\Lambda)$-Bessel period of $F$ is defined by 
\begin{align}
B^{T_\theta,\Lambda}(F:g):=\int_{\A_{E}^\times/\A^\times E^\times} \Lambda(\tau)^{-1}\biggl\{\int_{\sN(\Q)\bsl \sN(\A)} F(n \sm_{\theta}(\tau)g)\,\psi_{T_\theta}(n)^{-1}\,\d n \biggr\}\,\d^\times \tau, \quad g\in \sG(\A).
\label{Def-BesselPer}
\end{align}
Note that the quotients $\A_{E}^\times/\A^\times E^\times$ and $\sN(\Q)\bsl \sN(\A)$ are compact.

\subsection{Godement section} \label{sec:JaqS}
Let $p\leq \infty$. Given characters $\Lambda_p:E_p^\times/\Q_p^\times \rightarrow \C^1$ and $\mu_p:\Q_p^\times \rightarrow \C^1$, let ${\mathscr V}^\#(s,\Lambda_p,\mu_p)$ be the smooth representation of $\sG(\Q_p)$ induced from the character 
$
\left[\begin{smallmatrix}\tau & \beta \\ 0& a \bar \tau\end{smallmatrix}\right] \longmapsto \Lambda_p^{-1}(\tau)\mu_p^{-1}(a)|a|_p^{-(s+1)}
$ of $\sB^\#(\Q_p)$. Let $\cS(E_p^2)$ denote the Schwartz-Bruhat space on $E_p^2$. For $\phi\in\cS(E_p^2)$, set
\begin{align}
f_{\phi}^{(s,\Lambda_p,\mu_p)}(g^\#):= \mu_p(\det g^\#)|\det g^\#|_p^{s+1}
\int_{E_p^\times}\phi\left(\left[\begin{smallmatrix}0 & 1\end{smallmatrix}\right] \left[\begin{smallmatrix}\tau & 0 \\ 0 & \bar \tau \end{smallmatrix}\right] g^{\#}\right)(\Lambda\mu_{E})_p(\tau)|\tau \bar \tau|_p^{s+1}\d\eta_{E_p^\times}(\tau) \label{DefJaqS}
\end{align}
for $s\in \C$, $g^\#\in \sG^\#(\Q_p)$, which is absolutely convergent on the region $\Re(s)>0$. When $p=\infty$, we suppose $\phi$ is $\bK_\infty^\#$-finite. By the local Tate's theory, the function $s\mapsto f_\phi^{(s,\Lambda_p,\mu_p)}(g^\#)$ is continued meromorphically to $\C$ in such a way that 
\begin{align}
\sf_{\phi}^{(s,\Lambda_p,\mu_p)}(g^\#):=L_p(s+1,\Lambda\mu_{E})^{-1}\,f_{\phi}^{(s,\Lambda_p,\mu_{p})}(g^\#)
\label{normsec}
\end{align}
is holomorphic on $\C$. As a function of $g^\#$, it belongs to $\cV^\#(s,\Lambda_p,\mu_p)$ for all regular points $s$. When both $\Lambda_p$ and $\mu_p$ are unramified and $\phi=\cchi_{\cO_{E,p}\oplus \cO_{E,p}}$, then $\sf_{\phi}^{(s,\Lambda_p,\mu_p)}$ yields a $\bK_p^\#$-invariant vector in ${\mathscr V}^\#(s,\Lambda_p,\mu_p)$ whose restriction to $\bK_p^\#$ is the constant $1$. 

Set $w_0:=\left[\begin{smallmatrix}0 & -1\\ 1& 0\end{smallmatrix}\right] \in \sG^\#(\Q_p)$. The standard intertwining operator 
$$
M(s):{\mathscr V}^\#(s,\Lambda_p,\mu_p)\longrightarrow {\mathscr V}^\#(-s,\Lambda_p^{-1},\mu_p^{-1})
$$
is defined on $\Re(s)>0$ as the absolutely convergent integral for $f\in \cV^\#(s, \Lambda_p,\mu_p)$:
\begin{align}
 M(s)f(g^\#)=\int_{E_p}f\left( 
w_0  
\left[\begin{smallmatrix}1 & \beta \\ 0& 1 \end{smallmatrix}\right]g^\#\right)\,\d\eta_{E_p}(\beta), \quad g^\#\in \sG^\#(\Q_p).  
\label{Def-LIntOpr}
\end{align}
The effect of $M(s)$ on the section $f_{\phi}^{(s,\Lambda_p,\mu_p)}$ is described  by the Fourier transform  
\begin{align}
\widehat \phi(x,y):=\int_{E_p^2} \phi(u,v)\,\psi_{E_p}\left(\tfrac{xv-yu}{\sqrt{D}}\right)^{-1}\,\d\eta_{E_p}(u)\,\d\eta_{E_p}(v), \quad (x,y)\in E_p^2
\label{JacS-f0}
\end{align}
as 
\begin{align}
f_{\widehat \phi}^{(-s,\Lambda_p^{-1},\mu_p^{-1})}(g^\#)
=|D|_p^{-s+1/2}\,c_p\,
\varepsilon_p(s,\Lambda\mu_E,\psi_{E_p})\frac{L_p(1-s,\Lambda^{-1}\mu_{E}^{-1})}{L_p(s,\Lambda\mu_{E})}\times M(s)f_{\phi}^{(s,\Lambda_p,\mu_p)}(g^\#).
\label{JaqS-f1}
\end{align}
with $c_p=1$ if $p<\infty$ and $c_\infty=\frac{|D|}{2}$ if $p=\infty$, where $\varepsilon_p(s,\cdot ,\psi_{E_p})$ is Tate's local $\varepsilon$-factor by the additive character $\psi_{E_p}:=\psi_p \circ {\rm tr}_{E_p/\Q_p}$ and the associated self-dual measure on $E_p$ as usual.

\subsection{Section for the archimedean place}\label{sec:SecArch}
Recall that $\Lambda_\infty$ is trivial. Set ${\bf S}^\#:=(\sB^\#(\R)\cap (\bK_\infty^\#)^0)\bsl (\bK_\infty^\#)^0$; since $\sB^\#(\R)\cap (\bK_\infty^\#)^0=\{\left[\begin{smallmatrix} t & 0\\ 0 & t^{-1}\end{smallmatrix} \right]\mid t\in \C^1\}$ and $(\bK_\infty^\#)^0={\bf SU}(2)$, the manifold ${\bf S}^\#$ is diffeomorphic to the unit sphere in $W:=\{X\in {\bf M}_2(\C)\mid {}^t\bar X =X,\,\tr(X)=0\}\cong \R^3$ with the inner product $(X|Y):=2^{-1}\tr(X Y)$, on which $(\bK_\infty^\#)^0$ acts by $(X,k_0)\mapsto {}^t\bar k_0 X k_0$. Due to $\sG^\#(\R)=\sB^\#(\R) (\bK^\#)^0$, by the restriction to $(\bK_\infty^\#)^0$, we have ${\mathscr V}^\#(s,\Lambda_\infty,\mu_\infty)\cong C^\infty({\bf S}^\#)_0$, the space of $(\bK_\infty^\#)^0$-finite $C^\infty$-functions on ${\bf S}^\#$; for $f_\infty\in C^\infty({\bf S}^\#)_0$, let $f_\infty^{(s,\Lambda_\infty,\mu_\infty)}\in {\mathscr V}^\#(s,\Lambda_\infty,\mu_\infty)$ denote the inverse image of $f_\infty$ by this map. 
For each $d\in \Z_{\geq 0}$, let $(\varrho_d,V_d)$ be the  representation of ${\bf SU}(2)=(\bK_\infty^\#)^0$ on the space $V_d$ of homogeneous polynomials of degree $d$ in $\C[X,Y]$ (\cite[p.101]{Jacquet}), which carries a non-degenerate paring $\langle\cdot,\cdot\rangle$ given by \cite[(18.3.4)]{Jacquet}. For $d\in 2\Z_{\geq 0}$, set $\xi_d:=
(XY)^{d/2}$, and define $\sf_{d}(k^\#):=(d+1)^{1/2} {\langle \varrho_d(k^\#)\xi_d, \xi_d\rangle}/{\langle \xi_d,\xi_d\rangle}$ for $k^\#\in (\bK^\#_\infty)^0$. Since $\xi_d\in V_d$ is fixed by $\sB^\#(\R)\cap (\bK^\#_\infty)^0$
, the function $\sf_d$ descend to a function on ${\bf S}^\#$. The $\C$-span of $\sf_d\,(d\in 2\Z_{\geq 0})$ coincides with the $\sB^\#(\R)\cap (\bK_\infty^\#)^0$-fixed vectors of $C^\infty({\bf S}^\#)_0$ ({\it cf}. \cite[\S4.4]{Tsud2011-1}).
Define $\phi_\infty^{(d)}\in \cS(\C^2)$ by 
\begin{align}
\phi_\infty^{(d)}(x,y):=\sum_{j=0}^{d/2}\tfrac{(x \bar x)^{j}}{(j!)^2}\tfrac{(-y\bar y)^{d/2-j}}{((d/2-j)!)^2}\,\tfrac{2}{\pi}\exp\left(-\tfrac{2\pi}{\sqrt{|D|}}(x\bar x+y \bar y)\right), \quad (x,y)\in \C^2.
 \label{Def-phiInfd}
\end{align}
By \eqref{DefJaqS} and \eqref{normsec}, 
\begin{align}
\sf_{\phi_\infty^{(d)}}^{(s,\Lambda_\infty,\mu_\infty)}=|D|^{\frac{s+1}{2}}(-1)^{d/2}\tfrac{\Gamma_{\C}(s+d/2+1)}{\Gamma_\C(s+1)}\,
\sf_d^{(s,\Lambda_\infty,\mu_\infty)}.
 \label{Ksphericalsection}
\end{align}
In particular, $|D|^{-\frac{s+1}{2}}\sf_{\phi_\infty^{(0)}}^{(s,\Lambda_\infty,\mu_\infty)}$ corresponds to the constant $1\in C^\infty({\bf S}^\#)_0$.

\subsection{Eisenstein series and Rankin-Selberg integrals}\label{sec:EISSER} 
Let $\Lambda=\otimes_{p\leq \infty}\Lambda_p \in \widehat{{\rm Cl}(E)}$ and $\mu=\otimes_{p\leq \infty}\mu_p$ a character of $\A^\times /\Q^\times$ of finite order. For a section $f^{(s)}$ of the adelic principal series $\otimes_{p\leq \infty}{\mathscr V}^\#(s,\Lambda_p,\mu_p)$ such that $f^{(s)}|\bK^\#$ is $\bK^\#$-finite and meromorphic, define the Eisenstein series
$$
{\rm Eis}(f^{(s)};g^\#):=
\widehat L(s+1,\Lambda\mu_{E})\times
\sum_{\delta \in \sB^\#(\Q)\bsl \sG^\#(\Q)}f^{(s)}(\delta g^\#), \quad g^\#\in \sG^\#(\A), 
$$
which is absolutely convergent on $\Re(s)>1$ and has a meromorphic continuation to $\C$. Let $\phi=\otimes_{p\leq \infty} \phi_p$ be a decomposable element of $\cS(\A_{E}^2)$ such that $\phi_p=\cchi_{\cO_{E,p}\oplus \cO_{E,p}}$ for almost all $p<\infty$. Define 
\begin{align}
\sf_{\phi}^{(s,\Lambda,\mu)}(g^\#):=\prod_{p\leq \infty} \sf_{\phi_p}^{(s,\Lambda_p,\mu_p)}(g_p^\#), \quad g^\#=(g_p^\#)_{p}\in \sG^\#(\A),
 \label{ad-normsec}
\end{align}
which is a $\bK^\#$-finite holomorphic section of the adelic principal series $\otimes_{p\leq \infty}{\mathscr V}^\#(s,\Lambda_p,\mu_p)$. Let 
$$
E(\phi,s, \Lambda,\mu;g^\#):={\rm  Eis}(\sf^{(s,\Lambda,\mu)}_\phi;g^\#), \quad g^\#\in \sG^\#(\A)
$$
be the associated Eisenstein series, which is known to be holomorphic except for possible simple poles at $s=\pm 1$ and is vertically bounded; these poles occur if and only if $\Lambda\mu_{E}={\bf 1}$. For $\phi \in \cS(\A_{E}^2)$, let $\widehat \phi$ be the Fourier transform of $\phi$ with respect to the self-dual Haar measure $\eta_{\A_E^2}:=\eta_{\A_E}\otimes \eta_{\A_{E}}$ on $\A_{E}^2$:
\begin{align}
\widehat{\phi}(x,y):=\int_{\A_{E}^2} \phi(u,v)\psi\left(\tr_{E/\Q}\tfrac{yu-xv}{\sqrt{D}}\right)\,\d\eta_{\A_E^2}(u,v), \quad (x,y)\in \A_{E}^2.
 \label{AdelicFT}
\end{align}
Thus, when $\phi=\otimes_{p}\phi_p$, its adelic Fourier transform \eqref{AdelicFT} is related to the product of the local Fourier transforms \eqref{JacS-f0} as $\widehat \phi=\frac{4}{|D|}\otimes_{p}\widehat \phi_p$ due to the relation between $\eta_{\A_E}$ and $\otimes_{p}\eta_{E_p}$ mentioned in \S\ref{sec:HaarDGS}. We have the functional equation 
\begin{align}
E(\widehat \phi,-s, \Lambda^{-1}, \mu^{-1})=E(\phi, s, \Lambda,\mu) .
 \label{FeEisSer}
\end{align}
Let $\varphi$ be a cusp form on $\sZ(\A)\sG(\Q)\bsl \sG(\A)$. For $f_\infty\in C^\infty({\bf S}^\#)_0$, $\phi \in \cS(\A_{E,\fin}^2)$, and for $g\in \sG(\A)$, define
\begin{align}
Z_\phi^{(s,\Lambda,\mu)}(f_\infty\mid \varphi,g)&:=\int_{\sZ^\#(\A)\sG^\#(\Q)\bsl \sG^\#(\A)}
{\rm Eis}(f_\infty^{(s,\Lambda_\infty,\mu_\infty)}\otimes \sf_\phi^{(s,\Lambda_\fin, \mu_\fin)};h^\#) \varphi(\iota_\theta(h^{\#}g))\,\d h^\#.
 \label{Def-RSInt}
\end{align}
Due to the cuspidality of $\varphi$, this converges absolutely to define a holomorphic function on any domain where the Eisenstein series is regular.

\section{Smoothed Rankin-Selberg integral: the spectral side}

\subsection{Shintani functions}\label{sec:ShinFtn} 
The complex power function $z^{\alpha}\,(z\in \C^\times,\alpha \in \C)$ is defined as 
$$
z^{\alpha}:=|z|^{\alpha} \exp(i \alpha {\rm Arg}(z)), \quad {\rm Arg}(z) \in (-\pi,\pi].
$$
Then, $(z,\alpha)\mapsto z^{\alpha}$ is holomorphic on $(\C-\R_{\leq 0})\times \C$ and 
{\allowdisplaybreaks\begin{align}
 (z^{-1})^{\alpha}&=z^{-\alpha}, \quad (\Im(z) \not=0), 
 \label{CpxPw-f1}\\
 (zw)^{\alpha}&=z^{\alpha} w^{\alpha}, \quad \text{if ${\rm Arg}(z)+{\rm Arg}(w) \in (-\pi, \pi]$,}
 \label{CplxPw-f2} \\
 |z^{\alpha}|&\leq \exp\left(\tfrac{\pi}{2}|\Im(s)|\right)\,|z|^{\Re(\alpha)}\quad (\Re(z)>0).
 \label{CplxPw-f3}
\end{align}}For $l \in 2\Z_{>0}$, a character $\mu_\infty:\R^\times \rightarrow \{\pm 1\}$ and $s\in \C$, set
\begin{align}
\Phi_{l}^{(s,\mu_\infty)}(g):=J(g,i1_2)^{-l}\mu_\infty(\nu(g))\,\left({\rm sgn}(\nu(g)) \frac{\tr(T_\theta\,g\langle i 1_2\rangle)}{i \sqrt{|D|}} \right)^{s-l+1}, \quad g\in \sG(\R).
 \label{Def-ShinFtn}
\end{align}

\begin{lem} \label{ShinFtn-L1}
The function $\Phi_{l}^{(s,\mu_\infty)}:\sG(\R) \longrightarrow \C$ is $C^\infty$ and 
\begin{align}
\Phi_{l}^{(s,\mu_\infty)}(gk)&=J(k,i1_2)^{-l}\,\Phi_{l}^{(s,\mu_\infty)}(g), \quad (k\in \bK_\infty,\,g\in \sG(\R)), 
 \label{ShinFtn-L1-f1}
\\
R_{\bar X}\Phi_{l}^{(s,\mu_\infty)}(g)&=0, \qquad X\in \fp^{+},\,g\in \sG(\R), \\
\Phi_{l}^{(s,\mu_\infty)}(b_\R^\theta)&=1.
 \label{ShinFtn-L2-f2}
\end{align}
For each $g\in \sG(\R)$, the function $s\mapsto \Phi_l^{(s,\mu_\infty)}(g)$ is holomorphic and vertically bounded on $\C$. Moreover, it has the following equivariance with respect to the left $\iota(\sB^{\#}(\R))$ action:
\begin{align}
\Phi_{l}^{(s,\mu_\infty)}\left( \iota\left(\left[\begin{smallmatrix} \tau & \beta \\ 0 & a \bar \tau \end{smallmatrix} \right]\right)\,g\right)&=\mu_\infty(a)|a|^{-(s+1)}\,\Phi_{l}^{(s,\mu_\infty)}(g), \quad (\tau \in \C, a\in \R^\times, \beta\in \C, g\in \sG(\R)).
 \label{ShinFtn-L1-f0}
\end{align}
\end{lem}
\begin{proof} Since $\Im(Z)$ for $Z\in \fh_2$ is positive definite, by ${}^tA_\theta A_\theta=T_\theta$, $\Im(\tr(T_\theta Z))=\tr({}^t A_\theta\Im(Z)A_\theta)>0$, which implies $\Re\left({\tr(T_\theta\,Z)}/{(i \sqrt{|D|})} \right)>0$ for $Z\in \fh_2$. Thus, $(s, Z)\longmapsto \left({\tr(T_\theta\,Z)}/({i \sqrt{|D|}}) \right)^{s-l+1}$ is holomorphic on $\C \times \fh_2$. Since ${\rm sgn}(\nu(g))\,g\langle i 1_2\rangle\in \fh_2$, we obtain \eqref{ShinFtn-L1-f1} and \eqref{ShinFtn-L2-f2} as well as the holomorphicity of $\Phi_l^{(s,\mu_\infty)}(g)$ in $s\in \C$. The remaining formulas are confirmed by a direct computation.  
\end{proof}

We consider \eqref{Def-RSInt} for $g=a_\infty^{(t)}b_\R^\theta$. The group $b_\R^\theta \iota_\theta(\sB^\#(\R)\cap \bK_\infty^\#) (b_\R^\theta)^{-1}$ is contained in $\sM(\R)$, so that it is in the centralizer of $A$. Then, by \cite[Lemma 2.2]{KugaTsuzuki}, we have $Z_\phi^{(s,\Lambda,\mu)}(R(k^\#)f_\infty\mid \varphi,a_\infty^{(t)}b_\R^\theta )=Z^{(s,\Lambda,\mu)}_\phi(f_\infty\mid \varphi,a_\infty^{(t)}b_\R^\theta )$ for all $k^\#\in \sB^\#(\R)\cap (\bK_\infty^\#)^0$. We have $Z^{(s,\Lambda,\mu)}_\phi(f_\infty\mid \varphi,a_\infty^{(t)}b_\R^\theta)=0$ unless $f_\infty$ is right $\sB^\#(\R)\cap (\bK_\infty^\#)^0$-invariant, i.e., $f_\infty$ is a zonal spherical function on ${\bf S}^\#$. As such, $f_\infty$ is in the linear span of $\sf_d\,(d\in 2\Z_{\geq 0})$ (see \S\ref{sec:SecArch}). Define\begin{align}
\Phi_{l,d}^{(s,\mu_\infty)}(t):=\int_{\bK_\infty^{\#}} \Phi_l^{(s,\mu_\infty)}(\iota_\theta(k^\#) a_\infty^{(t)}b_\R^\theta)\,{\sf_d(k^\#)}\,\d k^\#, \quad t\in \R,\,d\in 2\Z_{\geq 0}. 
\label{DefPhiLD}
\end{align}
({\it cf}. \cite[(4.25)]{Tsud2011-1}). The following formula is used in the proof of Proposition \ref{SpectralExp-P}. 

\begin{lem} \label{RSInt-L2}
For all $d\in \Z_{\geq 0}$, 
\begin{align*}
Z_\phi^{(s,\Lambda,\mu)}(\sf_d\mid \varphi, a_\infty^{(t)}b_\R^\theta g_\fin)=Z^{(s,\Lambda,\mu)}_\phi(1\mid \varphi, b_\R^\theta g_\fin)\,\Phi_{l,d}^{(-s,\mu_\infty)}(t), \quad t\in \R,\,g_\fin \in \sG(\A_\fin).
\end{align*}
\end{lem}
\begin{proof} We can use \cite[Proposition 29]{Tsud2011-1} due to ${\bf PGSp}_2(\R)\cong {\bf SO}(Q)$ (see \S\ref{sec: ExIso}). 
\end{proof}

\subsection{Hecke functions}\label{sec:HeckeFtn}
Let $\Lambda=\otimes_{p}\Lambda_p$ and $\mu=\otimes_{p}\mu_p$ be as in \S\ref{sec:EISSER}. 
Let $\phi=\otimes_{p<\infty}\phi_p$ be any decomposable element of $\cS(\A_{E,\fin}^2)$ with $\phi_p=\cchi_{\cO_{E,p}\oplus \cO_{E,p}}$ for almost all $p$. Let $f=\otimes_{p<\infty}f_p\in \cH(\sG(\A_\fin))$ with $f_p=\cchi_{\bK_p}$ for almost all $p$. For $s\in \C$, set
\begin{align}
\Phi_{\phi_p,f_p}^{(s,\Lambda_p,\mu_p)}(g_p):=\int_{\sG^\#(\Q_p)}\sf_{\phi_p}^{(s,\Lambda_p,\mu_p)}(g^\#)\,f_p(\iota(g_\#)^{-1}g_p)\,\d \eta_{\sG^\#(\Q_p)}(g^\#), \quad  g_p \in \sG(\Q_p).
\label{HeckeFtn-f0}
\end{align}
 Since $f_p$ is of compact support on $\sG(\Q_p)$, the integral converges absolutely defining a smooth function $\Phi_{\phi_p,f_p}^{(s,\Lambda_p,\mu_p)}$ on $\sG(\Q_p)$; it is identically $1$ on $\bK_p$ when $\phi_p=\cchi_{\cO_{E,p}\oplus \cO_{E,p}}$ and $f_p=\cchi_{\bK_p}$. In a similar manner, a smooth function $\Phi_{\phi,f}^{(s,\Lambda,\mu)}:\sG(\A_\fin)\rightarrow \C$ is defined by the integral of $\sf_\phi^{(s,\Lambda,\mu)}(h)f(\iota(h)^{-1}g)$ over $h \in \sG^{\#}(\A_\fin)$, so that
$$
\Phi_{\phi,f}^{(s,\Lambda,\mu)}(g)=\prod_{p<\infty}\Phi_{\phi_p,f_p}^{(s,\Lambda_p,\mu_p)}(g_p), \quad g=(g_p)_{p<\infty} \in \sG(\A_\fin).
$$
The proof of the next lemma is immediate. 
\begin{lem} \label{HeckeFtn-L1}
 Let $p<\infty$ and $s\in \C$. Then,   
\begin{align}
\Phi_{\phi_p,f_p}^{(s,\Lambda_p,\mu_p)}(\left[\begin{smallmatrix} \tau & \beta \\ 0 & a\bar \tau\end{smallmatrix} \right]g)=\Lambda_p^{-1}(\tau)\mu_p(a)|a|_p^{-(s+1)}\,\Phi_{\phi_p,f_p}^{(s,\Lambda_p,\mu_p)}(g), \quad (
\left[\begin{smallmatrix}\tau & \beta \\ 0& a \bar \tau\end{smallmatrix}\right] \in \sB^\#(\Q_p),\,
 g\in \sG(\Q_p)). 
 \label{HeckeFtn-L1-f0}
\end{align}
 Moreover, $s\mapsto\Phi_{\phi_p,f_p}^{(s,\Lambda_p,\mu_p)}(g)$ is entire and is  vertically bounded.
We have that $
{\rm supp}(\Phi_{\phi_p,f_p}^{(s,\Lambda_p,\mu_p)})\subset \iota(\sB^\#(\Q_p)){\rm supp}(f_p)$. When $f_p=\cchi_{\bK_p}$ and $\phi_p=\cchi_{\cO_{E,p}\oplus \cO_{E,p}}$, we have $\Phi_{\phi_p,f_p}^{(s,\Lambda_p,\mu_p)}(g)=0$ for $g\not\in \iota(\sB^\#(\Q_p))\bK_p$,  $$
\Phi_{\phi_p,f_p}^{(s,\Lambda_p,\mu_p)}(\left[\begin{smallmatrix}\tau & \beta \\ 0& a \bar \tau
\end{smallmatrix}\right])=\Lambda_p^{-1}(\tau)\mu_p(a)|a|_p^{-(s+1)} \quad (\left[\begin{smallmatrix}\tau & \beta \\ 0& a \bar \tau
\end{smallmatrix}\right] \in \sB^\#(\Q_p),\,k\in \bK_p). 
$$
\end{lem}

\subsection{Kernel functions} \label{sec:KerFtn}
Let $l\in 2\Z_{>0}$. Let $f=\otimes_{p<\infty}f_p \in \cH(\sG(\A_\fin))$ and $\phi=\otimes_{p<\infty} \phi_p \in \cS(\A_{E,\fin}^2)$. For $s\in \C$ and $(\Lambda,\mu)$ as in \S\ref{sec:EISSER}, we define a function ${\bf \Phi}_{\phi,f,l}^{(s,\Lambda,\mu)}$ on $\sG(\A)$ by setting 
\begin{align*}
{\bf\Phi}_{\phi, f,l}^{(s,\Lambda,\mu)}(g):=\Phi_{l}^{(s,\mu_\infty)}(g_\infty)\,\prod_{p<\infty} \Phi_{\phi_p,f_p}^{(s,\Lambda_p,\mu_p)}(g_p), \qquad g=(g_p)_{p\leq \infty} \in \sG(\A),
\end{align*}
where $\Phi_{l}^{(s,\mu_\infty)}(g_\infty)$ is the holomorphic Shintani function on $\sG(\R)$ of weight $l$ (see \S\ref{sec:ShinFtn}) and $\Phi_{\phi_p,f_p}^{(s,\Lambda_p,\mu_p)}(g_p)$ is the function on $\sG(\Q_p)$ defined by \eqref{HeckeFtn-f0}.

Let $\beta(s)$ be an entire function on $\C$ of the form 
 \begin{align}
P(s) e^{as^2}\quad (a\in \R_{>0},\,P(s)\in \C[s]).
 \label{EstBeta}
\end{align}
Note that $\beta(s)=O_{\sigma}(e^{-a|\Im(s)|^2})$ for $s\in \cT_{[-\sigma,\sigma]}$. 
For such $\beta(s)$, set 
\begin{align}
\widehat {{\bf\Phi}_{\phi,f,l}}^{(\beta, \Lambda,\mu)}(g):=\int_{c-i\infty}^{c+i\infty} \beta(s)s(s^2-1)|D|^{s/2}\widehat L(1+s,\Lambda\mu_E)\,{\bf\Phi}_{\phi,f,l}^{(s,\Lambda,\mu)}(g)\,\d s, \quad g\in \sG(\A)
 \label{KerFtn-f1}
\end{align}
with any $c\in \R$. By Lemmas \ref{ShinFtn-L1} and \ref{HeckeFtn-L1}, the function ${\bf\Phi}_{\phi,f,l}^{(s,\Lambda,\mu)}(g)$ of $s$ is entire and vertically bounded, so is $s(s+1)\hat L(s+1, \Lambda\mu_E)$. Hence, the integral converges absolutely and is indepdendent of $c$. By Lemmas \ref{ShinFtn-L1} and \ref{HeckeFtn-L1}, due to the product formula $|a|_\A=1\,(a\in \Q^\times)$, we have
\begin{align}
\widehat {{\bf\Phi}_{\phi,f,l}}^{(\beta,\Lambda,\mu)}(\iota(b)g)=\widehat{{\bf\Phi}_{\phi,f,l}}^{(\beta, \Lambda,\mu)}(g), \quad b\in \sB^{\#}(\Q), \, g\in \sG(\A).
 \label{KerFtn-f0}
\end{align}

\subsection{Poincar\'{e} series} \label{sec:PoinSer}
Due to \eqref{KerFtn-f0}, the Poincar\'{e} series 
\begin{align}
{\mathbb F}_{\phi,f,l}^{(\beta,\Lambda,\mu)}(g):=\sum_{\gamma \in \iota(\B^\#(\Q))\bsl \sG(\Q)} \widehat {{\bf\Phi}_{\phi,f,l}}^{(\beta, \Lambda,\mu)}(\gamma g), \quad g\in \sG(\A),
\label{PoinSer-L1}
\end{align}
is well-defined if convergent. The convergence is proved by constructing a majorant of $\widehat {{\bf\Phi}_{\phi,f,l}}^{(\beta, \Lambda,\mu)}(g)$ ({\it cf}. \cite[Lemma 47]{Tsud2011-1}) as follows. For any prime $p$ and $n\in \Z_{\geq 0}$, set $G_p(T_\theta;n):=\{g\in \sG(\Q_p)\mid \sr(g)^{-1}[T_\theta^\dagger] \in p^{-n}(\sX(\Z_p)^{*}-p\sX(\Z_p)^{*})\}$ with $\sX(\Z_p)^{*}=\Z_p\xi_1\oplus [{\mathcal Q}_{\Z_p})]\oplus \Z_p \xi_5$ being the dual lattice of $\sX(\Z_p)=\Z_p\xi_1 \oplus [\sV(\Z_p)]\oplus \Z_p\xi_5$. Then, $\sG(\Q_p)$ is a disjoint union of $G_p(T_\theta,n)\,(n\in \Z_{\geq 0})$; from \cite[Proposition 2.7]{MS98}, 
\begin{align}
G_p(T_\theta;n)=\iota_\theta(\sG^{\#}(\Q_p))a_{p,n}\bK_{p} \quad (n\in \Z_{\geq 0})
 \label{pAdicCartanDec}
\end{align}
with $a_{p,n}:=\sm(\left[\begin{smallmatrix} p^{-n} & 0 \\ 0 & 1\end{smallmatrix} \right],1)$. 
Set 
$$\fA:=\{a_{\fin}=(a_{p,n_{p}})_{p<\infty}\mid n_{p}=0\,\text{for almost all $p$}\}.$$ 
Then, 
\begin{align}
\sG(\A_\fin)=\iota_\theta(\sG^\#(\A_\fin))\,\fA\,\bK_\fin=\bigsqcup_{a_\fin \in \fA}\iota_\theta(\sG^\#(\A_\fin))\, a_\fin\, \bK_{\fin}.
 \label{finCartanDec}
\end{align}
Combining the decomposition of $\sG(\R)$ in Lemma \ref{RSInt-L1} with \eqref{finCartanDec}, we have $\sG(\A)=\iota_\theta(\sG^\#(\A))A^{+}b_\R^\theta\, \fA\,\bK_{\infty}\bK_\fin$ by which a general point $g\in \sG(\A)$ is written as
\begin{align}
&g=\iota_{\theta}(h)a_\infty^{(t)}b_\R^\theta a_{\fin}\,k_\infty k_\fin
 \label{elementgCartanDec}
\end{align}
with $h\in \sG^{\#}(\A)$, $a_\infty^{(t)}\in A^{+}$, $a_\fin=(a_{p,n_p})_{p<\infty}\in \fA$ and $k_\infty k_\fin\in \bK_\infty\bK_\fin$.
We define functions $\Xi_{q,r}^{S},\,{\mathcal F}_r^{S}:\sG(\A)\longrightarrow \R_{\geq 0}$ ({\it cf}. \cite[\S 6]{Tsud2011-1}), where $q,\,r>0$ and $S$ is a finite set of prime numbers, by demanding the values $\Xi_{q,r}^{S}(g)$ and ${\mathcal F}_r^{S}(g)$ at a point \eqref{elementgCartanDec} are 
\begin{align*}
{\mathcal F}_r^{S}(g)&:=(\cosh 2t)^{-r}\prod_{p\in S}p^{-rn_p}\prod_{p\not \in S}\delta_{n_p,0}, \\
\Xi_{q,r}^{S}(g)
&:=\inf(y^{\#}(h)^{q},1) \,{\mathcal F}_r^{S}(g), 
\end{align*}
where $y^{\#}:\sG^{\#}(\A)\longrightarrow \R_{>0}$ is a function defined by $y^{\#}\left(\left[\begin{smallmatrix} \tau & \beta \\ 0 & a \bar \tau \end{smallmatrix} \right]k_0\right)=|a^{-1}|_\A$ for $\left[\begin{smallmatrix} \tau & \beta \\ 0 & a \bar \tau \end{smallmatrix} \right]\in \sB^\#(\A)$ and $k^\# \in \bK^\#$ ({\it cf}. \cite[(6.5) and (6.6)]{Tsud2011-1}). Note that ${\mathcal F}_r^{S}$ is left $\iota_{\theta}(\sG^\#(\A))$-invariant and that $\Xi_{q, r}^{S}$ is left $\sZ(\A)\iota_\theta(\sB^\#(\Q))$-invariant. We need the following bound. 
\begin{lem} \label{PoiSer-L2}
For $k_{\infty}^\#\in \bK^{\#}$, $t\in \R$ and $s\in \cT_{(-\infty,1]}$, 
\begin{align}
|\Phi_l^{(s,\mu_\infty)}(\iota_{\theta}(k_{\infty}^\#) a_\infty^{(t)}b_\R^\theta)|
&\leq \exp(\tfrac{\pi}{2}|\Im(s)|)\,(\ch 2t)^{\Re(s)+1-l}.
\label{PoiSer-L2-f1} 
\end{align}
\end{lem}
\begin{proof}
Let $k_{\infty}^\#=\left[\begin{smallmatrix} a & -\bar b \\ b & \bar a\end{smallmatrix}\right]$ with $a,b\in \C, |a|^2+|b|^2=1$ and define $A,B\in \Mat_2(\R)$ as in \eqref{kAB-def}. Using $\tr(AA^\dagger)=2|a|^2$, $\tr(BB^\dagger)=-2|b|^2$ and $\tr(AB^{\dagger})=\tr(BA^\dagger)=0$, from \eqref{Def-ShinFtn}, we get 
\begin{align}
\Phi_l^{(s,\mu_\infty)}(\iota_{\theta}(k_{\infty}^\#) a_\infty^{(t)}b_\R^\theta)
&=u^{-2l}\,\det\left(A-iB\tfrac{\bar u}{u}\right)^{-l}\left\{\det\left(A-iB\tfrac{\bar u}{u}\right)^{-1}\tfrac{\bar u}{u}\right\}^{s-l+1}
 \notag
\\
&=(\ch 2t)^{s-l+1}\{1+i \sh 2t(|a|^2-|b|^2\}^{-s-1}
 \label{PoiSer-L2-f2}
\end{align}
with $u:=\ch t+i\sh t$. Note that, by $\det X=2^{-1}\tr(X X^\dagger)$, 
$$
\det\left(A-iB\tfrac{\bar u}{u}\right)=u^{-2}\{1+i\sh 2t(|a|^2-|b|^2)\}.
$$
Due to $|1+i\sh 2t (|a|^2-|b|^2)|\geq 1$ and \eqref{CplxPw-f3}, from \eqref{PoiSer-L2-f2}, we get 
\begin{align*}
|\Phi_l^{(s,\mu_\infty)}(\iota_{\theta}(k_{\infty}^\#) a_\infty^{(\infty)}b_\R^\theta)|
&\leq \exp(\tfrac{\pi}{2}|\Im(s)|)\,(\ch 2t)^{\Re(s)+1-l}. \qedhere
\end{align*}
\end{proof}

\begin{lem} \label{MMajorant}
Let $S$ be a finite set of primes such that $f_p=\cchi_{\bK_p}$ and $\phi_{p}=\cchi_{\cO_{E,p}\oplus \cO_{E,p}}$ for $p\not\in S$. For any small $\e>0$, \begin{align*}
|\widehat {{\bf\Phi}_{\phi,f,l}}^{(\beta, \Lambda,\mu)}(g)|\ll_{\e, l} \Xi_{2+\e,l-2-\e}^{S}(g), \quad g\in \sG(\A).
\end{align*}
\end{lem}
\begin{proof}
Let $g\in \sG(\A)$ be as in \eqref{elementgCartanDec}. Then, for $s\in \C$,  \begin{align}
{\bf\Phi}_{\phi,f,l}^{(s,\Lambda,\mu)}(g)
&=
\int_{\sG^\#(\A_\fin)}\sf_\phi^{(s,\Lambda,\mu)}(x)f(\iota_\theta(x^{-1}h_\fin) a_\fin k_\fin)\, \d x\times \Phi_l^{(s,\mu)}(\iota_\theta(h_\infty)a_\infty^{(t)}b_\R^\theta k_\infty)
 \notag
\\
&=
\int_{\sG^\#(\A_\fin)}\sf_\phi^{(s,\Lambda,\mu)}(h_\fin x)f(\iota_\theta(x^{-1})a_\fin k_\fin)\, \d x\times 
|\lambda|_\infty^{-s-1}\Phi_l^{(s,\mu)}(\iota_\theta(k^\#)a_\infty^{(t)}b_\R^\theta k_\infty),
 \label{MMajorant-f1}
\end{align}
where $h_\infty=\left[\begin{smallmatrix} \tau_\infty & \beta_\infty \\ 0 & \bar \tau_\infty \lambda_\infty
\end{smallmatrix} \right]k_{\infty}^\#$ $(\tau_\infty\in \C^\times,\,\lambda_\infty\in \R^\times,\,\beta_\infty \in \C, \, k_{\infty}^\#\in \bK_\infty^{\#})$ is the Iwasawa decomposition of $h_\infty \in \sG^\#(\R)$. Let $p\in S$; since ${\rm supp}(f_p)$ is compact, there exists $k_p\in \Z_{>0}$ such that ${\rm supp}(f_p)\subset \bigcup_{0\leq n\leq k_p}G_{p}(T_\theta;n)$. From this, for any $r>0$ there exists constants $C_r>0$ such that the $x$-integral in \eqref{MMajorant-f1} is no greater than $C_r |\lambda_\fin|_\fin^{-\Re(s)-1}\,\prod_{p\in S}p^{-r n_p}\prod_{p\not \in S}\delta_{n_p,0}$ for $s\in \C$, $h_\fin \in \sG^\#(\A_\fin)$ and $a_{\fin}=(a_{p,n_p})_{p}\in \fA$, where $h_\fin=\left[\begin{smallmatrix}  \tau_\fin & \beta_\fin \\ 0 & \bar \tau_\fin \lambda_\fin \end{smallmatrix} \right] k_{0,\fin}$ $(\tau_\fin \in \A_{E,\fin}^\times,\,\lambda_\fin \in \A_\fin^\times,\beta_\fin\in \A_\fin,k_{0,\fin}\in \bK_\fin^{\#})$ is the Iwasawa decomposition of $h_\fin \in \sG^\#(\A_\fin)$. Combining this with Lemma \ref{PoiSer-L2}, we have
\begin{align}
|{\mathbf \Phi}_{l,f}^{(s,\Lambda,\mu)}(g)|\ll_{r} \exp\left(\tfrac{\pi}{2}|\Im (s)|\right)\,y^\#(h)^{\Re(s)+1}\,(\ch 2t)^{\Re(s)+1-l}\,\prod_{p\in S}p^{-rn_p}\prod_{p\not \in S}\delta_{n_p,0}
 \label{MMajorant-f2}
\end{align}
uniformly in $s \in \cT_{[-1,1+\e]}$. Since $\int_{c-i \infty}^{c+i \infty} |\beta(s)\widehat L(1+s,\Lambda\mu_E)|e^{\frac{\pi}{2}|\Im(s)|}\,|\d s|<+\infty$ due to \eqref{EstBeta}, by taking $c=1+\e$ and $c=-1$, we get the required majorization.
\end{proof}

\begin{lem}\label{MintConvL} If $r>3$ and $q>2$, then 
\begin{align}
 \int_{\sZ(\A)\iota_{\theta}(\sB^\#(\Q))\bsl \sG(\A)}\Xi_{q,r}^{S}(g)\,\d g<\infty. 
\label{MintConvL-f0}
 \end{align}
\end{lem}
\begin{proof} The integral in \eqref{MintConvL-f0} is a product of the following two integrals: 
\begin{align}
&\int_{\sZ^\#(\A)\sB^\#(\Q)\bsl \sG^\#(\A)} \inf(y^\#(h)^{q}, 1)\,\d h, 
\label{MCL-f1}
\\
&\int_{\iota_\theta(\sG^\#(\A))\bsl \sG(\A)} {\mathcal F}_r^{S}(g)\,\d g \label{MCL-f2}.
\end{align}
As the integration domain in \eqref{MCL-f1}, we may take all the points of the form \eqref{Def-HaarMeasSGP-f0} with $\beta\in \A_E/E$, $\tau\in \A_E^\times/\A^\times E^\times$, $a=a_0\lambda\,(a_0 \in \A^1/\Q^\times\,\lambda \in \R_{>0})$ and $k^\#\in \bK^\#$, the integral in \eqref{MCL-f1} is dominated by a constant times of $\int_{0}^{\infty} \inf(\lambda^{-q},1)\lambda^{2}\d^\times \lambda$, which is finite if $q>2$. By Lemma \ref{RSInt-L1}, in terms of the decomposition \eqref{elementgCartanDec}, the Haar measure $\d g$ on $\sG(\A)$ is given as
\begin{align}
\d g=c_{\sG}\,\d\eta_{\sG^\#(\A)}(h)\,(\sh 2t)^{2}(\ch 2t)\,\d t\,\d\eta_{\fA}(a_\fin)\, \d\eta_{\bK_\fin}(k_\fin)\,\d\eta_{\bK_\infty}(k_\infty),
 \label{MCL-f3}
\end{align}
where $c_\sG=2^6\pi^{3/2}\Gamma(3/2)^{-1}$, and $\d\eta_{\fA}(a_\fin)=\otimes_{p<\infty} \d\eta_{\fA_p}(a_p)$ is a measure on the discrete space $\fA$ such that $\eta_{\fA_p}(\{a_{p,n_p}\})\asymp p^{3n_p}$ (see \cite[Lemma 20]{Tsud2011-1}). By this, we see that the integral \eqref{MCL-f2} is majorized by the product of 
$$
I_{\infty}:=\int_{0}^{\infty} (\cosh 2t)^{-r}\,(\sinh 2t)^2\cosh 2t\, \d t, \quad I_\fin:=\prod_{p\in S}\{\sum_{n_p=0}^{\infty} p^{-rn_p}\,\eta_{\fA_p}(\{a_{p,n_p}\})\}.  
$$
If $r>3$, then both are finite.
\end{proof}

Let $\cK_f\subset \bK_\fin$ be an open compact subgroup such that $f$ is right $\cK_f$-invariant. 
\begin{lem}\label{MConvergenceL}
If $l \in 2\Z_{\geq 3}$, then, for any compact set $\mathcal N\subset \sG(\A)$, the series in \eqref{PoinSer-L1} is convergent absolutely and uniformly in $g\in \cN$ defining a smooth function ${\mathbb F}_{\phi,f,l}^{(\beta,\Lambda,\mu)}$, which is left $\sZ(\Q)\sG(\Q)$-invariant, right $\mathcal K_f$-invariant and integrable on $\sZ(\A)\sG(\Q)\bsl \sG(\A)$. 
\end{lem}
\begin{proof}
({\it cf}. \cite[\S6]{Tsud2011-1}). The normal convergence of $\sum_{\gamma \in \iota_{\theta}(\sB^\#(\Q))\bsl \sG(\Q)}\Xi_{q,r}^{S}(\gamma g)$ results from Lemma \ref{MintConvL} when $r>3$ and $q>2$; these inequalities are satisfied for $r=l-2-\e$ and $q=2+\e$ with a small $\e>0$ if $l>5$. The integral in \eqref{MintConvL-f0} equals
\begin{align*}
 \int_{ \sZ(\A) \sG(\Q) \bsl \sG(\A)} 
 \sum_{\gamma \in \iota_\theta(\sB^\#(\Q))\bsl \sG(\Q)} 
\Xi_{q,r}(\gamma g)\, \d g,
\end{align*}
which majorizes $\int_{\sZ(\A)\sG(\Q) \bsl \sG(\A)} |{\mathbb F}_{f,l}^{(\beta,\Lambda,\mu)}(g)|\,\d g$ by Lemma \ref{MMajorant}. \end{proof}

Let $S_l(\cK_f)$ denote the space of all the smooth bounded functions $F$ on $\sZ(\A)\sG(\Q)\bsl \sG(\A)$ with the $\cK_f\bK_\infty$-equivariance 
$$
F(gk_\fin k_\infty)=J(k_\infty,i1_2)^{-l} 
F(g), \quad k_\fin\in \cK_f,\, k_\infty\in \bK_\infty$$
and satisfies the Cauchy-Riemann condition $R(\fp^{-})F=0$. 

\begin{prop} \label{HolomorphyPhi-L2}
For any $l\in 2\Z_{\geq 3}$, we have ${\mathbb F}_{\phi,f,l}^{(\beta,\Lambda,\mu)}\in S_l(\cK_f).$
\end{prop}
\begin{proof} Since the function $\widehat {\mathbf \Phi}_{\phi,f,l}^{(\beta,\Lambda,\mu)}$ has the correct $\cK_f\bK_\infty$-equivariance as above and satisfies the Cauchy-Riemann condition $R(\fp^{-})\widehat {\mathbf \Phi}_{\phi,f,l}^{(\beta,\Lambda,\mu)}=0$, so does the function $\widehat{\mathbb F}_{\phi,f,l}^{(\beta,\Lambda,\mu)}(g)$ by Lemma \ref{MConvergenceL}. Since we already have the integrability of ${\mathbb F}_{\phi,f,l}^{(\beta,\Lambda,\mu)}$ on $\sZ(\A)\sG(\Q)\bsl \sG(\A)$, we complete the proof by\footnote{\cite[Theorem 2]{Tsud2011-1} is proved only for full level case, but the proof with a minor modification works on general levels.} \cite[Theorem 2]{Tsud2011-1} with $m=3$. 
\end{proof}

\subsection{The spectral expansion} \label{subsecSPECEXP}
For $l\in 2\Z_{\geq 2}$ and $s\in \C$, set 
\begin{align}
B_{l}{(s)}:=2\int_{0}^{+\infty}(\sh 2t)^{2}(\ch 2t)^{3-2l} \int_{(\bK_{\infty}^\#)^0}|q(t;k^\#)|^{-2}\left(q(t;k^\#)/{\overline {q(t;k^\#)}}\right)^{-s}
\,\d k^\#\,\d t
\label{def:BlXiS}
\end{align}
with $q(t;k^\#):=1+i(|a|^2-|b|^2)\sh 2t\,(t>0)$ and $k^\#=\left[\begin{smallmatrix} a & -\bar b \\ b & \bar a \end{smallmatrix} \right]\in (\bK_\infty^\#)^0$. Note that the factor $2$ on the right-hand side is $[\bK_\infty^\#:(\bK_\infty^\#)^0]$, and that $|a|^2-|b|^2=(\overline{{}^t{k^\#}}v_0 k^\#|v_0)$ with $v_0:=\left[\begin{smallmatrix} 1 & 0 \\ 0 & -1 \end{smallmatrix}\right]\in W$ (see \S\ref{sec:EISSER}). By the same reasoning as in the proof of \cite[Proposition 30]{Tsud2011-1}, the integral \eqref{def:BlXiS} converges absolutely for all $s\in \C$ defining an entire function satisfying \footnote{This functional equation is erroneously stated as $B_l(-s)=B_l(s)$ in \cite[Proposition 30]{Tsud2011-1}} $\overline{B_l(-\bar s)}=B_l(s)$ and
\begin{align}
|B_l(s)|\leq B_{l}(0)\exp({\pi}|\Im (s)|).
\label{Bl-ESTMT}
\end{align}
Moreover ({\it cf}. \cite[Proposition 30 (2)]{Tsud2011-1}), \begin{align}
B_l(0)=\frac{\Gamma(3/2)\Gamma(l-3/2)\Gamma(l-2)}{4\,\Gamma(l-1)^2}. 
\label{ValueBl0}
\end{align}
From \eqref{def:BlXiS} and the expression in \eqref{PoiSer-L2-f2}, we get 
$$
B_{l}(s)=\int_{0}^{\infty} (\sh 2t)^{2}(\ch 2t)^{3-2l} \int_{\bK_{\infty}^\#}\Phi_{l}^{(s,\mu_\infty)}(\iota_{\theta}(k^\#) a_\infty^{(t)}b_\R^\theta)\,\overline{\Phi_{l}^{(-\bar s,\mu_\infty)}(\iota_{\theta}(k^\#) a_\infty^{(t)}b_\R^\theta)}\,\d k^\#\,\d t.
$$

Let $\phi\in \cS(\A_{E,\fin}^2)$ and $f\in \cH(\sG(\A_\fin))$. Let $\cB_l(\cK_f)$ be an orthonormal basis of the space $S_l(\cK_f)$. Define an element of $S_l(\cK_f)$ depending holomorphically on $s\in \C$ by 
\begin{align}
\F_{\phi,f,l}^{(s,\Lambda,\mu)}(g):=
\tfrac{2^{5}\pi^{1/2}}{\Gamma(3/2)}\,B_l{(s)}
\sum_{\varphi \in \cB_l(\cK_f)}
|D|^{s/2}\overline{Z_{\bar \phi}^{(\bar s, \Lambda^{-1},\mu^{-1})}(1\mid R(\bar f)\varphi,b_\R^\theta)}
\,\varphi(g),
 \label{FFSpectExp}
\end{align}
where $Z_\phi^{(s,\Lambda,\mu)}(\varphi,b_\R^\theta)$ is as in \eqref{Def-RSInt}. We consider the series 
\begin{align*}
\widehat\EE^{(\beta)}(g):=\sum_{\delta \in \sB^\#(\Q))\bsl \sG^\#(\Q)}\widehat {{\bf\Phi}_{\phi,f,l}}^{(\beta, \Lambda,\mu)}(\iota_\theta(\delta) g)
,\end{align*}
which, by Lemma \ref{MConvergenceL}, is absolutely convergent normally on $\sG(\A)$ from  as a sub-series of $\widehat {\mathbb F}_{\phi,f,l}^{(\beta,\Lambda,\mu)}(g)$; indeed, $$
\widehat {\mathbb F}_{\phi,f,l}^{(\beta,\Lambda,\mu)}(g)=\sum_{\gamma\in \iota_\theta(\sG^\#(\Q))\bsl \sG(\Q)} \widehat\EE^{(\beta)}(\gamma g).
$$
We also need yet another series
\begin{align}
\EE^{(s)}(g):
=\widehat L(1+s,\Lambda\mu_{E})\sum_{\delta \in \sB^\#(\Q)\bsl \sG^\#(\Q)}{{\bf\Phi}_{\phi,f,l}}^{(s, \Lambda,\mu)}(\iota_\theta(\delta) g)
, \quad g\in \sG(\A),
 \label{HolomrphyPhi-f0}
 \end{align} 
depending on $s\in \C$, whose absolute convergence on $\Re(s) >1$ follows from the majorization \eqref{MMajorant-f2}. We have the formula
\begin{align}
\EE^{(s)}(g)
=\sum_{d\in 2\Z_{\geq 0}} \biggl\{\int_{\sG^\#(\A_\fin)} {\rm Eis}(\sf_{d}^{(s)
}\otimes \sf_{\phi}^{(s,\Lambda_\fin,\mu_\fin)};hx)\,f(\iota_\theta(x)^{-1}a_\fin k_\fin)\d x\biggr\}\, J(k_\infty,i 1_2)^{-l}\,\Phi_{l,d}^{(s,\mu_\infty)}(t), 
 \label{HolomrphyPhi-f1}
\end{align}
for any point $g \in \sG(\A)$ in \eqref{elementgCartanDec}, where $\Phi_{l,d}^{(s,\mu_\infty)}(t)$ is as in \eqref{DefPhiLD} and $\sf^{(s,\Lambda_\infty,\mu_\infty)}_d \in {\mathscr V}^\#(s)$ is the flat section from $\sf_d\in C^\infty({\bf S}^\#)_0$. By the same proof as \cite[Lemma 50]{Tsud2011-1}, the expression \eqref{HolomrphyPhi-f1} is shown to be absolutely convergent for $\Re(s)>1$ ({\it cf}. \cite[Lemma 51]{Tsud2011-1}).

\begin{prop} \label{SpectralExp-P}
Let $l\in 2\Z_{\geq 3}$. For any $a>0$, the function $s\longmapsto e^{as^2}\F_{\phi,f,l}^{(s,\Lambda,\mu)}(g)$ is vertically bounded on $\C$; the function $\widehat{\F}_{\phi,f,l}^{(\beta,\Lambda,\mu)} \in S_l(\cK_f)$ has the following contour integral expression.
$$ 
\widehat{\F}_{\phi,f,l}^{(\beta,\Lambda,\mu)}(g)=\int_{c-i \infty}^{c+i \infty} \beta(s)\,s(s^2-1)\,\F_{\phi,f,l}^{(s,\Lambda,\mu)}(g)\,\d s, \quad g\in \sG(\A) \qquad (c>1).
$$ 
\end{prop}
\begin{proof} For the first claim, we note that $e^{as^2}B(s)$ is vertically bounded due to \eqref{Bl-ESTMT}. In this proof, we set $\sf_\phi^{(s)}=\sf_\phi^{(s,\Lambda,\mu)}$ and $\sf_d^{(s)}=\sf_d^{(s,\Lambda_\infty,\mu_\infty)}$. From Proposition~\ref{HolomorphyPhi-L2}, we have the expansion
\begin{align*}
\widehat{\mathbb F}_{\phi,f,l}^{(\beta,\Lambda,\mu)}(g)
=\sum_{\varphi\in \cB_l(\cK_f)}\langle 
\widehat{\mathbb F}_{\phi,f,l}^{(\beta,\Lambda,\mu)}\mid \varphi\rangle_{L^2}\,\varphi(g), \quad g\in \sG(\A).
\end{align*}
We compute the $L^2$-inner-product $\langle 
\widehat{\F}_{\phi,f,l}^{(\beta,\Lambda,\mu)}\mid \varphi\rangle$ (in $L^2(\sZ(\A)\sG(\Q)\bsl \sG(\A))$) following the long displayed formula in the proof of \cite[Proposition 56]{Tsud2011-1}. The first 5 lines of the computation up to \cite[(8.4)]{Tsud2011-1} are the same, which give us the equality
\begin{align*}
\langle 
\widehat{\F}_{\phi,f,l}^{(\beta,\Lambda,\mu)}\mid \varphi\rangle
=\int_{c-i \infty}^{c+i \infty} \beta(s)s(s^2-1)\biggl\{
\int_{\sG^\#(\A)\bsl \sG(\A)}
\d g\int_{\sZ^{\#}(\A)\sG^\#(\Q)\bsl\sG^\#(\A)}
\EE^{(s)}(hg)\bar \varphi(hg)\,\d h
\biggr\}\,\d s,
\end{align*}
where $c>1$ and $\EE^{(s)}$ is as in \eqref{HolomrphyPhi-f0}. Substituting the expression \eqref{HolomrphyPhi-f1} and using the formula in \eqref{MCL-f3}, we compute the inner integral of the last formula as follows.
{\allowdisplaybreaks
\begin{align}
&\int_{\sZ(\A)\sG^{\#}(\Q)\bsl \sG(\A)}\EE^{(s)}(g)\,\bar \varphi(g)\,\d g
 \label{SpectralExp-P-1}
 \\
&=\int_{\sG^\#(\A)\bsl \sG(\A)}\d g\int_{\sZ^\#(\A)\sG^\#(\Q)\bsl\sG^\#(\A)}
\sum_{d \in 2\Z_{\geq 0}} \biggl\{\int_{\sG^\#(\A_\fin)}{\rm Eis}(\sf_d^{(s)}\otimes \sf_{\phi}^{(s,\Lambda_\fin,\mu_\fin)},hx)f(\iota_\theta(x)^{-1}a_\fin k_\fin)\,\d x\biggr\}
\notag
\\
&\qquad \times \Phi_{l,d}^{(s,\mu)}(t)\,\bar \varphi(ha_\infty^{(\infty)}b_\R^\theta a_\fin k_\fin)\,\d h 
 \notag
\\
&=c_{\sG}\sum_{d=0}^{\infty} \int_{0}^{\infty}(\sh t)^{2}(\ch t)\,\Phi_{l,d}^{(s,\mu_\infty)}(t)\d t\, 
\int_{\fA\times K_\fin}\d \eta_{\fA}(a_\fin)\,\d k_\fin
 \notag
\\
&\quad \times 
\int_{\sG^\#(\A_\fin)}f(\iota_\theta(x)^{-1}a_\fin k_\fin)\,
\biggl\{\int_{\sZ^\#(\A)\sG^\#(\Q)\bsl \sG^\#(\A)}{\rm Eis}(\sf_{d}^{(s)}\otimes \sf_{\phi}^{(s,\Lambda_\fin,\mu_\fin)}, h x)\,\bar \varphi(h a_\infty^{(t)}b_\R^\theta a_\fin k_\fin) \,\d h\biggr\}\,\d x.
\notag
\end{align}}By Lemma \ref{RSInt-L2}, the integral in the bracket, which is 
$Z_{\bar \phi}^{(\bar s,\Lambda^{-1},\mu^{-1})}
(\sf_d\mid \varphi, a_\infty^{(t)}b_\R^\theta a_\fin k_\fin)$, equals the complex conjugate of \begin{align*}
\int_{\sZ^\#(\A)\sG^\#(\Q)\bsl \sG^\#(\A)}{\rm Eis}(\sf_{0}^{(\bar s)}\otimes \sf_{\bar \phi}^{(\bar s,\Lambda^{-1},\mu^{-1})}, h)\varphi(\iota_\theta(hx^{-1}) b_\R^\theta a_\fin k_\fin)\,\d h\times \
\,\Phi_{l,d}^{(-\bar s,\mu_\infty)}(t).
\end{align*}
Substitute this to the last expression of \eqref{SpectralExp-P-1}; then by combining the $a_\fin$-integral and $x$-integral over $\sG^\#(\A_\fin)$ to an integral over $\sG(\A_\fin)=\iota_\theta(\sG^\#(\A_\fin))\fA\bK_\fin$, we compute it as {\allowdisplaybreaks\begin{align*}
&\int_{\fA\times K_\fin}\d\eta_{\fA}(a_\fin)\,\d k_\fin \int_{\sG^\#(\A_\fin)}f(\iota_\theta(x)^{-1}a_\fin)\,\bar \varphi (\iota_\theta(h)b_\R^\theta \iota_\theta(x)^{-1}a_\fin)\d h \\
&=\int_{\sG(\A_\fin)} f(g_\fin)\,\bar \varphi (\iota_\theta(h)b_\R^\theta g_\fin)\,\d g_\fin
=R(f)\bar \varphi(\iota_\theta(h)b_\R^\theta).
\end{align*}}Hence, we obtain the following   
\begin{align*}
\langle 
\widehat{\F}_{\phi,f,l}^{(\beta, \Lambda,\mu)}\mid \varphi\rangle_{L^2}
={c_{\sG}}\,\int_{c-i \infty}^{c+i \infty} \beta(s)s(s^2-1)\,B_{l}(s)\,
\overline{Z^{(\bar s, \Lambda^{-1},\mu^{-1})}(1\mid R(\bar f)\varphi, b_\R^\theta)
}\,\d s
\end{align*}
with 
$$
B_l(s)=\sum_{d=0}^\infty\int_{0}^\infty(\sh 2t)^{2}(\ch 2t)\,\Phi_{l,d}^{(s,\mu_\infty)}(t)\,\overline{\Phi_{l,d}^{(-\bar s,\mu_\infty)}(t)}\,\d t. 
$$
To complete the proof, we use \cite[Propositions 30 and 17]{Tsud2011-1}.  
\end{proof}
Define 
\begin{align*}
\widehat \II_{\phi,f,l}{(\beta,\Lambda,\mu)}:=B^{T_\theta,\Lambda^{-1}}(\widehat\F_{f,l}^{(\beta,\Lambda,\mu)};b_\R^\theta), \quad \II_{l}(\phi,f\mid s,\Lambda,\mu):=B^{T_\theta,\Lambda^{-1}}(\F_{\phi,f,l}^{(s,\Lambda,\mu)};b_\R^\theta).
\end{align*}
By \eqref{FFSpectExp} 
\begin{align}
\II_{l}(\phi, f\mid s,\Lambda,\mu)=
\tfrac{2^5\pi^{1/2}}{\Gamma(3/2)}B_l{(s)}\,\sum_{\varphi \in \cB_l(\cK_f)} |D|^{s/2}\overline{Z^{(\bar s,\Lambda^{-1},\mu^{-1})}(1\mid R(\bar f)\varphi,b_\R^\theta)}\,B^{T_\theta,\Lambda^{-1}}(\varphi:b_\R^\theta). 
 \label{FFSpectExp-v2}
\end{align} 

\begin{lem} \label{IISpectEx} 
The function $s\longmapsto \II_{l}(\phi,f\mid s,\Lambda,\mu)$ is entire and satisfies the functional equation $\II_{l}{(\phi, f\mid s,\Lambda,\mu)}=\overline{\II_{
l}{(\bar \phi, \bar f \mid -\bar s,\Lambda^{-1},\mu^{-1})}}$; moreover, when multiplied by $e^{as^2}\,(a>0)$, it is vertically bounded. We have the contour integral expression
\begin{align}
\widehat \II_{\phi,f,l}{(\beta,\Lambda,\mu)}=
\int_{c-i \infty}^{c+i \infty} \beta(s)s(s^2-1)\II_{l}{(\phi, f\mid s,\Lambda,\mu)}\,\d s \quad (c>1). \label{IISpectEx-f2}
\end{align}
\end{lem}
\begin{proof} The functional equation of $\II_{l}{(\phi,f\mid s,\Lambda, \mu)}$ follows from \eqref{FeEisSer}, combined with $\overline{B_l(-\bar s)}=B_l(s)$ recalled above.\end{proof}

\section{Double coset decomposition} \label{sec:DoubleCoset}
Let $w_2:=\iota_\theta(\left[\begin{smallmatrix} 0 & 1 \\ 1 & 0 \end{smallmatrix}\right])$ be as in \eqref{Def-iotaf3}, and set 
\begin{align}
s_1:=\left[\begin{smallmatrix} 0 & 1 & 0 & 0 \\
 1 & 0 & 0 & 0 \\
0 &0& 0& 1 \\
0& 0& 1 &0 \end{smallmatrix} \right], 
\quad s_2:=
\left[\begin{smallmatrix}
0 & 0 & 1 & 0 \\
0 & 1 & 0 & 0 \\
-1 &0& 0& 0 \\
0& 0& 0 &1 \end{smallmatrix} \right], \quad w_1:=\left[\begin{smallmatrix} 1 & 0 & 0 & 0 \\
 0 & 0 & 0 & 1 \\
0 &0& 1& 0 \\
0& -1& 0 &0 \end{smallmatrix} \right].
 \label{Def:WeyELM}
\end{align}
The Weyl group $W:={\rm Norm}_{\sG}(\sT)/\sT$ of $\sG$ is generated by the images of $s_1$ and $s_2$. We have the relations
$$
w_1=s_1s_2s_1, \quad 
w_2=s_2s_1s_2\sm(\left[\begin{smallmatrix} -1 & -\tt \\ 0 & -1 \end{smallmatrix} \right],-1),
$$
which yields $\sP(\Q) w_1\sP(\Q)=\sP(\Q)s_2\sP(\Q)$ and $\sP(\Q)w_2\sP(\Q)=\sP(\Q)s_2s_1s_2\sP(\Q)$. Let $W_{\sM}$ denote the Weyl group of $\sM$. Then, since $\{1,s_2,s_2s_1s_2\}$ is a complete set of representatives of $W_{\sM}\bsl W/W_\sM$, we have the Bruhat decomposition \begin{align}
\sG(\Q)=\sP(\Q)\bigsqcup \sP(\Q) w_1 \sP(\Q)\bigsqcup \sP(\Q) w_2 \sP(\Q).
 \label{BruhatDec}
\end{align}
Set $\sB:=\{\left[\begin{smallmatrix} * & * \\ 0 & *\end{smallmatrix} \right]\}\cap {\bf GL}_2$. 
\begin{lem} \label{DoubleCoset-L1} The set $\{1_4,\,w_1,\, w_2,\,\sn(T_\theta^\dagger)w_2\}$ is a complete set of representatives of the double coset space $\iota(\sB^\#(\Q))\bsl \sG(\Q)/\sP(\Q)$. 
\end{lem}
\begin{proof} Since $\iota(\sB^\#(\Q))=\{\sm(I_\theta(\tau),\lambda)\sn(X_\beta) \mid  \tau \in  E^\times,\, \lambda \in \Q^\times,\,\beta\in E\}$, we have $\iota(\sB^\#(\Q)))\subset \sP(\Q)$, which means $\sP(\Q)=\iota(\sB^\#(\Q))\sP(\Q)$; moreover, since $\sV(\Q)=\Q T_\theta^\dagger \oplus \{X_\beta\mid \beta\in E\}$, any coset in $\iota(\sB^\#(\Q))\bsl \sG(\Q)$ is uniquely represented by a matrix of the form
$$
\sn(a T_\theta^\dagger)\,\sm(A,1)\quad (a\in \Q^\times,\,A\in I_{\theta}(E^\times)\bsl {\bf GL}_2(\Q)).
$$
Since $E$ is a quadratic field, we have ${\bf GL}_2(\Q)=I_{\theta}(E^\times)\sB(\Q)$. Note that $I_\theta(E^\times)\cap \sB(\Q)=\{a 1_2\mid a\in \Q^\times\}$. Thus, any coset in $I_\theta(E^\times)\bsl {\bf GL}_2(\Q)$ is represented by a unique matrix of the form  $\left[\begin{smallmatrix} 1 & x \\ 0 & d \end{smallmatrix}\right]$ with $x\in \Q$ and $d\in \Q^\times$. Therefore, 
\begin{align}
\iota(\sB^\#(\Q))\bsl \sG(\Q)=\{\sn(a T_\theta^\dagger)\sm(\left[\begin{smallmatrix} 1 & x \\ 0 & d \end{smallmatrix}\right],1)\mid  d\in \Q^\times,\,a,x\in \Q\}.
\label{DoubleCoset-L1-f0}
\end{align}
By the easily confirmed relations 
\begin{align*}
&w_1 \sm((\left[\begin{smallmatrix} 1 & x \\ 0 & 1 \end{smallmatrix}\right],1)=\sn(\left[\begin{smallmatrix} 0 &  x \\ x & 0\end{smallmatrix} \right]), \quad 
w_1 \sm(\left[\begin{smallmatrix} 1 & 0 \\ 0 & d \end{smallmatrix}\right],1)=\sm(\left[\begin{smallmatrix} 1 & 0 \\ 0 & d^{-1}\end{smallmatrix} \right]), \\
&\sn(a T_\theta^\dagger)w_1 \in \iota(\left[\begin{smallmatrix} 1 & a \bar \theta^{-1} \\ 0 & 1 \end{smallmatrix} \right])\,w_1 \sP(\Q),
\end{align*}
we obtain $\sP(\Q)w_1\sP(\Q)= \iota(\sB^\#(\Q)) w_1 \sP(\Q)$. \\
 By \eqref{DoubleCoset-L1-f0}, $w_2^{-1}\sP w_2\cap \sP=\sM$ and by $\sm(1_2,a)\sn(a T_\theta^\dagger)\sm(1_2,a)^{-1}=\sn(T_\theta^\dagger)$ for $a\in \Q^\times$, we have
$$
\sP(\Q)w_2\sP(\Q)=\iota(\sB^\#(\Q)) w_2 \sP(\Q) \cup \iota(\sB^\#(\Q))\sn(T_\theta^\dagger)w_2\sP(\Q).
$$
It remains to show that the union is disjoint. To argue by contradiction, suppose we are given the relation $
\sn(T_\theta^\dagger)w_2 \in \iota(\left[\begin{smallmatrix}\tau & 0 \\ 0 & \tau^{-1}\lambda \end{smallmatrix}\right] \left[\begin{smallmatrix} 1 & \beta \\ 0  & 1 \end{smallmatrix}\right] )w_2\sP(\Q)$ with $\tau \in E^\times,\lambda \in \Q^\times$ and $\beta\in E$. Then, 
$$
\sn(-X_\beta)\sm(I_\theta(\tau),\lambda)^{-1} \sn(T_\theta^\dagger)\in \sP(\Q)\cap w_2 \sP(\Q) w_2^{-1}=\sM(\Q),
$$
which in turn yields an impossible relation $X_\beta=\lambda \nr_{E/\Q}(\tau)^{-1}T_\theta^\dagger$.  
\end{proof}

\begin{lem} \label{DoubleCoset-L3}
We have the disjoint decomposition: 
\begin{align}
\iota(\sB^\#(\Q))\su \sP(\Q)&=\bigsqcup_{(\lambda, \delta )\in \fX(\su)}
\iota(\sB^\#(\Q))\, \su\, \sm(\delta, \lambda \det \delta)\,\sN(\Q) 
\label{DC-L3-f1}
\end{align}
for $\su\in\{1_4,w_2,w_1,\sn(T_\theta^\dagger)w_2\}$ with 
\begin{align}
\fX(\su):=\begin{cases}
\{1\} \times (I_\theta(E^\times)\bsl {\bf GL}_2(\Q))\quad &(\su=1_4), \\
\{1\} \times (I_\theta(E^\times)\bsl {\bf GL}_2(\Q))\quad &(\su=w_2), \\
\Q^\times \times(\sB(\Q)\bsl {\bf GL}_2(\Q)), \quad &(\su=w_1), \\
\Q^\times \times (I_\theta(E^\times)\bsl {\bf GL}_2(\Q)), \quad &(\su=\sn(T_\theta^\dagger)w_2).
\end{cases}
\label{def-fXsu}
\end{align}
\end{lem}
\begin{proof} Since $\sm(1_2,\lambda)=\iota(\diag(1,\lambda))$ for $\lambda \in \Q^\times$, we have that $\sP(\Q)$ is a union of $\iota(\sB^\#(\Q))\sm(\delta,\det \delta)\sN(\Q)$ ($\delta \in {\bf GL}_2(\Q))$. The relation   
$$\sm(\delta_1,\det \delta_1)=\iota(\left[ \begin{smallmatrix} \tau & 0 \\ 0 & \tau^{-1}\lambda \end{smallmatrix}\right]\left[ \begin{smallmatrix} 1 & \beta \\ 0 & 1\end{smallmatrix}\right])\sm(\delta_2,\det \delta_2)\sn(X)
$$
for $\delta_1,\delta_2\in {\bf GL}_2(\Q)$ implies $\delta_1=I_{\theta}(\tau)\delta_2$. Hence we get \eqref{DC-L3-f1} for $\su=1_4$.  

To obtain \eqref{DC-L3-f1} for $\su=w_1$, it suffices to show that the relation \begin{align}w_1 \sm(\delta_1,\lambda_1 \det \delta_1)=\iota(\left[ \begin{smallmatrix} \tau & 0 \\ 0 & \tau^{-1}\lambda \end{smallmatrix}\right]\left[ \begin{smallmatrix} 1 & \beta \\ 0 & 1\end{smallmatrix}\right])w_1 \sm(\delta_2,\lambda_2 \det \delta_2)\sn(X)
 \label{DC-L3-f5}
\end{align}
for $(\delta_j,\lambda_j) \in {\bf GL}_2(\Q) \times \Q^\times$ implies $\delta_1 \in \sB(\Q)\delta_2$ and $\lambda_1=\lambda_2$. Set $\beta=b_2+b_3\theta$ $(b_2,b_3\in \Q)$, $b_1=-b_2\tr_{E/\Q}(\theta)-b_3\nr_{E/\Q}(\theta)$ and $\tau=a+b\theta$ $(a,b\in \Q)$. Then, we see that the lower-left $2\times 2$-block of the matrix $w_1^{-1} \iota(\left[ \begin{smallmatrix} \tau & 0 \\ 0 & \tau^{-1}\lambda \end{smallmatrix}\right]\left[ \begin{smallmatrix} 1 & \beta \\ 0 & 1\end{smallmatrix}\right])w_1=w_1^{-1}\sm(I_{\theta}(\tau),\lambda)\sn(X_\beta)w_1$ equals $
\left[\begin{smallmatrix} 0 & \lambda b \nr_{E/\Q}(\tau)^{-1} \\ b & -bb_2-(a+\tt b) b_3 \end{smallmatrix} \right],
$ which should be the zero matrix due to the identity in \eqref{DC-L3-f5}. Hence $b=ab_3=0$. Since $b=0$ means $a=\tau\in \Q^\times$, we have $b_3=0$. Then, \eqref{DC-L3-f5} is written as 
$$
\sm(\delta_1,\lambda_1 \det \delta_1)= 
\left[\begin{smallmatrix} a &{} &{} &{} \\ {} & a^{-1}\lambda &{} &{} \\ {} & {}  &{a^{-1}\lambda} &{} \\ {} & {} &{} &a \end{smallmatrix} \right]
\left[\begin{smallmatrix} 1 &{-b_2} &{b_1} &{} \\ {} & 1 &{} &{} \\ {} & {}  &{1} &{} \\ {} & {} &{b_2} &1 \end{smallmatrix} \right] \,\sm(\delta_2,\lambda_2\det \delta_2)\sn(X), 
$$
which gives us the relations 
$$
\delta_1=\left[\begin{smallmatrix} a & {} \\ {} & a^{-1}\lambda \end{smallmatrix} \right] \left[\begin{smallmatrix}1 & {-b_2} \\ {} & 1 \end{smallmatrix} \right] \delta_2, \quad \lambda_1\det \delta_1=\lambda \lambda_2\det \delta_2. 
$$
The first formula implies $\det \delta_1=\lambda \det \delta_2$, which combined with the second formula yields $\lambda_1=\lambda_2$ as desired. 

Since $w_2 \sm(1_2,\lambda)w_2^{-1}=\iota(\diag(\lambda,1))$ $(\lambda \in \Q^\times)$, we have that $\iota(\sB^\#(\Q))w_2 \sP(\Q)$ is a union of the cosets $\iota(\sB^\#(\Q)) w_2 \sm(\delta,\det \delta)\sN(\Q)$ for $\delta \in {\bf GL}_2(\Q)$. To prove that \eqref{DC-L3-f1} for $\su=w_2$ is disjoint, it sufficed to show that
 \begin{align}w_2 \sm(\delta_1, \det \delta_1)=\iota(\left[ \begin{smallmatrix} \tau & 0 \\ 0 & \tau^{-1}\lambda \end{smallmatrix}\right]\left[ \begin{smallmatrix} 1 & \beta \\ 0 & 1\end{smallmatrix}\right])w_2 \sm(\delta_2, \det \delta_2)\sn(X) \label{DC-L3-f6}
\end{align}
for $\delta_j \in {\bf GL}_2(\Q)\,(j=1,2)$ implies $\delta_1 \in I_{\theta}(E^\times) \delta_2$. The identity in \eqref{DC-L3-f6} means that $w_2^{-1} \sn(X_\beta)w_2\in \sP(\Q)$, which yields $X_\beta=0$, or equivalently $\beta=0$. Since $
w_2^{-1} \iota(\left[ \begin{smallmatrix} \tau & 0 \\ 0 & \tau^{-1}\lambda \end{smallmatrix}\right])w_2=\sm(\lambda I_{\theta}(\tau^{-1}),\lambda)$, \eqref{DC-L3-f6} can be written as 
\begin{align*}
\sm(\delta_1,\det \delta_1)= \sm( I_\theta(\lambda \tau^{-1}),\lambda)\cdot\sm(\delta_2, \det \delta_2)\sn(X),
\end{align*}
which gives us $\delta_1\in I_{\theta}(E^\times)\delta_2$ as desired. 

To show the decomposition in \eqref{DC-L3-f1} for $\su=\sn(T_\theta^\dagger)w_2$, it suffices to prove that the relation 
\begin{align}
\sn(T_\theta^\dagger) w_2 \sm(\delta_1, \lambda_1 \det \delta_1)=\iota(\left[ \begin{smallmatrix} \tau & 0 \\ 0 & \tau^{-1}\lambda \end{smallmatrix}\right]\left[ \begin{smallmatrix} 1 & \beta \\ 0 & 1\end{smallmatrix}\right]) \sn(T_\theta^\dagger) w_2 \sm(\delta_2, \lambda_2 \det \delta_2)\sn(X) \label{DC-L3-f7}
\end{align}
for $(\delta_j, \lambda_j) \in {\bf GL}_2(\Q) \times \Q^\times \,(j=1,2)$ implies $\delta_1 \in I_{\theta}(E^\times) \delta_2$ and $\lambda_1=\lambda_2$. We have
\begin{align*}
(\sn(T_\theta^\dagger) w_2)^{-1} \iota(\left[ \begin{smallmatrix} \tau & 0 \\ 0 & \tau^{-1}\lambda \end{smallmatrix}\right] \left[ \begin{smallmatrix} 1 & \beta \\ 0 & 1\end{smallmatrix}\right])
\sn(T_\theta^\dagger) w_2 
=\sm(\lambda I_{\theta}(\tau^{-1}),\lambda)\cdot w_2^{-1} \sn(X_\beta+(1-\lambda \nr_{E/\Q}(\tau^{-1} ))T_\theta^\dagger)w_2. 
\end{align*}
By this and \eqref{DC-L3-f7}, the element $w_2^{-1} \sn(X_\beta+(1-\lambda \nr_{E/\Q}(\tau^{-1} ))T_\theta^\dagger)w_2$ should be in $\sP(\Q)$, which happens only when $X_\beta+(1-\lambda \nr_{E/\Q}(\tau^{-1}))T_\theta^\dagger=0$. Since $X_\beta$ and $T_\theta^\dagger$ are linearly independent, this yields $\lambda=\nr_{E/\Q}(\tau)$ and $\beta=0$. Then, \eqref{DC-L3-f7} becomes
\begin{align*}
\sm(\delta_1, \lambda_1\det \delta_1)=\sm(I_{\theta}(\bar \tau),\nr_{E/\Q}(\tau))\cdot\sm(\delta_2, \lambda_2 \det \delta_2)\sn(X),
\end{align*}
which gives us $\delta_1=I_{\theta}(\bar \tau)\,\delta_2$ and $\lambda_1\det \delta_1= \nr_{E/\Q}(\tau)\,\lambda_2\det \delta_2$. Thus, $\delta_1\in I_\theta(E^\times)\delta_2$ and $\lambda_1=\lambda_2$ as desired. \end{proof}

The proof of the following lemma is straightforward. 
\begin{lem} \label{DoubleCoset-L2}
\begin{itemize}
\item[(1)] Let $\gamma=\sm(A,\lambda)$ with $A\in {\bf GL}_2(\Q)$ and $\lambda \in \Q^\times$. Then, 
$$
\sN(\Q)\cap \gamma^{-1}\iota(\sB^\#(\Q))\gamma=\{\sn(X)\mid \tr({}^t A T_\theta A\,X)=0\}.
$$
\item[(2)] Let $\gamma=u\sm(A, \lambda)$ with $u\in \{w_1,w_2,\sn(T_\theta^\dagger)w_2\}$, $A\in {\bf GL}_2(\Q)$ and $\lambda \in \Q^\times$. Then, 
$$
\sN(\Q)\cap \gamma^{-1}\iota(\sB^\#(\Q))\gamma=\{1_4\}.
$$
\end{itemize}
\end{lem}

\section{The Bessel period of a Poincar\'{e} series} \label{sec:BesselPer}
In this section, suppose we are given a left $\sZ(\A)\iota(\sB^\#(\Q))$-invariant continuous function $\Phi$ on $\sG(\A)$ such that the Poincar\'{e} series \begin{align}
F(g):=\sum_{\gamma \in \iota(\sB^\#(\Q))\bsl \sG(\Q)} \Phi(\gamma g), \quad g\in \sG(\A)
\label{BesselPer-f0}
\end{align}
is convergent absolutely and locally uniformly when $g$ is in any given compact subset of $\sG(\A)$, so that $F(g)$ yields a left $\sZ(\A)\sG(\Q)$-invariant continuous function on $\sG(\A)$. We consider the Bessel period $B^{T_\theta, \Lambda^{-1}}(F;b_\R^\theta)$ (see \eqref{Def-BesselPer} and \eqref{b-def}). Note that coset spaces $E^\times \A^\times \bsl \A_E^\times$ and $\sN(\Q)\bsl \sN(\A)$ are compact, so that 
$$
 \int_{E^\times \A^\times\bsl \A_E^\times}\Lambda(\tau)\biggl\{\int_{\sN(\Q)\bsl \sN(\A)} \sum_{\gamma \iota(\sB^\#(\Q)\bsl \sG(\Q)} 
|\Phi (\gamma n \sm_{\theta}(\tau)b_{\R}^\theta)|\,\d n\biggr\}\,\d^\times \tau
<+\infty.
$$ 
Due to this, we apply Fubini's theorem in the following argument. By writing the sum in \eqref{BesselPer-f0} as 
$$
F(g)=\sum_{\gamma \in \iota(\sB^\#(\Q))\bsl \sG(\Q)/\sN(\Q)} \sum_{\delta \in \sN_\gamma(\Q) \bsl \sN(\Q)} \Phi(\gamma \delta g)
$$
with $ \sN_\gamma:=(\sN\cap \gamma^{-1}\iota(\sB^\#) \gamma$, and substituting this to \eqref{Def-BesselPer}, we have that $B^{T_\theta, \Lambda^{-1}}(F;b_\R^\theta)$ equals  
{\small \begin{align*}
&\int_{E^\times \A^\times\bsl \A_E^\times}\Lambda(\tau)\biggl\{
\sum_{\gamma \in \iota(\sB^\#(\Q))\bsl \sG(\Q)/\sN(\Q)} \int_{\sN(\Q)\bsl \sN(\A)} \sum_{\delta \in \sN_\gamma(\Q) \bsl \sN(\Q)} 
\Phi (\gamma \delta n \sm_{\theta}(\tau)b_{\R}^\theta)\,\psi_{T_\theta}(n)^{-1}\,\d n\biggr\}\,\d^\times \tau. 
\end{align*}}
Since $\psi_{T_\theta}$ is trivial on $\sN_{\gamma}(\Q)$ (see \eqref{def-psiT}), the $n$-integral and the $\delta$-summation are combined to an integral over $\sN_\gamma(\Q)\bsl \sN(\A)$; thus, 
 \begin{align}
&B^{T_\theta, \Lambda^{-1}}(F;b_\R^\theta)
\label{BesselPer-f3}
\\
&=\int_{E^\times \A^\times\bsl \A_E^\times}\Lambda(\tau)\biggl\{
\sum_{\gamma \in \iota(\sB^\#(\Q))\bsl \sG(\Q)/\sN(\Q)}\int_{\sN_\gamma (\Q)\bsl \sN(\A)}
\Phi (\gamma n \sm_{\theta}(\tau)b_{\R}^\theta)\,\psi_{T_\theta}(n)^{-1}\,\d n\biggr\}\,\d^\times \tau. 
\notag
\end{align}
 By Lemma \ref{DoubleCoset-L1}, there is a surjection $\iota(\sB^\#(\Q))\bsl \sG(\Q)/\sN(\Q) \rightarrow \{1_4,w_1,w_2,\sn(T_\theta^\dagger)w_2\}$. Let $\JJ_{\su}(\Phi)$ for $\su \in\{1_4,w_1,w_2,\sn(T_\theta^\dagger)w_2\}$ denote the sub-series-integral of \eqref{BesselPer-f3} obtained by reducing the $\gamma$-summation range to the fiber of $\su$. Then, $B^{T_\theta, \Lambda^{-1}}(F;b_\R^\theta)= \sum_{\su} \JJ_{\su}(\Phi)$. Recall the element $\sigma$ in \eqref{Def-sigma}. 

\begin{prop} \label{BesselPer-L1}
We have
{\allowdisplaybreaks
\begin{align}
\JJ_{1_4}(\Phi)&=
\sum_{\delta \in \{1_2, \sigma\}} \int_{E^\times \A^\times\bsl \A_E^\times}\Lambda(\tau)\biggl\{ \int_{\A} \psi(Da/2)\, \Phi ( \sm(\delta,\det \delta)\,\sn(a T_\theta^\dagger)\,\sm_{\theta}(\tau)b_{\R}^\theta)\,\d a \biggr\}\,\d^\times \tau, \label{BesselPer-L1-f1}
\\
\JJ_{\su}(\Phi)&=\int_{E^\times \A^\times\bsl \A_E^\times}\Lambda(\tau) \biggl\{ \int_{\sN(\A)} \psi_{T_\theta}(n)^{-1}\, \sum_{(\lambda,\delta) \in \fX(\su)} \Phi (\su \sm(\delta,\lambda \det \delta)\,n\,\sm_{\theta}(\tau)b_{\R}^\theta)\,\d n
\biggr\}\,\d^\times \tau, 
\label{BesselPer-L1-f2}
\end{align}}
for $\su \in \{w_2,w_1,\sn(T_\theta^\dagger)w_2\}$ with $\fX(\su)$ being as in \eqref{def-fXsu}.
\end{prop}
\begin{proof} This is obtained by Lemmas \ref{DoubleCoset-L3} and \ref{DoubleCoset-L2}; the expressions for $\JJ_{w_1}(\Phi)$, $\JJ_{w_2}(\Phi)$ and $\JJ_{\sn(T_\theta^\dagger)w_2}(\Phi)$ are direct. An argument is in order to arrive at the expression for $\JJ_{1_4}(\Phi)$. Initially, $\JJ_{1_4}(\Phi)$ is set to be 
$$
\int_{E^\times \A^\times\bsl \A_E^\times}\Lambda(\tau) \biggl\{ \int_{\sN^{[\delta]}(\Q)\bsl \sN(\A)}\sum_{{\delta\in I_\theta(E^\times)\bsl{\bf GL}_2(\Q)
}} \psi_{T_\theta}(n)^{-1}\, \Phi (\sm(\delta,\det \delta)\,n\,\sm_{\theta}(\tau)b_{\R}^\theta)\,\d n\biggr\}\,\d^\times \tau
$$
with $\sN^{[\delta]}(\Q):=\{\sn(X)\mid \tr({}^t\delta T_\theta \delta X)=0\}$. The $n$-integral can be written as iterated integrals first over $ \sN^{[\delta]}(\Q)\bsl \sN^{[\delta]}(\A)$ and then over $\sN^{[\delta]}(\A)\bsl \sN(\A)$. Since $n\mapsto \Phi(\sm(\delta,\det\delta)\,n\,\sm_\theta(\tau)b_\R^\theta)$ is left $\sN^{[\delta]}(\A)$-invariant, the integral over $\sN^{[\delta]}(\Q)\bsl \sN^{[\delta]}(\A)$ yields the factor $\int_{\sN^{[\delta]}(\Q)\bsl \sN^{[\delta]}(\A)}\psi_{T_\theta}(n')\,\d n'$, which is $0$ unless $\psi_{T_\theta}|\sN^{[\delta]}(\A)$ is trivial, i.e., $\tr(T_\theta X)=0$ for all $X\in \sV(\A)$ such that $\tr({}^t\delta T_\theta\delta X)=0$, or equivalently $(T_{\theta})^\bot=({}^t\delta T_\theta \delta)^{\bot}$. Here $(Y)^{\bot}:=\{X\in \sV(\Q) \mid \tr(YX)=0\}$ for $Y\in \sV(\Q)$. The equality $(T_\theta)^\bot=({}^t\delta T_\theta \delta)^\bot$ holds if and only if ${}^t\delta T_\theta \delta=\mu T_\theta$ for some $\mu \in \Q^\times $, i.e., $\delta\in {\bf GO}_{T_\theta}(\Q)$ (see \S\ref{sec:GO}). Due to \eqref{GOpoints}, the left-coset $I_\theta(E^\times)\delta$ is represented by either $1_2$ or by $\sigma$. As such, we have $\sN^{[\delta]}(\A)\bsl \sN(\A) \cong \{aT_\theta^\dagger \mid a\in \A\}$. To obtain the desired expression of $\JJ_{1_4}(\Phi)$, we also use $
\psi_{T_\theta}(\sn(aT_\theta^\dagger))=\psi(a\tr(T_\theta T_\theta^\dagger))=\psi(-aD/2)$. \end{proof}
Set
{\allowdisplaybreaks\begin{align}
\JJ_{w_2({\rm s})}(\Phi)&:=\int_{E^\times \A^\times\bsl \A_E^\times}\Lambda(\tau) \biggl\{ \int_{\sN(\A)} \psi_{T_\theta}(n)^{-1}\, \sum_{\delta \in \{1_2, \sigma\}} \Phi (w_2\sm(\delta,\det \delta)\,n\,\sm_{\theta}(\tau)b_{\R}^\theta)\,\d n\biggr\}\,\d^\times \tau,
 \label{BesselPer-f5}
\\
\JJ_{w_2({\rm r})}(\Phi)&:=\JJ_{w_2}(\Phi)-\JJ_{w_2({\rm s})}(\Phi). 
 \label{BesselPer-f6}
\end{align}}Then 
\begin{align}
B^{T_\theta, \Lambda^{-1}}(F;b_\R^\theta)=
(\JJ_{1_4}(\Phi)+\JJ_{w_2({\rm s})}(\Phi))+\JJ_{w_2({\rm r})}(\Phi)+
\JJ_{w_1}(\Phi)+\JJ_{\sn(T_\theta^\dagger)w_2}(\Phi). 
\label{BesselPer-f4}
\end{align}

\section{Proof of the main theorem} \label{sec:PrfMT}
In what follows, we fix characters $\Lambda\in \widehat{{\rm Cl}(E)}$ and $\mu=\otimes_{p}\mu_p:\A^\times/\Q^\times\R_{>0} \rightarrow \C^1$. Let $M$ denotes the conductor of $\mu$; we suppose that any $p\in S(M)$ is an odd prime inert in $E/\Q$.

\subsection{Functional equations}
Let $\phi\in \cS(\A_{E,\fin}^2)$ and $f\in \cH(\sG(\A))$. By Lemma \ref{MConvergenceL}, we can apply the result of \S\ref{sec:BesselPer} to $\Phi=\widehat{{\bf \Phi}_{\phi,f,l}}^{(\beta,\Lambda,\mu)}$. The following will be proved in \S\ref{sec:JJ14} and \S\ref{sec:MainTSg}. 
\begin{prop} \label{MainT-P1}
For $\su\in \{1_4,w_2({\rm s})\}$, there exists a meromorphic function $
\JJ_{l}^{\su}(\phi,f|s,\Lambda,\mu)$ of $s$ on $\C$, which becomes holomorphic and of vertically bounded when multiplied by $s(s^2-1)\widehat L(s+1,\Lambda\mu_E)$, and for any $c\in \R$ such that $1<c<l-2$, 
\begin{align}
\JJ_{\su}(\widehat {\Phi_{\phi,f,l}}^{(\beta,\Lambda,\mu)})&=\int_{c-i\infty}^{c+i\infty} |D|^{s/2}\widehat{L}(s+1,\Lambda\mu_{E})\,\JJ_{l}^{\su}(\phi,f\mid s, \Lambda,\mu)\,\beta(s)\,s(s^2-1)\d s \times e^{-2\pi \sqrt{|D|}}.
 \label{MainT-P1-f1}
\end{align}
Moreover, we have the functional equation: 
\begin{align}
|D|^{s/2}\widehat{L}(s+1,\Lambda\mu_{E})\,{\JJ}_{l}^{w_2({\rm s})}(\phi,f\mid s,\Lambda,\mu)=|D|^{-s/2}\widehat{L}(-s+1,\Lambda^{-1}\mu_{E}^{-1})\,{\JJ}_{l}^{1_4}(\widehat \phi,f\mid -s,\Lambda^{-1},\mu^{-1}).
\label{MainT-P1-f2}
\end{align}
\end{prop}
Concerning the remaining terms $\JJ_{\su}(\Phi)\,(\su\in \{w_2({\rm r}), w_1,\sn(T_\theta^\dagger)w_2\})$, we have the next proposition to be proved in \S\ref{sec:ErrorEST}. 

\begin{prop}
 For $\su \in \{w_2({\rm r}), w_1,\sn(T_\theta^\dagger)w_2\}$, there exists a holomorphic function $s\longmapsto \JJ_{l}^{\su}(\phi,f|s,\Lambda,\mu)$ on $\cT_{(2,l-2)}$, which becomes vertically bounded on $\cT_{(2,l-2)}$ when multiplied by $e^{s^2}$, such that 
\begin{align*}
\JJ_{\su}(\widehat{\Phi_{\phi,f,l}}^{(\beta,\Lambda,\mu)})=\int_{c-i \infty}^{c+i \infty} |D|^{s/2}\widehat L(s+1,\Lambda\mu_E)\JJ_{l}^{\su}(\phi,f|s,\Lambda)\,\beta(s)\,s(s^2-1)\,\d s \times e^{-2\pi \sqrt{|D|}}. 
\end{align*}
\end{prop}

\subsection{Error term estimates} \label{sec:ErrTEST}
We specify $\phi$ and $f$ as follows.
\begin{itemize}
\item[(i)] ({\it The choice of Schwartz-Bruhat functions})\\ 
 For $e\in \Z_{\geq 0}$ and a prime number $p$, define $\phi_p^{(e)} \in \cS(E_p^2)$ as
\begin{align}
\phi_p(x,y):=\cchi_{p^{e}\cO_{E,p}}(x)\,\cchi_{1+p^{e}\cO_{E,p}}(y), \quad (x,y)\in E_p^2.
\label{SBftn-f0}
\end{align}
 When $e=0$, $\phi_p^{(0)}$ is the characteristic function of $\cO_{E,p}^2$ so that $\widehat {\phi_p^{(0)}}=\phi_{p}^{(0)}$. When $e>0$, a computation reveals 
\begin{align}
\widehat {\phi^{(e)}_p}(x,y)=p^{-4e}\psi_{E_p}({x}{\sqrt{D}}^{-1})\,\cchi_{p^{-e}\cO_{E,p}\oplus p^{-e}\cO_{E,p}}(x,y).
\label{FrTr-phie}
\end{align}
Recall $\phi_\infty^{(0)}$ in \eqref{Def-phiInfd}. The function $\phi:=\phi_\infty^{(0)}\otimes (\bigotimes_{p<\infty}\phi_p^{({\rm ord}_p(M))}) \in \cS(\A_{E}^2)$ will be referred to as the Schwartz-Bruhat function associated with $\mu$. 

\item[(ii)] ({\it The choice of Hecke functions}) \\
Let $N$ be a square free integer prime to $M$. We suppose that any $p\in S(N)$ is inert in $E/\Q$. For $p<\infty$, define a function $f_p^0\in \cH(\sG(\Q_p))$ as  
$$
f_{p}^0:=\begin{cases} \cchi_{\bK_0(N\Z_p)}, \qquad &(p\nmid M), \\ 
\cchi_{b_p^{M}\,\bK_p}, \qquad &(p\mid M),
\end{cases}   
$$
where $b_p^{M}:=\left[\begin{smallmatrix} p^{e}1_2  & T_\theta^\dagger \\ 0 & 1_2 \end{smallmatrix} \right] \in \sP(\Q_p)$ with $e:={\rm ord}_p(M)$ for $p\mid M$. Note that $f_p^0=\cchi_{\bK_p}$ if $p \nmid NM$. Let $S$ be a finite set of primes disjoint from $S(DNM)$, define 
$$
f^{0,S}(g^S):=\prod_{p\not \in S}f_p^{0}(g_p), \quad g^S=(g_p)_{p\not\in S} \in \prod_{p\not\in S}\sG(\Q_p). 
$$
For $f_S=\otimes_{p\in S}f_p \in \cH_S$ (see \S\ref{sec:Def-HeckeAlg}) define $f:=R(\eta)(f_S\otimes f^{0,S})$, where $\eta=(\eta_p)_{p<\infty}$ with $\eta_p=1_4$ if $p\nmid N$ and $\eta_q:=
\left[\begin{smallmatrix} {} & {} & {} & {-1} \\
 {} & {} & {1} & {} \\
{} & {q} & {} & {} \\
{-q} & {} & {} & {}
\end{smallmatrix} \right]\in \sG(\Q_q)$, the Atkin-Lehner element for $q\mid N$. Since $\eta$ normalizes $\bK_0(N\widehat \Z)$, $f$ is right $\bK_0(N\widehat\Z)$-invariant. We mainly focus on the case when $N=1$ or a prime. \end{itemize} 
For this $(\phi,f)$, the terms ${\JJ}_{l}^{\su}(\phi,f\mid s,\Lambda,\mu)$ $(\su\in \{1_4,w_2({\rm s})\})$ will be computed exactly in \S\ref{sec:MainT}, whereas the term $\JJ_{l}^{\su}(\phi,f\mid s,\Lambda,\mu)$ for $\su\in \{w_2({\rm r}),w_1,\sn(T_\theta^\dagger)w_2\}$ are estimated in \S\ref{sec:ErrTermEstimate} as errors. In this section, we collect final outputs obtained in these highly technical sections and complete the proof of Theorems \ref{MainT-1} and \ref{MainT-2}.  Recall the quantities $\|f_S\|_1$ and $\fa(f_S)$ from \S\ref{sec:Def-HeckeAlg}. Set $\omega(N):=\# S(N)$.  

\begin{prop} \label{ErrorT-P000} For any compact interval $I\subset (1,l-2)$,
$$
|D|^{s/2}\widehat L(s+1,\Lambda\mu_{E})\,\JJ_{l}^{w_2({\rm r})}(\phi,f\mid s,\Lambda,\mu) \ll_{I} \tfrac{(2\pi\sqrt{|D|})^{l-1}}{\Gamma(l-1)} 2^{-l}\|f_S\|_1 N^{-\Re(s)}\fa(f_S)^{l-3}
$$
uniformly in $s\in\cT_{I}$, $l\in 2\Z_{\geq 3}$ and in $N$. 
\end{prop}

\begin{prop} \label{ErrorT-P1}
 Let $\su \in \{w_1,\sn(T_\theta^\dagger)w_2\}$. For any compact interval $I\subset (1,l-7)$, we have the majorizations: 
\begin{align*}
\JJ_{l}^{w_1}(\phi,f\mid s,\Lambda,\mu)&\ll_{I} \frac{C_1^{l}\,\|f_S\|_12^{\omega(N)}N^{3\Re(s)-2l+4}\,\fa(f_S)^{2l+\Re(s)-1}}{|\Gamma(l-s+1)|\,\Gamma(l-1/2)}, \\
\JJ_{l}^{\sn(T_\theta^\dagger)w_2}(\phi,f\mid s,\Lambda,\mu)&\ll_{I} (1+|s|)^2
\,\frac{l^2\,C_1^{l}\|f_S\|_1 2^{\omega(N)}N^{3\Re(s)-l+6}\,\fa(f_S)^{l+2\Re(s)-2}}{|\Gamma(l-s-1)|\Gamma(l-1/2)}
\end{align*}
uniformly in $s\in \cT_{I}$, $l\in 2\Z_{\geq 5}$ and in $N$, where $C_1>1$ is a constant independent of $N$, $l$ and $f_S$. 
\end{prop}
For proofs of these propositions, we refer to \S\ref{sec:JJw2reg}, \S\ref{sec:JJw1} and \S\ref{sec:JJTtw2}. Define 
\begin{align}
{\bf R}_{\phi,f,l}(s,\Lambda,\mu)&:=\tfrac{\Gamma(l-1)}{(2\pi \sqrt{|D|})^{l-1}}s(1-s^2)\Bigl\{\II_{l}(\phi,f\mid s,\Lambda,\mu)
\label{Def-ResidueTerm}
\\
&\quad
-|D|^{s/2}\widehat L(s+1,\Lambda\mu_{E})\bigl(\JJ_{l}^{1_4}(\phi,f\mid s,\Lambda,\mu)
+\JJ_{l}^{w_2({\rm s})}(\phi,f\mid s,\Lambda,\mu)\bigr)\Bigl\},
\notag
\end{align}
where $\II_l(\phi,f\mid s,\Lambda,\mu)$ is given by \eqref{FFSpectExp-v2}. 
\begin{prop} \label{ResidueEst-L1}
The function $s\longmapsto {\bf R}_{\phi,f,l}(s,\Lambda,\mu)$ is entire on $\C$ and satisfies the functional equation
$$
 \overline{{\bf R}_{\widehat \phi,f,l}(\bar s,\Lambda,\mu)}={\bf R}_{\phi,\bar f,l}(-s,\Lambda^{-1},\mu^{-1}).
$$
Let $\sigma>1$ and $l\in 2\Z_{\geq 5}$. The quantity $ e^{s^2}{\bf R}_{\phi,f,l}(s,\Lambda,\mu)$ is majorized by 
\begin{align}
 &\|f_S\|_1O_{l,\sigma}
\Bigl(N^{-\sigma}\fa(f_S)^{l-3}+2^{\omega(N)}N^{-2l+3\sigma+4} 
\fa(f_S)^{2l+\sigma-1}+2^{\omega(N)}N^{3\sigma-l+6}\fa(f_S)^{l+2\sigma-2}
\Bigr)
\label{ErrorT-P1-f0}
\end{align} 
for $s\in \cT_{(-1,1)}$ uniformly in $N$ and $f_S$. Given $T>1$, there exists $l_0\in 2\Z_{\geq 5}$ such that the implied constant can be taken to be $O_{{\rm T}, \sigma}({\rm T}^{-l})$ for $l\geq l_0$. 
\end{prop}   
\begin{proof}
The first two assertions follow from Propositions~\ref{IISpectEx} and \ref{ExplicitMT-L}. Note that our $\phi$ is real-valued. Let $I\subset (1,l-7)$ be any interval; $l\geq 10$ is necessary for $I\not=\emp$. The function $s(s^2-1)\, \widehat L(s+1,\Lambda\mu_{E})$ is entire on $\C$. We shall prove
\begin{align*}
{\bf R}_{\phi,f,l}(s,\Lambda,\mu)=&\tfrac{\Gamma(l-1)}{(2\pi \sqrt{|D|})^{l-1}}s(1-s^2)|D|^{s/2}\widehat L(s+1,\Lambda\mu_{E})\\
&\quad\times\left\{\JJ_{l}^{w_2({\rm r})}(\phi,f\mid s,\Lambda,\mu)+\JJ_{l}^{w_1}(\phi,f\mid 
 s,\Lambda,\mu)+\JJ_{l}^{\sn(T_\theta^\dagger)w_2}(\phi,f\mid s,\Lambda,\mu)\right\}
\end{align*}
on $\cT_{(1,l-7)}$. Let $F(s)$ denote the difference ; by Propositions \ref{ErrorT-P000} and \ref{ErrorT-P1}, $F(s)$ is entire and is bounded by $O_{f,l}\left(\exp(\tfrac{\pi}{2}|\Im(s)|\right)$ on $\cT_{I}$; moreover, by \eqref{BesselPer-f4} and Proposition \ref{MainT-P1} and Lemma \ref{IISpectEx}, we have $\int_{c-i \infty}^{c+i \infty} \beta(s)\,F(s)\,\d s=0$ for all $\beta(s)$ as in \eqref{EstBeta}, where $c\in I$ is fixed. Thus, $t\mapsto e^{(c+i t)^2}\,F(c+i t)$ is in $L^2(\R)$ orthogonal to all the functions $t\mapsto  \beta(c+i t)$, which are everywhere dense in $L^2(\R)$. Thus, $F(c+i t)=0$ for all $t\in \R$. Therefore, $F=0$ on $\cT_{(2,l-7)}$ as desired. First, we shall show \eqref{ErrorT-P1-f0} on the strip $\cT_{I}$; by Propositions \ref{ErrorT-P000} and \ref{ErrorT-P1}, $|e^{s^2}\SS_{\phi,f,l}(s,\Lambda,\mu)|$ for $s\in \cT_{I}$ is majorized uniformly in $l\in 2\Z_{\geq 5}$ and $N$ by 
\begin{align}
&e^{-|\Im(s)|^2}\times|(s^2-1)\widehat L(s+1,\Lambda\mu_{E})|\biggl\{\tfrac{C_2^l}{\Gamma(l-1)}+\tfrac{\Gamma(l-1)\,C_2^l}{|\Gamma(-s+l-1)|\Gamma(l-1/2)}
+\tfrac{(1+|s|)^2
\Gamma(l-1)\,l^2\,C_2^{l}}{|\Gamma(s+1)\Gamma(-s+l-1)|\Gamma(l-1/2)}
\biggr\}\label{ErP1-f1}\\
&\qquad \times C(\Re(s),f_S,N,l),
 \notag
\end{align}
where $C(x,f_S,N,l)$ denotes the quantity 
$$\|f_S\|_1\,\sup(N^{-x}
\fa(f_S)^{l-3}, 2^{\omega(N)}N^{-2l+3x+4} 
\fa(f_S)^{2l+x-1},
2^{\omega(N)}N^{3x-l+6}\fa(f_S)^{l+2x-2})
$$
and $C_2>1$ is a constant. We have that $s(s^2-1)\widehat L(s+1,\Lambda\mu_{E})$ is bounded on $\cT_{I}$; by Stirling's formula, we have $\Gamma(l-1)/\Gamma(l-1/2)\sim (l-1/2)^{-1/2}=O(1)$ as $l\rightarrow \infty$ and $$e^{-|\Im(s)|^2/2}\times (1+|s|)^2|\Gamma(s+1)|^{-1}=O(1), \quad s\in \cT_{I}.$$
Thus, \eqref{ErP1-f1} is further majorized by 
\begin{align*}
e^{-|\Im(s)|^2/2}\times \tfrac{C_2^{l}}{|\Gamma(l-s-1)|}\times (l^2+1)\times C(\Re(s),f_S,N,l)
\end{align*}
uniformly in $s\in \cT_{I}$, $l\in 2\Z_{\geq 5}$ and $N$. To estimate this, we apply Lemma~\ref{ErP1-L1}. As a consequence, for a given arbitrary ${\rm T}>1$, we have the following majorant of $e^{s^2}\SS_{\phi,f,l}(s,\Lambda,\mu)$ valid uniformly in $s\in \cT_{I}$, $l\in 2\Z_{\geq 5}$ and $N$. 
\begin{align}
e^{-|\Im(s)|^2/2}\times
C_2^{l}\left(\tfrac{l^{\Re(s)}}{\Gamma(l-1)}+\tfrac{(2C_2{\rm T})^{-l}}{|\Gamma(-s-1)|}\right)
\times (l^2+1)\times C(\Re(s),f_S,N,l)
. 
\label{ErP1-f11}
\end{align}
Thus, there exists $B(l)>0$ independent of $N$ and $f_S$ such that $|e^{s^2}\SS_{\phi,f,l}(s,\Lambda,\mu)|\leq B(l)\,
C(\Re(s),f_S,N,l)
$ for $s\in \cT_{I}$, $l\in 2\Z_{\geq 5}$ and $N$.
The dependence of $B(l)$ on a large enough $l$ is simplified as follows. Since $C_2^{l}(l^2+1)l^{\Re(s)+2}/\Gamma(l-1)=O_{{\rm T}}({\rm T}^{-l})$ and $e^{-|\Im(s)|^2/2}/\Gamma(-s-1)=O(1)$ by Stirling's formula and since $(l^2+1)2^{-l}=O(1)$, the expression \eqref{ErP1-f11} is majorized by $O_{{\rm T},I}({\rm T}^{-l}C(\Re(s),f_S,N,l)$ for large $l$, i.e., there exists $l_0$ such that $B(l)=O_{\rm T}({\rm T}^{-l})$ for $l\in 2\Z_{\geq l_0}$. By the functional equation, the bound 
$|e^{s^2}\SS_{\phi,f,l}(s,\Lambda,\mu)|\leq B(l)C(\Re(s),f_S,N,l)$ holds on the strip $\cT_{-I\cup I}$, a fortiori on the union of vertical lines $\Re(s)=\pm \sigma$ if $\sigma \in (1,l-7)$. Since $e^{s^2}\SS_{\phi,f,l}(s,\Lambda,\mu)$ is vertically of exponential growth on $\C$, the same bound remains valid on the vertical strip $\cT_{[-\sigma, \sigma]}$ by the Phragmen-Lindel\"{o}f principle. \end{proof}

\begin{lem} \label{ErP1-L1}
For any interval $I \subset (1,a-1)$ and for any ${\rm T}>1$, the following bound holds for all $s\in \cT_I,\,l\in \Z_ {\geq a}$.
\begin{align*}
{1}/{|\Gamma(l-s-1)|}
\ll_{{\rm T}} {l^{\Re(s)}}/{\Gamma(l-1)}+{{\rm T}^{-l}}/{|\Gamma(-s-1)|}.
\end{align*}
\end{lem}
\begin{proof} Set $\cT({\rm T})=\cT_I\cap\{|\Im (s)|\leq \rm T\}$. Since $\cT({\rm T})$ is relatively compact, Stirling's formula yields the uniform bound $
\left|{\Gamma(l-1)}/{\Gamma(l-s-1)}\right|\ll (l-1)^{\Re s}$ for $s\in \cT({\rm T})$ and $l\in \Z_{\geq 
 a}$. Suppose $s=\sigma+i t\in \cT_I-\cT(T)$. Then 
\begin{align*}
 |\Gamma(l-s-1)/\Gamma(-s-1)|&=\prod_{k=0}^{l}\bigl\{t^2+(-\sigma-1+k)^2\bigr\}^{1/2}\geq |t|^{l+1}\geq {\rm T}^{l+1}.
\qedhere
\end{align*}
\end{proof}

\subsection{Explicit formulas of the main terms}
Let $p\in S$. For any $f_p \in \cH_p$, define a function $\cW_p^{(T_\theta, \Lambda_p)}(f_p):\sG(\Q_p) \rightarrow \C$ by \begin{align}
\cW^{(T_\theta, \Lambda_p)}(f_p;g):=\int_{E_p^\times}\Lambda_p(\tau)\biggl\{\int_{\sN(\Q_p)} f_p(g^{-1} \sm_\theta(\tau)^{-1}n^{-1})\,\psi_{T_\theta,p}(n)\,\d n\biggr\}\,\d ^\times \tau, \quad g\in \sG(\Q_p), 
\label{MainT-f1}
\end{align}
where the measure on $\sN(\Q_p)\cong \sV(\Z_p)$ is normalized by $\vol(\sN(\Z_p))=1$, and set 
\begin{align}
\widehat {\cW^{(T_\theta, \Lambda_p)}}(f_p;\mu_p,s):=\int_{\Q_p^\times}\cW_p^{(T_\theta, \Lambda_p)}(f_p;\sm(1_2,\lambda))|\lambda|_p^{s-1}\,\mu_p(\lambda)\,\d^\times \lambda, \quad s\in \C.
 \label{MainT-f2}
\end{align} 
Since ${\rm supp }(f_p)\subset \sG(\Q_p)$ is compact, the integral in \eqref{MainT-f1} converges absolutely and $\lambda\mapsto \cW^{(T_\theta, \Lambda_p)}(f_p;\sm(1_2,\lambda))$ is compactly supported on $\Q_p^\times$; thus, $s\mapsto \widehat{\cW^{(T_\theta,\Lambda_p)}}(f_p;\mu_p, s)$ is entire and vertically bounded on $\C$. For $f$ as in \S\ref{sec:ErrTEST} (ii), set
\begin{align}
\MM_l(f\mid s,\Lambda,\mu):=&|D|^{s/2}\widehat {L}(s+1,\Lambda\mu_{E})\,\frac{\pi h_D}{4w_D}\,
\,\frac{(2\pi\sqrt{|D|})^{-s+l-1}}{\Gamma(-s+l-1)}\tilde \mu(N)^{-1}N^{s-1}
 \label{Def-MM}
\\
&\quad \times \tilde \mu(2)\tilde \mu\left({D}\right)^{-1}\,M^{s-6}\,
{G(\tilde \mu)}\zeta_{M}(1)\zeta_{M}(4)
\notag
\\
&\quad \times \{\delta_{\Lambda^2,{\bf 1}} \widehat{\cW_S^{(T_\theta,\Lambda)}}(f_S;\mu_S,s)+\widehat{\cW_S^{(T_\theta,\Lambda^{-1})}}(f_S;\mu_S,s)\}
\notag
\end{align}
with $\widehat{\cW_S^{(T_\theta,\Lambda)}}(f_S;\mu_S,s):=\prod_{p\in S}\widehat{ \cW_p^{(T_\theta,\Lambda_p)}}(f_p;\mu_p,s)$, $\zeta_M(s):=\prod_{p\mid M}(1-p^{-s})^{-1}$, and $G(\tilde \mu)$ is the Gauss sum of $\tilde \mu$. The following will be proved in \S\ref{sec:MainT}. 

\begin{prop}\label{ExplicitMT-L} For the pair $(\phi,f)$ defined in \S\ref{sec:ErrTEST}(i) and (ii), we have
\begin{align*}
|D|^{s/2}\widehat L(s+1,\Lambda\mu_{E})\JJ_{l}^{1_4}(\phi,f\mid s,\Lambda,\mu)&=\MM_l(f\mid s,\Lambda,\mu), \\
|D|^{s/2}\widehat L(s+1,\Lambda\mu_{E})\JJ_l^{w_2({\rm s})}(\phi,f\mid s,\Lambda,\mu)&=M^{-2s}\left(\tfrac{G(\tilde \mu^{-1})}{\sqrt{M}}\right)^2\,\MM_l(f\mid -s,\Lambda^{-1},\mu^{-1}).
\end{align*}
\end{prop}
 By \eqref{Def-MM} and by the last assertion of Proposition \ref{MellinWphi}, the following is immediate.

\begin{lem} \label{ErrorTerm-L1(Mod)} The constant term at $s=0$ of the sum
$$\MM_l(f\mid s,\Lambda,\mu)+M^{-2s}\left(\tfrac{G(\tilde \mu^{-1})}{\sqrt{M}}\right)^2\,\MM_l(f\mid -s,\Lambda^{-1},\mu^{-1})$$
is equal to  
{\small
\begin{align*}
&N^{-1}\,M^{-6}\zeta_M(1)\zeta_M(4)\,\tfrac{\pi h_D}{4w_D}
\tfrac{(2\pi \sqrt{|D|})^{l-1}}{\Gamma(l-1)}\,(2\pi)^{-1}(\delta_{\Lambda^2,{\bf 1}}+1)\,\tilde \mu(2)\tilde \mu(DN)^{-1}G(\tilde \mu)\\
&\cdot\biggl\{L(1,\Lambda\mu_E)
\prod_{p\in S}\widehat{\cW^{(T_\theta,\Lambda)}}(f_p,\mu,0)+\tilde\mu(4)^{-1}\tilde\mu(DN)\left(\tfrac{G(\tilde \mu^{-1})}{\sqrt{M}}\right)^4 L(1,\bar \Lambda\bar \mu_{E})
\prod_{p\in S}\widehat{\cW^{(T_\theta,\Lambda)}}(f_p,\bar \mu,0)\biggr\}
\end{align*}}
if $\Lambda\mu_{E}\not={\bf 1}$, and to 
{\small \begin{align*}
&N^{-1}L(1,\eta_D)\,\tfrac{ \pi h_D}{w_D}\tfrac{(2\pi \sqrt{|D|})^{l-1}}{\Gamma(l-1)}\,(2\pi)^{-1} \\
&\cdot\biggl\{\prod_{p\in S}\widehat{\cW^{(T_\theta,{\bf 1})}}(f_p,{\bf 1},0)
\left(\log q-\log(4\pi^2)+\psi(l-1)+\tfrac{L'(1,\eta_D)}{L(1,\eta_D)}\right)
+\left(\tfrac{\d}{\d s}\right)_{s=0}\prod_{p\in S}\widehat{\cW^{(T_\theta,{\bf 1})}}(f_p,{\bf 1},0)\biggr\}
\end{align*}}
if $\Lambda\mu_{E}={\bf 1}$, which means both $\Lambda$ and $\mu$ are trivial. 
\end{lem}

\subsection{Evaluation of the $S$-factor} \label{EVSfactro}
Let $p\in S$. Recall the notation in \S\ref{sec:Intro2}. For any $\Lambda_p\in \widehat {E_p^\times/\Q_p^\times}$, ${\mathcal {AI}}(\Lambda_p)$ denotes its lift to ${\bf GL}_2(\Q_p)$, so that $L(s,\Lambda_p)=L(s,{\mathcal {AI}}(\Lambda_p))$. 
%Since $\Lambda_p|\Q_p^\times$ is trivial, the central character of ${\mathcal AI}(\Lambda_p)$ is trivial, which in turn means ${\mathcal AI}(\Lambda_p)$ is self dual. 
We need the notion of the stable integral (\cite{Liu}). Let ${\rm OC}(\sN(\Q_p))$ denote the set of all the open compact subgroups of $\sN(\Q_p)$. From \cite[Proposition 3.1]{Liu}, for any $\tau \in E_p^\times$ and $g\in \sG(\Q_p)$, the function $\varphi(n):=\omega_p(y; \sm_\theta(\tau) n g)\,\psi_{T_\theta}(n)$ on $\sN(\Q_p)$ is compactly supported after averaging, where $\omega_p(y;-)$ is as in \eqref{Def-SphFtn}. Thus there exists ${\mathcal U}_0(\tau,g) \in {\rm OC}(\sN(\Q_p))$ such that for any ${\mathcal U}\in {\rm OC}(\sN(\Q_p))$ containing ${\mathcal U}_0(h_0,g)$, the integral $\int_{{\mathcal U}}\varphi(u)\d u$ is independent of ${\mathcal U}$; the stable integral $\int_{\sN(\Q_p)}^{\rm st}\varphi(n)\d n$ is defined to be this common value of integrals. Set $\Delta_{\sG,p}:=\zeta_p(1)^{-1}\zeta_p(2)\zeta_p(4)$. 

\begin{lem}
Let $y\in [Y_p^0]$. Then, for any $g\in \sG(\Q_p)$, 
\begin{align}
\int_{E_p^\times/\Q_p^\times}
\left|\int_{\sN(\Q_p)}^{\rm st}
\omega_p(y; g^{-1} \sm_\theta(\tau)^{-1}n^{-1})\,\psi_{T_\theta}(n)\,\d n \right|\,\d^\times \tau<+\infty.
 \label{Liu-f1}
\end{align}
We have 
{\small \begin{align}
\int_{E_p^\times/\Q_p^\times} \omega_p(\tau)\,\biggl\{
\int_{\sN(\Q_p)}^{\rm st}\omega_p(y^{-1}; \sm_\theta(\tau)^{-1} n^{-1})\,
\psi_{T_\theta}(n)\,\d n\biggr\} \,\d^\times \tau
=\frac{\Delta_{\sG,p}\,
L(1/2, \pi_p^{\rm ur}(y)\times {\mathcal AI}(\Lambda_p))}
{L(1,{\mathcal AI}(\Lambda_p) ;{\rm Ad})\,L(1,\pi^{\rm ur}(y);{\rm Ad})}. 
\label{Liu-f2}
\end{align}}
\end{lem}
\begin{proof} Due to $\sr:\sZ(\Q_p)\bsl \sG(\Q_p)\cong{\bf SO}(Q)_{\Q_p}$, this follows from \cite[Theorems 2.1 and 2.2]{Liu}. 
\end{proof}

\begin{prop} \label{MellinWphi}
 Set $\sigma_p:={\mathcal {AI}}(\Lambda_p)$. Let $\mu_p$ be an unramified character of $\Q_p^\times$. For any $f_p\in \cH_p$ and for $\Re(s)>-1/2$, 
\begin{align}
&\widehat{\calW^{(T_\theta,\Lambda_p)}}(\check f_p;\mu_p,s)
 \label{MellinWphi-f0}
\\
&=
\int_{[Y_p^0]}\widehat{f_p}(y)
\,\frac{\Delta_{\sG,p}\,L(1/2, \pi_p^{\rm ur}(y)\times \sigma_p)}{L(1,\sigma_p;{\rm Ad})\,L(1,\pi_p^{\rm ur}(y);{\rm Ad})}\frac{L(s+1/2,\pi_p^{\rm ur}(y)\otimes \mu_p)}{L(s+1,\sigma_p\otimes \mu_p))}\,\d\mu_p^{\rm Pl}(y),
\notag
\end{align}
where $\check f_p(g)=f_p(g^{-1})$. We have $\widehat{\calW^{(T_\theta,\Lambda_p)}}(\check f_p;\mu_p,s)=
\widehat{\calW^{(T_\theta,\Lambda_p^{-1})}}(\check f_p;\mu_p,s)$. 
\end{prop}
\begin{proof} 
Since ${\rm supp}(f_p)\subset \sG(\Q_p)$ is compact, by the Iwasawa decomposition, there exists $\cN_0(g) \in {\rm OC}(\sN(\Q_p))$, independent of $\tau\in E_p^\times$, such that $\sN(\Q_p)$ in \eqref{MainT-f1} is replaced with $\cN_0(g)$. Then, substituting \eqref{Def-FInvF} to \eqref{MainT-f1}, we have that $\calW_p^{(T_\theta,\Lambda_p)}(\check f_p;g)$ equals 
\begin{align*}
&\int_{E_p^\times/\Q^\times_p} \Lambda_p(\tau)\,\d \tau^\times \int_{\cN_0(g)} \biggl\{\int_{
[Y_p^0]} \widehat{f_p} (y)\omega_p(y^{-1}; g^{-1} \sm_\theta(\tau)^{-1}n^{-1})\,\d\mu_p^{\rm Pl}(y)\biggr\}\,\psi_{T_\theta}(n)\,\d n. 
\end{align*}
Since $\cN_0(g)$ and $[Y_p^0]$ are both compact, we exchange the order of integrals to see that the last integral becomes
\begin{align*}
\int_{E_p^\times/\Q_p^\times}
\Lambda_p(\tau)\,\d^\times \tau \int_{[Y_p^0] }\widehat{f_p
} (y) \biggl\{ \int_{\cN_0(g)} \omega_{p}(y^{-1}; g^{-1}\sm_\theta(\tau)^{-1} n^{-1})\,\psi_{T_\theta}(n)\,\d n \biggr\}\,\d\mu_p^{\rm Pl}(y).
\end{align*}
For each $\tau \in E_p^\times$, the $n$-integral is unchanged when $\cN_0(g)$ is enlarged to any $\cN \in {\rm OC}(\sN(\Q_p))$ containing $\cN_0(g)\,{\mathcal U}_0(\tau,g)$; then the $n$-integral over such $\cN$ is identified with the stable integral by definition. Therefore, we have that  $\calW_p^{(T_\theta,\Lambda_p)}(\check f_p;g)$ equals  
\begin{align*}
&\int_{E_p^\times/\Q_p^\times}\Lambda_p(\tau)\,\d^\times \tau \int_{[Y_p^0]}\widehat{f_p}(y) \biggl\{ \int_{\sN(\Q_p)}^{\rm st} \omega_p(y^{-1}; g^{-1}\sm_\theta(\tau)^{-1}n)\,\psi_{T_\theta} (n)\,\d n \biggr\}\,\d\mu_p^{\rm Pl}(y) \\
&=\int_{[Y_p^0] }\widehat{f_p}(y) \biggl\{\int_{E_p^\times/\Q_p^\times}\Lambda_p(\tau)
\,\d ^\times\tau \, \int_{\sN(\Q_p)}^{\rm st} \omega_p(y^{-1}; (g^{-1}\sm_\theta(\tau)^{-1} n^{-1})\,\psi_{T_\theta}(n)\,\d n \biggr\}\,\d\mu_p^{\rm Pl}(y) \\
&=\int_{[Y_p^0]}\widehat{f_p}(y) B^{\Lambda_p}(y^{-1};g) \d \mu_{p}^{\rm Pl}(\nu),\end{align*}
where 
$$
B^{\Lambda_p}(y; g):=\int_{E^\times_p/\Q_p^\times}\Lambda_p(\tau)\,\d^\times \tau \, \int_{\sN(\Q_p)}^{\rm st} \omega_p(y; g^{-1} \sm_\theta(\tau)^{-1} n^{-1})\,\psi_{T_\theta}(n)\,\d n.
$$
In the above computation, the second equality is legitimized by \eqref{Liu-f1}. When viewed as a function in $g\in \sG(\Q_p)$, $B^{\Lambda_p}(y;g)$ possesses the properties:
\begin{itemize}
\item[(i)] $B^{\Lambda_p}(y;\sm_\theta(\tau) n g k)=\Lambda_p(\tau)\psi_{T_\theta}(n) B^{\Lambda_p}(y;g)$ for $(\tau,n,g,k)\in E_p^\times \times \sN(\Q_p)\times \sG(\Q_p)\times \bK_{p}$. 
\item[(ii)] $B^{\Lambda_p}(y)* f=\widehat{f}(y)\,B^{\Lambda_p}(y)$ for $f \in \cH_p$.
\end{itemize}
By \cite[Theorem 0.5]{MSK}, the $\C$-vector space of all the functions satisfying these two\ properties form a one dimensional space generated by a unique function $B_0^{(\Lambda_p,y)}$ such that $B^{(\Lambda_p,y)}_0(1_4)=1$. Thus, $B^{\Lambda_p}(y;g)=C\,B_0^{(\Lambda_p,y)}(g)$ for all $g\in \sG(\Q_p)$ with $C=B^{\Lambda_p}(y;1_4)$ given by \eqref{Liu-f2}. Hence, for $g\in \sG(\Q_p)$, 
\begin{align*}
\calW_p^{(T_\theta,\Lambda_p)}(\check f_p;g)&=\int_{[Y_p^0]}\widehat{f_p}(y)
B^{\Lambda_p}(y^{-1};g)\,\d \mu_p^{\rm Pl}(y)
\\
&=\int_{[Y_p^0]}\widehat{f_p}(y)\, \frac{\Delta_{\sG,p}L(1/2, \pi_p^{\rm ur}(y^{-1}) \times \sigma_p)}
{L(1,\sigma_p;{\rm Ad})\,L(1,\pi_p^{\rm ur}(y^{-1});{\rm Ad})}\,B_0^{(\Lambda_p,y^{-1})}(g)\,\d \mu_p^{\rm Pl}(y).
\end{align*}
By \cite[Proposition 2]{Sugano85}, 
 \begin{align*}
\int_{\Q_p^\times} B_0^{(\Lambda_p,y)}(\sm(t;1))\mu_p(t)|t|_p^{s-3/2}\,\d^\times t
=\frac{L(s,\pi_p^{\rm ur}(y)\otimes \mu_p)}{L(s+\frac{1}{2},\sigma_p\otimes \mu_p)}
\end{align*}
for $\Re (s)\gg 0$. Since $[Y_p^0]$ is compact, by changing the order of integrals and by applying this formula, we get \eqref{MellinWphi-f0} for $\Re(s) \gg 0$. The right-hand side of \eqref{MellinWphi-f0} defines a holomorphic function on $\Re(s)>-1/2$. Thus, by analytic continuation, the same formula is true for $\Re(s)>-1/2$. Since $\sigma_p$ is self-dual, from \eqref{MellinWphi-f0}, the last assertion follows. \end{proof}

\begin{thm} \label{MainT-0} The asymptotic formula in \eqref{MainT-2-f0} is true for the whole family $\Pi_{\rm cusp}(l,N)$ with a fixed $l\in 2\Z_{\geq 5}$ and for $\alpha=\widehat {f_S}$ $(f_S \in \cH_S)$ with the error term 
\begin{align}
O_{l,\sigma}
\Bigl(\|f_S\|_1(
N^{-\sigma}\fa(f_S)^{l-3}+2^{\omega(N)}N^{-2l+3\sigma+4} 
\fa(f_S)^{2l+\sigma-1}+2^{\omega(N)}N^{3\sigma-l+6}\fa(f_S)^{l+2\sigma-2}\Bigr)
\label{MainT-0-f0}
\end{align}
for any $\sigma\in (1,l-7)$, where the implied constant is independent of $N$ and $f_S$. For any ${\rm T}>1$, the implied constant is $O_{{\rm T},\sigma}({\rm T}^{-l})$ for large $l$. 
\end{thm}
\begin{proof} 
By \eqref{Def-ResidueTerm} and \eqref{ErrorT-P1-f0} applied to $s=0$, 
\begin{align*}
&\tfrac{\Gamma(l-1)}{(2\pi \sqrt{|D|})^{l-1}}\,\II_{l}(\phi,f\mid 0,\Lambda,\mu) \\
&=
\tfrac{\Gamma(l-1)}{(2\pi \sqrt{|D|})^{l-1}}
\lim_{s\rightarrow 0}\left\{\MM_l(f|s,\Lambda,\mu)+M^{-2s}\left(\tfrac{G(\tilde \mu^{-1})}{\sqrt{M}}\right)^2 \MM_l(f|-s,\Lambda^{-1},\mu^{-1})\right\}+\eqref{MainT-0-f0}. 
\end{align*}
Note that, due to \eqref{Def-RSInt} and \eqref{Ksphericalsection}, the quantity $|D|^{s/2}Z^{(s,\Lambda,\mu)}(1\mid \varphi,b_\R^\theta)$ in \eqref{FFSpectExp-v2} equals the pairing $|D|^{-1/2}\langle E(\phi_\infty^{(0)}\otimes \phi,s,\Lambda,\mu),R(b_\R^\theta)\varphi\rangle$ defined in \cite[\S3.2]{KugaTsuzuki}. The limit on the right-hand side is computed in Lemma \ref{ErrorTerm-L1(Mod)}. By \cite[Theorem 5.2]{KugaTsuzuki}, together with \eqref{FFSpectExp-v2}, \eqref{ValueBl0} and \eqref{Intro-defLambda}, the left-hand side is computed to be
\begin{align*}
&{2^{-2l+3}\pi^{-2l+\frac{7}{2}}|D|^{-l+\frac{3}{2}}}w_D^{-2}\Gamma\left(l-\tfrac{3}{2}\right)\Gamma(l-2)\times\tilde \mu(N)N^{-1}\prod_{q\in S(N)}(1+q^{-2})^{-1}\,M^{-6}\\
&\cdot \zeta_M(1)\zeta_M(4)\,
\frac{\tilde \mu(2)\tilde \mu(D)G(\tilde \mu)}{[\bK_\fin:\bK_0(N)]}\sum_{\pi \in \Pi^{(T_\theta,\Lambda)}_{\rm cusp}(N,l)} L\left(\tfrac{1}{2},\pi,\mu\right)\widehat{f_S}(\pi_S) \frac{|R(E,\Lambda^{-1},\varphi_\pi^0)|^2}{\langle \varphi_\pi^0 \mid \varphi_\pi^0\rangle_{L^2}}
\ft(\pi,\mu). \qedhere\end{align*}
\end{proof}

\subsection{Proof of Theorems~\ref{MainT-1}, \ref{MainT-2} and Corollary~\ref{MainT-3}}\label{sec:PrfMTHM} 
Theorem \ref{MainT-1} follows from Theorem \ref{MainT-0} with $\sigma=\frac{l-6}{4}$, which is in $(1,l-7)$ only if $l>10$, and with $S=\emp$ and ${\rm T}=2$. To prove Theorem \ref{MainT-2}, we invoke the following results: Suppose $\pi\cong \otimes_{p\leq \infty}\pi_p \in \Pi_{\rm cusp}(l,N)$ is non CAP.  
\begin{itemize}
\item[(i)] ({The Ramanujan property} \cite{Weissauer}): For all $p\nmid N$, $\pi_p$ is tempered. 
\item[(ii)] ({The existence of a transfer to ${\bf {GL}}_4$} \cite{Arthur} {\it cf}. \cite{Schmidt}) Either there exists an irreducible cuspidal automorphic representation $\Pi$ of ${\bf GL}_4(\A)$ of symplectic type such that $L(s,\Pi)=L(s,\pi)$, or there exists a pair of irreducible cuspidal automorphic representations $(\sigma_1,\sigma_2)$ of ${\bf GL}_2(\A)$ such that $L(s,\sigma_1)L(s,\sigma_2)=L(s,\pi)$. Then, by \cite{Lapid}, $L\left(\tfrac{1}{2},\pi,\mu\right)\geq 0$ for any $\mu$ with $\mu^2={\bf 1}$. 
\item[(iii)] {Refined form of Boecherer's conjecture} due to Liu \cite{Liu}, further computed by \cite[Theorem 3.10]{DPSS}, and proved by Furusawa and Morimoto \cite{FurusawaMorimoto3}. 
\end{itemize}
Set $[Y^0_S]:=\prod_{p\in S}[Y_p^0]$. The space $C([Y^0_S])$ is endowed with the topology of uniform convergence. Let $\mu_{l,N}^{\rm G,new}$ denote the linear functional on $C([Y_S^0])$ defined by the left-hand side of \eqref{MainT-2-f0}; we define yet another linear functional $\mu_{l,N}$ by extending the summation in \eqref{MainT-2-f0} to $\Pi_{\rm cusp}(l,N)$. Let ${\mes}={\mes}_{S}^{\Lambda,\mu}$, i.e., $
{\mes}(\alpha)=\int_{[Y_S^0]}\tilde \alpha(y)\,\fD_S(y)\,\d\mu_S^{\rm Pl}(y)$ for $\alpha \in C([Y_S^0])$
with $\fD_S(y):=\prod_{p\in S}\fD_p(y_p)$. When $\mu^2={\bf 1}$, the inequality $\fD_S(y)\geq 0$ for all $y\in [Y_S^0]$ is easily observed. In what follows, $\lim_{l,N}$ means the limit as $N\rightarrow$ or $N=1,l\rightarrow \infty$. By Theorem \ref{MainT-0}, we have $ \lim_{l,N}\mu_{l,N}(\widehat {f_S})=C\,{\mes}({\widehat{f_S}})$ with $C$ being $2L(1,\eta_D)$ or $2L(1,{\mathcal AI}(\Lambda)\times \mu)$, and by \cite[Theorem 5.4]{KugaTsuzuki}, we have $\lim_{l,N}(\mu_{l,N}(\widehat {f_S})-\mu_{l,N}^{\rm G,new}(\widehat{f_S}))=0$ for all $f_S \in \cH_S$. 
Note that when $N=1$ there is no contribution from the Yoshida lifts. Hence the limit formula \eqref{MainT-2-f0} is true for $\alpha=\widehat {f_S}$ with $f_{S}\in \cH_S$. By (ii), the measure $\mu_{l,N}^{\rm G,new}$ is non-negative. Hence, by an argument involving the Weierstrass approximation theorem (\cite{Serre}), we extend the result to all $\alpha \in C([Y_S^0])$. Corollary \ref{MainT-3} is deduced from Theorem \ref{MainT-2} by invoking (iii). \qed

\section{An application} \label{sec:APPL}
Let $p$ be a fixed odd prime number. Recall notation from \S\ref{sec:Def-HeckeAlg}. Let $c_1$, $c_2$ and $c_0$ be the characteristic functions of the double cosets $\bK_p\varpi(1,1;1)\bK_p$, $\bK_p\varpi(1,0;0)\bK_p$ and $\bK_p \varpi(1,1;2)\bK_p$, respectively. Note that $\varpi(1,1;2)=p 1_4 \in \sZ(\Q_p)$. We identify the dual $(\widehat \sG(\C),\widehat \sB(\C),\widehat \sT(\C))$ with $(\sG(\C), \sB(\C), \sT(\C))$ as usual. Let a general point of ${\bf T}(\C)$ as $\check t=\diag(p^{-s_1}, p^{-s_2}, p^{-s_0+s_1}, p^{-s_0+s_2})$ with $(s_0,s_1,s_2)\in \C^3$, and define a character $\chi_{\check t}:{\bf B}(\Q_p)\rightarrow \C^\times$ as $\chi_{\check t}(\diag(t_1,t_2,\lambda t_1^{-1}, \lambda t_2^{-1})n)=|t_1|_p^{-s_1-s_2+s_0}|t_2|_p^{-s_1+s_2}|\lambda|_p^{s_1}$. Let $\pi_{\check t}$ denote the unique irreducible $\bK_p$-spherical subquotient of ${\rm Ind}_{{\bf B}(\Q_p)}^{\sG(\Q_p)}(\chi_{\check t})$; $\check t \pmod W$ is the Satake parameter of $\pi_{\check t}$. Note that the central character of $\pi_{\check t}$ is trivial if and only of $\check t \in {\bf Sp}_2(\C)$; such $\check t$ is dentified with an element of $(\C^\times)^2/W$ under the map  
$(\C^\times)^2/W\ni (a,b) \mapsto \diag(a,b,a^{-1},b^{-1}) \in {\bf T}(\C)\cap {\bf Sp}_2(\C)/W$ (see \S\ref{sec:Intro2}). The spherical Fourier transform $\widehat {f_p}$ of $f_p\in \cH_p$ is defined to be the eigenvalue of $\pi_{\check t}(f)$ on $(\pi_{\check t})^{\bK_p}\cong \C$. Then, by \cite[]{Satake}, the map $f_p\mapsto \widehat {f_p}$ ({\it cf}. \S\ref{sec:Intro2}) yields a $\C$-algebra isomorphism from $\cH_p$ onto $\C[p^{-s_0},p^{-s_1},p^{-s_2}]^{W}$; moreover, $\{X_j:=\widehat {c_j}\mid j=0,1,2\}$ yields an algebraically independent generator of the $\C$-algebra $\C[p^{-s_0},p^{-s_1},p^{-s_2}]^{W}$.  

\begin{defn}
For ${\bf k}=(k_0,k_1,k_2)\in \Z_{\geq 0}^3$, let $f_{p,{\bf k}} \in \cH_p$ be the inverse image of the monomial $X_0^{k_0}X_1^{k_1}X_2^{k_2}\in \C[X_0,X_1,X_2]$ by the Fourier transform. 
 \end{defn}
\begin{lem}\label{APPL-L} $\|f_{p,{\bf k}}\|_1\leq p^{17(k_1+k_2)}$ and $\fa(f_{p,{\bf k}})\leq p^{\frac{34}{3}(k_1+k_2)+\frac{26}{3}}$
\begin{comment}, $\fc(f_{p,{\bf k}})=p^{2k_0+k_2}$
\end{comment}
for ${\bf k}\in \Z_{\geq 0}^{3}$.
\end{lem}
\begin{proof} We have $\|f*f_1\|_1\leq \|f\|_1\|f_1\|_1$ for any $f,f_1\in \cH_p$. Applying this inequality successively to $f_{p,{\bf k}}=c_0^{*k_0}*c_1^{*k_1}*c_2^{*k_2}$, and using \eqref{Def-HeckeAlg-f2}, we get 
$$
\|f_{p,{\bf k}}\|_1\leq \|c_0\|_1^{k_0}\|c_1\|_1^{k_1}\|c_2\|_1^{k_2}
\leq 1^{k_0} \times (p^{16})^{k_1}\times 
(p^{17})^{k_2}\leq p^{17(k_1+k_2)}. 
$$
For the second claim, use \eqref{Def-HeckeAlg-f4}. 
\begin{comment}
The second claim follows by noting that, for any two $\bK_p$-double cosets $X_1,X_2$, we have $\fc(\cchi_{X_1}*\cchi_{X_2})=\fc(\cchi_{X_1})\fc(\cchi_{X_2})$.
\end{comment}
\end{proof}
Let $D<0$, $\Lambda$, $\mu$ and $M$ be as in Theorem \ref{MainT-2}, so that $\mu^2={\bf 1}$. Let $N$ be a prime number such that $N\nmid DM$ and $p\nmid DMN$. Let $l\in 2\Z_{\geq 7}$ be fixed.

\begin{thm}\label{APPL-Thm}
Let ${\bf k}=(k_1,k_2;k_0)\in \Z_{\geq 0}^{3}$.Then, as $N\rightarrow \infty$, 
{\allowdisplaybreaks\begin{align}
&\frac{(1+N^{-2})^{-1}}{(\log Nl)^{\delta(\Lambda\mu_{E}={\bf 1})}
[{\bf Sp}_2(\Z):\Gamma_0^{(2)}(N)]}
\sum_{\pi \in \Pi_{\rm cusp}^{(E,\Lambda)}(l,N){}^{\rm G, new}}\widehat {f_{p,{\bf k}}}\,(y_p(\pi))\,{L\left(\tfrac{1}{2},\pi,\mu\right)}\,\omega^{\Phi_\pi^0}_{l,D,\Lambda^{-1}} \ft(\pi,\mu) 
\label{APPL-Thm-f}
\\
&=
2{\mes}^\Lambda_p(\widehat {f_{p,{\bf k}}})\begin{cases} 
L(1,\kappa_D), \quad &(\Lambda\mu_{E}={\bf 1}), \\ 
L(1,{\mathcal AI}(\Lambda)\times \mu)
, \quad &(\Lambda\mu_{E}\not={\bf 1}), 
\end{cases} \notag
\\
&\qquad +O_{\e}(N^{-2}
p^{a_1(k_1+k_2)+b_1}+ N^{-2l+10+\varepsilon}p^{a_2(k_1+k_2)+b_2}+N^{-l+12+\varepsilon}p^{a_3(k_1+k_2)+b_3})
\notag
\\
&\qquad +O(N^{\frac{-3}{2}}p^{17(k_1+k_2)}).
\notag
\end{align}}
where $a_i,b_i(i=1,2,3)$ are positive constants depending only on $l$. 
\end{thm}
\begin{proof}
From Theorem \ref{MainT-0} combined with \cite[Theorem 5.4]{KugaTsuzuki}, we obtain the formula with the error term \eqref{MainT-0-f0} ($\sigma=2+\varepsilon$) and $O(N^{-\frac{3}{2}}\|f_{p,{\bf k}}\|_1)$ which comes from the Saito-Kurokawa representations and the Yoshida type representations. By Lemma \ref{APPL-L}, the error term amounts to a sum of terms of the form $p^{A}N^{-B}\,(A\in \R,B>0)$ as in \eqref{APPL-Thm-f}. Note that $l>12$ is necessary for the $O$-terms to decay to $0$.    
\end{proof}

Like analogous theorems in \cite{KST} and \cite{Dickson}, our Theorem \ref{APPL-Thm} can be viewed as a weighted version of the quantitative automorphic density theorem for ${\bf PGSp}_2$ (\cite{ShinTemplier}, \cite{KWY}).  
In this article, we do not make any attempt to improve the error term. Instead, we shall content ourselves with one particular but interesting arithmetic application of Theorem \ref{APPL-Thm}; for this purpose, the error term of the form in \eqref{APPL-Thm-f} is good enough. For the notion of field of rationality of automorphic representations, we refer to \cite{Clozel} and \cite{ShinTemplier2}. 

\begin{thm}\label{APPL-T2} There exist constants  $N_{p}>DM$ and $C_p>0$ with the following properties: For any prime number $N>N_p$, there exists 
$\pi \in \Pi_{\rm cusp}(l,N)^{{\rm G}, {\rm new}}$ such that 
\begin{itemize}
\item $L\left(\tfrac{1}{2}, \pi\times \mu\right)\not=0$ and $L\left(\tfrac{1}{2}, \pi\right)L\left(\tfrac{1}{2}, \pi,\kappa_D\right)\not=0$, 
\smallskip
\item $[\Q(\pi):\Q]\geq C_p\sqrt{\log\log N}$, where $\Q(\pi)$ is the field of rationality of $\pi$. 
\end{itemize}    
\end{thm}
\begin{proof} We follow the argument described in \cite[\S3.2]{SakugawaSugiyama}.
We consider the set 
$$
{\mathcal Y}(p,N):=\{y_p(\pi) \in [Y_p^0] \mid \pi \in \Pi_{\rm cusp}^{E,{\bf 1}}(l,N)^{{\rm G}, {\rm new}},\,L(1/2, \pi,\mu)\not=0\}
$$
and estimate its cardinality $\# {\mathcal Y}(p,N)$ from above and from below. We have that $\#{\mathcal Y}(p,N)$ is bounded by $O((8p^A)^{3d^2})$ with $A>0$ being a constant (depending on $l$) and 
$$d:=\max\{[\Q(\pi):\Q] \mid \pi \in \Pi_{\rm cusp}^{E,{\bf 1}}(l,N)^{{\rm G}, {\rm new}},\,L(1/2, \pi,\mu)\not=0\}.
$$
At this step, it is critical to use the integrality of the complex numbers $p^{2l-3}{\widehat c_1}(y_p(\pi),1)$ and $p^{2(2l-3)}{\widehat c_2}(y_p(\pi),1)$, which follows from the integrality of arithmetically normalized Hecke operators $T(p)$ and $T_{1,1}(p)$ (\cite[\S2.3]{SakugawaSugiyama}) on $S_l(\Gamma_0(N))$ (\cite[Theorem 1.2] {SakugawaSugiyama}). We also use that 
%the condition $\pi \in \Pi_{\rm cusp}^{E,{\bf 1}}(l,N)^{{\rm G}, {\rm new}}$, $L(1/2, \pi,\mu)\not=0$ is invariant by $\pi \mapsto {}^\sigma \pi$ due to Lemma \ref{APPL-L1}, 
the map $[Y_p^0]\ni (a,b)\mapsto (p^{2l-3}{\widehat c_1}(a,b,1)$ and $p^{2(2l-3)}{\widehat c_2}(a,b,1))\in \C^2$ is injective. To have a lower bound, we need a version of \cite[Lemma 6.16]{ShinTemplier2} for ${\bf PGSp}_2$ with the Plancherel measure $\mu_{p}^{\rm Pl}$ therein being replaced with our ${\mathfrak m}_p^{\Lambda,\mu}$; the same  proof works since the Lebesgue measure on $(i\R)^2/W\,(\cong [Y_p^0])$ is absolutely continuous with respect to ${\mathfrak m}_p^{\Lambda,\mu}$. Then, we argue in the same way as \cite[\S3.2]{SakugawaSugiyama}, using Theorem \ref{APPL-Thm} instead of \cite[Theorem 1.1]{KWY2}, to have the lower bound $\#{\mathcal Y}(p,N)\gg_{p} (\log N)^{2}$. \end{proof}
We need the notion of regular-algebraicity of irreducible cuspidal representations of ${\bf GL}_n$ in the sense of \cite[Definition 3.12]{Clozel}. 

\begin{lem} \label{APPL-L1}
Let $\pi\in \Pi_{\rm cusp}^{E,\Lambda}(l,N)^{{\rm G}, {\rm new}}$ and $\sigma \in {\rm Aut}(\C)$. Let $\psi=\pi^{{\rm GL}}\boxtimes 1_{{\bf SL}_2(\C)}$ be the global Arthur parameter of $\pi$, where $\pi^{{\rm GL}}$ is an irreducible cuspidal automorphic representation of ${\bf GL}_4(\A)$ of symplectic type. 
\begin{itemize}
    \item[(i)] We have ${}^\sigma \pi \in \Pi_{\rm cusp}(l,N)^{{\rm G},{\rm new}}$. When $\Lambda=\chi\circ {\rm N}_{E/\Q}$ for some finite order character $\chi$ of $\A^\times/\Q^\times$, we have ${}^\sigma \pi \in \Pi_{\rm cusp}^{E,\sigma\circ \Lambda}(l,N)^{{\rm G},{\rm new}}$.  
    \item[(ii)] We have that $\pi^{{\rm GL}}$ is  regular algebraic, and ${}^\sigma \psi:={}^\sigma \pi^{{\rm GL}}\boxtimes 1_{{\bf SL}_2(\C)}$ is the global Arhtur parameter of ${}^\sigma \pi$.
    \item[(iii)]
    We have $\Q(\pi)=\Q(\pi^{\rm GL})$, and that the JPSS conductor of $\pi^{{\rm GL}}$ is $N^2$. 
    \item[(iv)] For any finite order character $\eta$ of $\A^\times/\Q^\times$, $L\left(\tfrac{1}{2}, \pi,\eta\right)\not=0$ if and only if $L\left(\tfrac{1}{2}, {}^\sigma \pi,\sigma\circ \eta\right)\not=0.$
\end{itemize}
\end{lem}
\begin{proof} By definition, ${}^\sigma \pi=\pi_\infty \otimes {}^\sigma \pi_\fin$, so that $\pi_\infty={}^\sigma \pi_\infty$ is a HDS of weight $l$. Due to Shimura, the space of $\Phi \in S_l(\Gamma_0(N))$ with rational Fourier coefficients is a $\Q$-structure of the $\C$-vector space $S_l(\Gamma_0(N))$. Thus, $S_l(\Gamma_0(N))$ has a $\sigma$-conjugate linear action; ${}^\sigma \pi_\fin$ occurs in the $\sG(\A_\fin)$-module generated by the $\sigma$-conjugate of the adelic lift to $\sG(\A)$ of $\Phi_\pi^0\in S_l(\Gamma_0(N))$. This shows ${}^\sigma\pi\in \Pi_{\rm cusp}(l,N)$. For any open compact subgroup $K\subset \sG(\A_\fin)$, it is obvious that $\pi_{\fin}^{K}\not=0$ if and only if $({}^\sigma \pi_\fin)^{K}\not=0$. From this, ${}^\sigma \pi \in \Pi_{\rm cusp}(l,N)^{\rm new}$. 
Due to the identification of the archimedean parameter ({\it cf}. \cite[Formula (243)]{PSS}), $\pi^{{\rm GL}}$ is seen to be regular algebraic of infinite type $p_\infty(\pi^{\rm GL})=(l,2,1,3-l)\in\Z^4/{\mathfrak S}_4$ (\cite[\S3.3, Definition 3.6]{Clozel}). For $\sigma\in {\rm Aut}(\C)$, ${}^\sigma \pi^{{\rm GL}}:=\pi^{{\rm GL}}_\infty \otimes {}^\sigma \pi^{{\rm GL}}_\fin$ is also cuspidal automorphic representation (\cite[Th\'{e}or\`{e}me 3.13]{Clozel}); moreover, by Jacquet-Shalika's criterion for a cuspidal automorphic representation of ${\bf GL}_{2n}$ to be of symplectic type (\cite[Theorem 1]{JacquetShalika}) and by the theorem in the appendix of \cite{GrRag}, ${}^\sigma \pi^{{\rm GL}}$ is also of symplectic type. Thus, we have a global $A$-parameter ${}^\sigma \psi:={}^\sigma \pi^{{\rm GL}} \boxtimes {\bf 1}_{{\bf SL}_2(\C)}$ of $\sG$. We claim that ${}^\sigma \pi$ is of general type belonging to the $A$-packet $\Pi_{{}^\sigma \psi}$. Indeed, for large enough $S$, the coincidence of unramified local $p$-factors $L(s,\pi_p)=L(s, \pi_p^{{\rm GL}})$ for $p\not\in S$ implies the equality $L(s,{}^\sigma \pi_p)=L(s, {}^\sigma \pi_p^{\rm{GL}})$ ({\it cf}, \cite[Lemma 5.3.1]{PSS}). Thus, $L^{S}(s,{}^\sigma \pi)=L^{S}(s, {}^\sigma \pi^{{\rm GL}})$. On the other hand, since ${}^\sigma \pi$ is a cuspidal automorphic representation of $\sG(\A)$, it belongs to some global $A$-packet $\Pi_{\psi'}$ with a global $A$-parameter $\psi'$. For sufficiently larger $S$, the partial $L$-function $L^{S}(s,{}^\sigma \pi)$ coincides with the standard partial $L$-function $L^{S}(s,\xi)$ of an isobaric automorphic representation $\xi$ of ${\bf GL}_4(\A)$ determined by $\psi'$. Thus, $L^{S}(s, {}^\sigma \Pi)=L^{S}(s, {}^\sigma \pi)=L^{S}(s, \xi)$. By \cite[Theorem (4.4)]{JSII}, we conclude that ${}^\sigma \pi^{{\rm GL}}=\xi$, i.e., ${}^\sigma \psi=\psi'$, which in particular means that ${}^\sigma \pi$ is also of general type. 

It remains to show that ${}^\sigma \pi$ has a global $(E,\sigma \circ \Lambda)$-Bessel model. At this point, we use the assumption that $\Lambda$ is of the form $\chi\circ \nr_{E/\Q}$, so that $\widehat L(s,{}^\sigma \pi\times {\mathcal AI}(\sigma\circ \Lambda))=\widehat L(s, {}^\sigma \pi^{\rm GL} \times \sigma\circ  \chi)\widehat L(s, {}^\sigma \pi^{{\rm GL}}\times \sigma\circ \chi \kappa_{D})$ for all $\sigma \in {\rm Aut}(\C)$. By \cite[Corollary 1.1]{FurusawaMorimoto3}, we have $L(1/2,\pi \times {\mathcal AI}(\Lambda))\not=0$ and $\pi_v$ has local $(E_v,\Lambda_v)$-Bessel model for all $v$. Thus, $\widehat L(1/2, \pi^{{\rm GL}} \times \chi)\widehat L(1/2, \pi^{{\rm GL}} \times \chi \kappa_{D})\not=0$. By \cite[Theorem 7.1.2]{GrRag}, there exists a system of elements $\Omega({}^\sigma \pi)\in \C^\times/\Q(\pi^{{\rm GL}})^\times $ such that, for any finite order $\eta \in \widehat {\A^\times/\Q^\times}$,
\begin{align}
\sigma\left(\frac{\widehat L(1/2,\pi^{{\rm GL}}\times \eta)}{G(\tilde \eta)^2\,\Omega(\pi)}\right)=
\frac{\widehat L(1/2,{}^\sigma \pi^{{\rm GL}} \times \sigma\circ \eta)}{G(\sigma \circ \tilde \eta)^2\,\Omega({}^\sigma \pi)}, \quad \sigma \in {\rm Aut}(\C) 
\label{APPL-L-f0}
\end{align} By this, we have $\widehat L(1/2, {}^\sigma \pi^{{\rm GL}} \times \sigma\circ  \chi)\widehat L(1/2, {}^\sigma \pi^{{\rm GL}} \times \sigma \circ\chi \kappa_{D})=\widehat L(1/2, {}^\sigma \pi\times{\mathcal AI}(\sigma \circ \Lambda))\not=0$. For $p<\infty$, there is a natural $\sigma$-linear $\sG(\Q_p)$-intertwining bijection $B_p \mapsto {}^\sigma B_p$ from the space of $(E_p,\Lambda_p)$-Bessel functions on $\sG(\Q_p)$ onto the space of $(E_p,\sigma\circ\Lambda_p)$-Bessel functions ({\it cf}. \cite[\S3.3.1]{GrSeb}). If $V(\pi_p)$ is the $(E_p,\Lambda_p)$-local Bessel model of $\pi_p$, then its $\sigma$-conjugate image ${}^\sigma(V(\pi_p))$ is the $(E_p,\sigma\circ\Lambda_p)$-Bessel model of ${}^\sigma \pi_p$. Thus, by \cite[Corollary 1.1]{FurusawaMorimoto3} once again, ${}^\sigma \pi$ admits a global $(E,\sigma\circ \Lambda)$-Bessel model. 
This complete the proof of (i) and (ii). For the rest of the proof, we address (iii) and (iv). Due to (ii), for $\sigma \in {\rm Aut}(\C)$, ${}^\sigma \pi\cong \pi$ implies ${}^\sigma\psi=\psi$, or equivalently ${}^\sigma(\pi^{\rm{GL}})=\pi^{{\rm GL}}$. Let us show the converse implication. Suppose that ${}^\sigma \pi^{{\rm GL}}=\pi^{{\rm GL}}$. Since $\pi \in \Pi_{\rm cusp}^{E,\Lambda}(l,N)^{\rm new}$, we know that the completed $L$-function $\widehat L(s,\pi)$ defined by Piatetski-Schapiro coincides with the principal $L$-function $\widehat L(s,\pi^{\rm GL})$ (\cite[\S6]{KugaTsuzuki}). Thus (iv) is immediate from \eqref{APPL-L-f0}. By (i) and (ii), we also have $\widehat L(s,{}^\sigma \pi)=\widehat L(s,{}^\sigma \pi)$ and $\pi_p\cong {}^\sigma \pi_p$ for almost all $p$. Thus, $\widehat L(s, \pi_p)=\widehat L(s,{}^\sigma \pi_p)$ for all $p<\infty$ by \cite[Lemma 3.1.1]{Schmidt}. As we recalled in \cite[\S4.3]{KugaTsuzuki}, $\pi_p$ and ${}^\sigma\pi_p$ are representations of type IIIa, which have degree $2$ local $L$-factors, or representations of type IVb, which have degree $1$ local $L$-factors. Noting this, we 
 get $\pi_p \cong {}^\sigma \pi_p$ from $L(s,\pi_p)=L(s,{}^\sigma \pi_p)$ as in the proof of \cite[Lemma 3.1.2]{Schmidt}. Therefore, $\pi_\fin \cong {}^\sigma \pi_\fin$ as desired. Thus, $\Q(\pi)=\Q(\pi^{{\rm GL}})$. In \cite[\S4.6]{KugaTsuzuki}, we established the functional equation $\widehat L(s,\pi)=N^{1-2s}\widehat L(1-s,\pi)$. Since $\widehat L(s,\pi)=\widehat L(s, \pi^{{\rm GL}})$, by comparing this with $\widehat L(s, \pi^{{\rm GL}})=\varepsilon N(\pi^{{\rm GL}})^{1/2-s}\widehat L(1-s, \pi^{{\rm GL}})$, we get $N^2=N(\pi^{{\rm GL}})$, where $N(\pi^{{\rm GL}})$ is the JPSS-conductor of $\pi^{{\rm GL}}$. \end{proof}
We deal with $\Lambda \in {\widehat {{\rm Cl}(E)}}$ in this work; however, this proof works on a general $\Lambda$ having ramifications. Corollary \ref{APPL-C1} follows from Theorem \ref{APPL-T2} applied to $(\Lambda,\mu)=({\bf 1}, \kappa_{M})$ and Lemma \ref{APPL-L1}.

\section{Deduction of the main terms} \label{sec:MainT}
In this section, we prove Propositions \ref{MainT-P1} and \ref{ExplicitMT-L}. Let $\phi=\otimes_{p<\infty}\phi_p\in \cS(\A_{E,\fin}^2)$ and $f=\otimes_{p<\infty} \in  \cH(\sG(\A_\fin))$; set $\Phi=\widehat{\bf\Phi_{\phi,f,l}}^{(\beta,\Lambda,\mu)}$.

\subsection{The term $\JJ_{1_4}(\Phi)$} \label{sec:JJ14}
We start from \eqref{BesselPer-L1-f1}. Substituting \eqref{KerFtn-f1} and formally changing the order of integrals, we find the following formula, whose proof will be given at the end of this section. 
\begin{align}
\JJ_{1_4}(\Phi)=\int_{c-i\infty}^{c+i\infty} \beta(s) s(s^2-1)\widehat L(1+s, \Lambda\mu_{E})\{{\mathcal I}_{\phi,f,l}^{(s,\Lambda,\mu)}(1_2)+{\mathcal I}_{\phi,f,l}^{(s,\Lambda,\mu)}({\sigma})
\}\,\d s
 \label{JJ14-f0}
\end{align}
with $c>1$, where, for $\delta \in \{1_4,\sigma\}$ with $\sigma$ being as in \eqref{Def-sigma}, we set  
\begin{align*}
{\mathcal I}_{\phi,f,l}^{(s,\Lambda,\mu)}(\delta)
:=\int_{\A} \int_{\A_{E}^\times/E^\times \A^\times} \Lambda(\tau) {\bf \Phi}_{\phi,f,l}^{(s,\Lambda,\mu)}(\sm(\delta, \det \delta)\sm_{\theta}(\tau)\sn(a T_\theta^\dagger)b_\R^\theta)\,\psi(aD/2)\,\d a \,\d^\times \tau.
\end{align*}
Since $\Lambda$ is trivial on $E_{\infty}^\times$ and since $\vol(E_\infty^\times/\R^\times)=\pi$, the integral ${\mathcal I}_{\phi,f,l}^{(s,\Lambda,\mu)}(\delta)$ is factored to the product of 
{\allowdisplaybreaks\begin{align*}
{\mathcal I}_{\phi,f}^{(s,\Lambda,\mu)}(\delta)&:=\int_{\A_\fin} \int_{\A_{E}^\times/E^\times \A^\times E_\infty^\times } \Lambda(\tau) {\Phi}_{\phi,f}^{(s,\Lambda,\mu)}(\sm(\delta, \det \delta)\sm_{\theta}(\tau)\sn(a T_\theta^\dagger))\,\psi_\fin(aD/2)\,\d a \,\d^\times \tau, \\
{\mathcal I}_{l}^{(s,\mu)}(\delta)&:=\pi \int_{\R} {\Phi}_{l}^{(s,\mu_\infty)}(\sm(\delta, \det \delta)
\sn(a T_\theta^\dagger)b_\R^\theta)\,\exp(\pi i aD)\,\d a.
\end{align*}}

\begin{lem} \label{MainT-L1} 
For any compact interval $I\subset \R$, we have 
\begin{align}
\int_{\A_{E}^\times/E^\times \A^\times E_\infty^\times }\int_{\A_\fin}|\Phi_{\phi,f}^{(s,\Lambda,\mu)}(\sm(\delta,\det \delta)\sm_\theta(\tau)\sn(a T_\theta^\dagger))|\,\d a\,\d^\times \tau \ll 1, \quad s\in \cT_{I},
 \label{MainT-L1-f1} 
\end{align}
so that the integral ${\mathcal I}_{\phi,f}^{(s,\Lambda,\mu)}(\delta)$ defines an entire function on $\C$, which is vertically bounded. Moreover, for $(\phi,f)$ as in \S\ref{sec:ErrTEST} (i) and (ii), 
\begin{align}
{\mathcal I}_{\phi, f}^{(s,\Lambda,\mu)}(\delta)=&C_\delta(\Lambda)\,
\frac{|D|h_D}{8w_D}\,
\tilde \mu(\det\delta)\tilde \mu(N)\,N^{s-1}\,M^{s-6}\label{MainT-L1-f0}\\
&\times\quad \zeta_{M}(1)\zeta_{M}(4)G(\tilde \mu)\tilde \mu^{-1}(2)\tilde \mu(D)\prod_{p\in S} \widehat {\cW^{(T_\theta,\Lambda_p^{\delta})}}(f_p;s),\notag
\end{align}
where $C_\delta(\Lambda)$ is $\delta_{\Lambda^2,{\bf 1}}$ if $\delta=1_2$ and is $1$ if $\delta=\sigma$. \end{lem}
\begin{proof} For $\delta \in \{1_4,\sigma\}$ and $\tau \in \A_{E}$, set $\tau^{\delta}:=\tau$ if $\delta=1_4$ and $\tau^{\delta}:=\bar \tau$ if $\delta=\sigma$; then, by \eqref{SigmaItau}, $\sm(\delta, \det \delta)\sm_{\theta}(\tau)\sm(\delta, \det \delta)^{-1}=\sm_{\theta}(\tau^{\delta})$ for $\tau \in \A_{E}^\times$. Using this on the way, by \eqref{HeckeFtn-f0} and by the Iwasawa decomposition on $\sG^\#(\Q_p)$ (see \S\ref{sec:HaarDGS}), we have that ${\mathcal I}_{\phi, f}^{(s,\Lambda,\mu)}(\delta)$ equals 
{\allowdisplaybreaks\begin{align*}
&
\frac{\sqrt{|D|}}{2}\int_{(a,\beta)\in \A_\fin \times \A_{E,\fin}} \int_{\tau \in \A_{E}^\times/E^\times \A^\times E_\infty^\times} \Lambda(\tau)\,
\psi_\fin(-\tr(T_\theta (-X_\beta+a T_\theta^{\dagger}))) \\
&\times \Bigl\{\int_{(\lambda,\tau_1,k_\fin^\#)\in \A_\fin^\times \times \A_{E,\fin}^\times \times \bK_\fin^{\#}} \Lambda^{-1}(\tau_1)\mu^{-1}(\lambda)|\lambda |_{\fin}^{-s+1}\sf_\phi^{(s,\Lambda,\mu)}(k_\fin^\#) \\
&\quad \times f(\iota_\theta(k_\fin^\#)^{-1} \sm_{\theta}(\tau_1^{-1}\tau^\delta)\sm(\delta, \lambda^{-1}\det \delta) \sn(-X_\beta+ a T_\theta^\dagger))
 \,\d^\times \lambda\,\d^\times \tau_1\,\d k_\fin^\#\lambda\Bigr\}
\,\d a\,\d\beta
\d^\times \tau, \end{align*}}where $\d \beta$ on $\A_{E,\fin}$ is the product of self-dual Haar measures on $E_{p}$ for $p<\infty$ with respect to the bi-character $\psi_{p}(\tr_{E_p/\Q_p}(\alpha \bar \beta))$. Since $\A_{E,\fin}^\times/ E^\times \A_\fin^\times$ is compact, we see that the integral in \eqref{MainT-L1-f1} is bounded from above by 
\begin{align*}
\int_{(a,\beta,\lambda,\tau_1,k_\fin^\#)} |\lambda |_{\fin}^{-\Re(s)+1} |f(\iota_\theta(k_\fin^\#)\sm_{\theta}(\tau_1)\sm(\delta,\lambda \det \delta)\sn(X_\beta+ a T_\theta^\dagger))|\,\d^\times \lambda\,\d^\times \tau_1\,\d k_\fin^\#\, \d a\,\d\beta. 
\end{align*}
Since ${\rm supp}(f)\subset \sG(\A_\fin)$ is compact, by $\sG(\A_\fin)=\bK_\fin \sP(\A_\fin)$, the domain of integral is restricted to a compact set of $(a,\beta, \lambda,\tau_1, k_\fin^\#)$. From this remark, the estimation \eqref{MainT-L1-f1} is evident. By \eqref{Ttheta-f1}, we see that $\d \eta_{\sV(\A_\fin)}(X):=|D/2|_\fin^{1/2} \d \beta \,\d a$ for $X=a T_\theta^\dagger+X_\beta\,(a\in \A_{\fin}, \beta\in \A_{E,\fin})$ is the self-dual Haar measure on $\sV(\A_\fin)$ with respect to the bi-character $\psi_{\fin}(X Y^\dagger)$. The group $\sN(\A_\fin)\cong \sV(\A_\fin)$ is endowed with the Haar measure $\d X$ such that $\vol(\sV(\widehat \Z))=1$, which coincides with $\prod_{p<\infty}[\cQ_p:\sV(\Z_p)]^{1/2} \times \d\eta_{\sV(\A_\fin)}(X)$. Note that $[\cQ_p:\sV(\Z_p)]=2^{\delta_{p,2}}$. Hence $\d \beta\,\d a=2^{-1/2}|D/2|_\fin^{-1/2}\,\d X$. By making a change of variables $\tau_1 \mapsto \tau_1(\tau^\delta)^{-1}$ on the way, ${\mathcal I}_{\phi,f}^{(s,\Lambda,\mu)}(\delta)$ equals
{\allowdisplaybreaks\begin{align*}
&\frac{|D|}{4}\int_{\A_{E}^\times/E^\times \A^\times E_\infty^\times} \Lambda(\tau)\Lambda^{-1}((\tau^\delta)^{-1})\,\d^\times \tau\int_{\A^\times \times \A_{E,\fin}^\times \times \bK_\fin^\#} \Lambda^{-1}(\tau_1)\mu^{-1}(\lambda)|\lambda|_\fin^{-s+1} 
\\
&\cdot \int_{\sV(\A)} 
 f(\iota_{\theta}(k_\fin^\#)^{-1} \sm_{\theta}(\tau_1^{-1})\sm(\delta, \lambda^{-1}\det \delta) \sn(X))\,\psi_\fin(-\tr(T_\theta X))\,\d X\biggr\}\,\d^\times \tau_1\,\d^\times \lambda\,\d k_\fin^\#.
 \end{align*}}We have that $E_\infty^\times/\R^\times\cong \C^\times/\R^\times$ is a closed subgroup of $\A_{E}^\times/E^\times \A^\times$; these two groups are naturally regarded as measure spaces. The measure on $\A_{E}^\times/E^\times\A^\times E_\infty^\times \cong (\A_{E}/E^\times \A^\times)/(E_\infty^\times/\R^\times)$ is the quotient measure. Then, by $\vol(E_\infty^\times/\R^\times)=\pi$ and \eqref{HaarEQ-f0}, we have $\vol(\A_{E}^\times/E^\times\A^\times E_\infty^\times)=\frac{\pi h_D}{2w_D} \times \frac{1}{\pi}=\frac{h_{D}}{2w_D}$. Hence, the $\tau$-integral of $\Lambda(\tau)\Lambda^{-1}((\tau^\delta)^{-1})$ is equal to $C_\delta(\Lambda)\,\frac{h_D}{2 w_D}$. Thus,   
\begin{align*}
{\mathcal I}_{\phi,f}^{(s,\Lambda,\mu)}(\delta)&=C_\delta(\Lambda)\frac{|D|h_D}{8w_D} \prod_{p<\infty}{\mathcal I}_{\phi_p,f_p}^{(s,\Lambda_p,\mu_p)}(\delta), 
\end{align*}
where ${\mathcal I}_{\phi_p,f_p}^{(s,\Lambda_p,\mu_p)}(\delta)$ is the local orbital integral defined by \eqref{Def-OBIIdCtp}. For $p\not\in S$ such that $p\nmid NM$, we easily have $\cW_p^{(T_\theta,\Lambda_p^{-1})}(f_p;\sm(1_2,\lambda))=\cchi_{\Z_p^\times}(\lambda)$, which yields ${\mathcal I}_{\phi_p,f_p}^{(s,\Lambda_p,\mu_p)}(\delta)=\widehat{\cW^{(T_\theta, \Lambda_p)}}(f_p;s,\mu_p)=1$. For $p\in S$, we have ${\mathcal I}_{\phi_p,f_p}^{(s,\Lambda_p,\mu_p)}(\delta)=\widehat{\cW^{(T_\theta, \Lambda_p^{\delta})}}(f_p;s,\mu_p)$. We apply formula \ref{OBIIdCt-f0} for $p\mid M$ and Proposition \ref{LocalSingw2} for $p\mid N$ to have that $\prod_{p\mid NM} {\mathcal I}_{\phi_p,f_p}^{(s,\Lambda_p,\mu_p)}(\delta)$ equals 
{\allowdisplaybreaks\begin{align*}
&\prod_{p\mid N} \mu_p(p)p^{s-1}\,\prod_{p\mid M }|M|_p^{-(s-11/2)}(1+p^{-2})^{-1}\zeta_{\Q_p}(1)\zeta_{E_p}(1)\mu_p(2^{-1}D\det \delta)W_{\Q_p}(\mu_p)
\\
&=\tilde \mu(\det\delta)\tilde \mu(N)N^{s-1}\,M^{s-11/2}\zeta_{M}(1)\zeta_{M}(4)\tilde \mu(2)^{-1}\tilde \mu(D)\,M^{-1/2}G(\tilde \mu),
\end{align*}}where the second expression is obtained by $\zeta_{E_p}(1)=(1-p^{-2})^{-1}$ for $p\mid M$ and by $\prod_{p\mid M} W_{\Q_p}(\mu)=M^{-1/2}G(\tilde \mu)$ and $\prod_{p\mid L}\mu_p(L)=\tilde \mu(L)$ for any $L\in \Z$ relatively prime to $M$. \end{proof}

\begin{lem} \label{MainT-L2}
Let $0<c<l-1$. Then, for $\delta \in \{1_2,\sigma\}$, 
\begin{align*}
\int_{\R}|\Phi_l^{(s,\mu_\infty)}(\sm(\delta, \det \delta)\sn(a T_\theta^\dagger)b_\R^\theta)|\,\d a\ll_c 1, \quad \Re(s)=c.
\end{align*}
For $\Re(s)<l-2$, we have
\begin{align}
{\mathcal I}_l^{s,\mu}(\delta)=\tilde \mu(\det\delta)\times\frac{2\pi}{{|D|}}\,\frac{(2\pi\sqrt{|D|})^{-s+l-1}}{\Gamma(-s+l-1)}\times (-1)^{l}e^{-2\pi \sqrt{|D|}}. 
 \label{MainT-L2-f0}
\end{align}
\end{lem}
\begin{proof}
By \eqref{Def-ShinFtn},  
\begin{align*}
\Phi_{l}^{(s,\mu_\infty)}(g)=\mu_\infty(\det\delta)(-1)^{l} \left(1-\tfrac{\sqrt{|D|}}{2}ai\right)^{s-l+1} \quad \text{for $g=\sm(\delta,\det \delta)\sn(aT_\theta^\dagger) b_\R^\theta$}.
\end{align*}
 Using this and by noting that $\mu_\infty(\det\delta)=\tilde \mu(\det\delta)$, we obtain 
\begin{align}
{\mathcal I}_l^{(s,\mu)}(\delta)&=\pi \cdot \tilde \mu(\det\delta)(-1)^{l}\frac{2}{\sqrt{|D|}}\int_{\R} (1+i x)^{s-l+1}\exp(2\pi i \sqrt{|D|} x)\,\d x . 
\label{MainT-L2-f1}
\end{align}
for $\Re(s)<l-2$. By \cite[3.382.6]{GR}, we are done. 
\end{proof}
By Lemmas~\ref{MainT-L1} and \ref{MainT-L2}, we can apply Fubini's theorem to have formula \eqref{JJ14-f0} for $0<c<l-2$. The function $\JJ_{l}^{1_4}(\phi,f\mid s,\Lambda,\mu)$ defined as the sum over $\delta\in\{1_2,\sigma\}$ of the product of \eqref{MainT-L1-f0} and \eqref{MainT-L2-f0} has the desired property in Proposition \ref{MainT-P1}. This completes the proof of formula \eqref{MainT-P1-f1}. Proposition \ref{ExplicitMT-L} for $\su=1_4$ follows from \eqref{MainT-L1-f0} and \eqref{MainT-L2-f0} since $\tilde \mu(\det\delta)^2=\tilde \mu(1)=1$.

\subsection{The term $\JJ_{w_2({\rm s})}(\Phi)$} \label{sec:MainTSg}
Starting from \eqref{BesselPer-f5} by substituting \eqref{KerFtn-f1}, we formally get the following formula, whose proof will be given at the end of this section. 
\begin{align}
\JJ_{w_2({\rm s})}(\Phi)=\int_{c-i\infty}^{c+i\infty} 
\beta(s) s(s^2-1)\widehat L(1+s, \Lambda\mu_E)
\{\cJ_{\phi,f,l}^{(s,\Lambda,\mu)}(1_2)
+\cJ_{\phi,f,l}^{(s,\Lambda,\mu)}(\sigma)\}\,\d s
 \label{JJw2-f0}
\end{align}
with $c>1$, where for $\delta \in \{1_4,\sigma\}$, we set  
\begin{align*}
\cJ_{\phi,f,l}^{(s,\Lambda,\mu)}(\delta):=\cJ_{l}^{(s,\mu)}
(\delta)\,\cJ_{\phi,f}^{(s,\Lambda,\mu)}(\delta)
\end{align*}
with 
\begin{align*}
\cJ_{\phi,f}^{(s,\Lambda,\mu)}(\delta)&:=
\int_{\A_{E}^\times/E^\times \A^\times E_\infty^\times } \Lambda(\tau) \biggl\{\int_{\sN(\A_\fin)} {\Phi}_{\phi,f}^{(s,\Lambda,\mu)}(w_2 \sm(\delta, \det \delta)n\sm_{\theta}(\tau))\,\psi_{T_\theta}(n)^{-1}\,\d n \biggr\}\,\d^\times \tau, \\
\cJ_{l}^{(s,\mu)}(\delta)&:=\pi
\int_{\sV(\R)} {\Phi}_{l}^{(s,\mu_\infty)}(w_2\sm(\delta, \det \delta)\sn(X)b_\R^\theta)\,\exp(-2\pi i\tr(T_\theta X))\,\d X.
\end{align*}

\begin{lem} \label{MainTSg-L1} 
Let $\delta \in \{1_2,\sigma\}$ and $I\subset (1,\infty)$ any compact interval. Then, 
\begin{align*}
\int_{\A_{E}^\times/E^\times \A^\times E_\infty^\times } \biggl\{\int_{\sN(\A_\fin)}|{\Phi}_{\phi,f,l}^{(s,\Lambda,\mu)}(w_2 \sm(\delta, \det \delta)n\sm_{\theta}(\tau)b_\R^\theta)|\,\d n \biggr\}\,\d^\times \tau \ll 1, \quad s\in \cT_{I}. 
\end{align*}
For $(\phi,f)$ as in \S\ref{sec:ErrTEST} (i) and (ii), 
\begin{align}
\cJ_{\phi,f}^{(s,\Lambda,\mu)}(\delta)
&=2^{1/2}|D|_\infty^{-s-1/2}\varepsilon(s,\Lambda\mu_E)^{-1}\frac{L(s,\Lambda\mu_{E})}{L(s+1,\Lambda\mu_{E})}M^{-2s-1}G(\tilde \mu^{-1})^{2}\,{\mathcal I}_{\phi,f}^{(-s,\Lambda^{-1},\mu^{-1})}(\delta).
\label{MainTSg-L1-f0}
 \end{align}
\end{lem}
\begin{proof}
By \eqref{HeckeFtn-f0}, 
\begin{align}
&\int_{\sN(\A_\fin)} {\Phi}_{\phi, f}^{(s,\Lambda, \mu)}(w_2 \sm(\delta, \det \delta)n\sm_{\theta}(\tau))\,\psi_{T_\theta}(n)^{-1}\,\d n
 \notag
\\
&=\int_{\sV(\A_\fin)} \biggl\{ \int_{\sG^\#(\A_\fin)}\sf_\phi^{(s,\Lambda,\mu)}(h)f(\iota(h)^{-1}w_2\sn(X) \sm(\delta, \det \delta)\sm_{\theta}(\tau))\,\d h \biggr\}\,\psi(\det \delta \tr(T_\theta X))^{-1}\,\d X
 \label{MainTSg-L1-f1}
\end{align}
By writing $X=a T_\theta^\dagger+X_\beta$ $(a\in \A_\fin,\,\beta\in \A_{E,\fin})$ and using \eqref{Def-iotaf2} and by making the change of variables $h\rightarrow \left[\begin{smallmatrix} 0 & 1 \\ 1 & 0 \end{smallmatrix} \right]\left[\begin{smallmatrix} 1 & -\beta  \\ 0 & 1 \end{smallmatrix} \right] h$, we get 
 \begin{align*}
& \int_{\sG^\#(\A_\fin)}\sf_\phi^{(s,\Lambda,\mu)}(h)\,f(\iota(h)^{-1}w_2\sn(X) \sm(\delta, \det \delta) \sm_{\theta}(\tau))\,\d h\\ 
&=\int_{\sG^\#(\A_\fin)}\sf_\phi^{(s,\Lambda,\mu)}(\left[\begin{smallmatrix} 0 & 1 \\ 1 & 0 \end{smallmatrix} \right]\left[\begin{smallmatrix} 1 & -\beta  \\ 0 & 1 \end{smallmatrix} \right] h)\,f(\iota(h)^{-1}\sn(aT_\theta^\dagger)\sm(\delta, \det \delta) \sm_{\theta}(\tau))\,\d h.
\end{align*}
Since the self-dual measure $\d X$ of $\sV(\A_\fin)$ is decomposed as $\d X= |D/2|^{1/2}_\fin \d\beta\,\d a$, the formula in \eqref{MainTSg-L1-f1} becomes
\begin{align*}
|D/2|_\fin^{1/2} \int_{\A_\fin} \int_{\sG^\#(\A_\fin)}
 [M(s)\sf_{\phi}^{(s,\Lambda,\mu)}](h) \,f(\iota(h)^{-1}\sn(aT_\theta^\dagger) \sm(\delta, \det \delta)\sm_{\theta}(\tau))\,\d h  \,\psi(a D (\det \delta)/2)\, \d a,
\end{align*}
where
$$
[M(s)\sf_{\phi}^{(s,\Lambda,\mu)}](h):=\int_{\A_{E,\fin}}\sf_\phi^{(s,\Lambda,\mu)}(\left[\begin{smallmatrix} 0 & 1 \\ 1 & 0 \end{smallmatrix} \right]\left[\begin{smallmatrix} 1 & -\beta  \\ 0 & 1 \end{smallmatrix} \right]
h) \,\d \beta.
$$
By the relation $\d \beta=|D|^{-1/2}\times \prod_{p<\infty} \d\eta_{E_p}(\beta_p)$ (see \S\ref{sec:HaarDGS}) and by \eqref{JaqS-f1} and \eqref{normsec}, $[M(s)\sf_\phi^{(s,\Lambda,\mu)}](h)$ equals 
$$|D|^{-1/2}\prod_{p<\infty}|D|_p^{s-1/2}\varepsilon_p(s,\Lambda\mu_E,\psi_{E_p})^{-1}\frac{L_p(s,\Lambda\mu_E)}{L_p(1+s,\Lambda\mu_E)}\sf_{\widehat \phi_p}^{(-s,\Lambda_p^{-1},\mu_p^{-1})}(h)
$$
Note that $\varepsilon_\infty(s,\Lambda\mu_E,\psi_{E,\infty})=1$. Moreover, $\widehat \phi=\otimes_{p<\infty}\widehat{\phi_p}$. 
Then, due to \eqref{HeckeFtn-f0}, we obtain the equality: 
{\small
\begin{align}
&\int_{\sN(\A_\fin)} {\Phi}_{\phi, f}^{(s,\Lambda,\mu)}(w_2 \sm(\delta, \det \delta)n\sm_{\theta}(\tau))\,\psi_{T_\theta}(n)^{-1}\,\d n
 \notag
\\
&=2^{1/2}|D|^{-s-1/2}
\varepsilon(s,\Lambda\mu_{E})^{-1}\frac{L(s,\Lambda\mu_E)}{L(s+1,\Lambda\mu_E)}\,\int_{\A_\fin} \Phi_{\widehat \phi, f}^{(-s,\Lambda^{-1}.\mu^{-1})}( \sm(\delta, \det \delta)\sm_{\theta}(\tau)\sn(aT_\theta^\dagger))\,\psi(a D/2)\,\d a
\label{MainTSg-L1-f2}
\end{align}} 
using the identity $\sn(aT_\theta^\dagger)\sm(\delta, \det \delta)\sm_\theta(\tau)=\sm(\delta, \det \delta)\sm_\theta(\tau)\sn(a(\det \delta) T_\theta^\dagger)$. Multiplying $\Lambda(\tau)$ and taking the integral in $\tau$, we get
\begin{align}
\cJ_{\phi,f}^{(s,\Lambda,\mu)}(\delta)=
2^{1/2}|D|^{-1/2-s}\varepsilon(s,\Lambda\mu_{E})^{-1}\frac{L(s,\Lambda\mu_E)}{L(s+1,\Lambda\mu_E)}{\mathcal I}_{\widehat \phi, f}^{(-s,\Lambda^{-1},\mu^{-1})}(\delta). 
 \label{MainTSg-L1-f3}
\end{align}
The same argument yields the inequality  
{\small\begin{align*}
&\int_{\A_E^\times/E^\times \A^\times E_\infty^\times} \biggl\{ \int_{\sN(\A_\fin)} |{\Phi}_{f}^{(\Lambda,s)}(w_2 \sm(\delta, \det \delta)n\sm_{\theta}(\tau))|\,\d n\biggr\}\d^\times \tau 
\\
&\leq 2^{1/2}|D|^{-1/2-\Re(s)}\tfrac{L(\Re(s),{\bf 1})}{L(\Re(s)+1, {\bf 1})}\,\int_{\A_E^\times/E^\times \A^\times E_\infty^\times}\biggl\{ \int_{\A_\fin} \Phi_{|\widehat \phi|, |f|}^{(-\Re(s),{\bf 1},{\bf 1})}( \sm(\delta, \det \delta)\sm_{\theta}(\tau)\sn(aT_\theta^\dagger))\,\d a\biggr\}\d^\times \tau
\end{align*}}
for $\Re(s)>1$. The majorant is convergent by Lemma~\ref{MainT-L1}. The formula \eqref{MainTSg-L1-f0} follows from \eqref{MainTSg-L1-f3} by \eqref{OBIIdCt-f1} and Proposition \ref{LocalSingw2} together with \eqref{MainT-L1-f0} easily. 
\end{proof}

\begin{lem} \label{MainTSg-L2}
For $\Re(s)=c>0$, 
\begin{align}
&\int_{\sV(\R)}|{\Phi}_{l}^{(s,\mu_\infty)}(w_2\sm(\delta, \det \delta)\sn(X)b_\R^\theta)|
\,\d X \ll_{c} 1, 
 \notag \\
&\cJ_{l}^{(s,\mu)}(\delta)=|D|^{1/2}2^{-1/2}\frac{\Gamma_\C(s)}{\Gamma_\C(s+1)}\,{\mathcal I}_{l}^{(-s,\mu)}(\delta). 
\label{MainTSg-L2-f1}
\end{align}
\end{lem}
\begin{proof}
Set $X=aT_{\theta}^\dagger+X_\beta$ $(a\in \R, \beta \in \C)$. By \eqref{Def-ShinFtn}, $\pi^{-1} \times \cJ_{l}^{(s,\mu)}({\delta})$ equals
\begin{align*}&\left|\tfrac{D}{2}\right|_\infty^{1/2}
\int_{\R} \int_\C \Phi_{l}^{(s,\mu_\infty)}(w_2 \iota(\left[\begin{smallmatrix} 1 & \beta \\ 0 & 1 \end{smallmatrix}\right])\,\sn(aT_\theta^\dagger)\,\sm(\delta, \det \delta)b_\R^\theta)\, \exp(-\pi i a(\det \delta)D)\, \d a \,\d\beta\\
&=|\tfrac{D}{2}|^{1/2}\mu_\infty(\det\delta)\int_{\R} \left(1-ia\det \delta\tfrac{\sqrt{|D|}}{2}\right)^{s-l+1}
\biggl\{\int_{\C}
\{\left(1-ia\det \delta \tfrac{\sqrt{|D|}}{2}\right)^2+\beta \bar \beta\}^{-s-1}\d \beta \biggr\} \\ 
&\quad \times \exp(-\pi i a(\det \delta)|D|)\,\d a.
\end{align*}
Then, using 
$$
\int_{\C}(A^2+\beta\bar\beta)^{-(s+1)}\,\d \beta
=\frac{\Gamma_\C(s)}{\Gamma_\C(s+1)}A^{-2s}, \quad \Re(A)>0,\, \Re(s)>0, 
$$
and by \eqref{MainT-L2-f1}, we immediately get \eqref{MainTSg-L2-f1}. 
\end{proof}

By the first parts of Lemmas~\ref{MainTSg-L1} and \ref{MainTSg-L2}, we can apply Fubini's theorem to have formula \eqref{JJ14-f0} for $1<c<l-2$. This completes the proof of formula \eqref{MainT-P1-f1}. By \eqref{MainTSg-L1-f3} and \eqref{MainTSg-L2-f1}, we have 
$$
\JJ_{l}^{w_2({\rm s})}(\phi,f\mid s,\Lambda,\mu):= |D|_\infty^{-s}\varepsilon(s,\Lambda\mu_{E})^{-1}\frac{\widehat L(s,\Lambda\mu_E)}{\widehat L(s+1,\Lambda\mu_E)}\sum_{\delta\in \{1_2,\sigma\}} {\mathcal I}_{\widehat \phi, f}^{(-s,\Lambda^{-1},\mu^{-1})}(\delta)\times \,{\mathcal I}_{l}^{(-s,\mu)}(\delta)
$$
which, combined with the functional equation of $\widehat L(s,\Lambda\mu_{E})$, shows the formula \eqref{MainT-P1-f2}.

\section{Preliminaries for error term estimates} \label{sec:ErrorEST}
To handle the terms $\JJ_{\su}(\Phi)$ with $\Phi={\widehat {\bf \Phi}_{\phi,f,l}}^{(\beta,\Lambda,\mu)}$ for $\su \in \{w_2({\rm r}),w_1,\sn(T_\theta^\dagger)w_2\}$, additional notation is in order. Let us extend \eqref{def-fXsu} by setting $\fX(w_2({\rm r})):=\{1\}\times (I_\theta(E^\times)\bsl ({\bf GL}_2(\Q)-I_\theta(E^\times)\{1_2,\sigma\}))$. We start from \eqref{BesselPer-L1-f2}; by changing the order of $n$-integral and the summation, and then by $n\rightarrow \sm(\delta,\lambda \det \delta)^{-1}n \sm(\delta,\lambda \det \delta)$ after that, we obtain that $\JJ_\su(\Phi)$ equals
{\small
\begin{align*}
\int_{\A_E^\times/E^\times \A^\times
}\Lambda(\tau)
\Bigl\{
\sum_{(\lambda,\delta)\in \fX(\su)}
\int_{\sN(\A)} \psi(\lambda \tr(T_\theta \ss(\delta^{-1})\,X))^{-1}\, 
  \Phi (\su \sn(X) \sm(\delta,\lambda \det \delta)\,\sm_{\theta}(\tau)b_{\R}^\theta)\,\d X\Bigr\}\,\d^\times \tau,
\end{align*}}
where the symbol $w_2({\rm r})$ appearing in the integral is simply $w_2$. Substituting \eqref{KerFtn-f1} to this, and formally changing the order of integrals and summation, we get
\begin{align}
\JJ_{\su}(\Phi)=\int_{c-i\infty}^{c+i \infty} \beta(s)s(s^2-1)\widehat L(s+1,\Lambda\mu_E)\,\JJ_{l}^{\su}(\phi,f\mid s,\Lambda,\mu)\,\d s \times e^{-2\pi \sqrt{|D|}}, 
 \label{JJu-f1}
\end{align}
where
\begin{align}
\JJ_{l}^{\su}(\phi,f \mid s,\Lambda,\mu)
:=\int_{\A_E^\times/E^\times \A^\times}\Lambda(\tau)\,\bigl\{\sum_{(\lambda, \delta) \in \fX(\su) }\cJ_{\phi,f,l}^{\su,(s,\Lambda,\mu)}(\lambda,\delta I_\theta(\tau))\bigr\}\,\d^\times \tau 
\label{JJu-f10}
\end{align}
with
{\small \begin{align}
\cJ_{\phi,f,l}^{\su,(s,\Lambda,\mu)}(\lambda,h)&:= e^{2\pi \sqrt{|D|}}\,\int_{\sV(\A)}\Phi_{\phi,f,l}^{(s,\Lambda,\mu)}(
\su \sn(X)\,\sm(h,\lambda^{-1}\det h)\,b_\R^\theta)\,\psi(\lambda^{-1}\tr(T_\theta \ss(h^{-1})\,X))^{-1}\, \d X
 \label{JJu-f2}
\end{align}}
for $(\lambda,h)\in \Q^\times \times {\bf GL}_2(\A)$. Define 
{\small
\begin{align}
\cJ_{\phi,f}^{\su,(s,\Lambda,\mu)}(\lambda, h_{\fin})&:=\int_{\sV(\A_\fin)}
\psi(\lambda^{-1} \tr((T_\theta \ss(h_\fin^{-1}))\,X))^{-1}\, {\Phi}_{\phi,f}^{(s,\Lambda,\mu)}(\su \sn(X)\sm(h_{1,\fin}, \lambda^{-1}\det h_{1,\fin}))\,\d X,
 \label{JJu-f5}
\\
\cJ_{l}^{\su,(s,\mu)} (\lambda, h_\infty)&:=e^{2\pi \sqrt{|D|}}\,\int_{\sV(\R)}\psi(\lambda^{-1} \tr((T_\theta \ss(h_\infty^{-1}))\,X))^{-1}\, {\Phi}_{l}^{(s,\mu_\infty)}(\su \sn(X)\sm(h_{\infty}, \lambda^{-1}\det h_{\infty})b_\R^\theta)\,\d X, 
 \label{JJu-f6}
\end{align}}
so that $\cJ_{\phi,f,l}^{\su,(s,\Lambda,\mu)}(\lambda,h)$ for $h=h_{\fin}h_{\infty}\in {\bf GL}_2(\A)$ equals $\cJ_{\phi, f}^{\sn,(s,\Lambda,\mu)}(\lambda, h_{\fin}) \times \cJ_{l}^{\su,(s,\mu)}(\lambda, h_{\infty})$. Since $\A_{E}^\times/E^\times\A^\times$ is compact, we can fix a compact set $\cN_\fin \subset \A_{E,\fin}^\times$ and $\cN_\infty \subset E_\infty^\times$ such that $\cN_\fin \cN_\infty$ is mapped surjectively onto $\A_{E}^\times/E^\times \A^\times$. To make the argument so far rigorous by Fubini's theorem, it is sufficient to establish the estimates, 
\begin{align}
&\int_{\sV(\A)} |{\bf \Phi}_{\phi,f,l}^{(s,\Lambda,\mu)} (\su \sn(X) \sm(h, \lambda^{-1}\det h) b_\R^\theta)|\,\d X \ll_{h}  1, \quad s\in \cT_I,
  \label{JJu-f3}
\\
& \sum_{(\lambda, \delta)\in \fX(\su)}|\cJ_{\phi,f,l}^{\su,(s,\Lambda,\mu)}(\lambda,\delta I_\theta(\tau))| \ll \exp(\tfrac{\pi}{2}|\Im (s)|), \quad s\in \cT_{I},\, \tau\in \cN_\fin\,\cN_\infty
 \label{JJu-f4}
\end{align}
for a compact interval $I\subset \R$, and choose the $c$ in \eqref{JJu-f1} from the interval $I$. To get \eqref{JJu-f3} with $h=h_\fin h_\infty \in {\bf GL}_2(\A)$, we shall prove estimations of its factors
\begin{align}
&\int_{\sV(\A_\fin)} |{\Phi}_{\phi,f}^{(s,\Lambda,\mu)} (\su \sn(X) \sm(h_\fin, \lambda^{-1} \det h_\fin) )|\,\d X,
 \label{PreJJw1-f5}
 \\
&\int_{\sV(\R)} |{\Phi}_{l}^{(s,\mu_\infty)} (\su \sn(X) \sm(h_\infty, \lambda^{-1}\det h_\infty) b_\R^\theta)|\,\d X
 \label{PreJJw1-f6}
\end{align}
for $\su=w_2({\rm r}), w_1, \sn(T_\theta^\dagger)w_2$ in \S\ref{sec:JJw2reg}, \S\ref{sec:JJw1} and \S\ref{sec:JJTtw2}, respectively. The integrals in \eqref{JJu-f6} are exactly evaluated (see Lemmas \ref{JJw2reg-L2}, \ref{JJw_1-L2} and \ref{JJTtw2-L2}) by the method to be explained in \S\ref{sec:CompArchInt}, whereas the integrals in \eqref{JJu-f5} are rather crudely bounded. Using these, \eqref{JJu-f4} will be established. 

\subsection{Convergence of the archimedean integrals}

In this section, we shall prove  

\begin{lem}\label{ConvAInt-P} The integral \eqref{PreJJw1-f6} converges on $\cT_{(1,l-3)}$ normally for $(\lambda,h_\infty)\in \R^\times \times {\bf GL}_2(\R)$.  
\end{lem}
\begin{proof} In the proof, a general element of ${\bf GL}_2(\R)$ will be written as $h$. From \eqref{Def-ShinFtn}, by a direct computation, for $X\in \sV(\R),\,h\in {\bf GL}_2(\R),\,\lambda \in \R^\times$, we get 
\begin{align}
\Phi_l^{(s,\mu_\infty)}(\su \sn(X)\sm(h,\lambda^{-1}\det h)b_{\R}^\theta)
=\mu_\infty(\lambda^{-1}\det h)\lambda^{-l}A^{-l}\left(\tfrac{{\rm sgn}(\lambda \det h)}{i \sqrt{|D|}}\tfrac{B}{A}\right)^{s-l+1}
 \label{ConvAInt-P-f0}
\end{align}
with
\begin{align}
A&=
\begin{cases}
\lambda^{-1} \langle \lambda^{-1}X^\dagger+\tfrac{2i}{\sqrt{|D|}}T_\theta\ss(h^{-1}),\e_3\rangle 
, \quad &(\su=w_1), \\
\frac{1}{2}\langle \lambda^{-1}X^{\dagger}+\tfrac{2i}{\sqrt{|D|}}T_\theta\ss(h^{-1})\rangle^2, \quad &(\su=w_2,\,\sn(T_\theta^\dagger)w_2),
\end{cases} 
 \label{ConvAInt-P-f1}
\\
B&=\begin{cases} 
\frac{1}{2}\langle \lambda^{-1}(X^\dagger-t\e_2)+\frac{2i}{ \sqrt{|D|}}T_\theta\ss(h^{-1}) \rangle^2-\frac{|D|}{4\lambda^2}, \quad &(\su=w_1), \\
-\lambda^{-1} \langle\lambda^{-1} X^\dagger +\tfrac{2i}{\sqrt{|D|}}T_\theta\ss(h^{-1}),T_\theta \rangle, \quad &(\su=w_2), \\
\frac{|D|}{4} \langle \lambda^{-1}(X^\dagger+\tfrac{2}{|D|}T_\theta)
+\tfrac{2i}{\sqrt{|D|}}T_\theta\ss(h^{-1})\rangle^{2}-\tfrac{1}{2\lambda^2}, \quad &(\su=\sn(T_\theta^\dagger)w_2),
\end{cases}
 \label{ConvAInt-P-f2}
\end{align}
where $w_2$ is as in \eqref{Def-iotaf3}, $w_1$ is in \eqref{Def:WeyELM} and $\e_1,\e_2,\e_3$ are as in \eqref{BasisV} and $\langle Z\rangle^2:=\langle Z, Z\rangle$. Thus, 
$$
|\Phi_l^{(s,\mu_\infty)}(\su \sn(X)\sm(h,\lambda^{-1}\det h)b_{\R}^\theta)|\ll_{l,s}
|A|^{-\Re(s)-1}|B|^{\Re(s)-l+1}. 
$$
We consider the case $\su=\sn(T_\theta^\dagger)w_2$. By Lemma \ref{JJTtw2-L3} (1), we obtain the estimate for $A,B$, viewed as functions in $X \in \sV(\R)$: 
$$
|A| \gg (x-\lambda^{-1}c)^2+1 , \quad |B|^2\gg (x^2+1)\left\{1+(x^2+1)^{-1}\|Y\|^2\right\}^2
$$
if we write $\lambda^{-1}(X^\dagger+\frac{2}{|D|}T_\theta)=x\tfrac{2}{\sqrt{|D|}}T_\theta\ss(h^{-1})+Y$ with $x\in \R$ and $Y\in (T_\theta\ss(h^{-1}))^{\bot}$ and set $a=\frac{2}{\lambda^2|D|}$, $c=2|D|^{-3/2}\langle T_\theta, T_\theta\ss(h^{-1})\rangle$, where $\|\cdot\|$ is a norm on $\sV(\R)$. From this, when $\Re(s)>-1$ and $\lambda$ is large, we can observe that \eqref{PreJJw1-f6} with $\su=\sn(T_\theta^\dagger)w_2$ is majorized, locally uniformly in $(\lambda,h)$, by 
\begin{align*}
\int_{\R} \int_{\R^2} \{x^2+1+(x^2+1)^{-1}\|Y\|^4\}^{(-l+\Re s+1)/2}\{(x-\lambda^{-1}c)^2+1\}^{-(\Re(s)+1)}\,\d x\,\d Y,
\end{align*}
which, by the change of variables $Y=(x^2+1)^{1/2}Y_1$, equals 
$$
\int_{\R} \int_{\R^2} \{1+\|Y_1\|^4\}^{(-l+\Re s+1)/2}\{(x-\lambda^{-1}c)^2+1\}^{-(\Re(s)+1)}(x^2+1)^{(-l+\Re(s)+1)/2+1}\,\d x\,\d Y_1.
$$
This is convergent for $s\in \cT_{(-1,l-2)}$. This completes the proof for $\su=\sn(T_\theta^\dagger)w_2$. For $\su=w_2({\rm r})$, we estimate $|B|\gg 1$ by Lemma \ref{JJTtw2-L3} (2) and $|A|$ by Lemma \ref{JJTtw2-L3} (1), and argue as above to get the convergence for $s\in \cT_{(1,l-1)}$. For $\su=w_1$, we estimate $|A|\geq 1$ by Lemma \eqref{JJTtw2-L3} (2) and $|B|$ by Lemma \eqref{JJTtw2-L3} (1) to get the convergence for $s\in \cT_{(-1,l-3)}$. 
\end{proof}

\begin{lem}\label{JJTtw2-L3}
\begin{itemize}
\item[(1)]  Let $Z\in \sV(\R)$ be a vector with $\det(Z)>0$. Let $a\geq 0$. Then, there exists a norm $\|\cdot\|$ on $\sV(\R)$  such that
\begin{align*}
|\langle X+iZ\rangle^{2}-a|^2 \geq (2\det Z)^2 (x^2+1) \left\{1+\frac{\|X_Z\|^2+a}{2(\det Z)(x^2+1)}\right\}^2
\end{align*}
for any $X=xZ+X_Z$ with $x\in \R$, $X_Z\in Z^\bot$. 
\item[(2)] Let $Z,Z'\in \sV(\R)$ be vectors such that $\det(Z)>0$ and $\det(Z')>0$. Then, $\langle Z,Z'\rangle\not=0$, $\langle Z,\e_3\rangle\not=0$, and 
\begin{align*}
|\langle X+iZ,Z'\rangle|\geq |\langle Z,Z'\rangle|, \quad |\langle X+iZ,\e_3\rangle|\geq |\langle Z,\e_3\rangle|, \quad X\in \sV(\R). 
\end{align*}
\end{itemize}
\end{lem}
\begin{proof} (1) Since $Z^\bot$ is a toatally negative subspace of $\sV(\R)\cong \R^3$, we may define a norm $\|X\|$ on $\sV(\R)$ by $\|X\|:=(x^2-\langle X_Z\rangle^2)^{1/2}$ for $X=xZ+X_Z\,(x\in \R,\,X_Z\in Z^\bot)$. Set $C:=\langle Z\rangle^2\,(=2\det Z)>0$ and $y:=C^{-1}(x^2+1)^{-1}(\|X_Z\|^2+a)$. Then a computation shows
\begin{align*}
|\langle X+iZ\rangle^2-a|^2
&=C^2 \{x^2+1-C^{-1}(\|X_Z\|^2+a)\}^2+4C(\|X_Z\|^2+a)
\\
&=
C^2(x^2+1)^{2}(1-y)^{2}+4C^2(x^2+1)y
\\
&\geq (x^2+1)\{C^2(1-y)^2+4C^2y\}
=(x^2+1)C^2(y+1)^2. 
\end{align*}
(2) If $\langle Z,Z'\rangle$ were $0$, then $\R Z+\R Z'$ would be a totally positive plane in $\sV(\R)$, which does not exist because the signature of $\langle\cdot,\cdot\rangle$ is $(1+,2-)$. Hence, $\langle Z,Z'\rangle\not=0$. Since $(\e_3)^\bot=\R\e_1+\R\e_2=\{\left[\begin{smallmatrix} 0 & b \\ b & c\end{smallmatrix}\right]\mid,b,c\in \R\}$ and $\det \left[\begin{smallmatrix} 0 & b \\ b & c\end{smallmatrix}\right]=-b^2\leq 0$, the space $(\e_3)^\bot$ does not contain a vector $Z$ with $\det Z>0$. Hence, $\langle Z,\e_3\rangle\not=0$. The inequalities are obvious. 
\end{proof}

\subsection{Computation of archimedean integrals}\label{sec:CompArchInt}
For $s\in \C$ on the convergence region of \eqref{PreJJw1-f6}, we consider the integral 
\begin{align}
W_{\su}^{Z}(g):=\int_{\sV(\R)} \Phi_l^{(s,\mu_\infty)}(\su \sn(X)g)\,\exp(-2\pi i \tr(ZX))\,\d X, \quad g\in \sG(\R)
 \label{JJw2reg-L2-f1}
\end{align}
for $Z\in \sV(\R)$ with $\det Z>0$ and $u \in \{w_2,w_1,\sn(T_\theta^\dagger)w_2\}$. At least on $\cT_{(1,l-3)}$, by Lemma \ref{ConvAInt-P}, we get the absolute convergence of the integral $W_\su^{Z}(g)$ for any $g\in \sG(\R)$ due to $\sG(\R)=\sP(\R)\bK_\infty$. Then, by Lemma \ref{ShinFtn-L1}, the function $W:=W_{u}^{Z}|\sG(\R)^0$ satisfies 
\begin{align}
W(\sn(X) g)&=\exp(2\pi i \tr(ZX))\,W(g), \quad X\in \sV(\R), \label{Whf1}
\\
W(gk)&=J(k,i 1_2)^{-l}\,W(g), \quad k\in \bK_\infty, \label{Whf2} \\
R_{\bar X}W(g)&=0, \quad X\in \fp^{+}. \label{Whf3}
\end{align}
\begin{lem}\label{CompArchIntL1} The function on $\sG(\R)^0$, 
\begin{align}
{\mathcal B}_l^{Z}:  g\longmapsto J(g,i1_2)^{-l}\,\exp(2\pi i \tr(Z\,g\langle i1_2\rangle)),
  \label{CompArchInt-L1-f}
  \end{align}is, up to a constant, a unique $C^\infty$-function that satisfies \eqref{Whf1}, \eqref{Whf2} and \eqref{Whf3}.
\end{lem}
\begin{proof} This is more or less well-known ({\it cf}. \cite{PitaleSchmidt09}). 
\end{proof}
Thus, there is a constant $C\in \C$ such that $W_u^{Z}(g)=C\,{\mathcal B}_l^{Z}(g)$ for $g\in \sG(\R)^0$. By $\sG(\R)=\sN(\R)\sM(\R)\bK_\infty$ and by \eqref{Whf1} and \eqref{Whf2}, to show the absolute convergence of the integral \eqref{JJw2reg-L2-f1}, it suffices to establish it for points $g=\sm(h,\lambda^{-1}\det h)b_\R^\theta$ $(h\in {\bf GL}_2(\R),\lambda \in \R^\times)$. A computation reveals that 
\begin{align}
    {\mathcal B}_l^{Z}(\sm(h,\lambda^{-1}\det h)b_\R^\theta)=\lambda^{l}\,\exp\left(\tfrac{-4\pi \lambda}{\sqrt{|D|}}\tr(Z^\dagger \,T_\theta\ss(h^{-1}))\right).
\label{CompArchInt-fff}
\end{align}
Hence, from $W_{\su}^{Z}(b_\R^\theta)=C {\mathcal B}_l^{Z}(b_\R^\theta)$, we get $C=e^{\frac{4\pi}{\sqrt{|D|}}\tr(Z^\dagger T_\theta)}\, W_u^{Z}(b_\R^\theta)$. Thus, 
\begin{align}
W_\su^{Z}(g)=e^{\frac{4\pi}{\sqrt{|D|}}\tr(Z^\dagger T_\theta)}\, W_u^{Z}(b_\R^\theta)\times 
{\mathcal B}_{l}^{Z}(g)
.\label{CompArchInt-f0}
\end{align}
When $\lambda^{-1}\det h>0$, we have $\sm(h,\lambda^{-1}\det h)b_\R^\theta\in \sG(\R)^0$ and
\begin{align}
\cJ_l^{\su,(s,\mu)}(\lambda,h)&=e^{2\pi \sqrt{|D|}}\,W_{\su}^{\lambda^{-1}T_\theta\ss(h^{-1})}(\sm(h,\lambda^{-1}\det h)b_\R^\theta)
\label{CompArchInt-f1}
\\
&=\lambda^l e^{\frac{4\pi}{\lambda|D|}\langle T_\theta, T_\theta\ss(h^{-1}\rangle}
\,W_\su^{\lambda^{-1} T_\theta\ss(h^{-1})}(b_\R^\theta)
 \notag
\end{align}
by comparing \eqref{JJu-f6} and \eqref{JJw2reg-L2-f1} and by using \eqref{CompArchInt-f0}. Noting $\sm(1_2,-1)=\iota\left(\left[\begin{smallmatrix} 1 & 0 \\ 0 & -1 \end{smallmatrix} \right]\right)$ and \eqref{ShinFtn-L1-f0}, we easily have $W_{\su}^{-Z}(g)=W_{\su}^{Z}(\sm(1_2,-1)g)$ for $g\in \sG(\R)$. Hence, when $\lambda^{-1}\det h<0$, we have $\sm(h,-\lambda^{-1}\det h)b_{\R}^\theta \in \sG(\R)^0$, and thus, by \eqref{CompArchInt-f1},  
\begin{align}
\cJ_{l}^{\su,(s,\mu)}(\lambda, h)
=\lambda^{l}
e^{\frac{-4\pi}{\lambda|D|}\langle T_\theta, T_\theta\ss(h^{-1}\rangle}
\,W_\su^{-\lambda^{-1} T_\theta\ss(h^{-1})}(b_\R^\theta).
\label{CompArchInt-f2}
\end{align}
The computation of the value $W_\su^{Z}(b_\R^\theta)$ will be done in \S\ref{sec:JJw2reg}, \S\ref{sec:JJw1} and \S\ref{sec:JJTtw2}.

\subsection{Convergence of certain infinite series}
We examine the infinite sum \begin{align}
\sum_{\delta \in I_\theta(E^\times)\bsl {\bf GL}_2(\Q)
} \ff_{R_1,R_2}(\delta h), \quad h\in {\bf PGL}_2(\A)
 \label{P6-L12-2}
\end{align}
with parameters $R_1,R_2>0$, where $\ff_{R_1,R_2}:{\bf GL}_2(\A) \longrightarrow \R_{\geq 0}$ is a function defined by 
\begin{align*}
\ff_{R_1,R_2}(h):=\|T_\theta\ss(h_\fin)\|_\fin^{-R_1}
\|T_\theta\ss(h_\infty)\|_\infty^{-R_2}, \qquad h\in {\bf GL}_2(\A)
\end{align*}
with $\|\cdot\|_\fin$ being as in \S\ref{sec:Ht} and $\|\cdot\|_\infty$ being any norm on $\sV(\R)$. Note that $\ff_{R_1,R_2}$ is left $I_{\theta}(\A_{E}^\times)$-invariant, and that $I_{\theta}(\A^\times_{E})$ contains the center of ${\bf GL}_2(\A)$. 

\begin{lem} \label{P6-L12}
If $R_1>2$ and $R_2>1$, then \eqref{P6-L12-2} converges normally on ${\bf PGL}_2(\A)$. 
\end{lem}
\begin{proof} We shall apply the integral test. Choose a complete set of representatives $D\subset {\bf PGL}_2(\Q)$ of $I_\theta(E^\times)\bsl {\bf GL}_2(\Q)$. Choose a compact neighborhood $\cV$ of the identity element $e$ in ${\bf PGL}_2(\A)$ such that $\cV\cV^{-1}\cap {\bf PGL}_2(\Q)=\{e\}$. Then, $D\cV:=\cup_{\delta\in D} \delta \cV$ is a disjoint union and the natural map ${\bf PGL}_2(\A)\rightarrow I_\theta(E^\times\A^\times)\bsl {\bf GL}_2(\A)$ is injective on $D\cV$.  Let $\cU$ be any compact subset of ${\bf PGL}_2(\A)$ then, for some constant $C>1$, 
$$
C^{-1}\ff_{R_1,R_2}(h u) \leq  \ff_{R_1,R_2}(h)\leq C\ff_{R_1,R_2}(hy), \qquad h\in {\bf GL}_2(\A),\,y \in {\mathcal V},\,u\in \cU.
$$
From the first inequality, the uniform convergence of \eqref{P6-L12-2} on $\cU$ results from the convergence at the point $h=1_2$. By the second inequality, $\vol({\mathcal V})\,\fF_{R_1,R_2}(1_2)$ is no greater than
\begin{align*} C\, \sum_{\delta \in D} 
\int_{y\in \cV} \ff_{R_1,R_2}(\delta y )\,\d y
\leq C_1\,\int_{I_\theta(E^\times\A^\times)\bsl{\bf GL}_2(\A)} \ff_{R_1,R_2}(h)\,\d h=C_1\prod_{p\leq \infty} I_p, 
\end{align*}
where $C_1=C\,\vol(\A_E^\times/E^\times \A^\times)$ and, for a place $p$ of $\Q$,  
$$I_p:=\int_{I_\theta(E_p^\times)\bsl{\bf GL}_2(\Q_p)} \|T_\theta \ss(h_p)\|_p^{-R_p} \,\d h_p
$$
with $R_p=R_1$ if $p<\infty$ and $R_\infty=R_2$. Let $S$ be the union of $\{\infty\}$ and the set of prime numbers $p$ such that $|D|_p\not=1$, $|\tt/2|_p\not=1$ or $|\ttN|_p\not=1$ ({\it cf}. \S\ref{sec:GO}). Then, it suffices to prove $I_p<\infty$ for all $p$ and that the infinite product $\prod_{p\not\in S}I_p$ converges. By \cite[Lemma 2-4]{Sugano84}, 
$$
{\bf GL}_2(\Q_p)=\bigcup_{n\in \Z_{\geq 0}} I_\theta(E_p^\times)\left[\begin{smallmatrix} 1 & 0 \\ 0 & p^{n-e_p} \end{smallmatrix} \right]{\bf GL}_2(\Z_p),
$$
where $e_p:={\rm ord}_{p}(D)$. Set $\cU_n:=\left[\begin{smallmatrix} 1 & 0 \\ 0 & p^{n-e_p} \end{smallmatrix} \right]{\bf GL}_2(\Z_p)\left[\begin{smallmatrix} 1 & 0 \\ 0 & p^{n-e_p} \end{smallmatrix} \right]^{-1}\cap I_\theta(E_p^\times)$, which is identified with the set of all the matrices $a+\theta b\in E_p^\times$ such that $a\in \Z_p$, $b\in p^{n-e_p}\Z_p\cap p^{-(n-e_p)}\ttN\Z_p$ and $\nr_{E_p/\Q_p}(a+\theta b)\in \Z_p^\times$. We have that $\vol(\cU_n):=\int_{\cU_n}\d\eta_{E_p^\times}=c_p\zeta_{E,p}(1)(1-p^{-1})\,p^{-n+e_p}$ if $n>e_p$. Thus, with respect to the quotient measure on $I_{\theta}(E_p^\times)\bsl {\bf GL}_2(\Q_p)$, the set 
$$Y_n:=(\left[\begin{smallmatrix} 1 & 0 \\ 0 & p^{n-e_p} \end{smallmatrix} \right]^{-1}\cU_n\left[\begin{smallmatrix} 1 & 0 \\ 0 & p^{n-e_p} \end{smallmatrix} \right]\cap {\bf GL}_2(\Z_p))\bsl {\bf GL}_2(\Z_p))$$ has the measure $\vol(\cU_n)^{-1}$ if $n>e_p$. Suppose $p\not\in S$, then $e_p=0$ and $\ttN\in \Z_p^\times$. Hence, $\vol(Y_0)=\vol(\cU_0)^{-1}=1$. Hence, for $p\not\in S$,
{\allowdisplaybreaks \begin{align*}
I_{p}&=\sum_{n=0}^{\infty}
 \vol(Y_n)\,\|T_\theta\ss(\left[\begin{smallmatrix} 1 & 0 \\ 0 & p^{n-e_p} \end{smallmatrix} \right])\|_p^{-R_1}
\\
&\leq 1+\sum_{n=1}^{\infty} \zeta_{E,p}(1)^{-1}(1-p^{-1})^{-1}p^{n}\times \max(p^{n},|\tt/2|_p,p^{-n}|\ttN|_p)^{-R_1}
\\
&=
1+\sum_{n=1}^{\infty}(1+\varepsilon p^{-1})^{-1} p^{-(R_1-1)n}=
\frac{1+\varepsilon p^{-R_1}}{1-p^{-R_1+1}}
\end{align*}}with $\varepsilon$ equal to $1$ or $-1$ according to $p$ being inert in $E$ or spliting in $E$. In the same way, $I_p<\infty$ for any $p\in S$. When $R_1>2$, then $\prod_{p<\infty}\frac{1+\varepsilon p^{-R_1}}{1-p^{-R_1+1}}$ converges, so does $\prod_{p<\infty} I_p$. 
Since $I_{\theta}(E_\infty^\times)$ is conjugate in ${\bf GL}_2(\R)$ to $\R^\times {\rm SO}(2)$, there exists a constant $C_2>0$ such that $$
\int_{{\bf PGL}_2(\R)}f(g)\,\d g=C_2\int_{0}^{\infty} f(\left[\begin{smallmatrix} e^{t} & 0 \\ 0 & e^{-t}\end{smallmatrix} \right])\,\sinh(2t)\,\d t
$$ 
for any $f\in C(\R^\times {\rm SO}(2)\bsl{ 
\bf GL}_2(\R))$. Apply this to $f(h)=\|T_\theta \ss(h)\|_\infty^{-R_2}$. Then, since $f(\left[\begin{smallmatrix} e^{t} & 0 \\ 0 & e^{-t}\end{smallmatrix} \right])=O(e^{-2R_2t})$, we have $I_\infty<\infty$ if $R_2>1$. \end{proof}

\subsection{A Dirichlet series from orbits} \label{sec:JJw2regPre}
By \eqref{GLaction} and \eqref{P2-f0}, 
\begin{align}
\langle X\ss(h), Y\ss(h)\rangle=\langle X,Y\rangle, \quad X,Y \in \sV(\Q),\, h \in {\bf GL}_2(\Q).
\label{productInv}
\end{align}
Recall \eqref{Def-Xbeta}. By means of the orthogonal decomposition $\sV(\Q)=\Q T_\theta\oplus \sV^{T_\theta}(\Q)^\dagger$, define $\Q$-linear maps $c:\sV(\Q) \rightarrow \Q$ and $\gamma:\sV(\Q)\rightarrow E$ such that $Z=c(Z)\,T_\theta +X_{\gamma(Z)}^\dagger$, where $X_\gamma^\dagger:=(X_\gamma)^\dagger$. A computation reveals 
\begin{align}
c(Z)=\tfrac{-2}{D} \langle T_\theta,  Z\rangle,\quad 
\nr_{E/\Q}(\gamma(Z))=\tfrac{-1}{D}\langle T_\theta,  Z\rangle^2-\det Z.
 \label{JJw2regPre-f0}
\end{align}
By \eqref{Ttheta-f1}, the point $T_\theta$ is in the set 
$$
\Omega_D(\Q):=\left\{Z\in \sV(\Q)\mid \det Z=\tfrac{-D}{4}\right\}.
$$

\begin{lem} \label{JJw2regPre-L1}
 Let $Z\in \Omega_D(\Q)$; then, $|c(Z)|\geq 1$ with the equality if and only if $Z=\pm T_\theta$.  
\end{lem}
\begin{proof}  By \eqref{JJw2regPre-f0}, 
we have
$\tfrac{-D}{4}(c(Z)^2-1)=
\tfrac{-D}{4}c(Z)^2-\det Z=\nr_{E/\Q}(\gamma(Z)) \geq 0$ with the equality if and only of $\gamma(Z)=0$, or equivalently $Z \in \Omega_{D}(\Q) \cap \Q T_\theta=\{\pm T_\theta\}$. \end{proof}

\begin{lem} \label{JJw2regPre-L2}The set $\Omega_D(\Q)$ coincides with the ${\bf GL}_2(\Q)$-orbit of the point $T_\theta$. The map $g\mapsto T_\theta\ss(g)$ induces a bijection from $I_{\theta}(E^\times)\bsl {\bf GL}_2(\Q)$ onto $\Omega_D(\Q)$. The coset $I_{\theta}(E^\times)$ and $I_{\theta}(E^\times)\sigma$ are mapped to $T_\theta$ and $-T_\theta$, respectively. 
\end{lem}
\begin{proof} By Witt's theorem, the set $\Omega_D(\Q)$ is a single ${\rm SO}(\sV(\Q))$-orbit. It turns out that $\ss({\bf GL}_2(\Q))={\rm SO}(\sV(\Q))$ by \eqref{productInv}.
\end{proof}
For $h\in {\bf GL}_2(\R)$, set 
\begin{align}
\fq(h):=|c(T_\theta \ss(h))|\,(=\tfrac{2}{|D|}|\langle T_\theta, T_\theta \ss(h) \rangle|).
 \label{Def:fq}
\end{align} 
For a $\Z$-lattice $\cM\subset \sV(\Q)$, define
\begin{align}
\zeta_D(s,\cM):=\sum_{\substack{\delta \in I_{\theta}(E^\times)\bsl {\bf GL}_2(\Q) \\ T_\theta \ss(\delta)\in \cM}} \fq(\delta)^{-s}, \quad s\in \R. 
 \label{Def:HypZeta2}
\end{align}

\begin{lem} \label{ZetaDCnv}
 The series $\zeta_{D}(s,\cM)$ in \eqref{Def:HypZeta2} converge absolutely for $s>1$. 
\end{lem}
\begin{proof} Since $\widehat \cM:=\cM\otimes_\Z\widehat \Z \subset \sV(\A_\fin)$ is compact, we have $\cchi_{\widehat \cM}(Z)\ll \|Z\|_\fin^{-3}$ for $Z\in \sV^\star(\A_\fin)$. Since $\langle T_\theta, T_\theta\rangle=2\det T_\theta=\frac{|D|}{2}>0$ and ${\rm sgn}\sV(\R)=(1+,2-)$, the orthogonal $T_\theta^{\bot}$ is negative definite; hence, by \eqref{JJw2regPre-f0}, we have a bound 
$\{|c(Z)|^2+1\}^{1/2}\asymp \|Z\|_\infty$ for $Z\in \Omega_D(\R)$.  Hence, \eqref{Def:HypZeta2} is majorized by \eqref{P6-L12-2} with $R_1=3$ and $R_2=s$. By Lemma \ref{P6-L12}, we have the absolute convergence of \eqref{Def:HypZeta2}.  
\end{proof}

%\begin{lem} \label{JJw2reg-L5}
% There exists $\varepsilon=\varepsilon_{\cM}>0$ such that $|c(Z)|\geq 1+\vareps%ilon$ for all $Z \in \cM-\{\pm T_\theta\}$. 
%\end{lem}
%\begin{proof} Choose $m\in \Z_{>0}$ such that $m\cM\subset \sV(\Z)$. Then, $Z
% \in \cM$ implies $m\,\tr(T_\theta Z^\dagger) \in \Z$, and hence $|c(Z)| \in \frac{2}{|D|M}\Z$. Combined with Lemma~\ref{JJw2regPre-L1}, this means that the points $|c(Z)|\,(Z\in \Omega_{D} \cap \cM)$ is in a discrete subset of $\R_{>1}$. By Lemma \ref{JJw2regPre-L1}, we are done.
%\end{proof}

\section{Error term estimates}\label{sec:ErrTermEstimate}
In this section, we shall prove Propositions \ref{ErrorT-P000} and \ref{ErrorT-P1}. To describe the dependence of error terms on the Hecke function $f_{S}\in \cH_S$, we recall $\fa(f_S)$ in \eqref{Def-HeckeAlg-f0} and the $L^1$-norm $\|f_S\|_1$.

\subsection{Estimation of $\JJ_{w_2({\rm r})}(\Phi)$} \label{sec:JJw2reg}
Recall the set $\cN_\fin\subset \A_{E,\fin}^\times$ from the preamble of \S\ref{sec:ErrorEST}.  

\begin{lem}\label{JJw2reg-L1}
Let $\phi=\otimes_{p<\infty}\phi_p\in\cS(\A_{E,\fin}^2)$ and $f=\otimes_{p<\infty}f_p\in\cH(\sG(\A_\fin))$ be as in \S\ref{sec:ErrTEST}. For a compact interval $I\subset (1,\infty)$, the integral \eqref{PreJJw1-f5} for $\su=w_2({\rm r})$ has the uniform bound
 \begin{align*}
\cJ_{\phi,f}^{w_2({\rm r}),(s,\Lambda,\mu)}(\lambda, h_{1,\fin})\ll_{M}N^{-\Re(s)}\,\fa(f_S)^{\Re(s)-2}\,\|f_S\|_1, \quad  s\in \cT_{I}.
\end{align*}
Moreover, there exists a lattice $\cM$ depending only on $E$ such that the integral for $h_{1,\fin}=\delta I_\theta(\tau)$ $(\delta \in {\bf GL}_2(\Q),\,\tau \in \cN_\fin)$ is zero unless $T_\theta \ss(\delta)\in \fa(f_S)^{-1}\cM$. 
\end{lem}
\begin{proof}
By \eqref{DefJaqS}, \eqref{normsec} and \eqref{ad-normsec}, we have $|\sf^{(s,\Lambda,\mu)}_\phi(g^\#)|\ll_{I} \sf_\phi^{(\Re(s),{\bf 1}, {\bf 1})}(g^\#)$ for $s\in \cT_{I},\,g^\#\in \sG^\#(\A)$. By arguing as in the proof of Lemma \ref{MainTSg-L1}, the integral in \eqref{PreJJw1-f5} with $\su=w_2$ is majorized by 
\begin{align*}
\int_{\sG^\#(\A_\fin)}\d h \biggl\{ \int_{\A_{E,\fin}} \sf_{\phi}^{(\Re(s),{\bf 1},{\bf 1})}
\left(\left[\begin{smallmatrix} 0 & 1 \\ 1 & 0 \end{smallmatrix} \right]\left[\begin{smallmatrix} 1 & -\beta \\ 0 & 1 \end{smallmatrix} \right]h \right)\,\d \beta\biggr\}\, \int_{\A_{\fin}}|f(\iota(h)^{-1} \sn(a T_\theta^\dagger)\sm(h_{1,\fin}, \det h_{1,\fin}))|\,\d a.
\end{align*}
Since the $\beta$-integral is evaluated to be $|D|^{1/2}\frac{L(\Re(s),{\bf 1})}{L(\Re(s)+1,{\bf 1})}\sf_{\widehat{\phi}}^{(\Re(s),{\bf 1},{\bf 1})}(h)$ when $\Re(s)>1$, and $|\widehat\phi|\leq \phi_0:=\cchi_{M^{-1}\widehat{\fo_E}\oplus M^{-1}\widehat{\fo_E}}$ by \eqref{FrTr-phie}, we have that 
\begin{align}
\int_{\A_\fin} \Phi_{\phi_0,|f|}^{(-\Re(s),{\bf 1},{\bf 1})}
(\sn(a T_\theta^\dagger)\sm(h_{1,\fin}, \det h_{1,\fin}))\,\d a\label{JJw2reg-int2}
\end{align}
 is a uniform majorant of the above expression for $s\in \cT_{I}$. Suppose that $f_S$ is the characteristic function of a single $\bK_S$-double coset; fix elements $y_i=\left[\begin{smallmatrix} \gamma_i & 0 \\ 0 & \lambda_i ({}^t\gamma_i)^\dagger \end{smallmatrix}\right]\left[\begin{smallmatrix} 1_2 & X_i \\ 0 & 1_2 \end{smallmatrix}\right]\in\sB(\A_\fin)$ $(1\leq i\leq n)$ with $y_{i,p}=1_4$ for $p\notin S\cup S(M)$ such that ${\rm supp}(f)=\cup_{i=1}^n \bK_0(N) y_i$. 
  Then, $|\lambda_i|_{\fin}^{-1} \leq \fa(f_S)$ by \eqref{Def-HeckeAlg-f5} and $\|f_S\|_1=n$. By noting that $f_p(k^\#)\leq 1$ for any $k^\#\in\bK_p$ for all $p$, the integral \eqref{JJw2reg-int2} is majorized by
{\allowdisplaybreaks \begin{align}
&\sum_{i=1}^n\int_{\A_\fin}\int_{\bK_\fin^\#}\int_{\A_{E,\fin}^\times}\int_{\A_\fin^\times}\int_{\A_{E,\fin}}\d\eta_{\A_{E,\fin}}(\beta)\d\eta_{\A_\fin^\times}(\lambda)\d\eta_{{\A_{E,\fin}}^\times}(\tau)\d\eta_{\bK_\fin^\#}(k^\#)\d\eta_{\A_\fin}(a)
\label{JJw2reg-int3}
\\
&\quad\cchi_{\bK_0(N)y_i}(\iota(k^\#)^{-1}\sm(I_\theta(\tau),\lambda \nr(\tau))^{-1}n(X_\beta)^{-1}\sn(a T_\theta^{\dagger})\sm(h_{1,\fin},\det h_{1,\fin})\eta)|\lambda|_\fin^{\Re(s)+1}
\notag
\end{align}}multiplied with a positive constant depending only on $M$. Recall the element $\eta=(\eta_p)_{p<\infty}\in\sG(\A_\fin)$ in \S\ref{sec:ErrTEST} (ii). By the inequality $\int_{\bK_\fin^\#}\cchi_{\bK_0(N)y_i}(\iota(k^\#)g)\d\eta_{\bK_\fin^\#}(k^\#)=[\bK_\fin^\#:\bK_0^\#(N)]^{-1}\cchi_{\iota(\bK_\fin^\#)\bK_0(N)}(gy_i^{-1})\leq \prod_{p|N}(p^2+1)^{-1}\cchi_{\bK_\fin}(gy_i^{-1})\leq N^{-2}\cchi_{\bK_\fin}(gy_i^{-1})$ for $g\in\sG(\A_\fin)$ and by the change of variables $\tau\mapsto\tau^{-1}$, $\lambda\mapsto\lambda^{-1}$, $\beta\mapsto-\beta$, and $X=aT_\theta^\dagger+X_\beta\mapsto X\ss({}^th_{1,\fin}\gamma_i^{-1})-X_i$, the integral \eqref{JJw2reg-int3} is majorized by
\begin{align}
N^{-2}\sum_{i=1}^n\int_{\sV(\A_\fin)}\int_{\A_{E,\fin}^\times}\int_{\A_\fin^\times}&\cchi_{\bK_\fin}\left(\left[\begin{smallmatrix} I_\theta(\tau)h_{1,\fin}\gamma_i^{-1} & \lambda_i^{-1}I_\theta(\tau)h_{1,\fin}\gamma_i^{-1}(1_2+X) \\ 0 & \lambda \lambda_i^{-1}({}^tI_\theta(\tau))^\dagger ({}^th_{1,\fin})^\dagger ({}^t\gamma_i^{-1})^\dagger\end{smallmatrix} \right]\eta\right)\notag\\
&\times|\lambda|_\fin^{-\Re(s)-1}\d\eta_{\A_\fin^\times}(\lambda)\d\eta_{\A_{E,\fin}^\times}(\tau)\d\eta_{\sV(\A_\fin)}(X).\label{JJw2reg-Int}
\end{align}
It is easy to check that the $4\times 4$-matrix inside the bracket is in $\bK_\fin$ if and only if $I_\theta(\tau)h_{1,\fin}\gamma_i^{-1}\in {\bf GL}_2(\widehat{\Z})$, $\lambda\in \lambda_i N^{-1}\widehat{\Z}^\times$, and $X\in -1_2+\lambda_iN^{-1}\sV(\widehat{\Z})$. The $\tau$-integral does not affect the evaluation of the inequality because the set of all $\tau\in \A_{E,\fin}^\times$ satisfying $I_\theta(\tau)h_{1,\fin}\gamma_i^{-1}\in {\bf GL}_2(\widehat{\Z})$ is either the empty set or a set of the form $\tau_0\widehat{\fo_E}^\times$ ($\tau_0\in\A_{E,\fin}^\times$), which has volume $1$. Hence, we have that $
\cJ_{\phi,f}^{w_2({\rm r}),(s,\Lambda,\mu)}(\lambda, h_{1,\fin})$ is majorized by 
\begin{align*}N^{-2}\cdot \sum_{i=1}^{n}(|\lambda_i|_\fin N)^{-\Re(s)-1}\cdot \vol(\tfrac{\lambda_i}{N}V(\widehat{\Z}))=N^{-\Re(s)}\,\fa(f_S)^{\Re(s)-2}\|f_S\|_1.  
\end{align*}
Moreover, $I_\theta(\tau)h_{1,\fin}\gamma_i^{-1}\in {\bf GL}_2(\widehat{\Z})$ yields $T_\theta\ss(h_{1,\fin})\in\cup_{i=1}^n T_\theta\ss(\gamma_i{\bf GL}_2(\widehat \Z))$. Suppose $h_{1,\fin}=\delta I_\theta(\tau)$ ($\delta \in {\bf GL}_2(\Q)$, $\tau \in \cN_\fin$). By \eqref{Def-HeckeAlg-f3}, we have $\cQ_{\widehat \Z} \ss(\gamma_i)\in \fa(f_S)^{-1}\cQ_{\widehat \Z}$ for $1 \leq i \leq n$; let $\cM$ be the smallest lattice such that $\cQ_{\widehat \Z}\ss(\cN_\fin)\subset \cM_{\widehat \Z}$. 
In general, $f_S$ is of the form $\sum_{j}c_j \cchi_{X_j}$ with $X_j$ being mutually disjoint $\bK_S$-double cosets and $c_j\in \C$. Then, $\|f_S\|_1=\sum_{j}|c_j| \|\cchi_{X_j}\|_1$. By this, we obtain the desired estimate for $f_S$ from estimates for $\cchi_{X_j}$. \end{proof}
 
Recall the function $\fq(h)$ in \eqref{Def:fq}. 
\begin{lem} \label{JJw2reg-L2}
For any compact interval $I\subset (1,\infty)$, the integral in \eqref{PreJJw1-f6} with $\su=w_2$ is uniformly bounded for $s\in \cT_{I}$; moreover, for $h\in {\bf GL}_2(\R)$, 
\begin{align}
\cJ_{l}^{w_2({\rm r}),(s,\mu)}(1,h)&=\frac{2^{5/2}\pi^{s+l}|D|^{(s+l-2)/2}\mu_\infty(\det h)}{\Gamma(s+1)\Gamma(l-1)} 
\,\int_{-1}^{1}(t^2-1)^{l-2}\left(t+\fq(h) \right)^{s-l+1}\,\d t, 
\label{JJw2reg-L2-f0} 
\end{align}
which, as a function in $h\in {\bf GL}_2(\R)$, is right $I_{\theta}(E_\infty^\times)$-invariant. For $h\not\in {\bf GO}_{T_\theta}(\R)$, the integral in \eqref{JJw2reg-L2-f0} converges absolutely for any $s\in \C$. For $q>1$, there exists a unique entire  function $\widetilde \cJ_{l}^{(s,\mu)}[q]$ in $s\in \C$, invariant by $s\rightarrow -s$ and such that 
$$
\widetilde {\cJ}_l^{(s,\mu)}[\fq(h)]=\left\{\tfrac{|D|}{4}(\fq(h)^2-1)\right\}^{-s/2}\,\Gamma_\C(s+1)\times \cJ_{l}^{w_2({\rm r}),(s,\mu)}(1,h).
$$
\end{lem}
\begin{proof} We only need to evaluate $W^{Z}_{w_2}(b_\R^\theta)$ (see \S\ref{sec:CompArchInt}). The value $\Phi_l^{(s,\mu_\infty)}(w_2 \sn(X) b_\R^\theta)$ is known by  \eqref{ConvAInt-P-f0} followed by \eqref{ConvAInt-P-f1} and \eqref{ConvAInt-P-f2}; thus, by writing $X=a T_\theta^\dagger+X_\beta$ $(a\in \R, \beta \in \C)$, we have that $W^{Z}_{w_2}(b_\R^\theta)$ equals
\begin{align*}
|\tfrac{D}{2}|^{1/2}\int_{\R} \int_\C e^{-2\pi i \tr(ZX)}\,\Bigl(1-ia\tfrac{\sqrt{|D|}}{2}\Bigr)^{s-l+1} \Bigl\{\Bigl(1-ia\tfrac{\sqrt{|D|}}{2}\Bigr)^2+\beta \bar\beta\Bigr\}^{-(s+1)}\,\d a\,\d\eta_\C(\beta)
.\end{align*}
Set $Z^\dagger =c T_\theta^{\dagger}+X_{\gamma}$ ($c\in \R$, $\gamma=u+iv\,u,v\in \R)$. Then, for $\Re(s)>1$, this becomes
{\allowdisplaybreaks\begin{align*}
|\tfrac{D}{2}|^{1/2}&\int_{\R} e^{\pi i acD} \Bigl(1-ia\tfrac{\sqrt{|D|}}{2}\Bigr)^{s-l+1}\,\d a \\
&\quad \times \iint_{\R^2}  \Bigl\{\Bigl(1-ia\tfrac{\sqrt{|D|}}{2}\Bigr)^2+
\beta \bar\beta \Bigr\}^{-(s+1)}\,e^{4\pi i \Re(\beta \bar\gamma}\,\d\eta_\C(\beta),
\end{align*}}which, by Lemma \ref{JJw2reg-L3} and by the change of variables $z=1-ia\sqrt{|D}/2$, equals the contour integral
\begin{align}
|\tfrac{D}{2}|^{1/2}\tfrac{2\pi (2\pi |\gamma|)^{s}}{\Gamma(s+1)}e^{-2\pi c \sqrt{|D|}}\,
\int_{\Re(z)=1} \exp(-2 \pi \sqrt{|D|}cz) z^{1-l} K_s(4\pi z |\gamma| )\, \tfrac{-2i}{\sqrt{|D|}}\,\d z.
\label{JJw2reg-L2-f00}
 \end{align}
 Due to $K_{s}(z)=K_{-s}(z)$, this yields the functional equation. By the formula $
K_s(z)=\tfrac{1}{2} \int_0^\infty \exp(-\tfrac{1}{2}(v+v^{-1})z)v^{s-1}\,\d v
$, and by using \cite[3.382.6]{GR}, we have that $W^{Z}_{w_2}(b_\R^\theta)$ equals
{\allowdisplaybreaks\begin{align*}
&-2^{1/2}\pi i \tfrac{(2\pi |\gamma|)^{s}e^{-2\pi \sqrt{|D|}c}}{\Gamma(s+1)} \, \int_{0}^{\infty}\left\{\int_{\Re(z)=1}z^{1-l}\exp(-2\pi \{\sqrt{|D|}c+(v+v^{-1})|\gamma|\}z)\,\d z\right\}\,v^{s-1}\,\d v
\\
&=-2^{1/2}\pi i \tfrac{(2\pi |\gamma|)^{s}e^{-2\pi \sqrt{|D|}c}}{\Gamma(s+1)\Gamma(l-1)}(-2\pi i )(2\pi|\gamma|)^{l-2} \int_{v_1(Z)}^{v_+(Z)} 
\{(v-v_{-}(Z))(v-v_{+}(Z))\}^{l-2}v^{s-l+1}\,\d v,
\end{align*}}where $v_{\pm}(Z):=|\gamma|^{-1}\left(\tfrac{-\sqrt{|D|}c}{2}\pm \sqrt{\det (Z)}\right)$ are the roots of the quadratic equation $|\gamma|v^2+\sqrt{|D|}cv+|\gamma|=0$. By  $v=2^{-1}(v_{+}(Z)-v_{-}(Z))t+2^{-1}(v_{+}(Z)+v_{-}(Z))$, 
\begin{align*}
e^{2\pi \sqrt{|D|}c}\, W_{w_2}^Z(b_\R^\theta)=-\frac{2^{-1/2+s+l}\pi^{s+l}(\det Z)^{(s+l-2)/2}}{\Gamma(s+1)\Gamma(l-1)}\,\int_{-1}^{1}(t^2-1)^{l-2}\left(t+\tfrac{\sqrt{|D|}c}{2\sqrt{\det Z}}\right)^{s-l+1}\,\d t.
\end{align*} 
By \eqref{CompArchInt-f1} and \eqref{CompArchInt-f2}, we have \eqref{JJw2reg-L2-f0}. The last assertion follows from the espression \eqref{JJw2reg-L2-f00} of $W^{Z}_{w_2}(b_\R^\theta)$ together with \eqref{CompArchInt-f1} and \eqref{CompArchInt-f2} because the function $\fq$ is right ${\bf GO}_{T_\theta}(\R)$-invariant. 
\end{proof}

\begin{lem} \label{JJw2reg-L3}
For $\Re(s)>1$ and $\gamma\in \C^\times$, 
\begin{align*}\iint_{\R^2}  (A^2+\beta \bar\beta)^{-(s+1)}\,\exp(4\pi i \Re(\beta \bar \gamma))\,\d\eta_{\C}(\beta)=\frac{2\pi(2\pi|\gamma|)^{s}A^{-s}}{\Gamma(s+1)} \,K_{s}(4\pi A|\gamma|).
\end{align*}
\end{lem}
\begin{proof} Let $I(\gamma)$ denote the integral. Since $I(\gamma)$ is a radial function, it suffices to evaluate it for $\gamma>0$. By using \cite[3.715.8]{GR}, $I(\gamma)$ equals
\begin{align*}
\int_{0}^{\infty} (A^2+r^2)^{-(s+1)}\biggl\{\int_{0}^{2\pi} \exp(2\pi i r \gamma \sin \theta)\,\d \theta\biggr\}\,r\d r
=2\pi \int_{0}^{\infty}(A^2+r^2)^{-(s+1)} J_{0}(2\pi \gamma r) r\,\d r.
\end{align*}
Then use \cite[6.564.4]{GR} to complete the proof. 
\end{proof}

\begin{lem} \label{JJw2reg-L4}
For any $t\in (1,l-1)$ and $u>1$, we have the uniform bound
\begin{align*}
\widetilde{\cJ}_{l}^{(s,\mu)}[q]\ll_{t,u}\tfrac{(\pi\sqrt{|D|})^{l-1}}{\Gamma(l-1)}\,\left(1+\tfrac{\Gamma(l-1)\sqrt{\pi}}{\Gamma(l-1/2)} \right)\,q^{-l+1}, \quad (s\in \cT_{[-t,t]},\, 
l\in 2\Z_{\geq 3},\, q>u).
\end{align*}
\end{lem}
\begin{proof} Fix $q>1$ and let $I_{-}$ (resp. $I_{+}$) denote the integral of $(1-t^2)^{l-2}(t+q)^{s-l+1}$ over $[-1,0]$ (resp. $[0,1]$). Since the function $t\mapsto f(t):=\frac{1-t^2}{t+q}$ is increasing on $[-1,0]$, we have $f(t)\leq f(0)=q^{-1}$ for $t\in [-1,0]$. Hence, 
\begin{align*}
 |I_{-}|\leq \int_{-1}^{0} \left(f(t)\right)^{l-2}(t+q)^{\Re(s)-1}\,\d t \leq q^{-l+2}\,q^{\Re(s)-1}=q^{l-\Re(s)+1},\quad \Re(s)>1, \end{align*}
As for $I_{+}$, by using the inequality $(t+q)^{\Re(s)-l+1}\leq q^{\Re(s)-l+1}$ valid for $\Re(s)<l-1$, 
\begin{align*}
 & |I_{+} |\leq q^{\Re(s)-l+1}\,\int_{0}^{1} (1-t^2)^{l-2} \,\d t=q^{\Re(s)-l+1} \tfrac{\Gamma(l-1)\sqrt{\pi}}{\Gamma(l-1/2)}. 
      \end{align*}
Thus, we have the inequality 
\begin{align*}
\widetilde {\cJ}_l^{(s,\mu)}[q] \leq  \left\{\tfrac{|D|}{4}(q^2-1)\right\}^{-\Re(s)/2}\,|\Gamma_\C(s+1)|\times \left(1+ \tfrac{\Gamma(l-1)\sqrt{\pi}}{\Gamma(l-1/2)}\right)\, q^{-l+\Re(s)+1}
\end{align*}
for $q>1$, $s\in \cT_{(1,l-1)}$, whose right-hand side is majorized by $O_{u,t}\left(\Bigl(1+\tfrac{\Gamma(l-1)\sqrt{\pi}}{\Gamma(l-1/2)} \Bigl)q^{-l+1}\right)$ uniformly in $s\in \cT_{(1,t)}$ and $q>u$. Since $\widetilde \cJ_l^{(s,\mu)}[q]$ is invariant by $s\rightarrow -s$, we conclude the proof by the Phragmen-Lindel\"{o}f principle. 
\end{proof}
Since $\Gamma(l-1)/\Gamma(l-1/2)\sim l^{-1/2}\ll 1$, the majorant in Lemma \ref{JJw2reg-L4} can be replaced by $2\fq(\delta)^{\Re(s)-l+1}$ for large $l$. Hence, by \eqref{JJu-f10} and by Lemmas \ref{JJw2reg-L1}, \ref{JJw2reg-L2}, and \ref{JJw2reg-L4}, we have the absolute convergence of \eqref{JJu-f10} and the majorization: 
\begin{align*}
\JJ_{l}^{w_2({\rm r})}(\phi,f\mid s,\Lambda,\mu) &\ll \tfrac{(\pi|D|^{1/2})^{s+l}}{|\Gamma(s+1)|\Gamma(l-1)}N^{-\Re(s)}\fa(f_S)^{\Re(s)-2}\|f_S\|_1\,\{\zeta_{D}(\sigma_l
,\cM)-2\}
\end{align*}
with $\sigma_l:=-\Re(s)+l-1$ uniformly in $s\in \cT_I$ and $l\in \Z_{>l_0}$ for a fixed compact interval $I\subset (1,l-2)$, where $\cM:=\fa(f_S)^{-1}\cQ$, and $l_0$ is an absolute constant. Note that $\sigma_l>1$, i.e., $\Re(s)<l-2$, is necessary to apply Lemma \ref{ZetaDCnv}. By \eqref{Def:HypZeta2} and Lemma \ref{ZetaDCnv}, $$\zeta_{D}(\sigma_l
,\cM)-2=\fa(f_S)^{\sigma_l}\{\zeta_D(\sigma_l
,\cQ)-2\} =O(\fa(f_S)^{\sigma_l}C^{-l})$$ for some $C>1$. Thus we get Proposition \ref{ErrorT-P000}.

\subsection{Estimation of $\JJ_{w_1}(\Phi)$} \label{sec:JJw1} Recall our choice $(\phi,f)$ as well as the definition of $\eta$ made in \S\ref{sec:ErrTEST}(i) and (ii), and note that $f$ depends on a square free integer $N$.

\begin{lem}\label{JJw1-L0} Let $\omega(N)$ denote the number of prime divisors of $N$. Then,  
$$
\Phi^{(s,\Lambda,\mu)}_{\phi,f}(k\eta)\ll 2^{\omega(N)} \|f_S\|_1(N\fa(f_S))^{\Re(s)-1}, \quad\Re(s)>1,\,k \in \bK_\fin,
$$  
with the implied constant independent of $(s,\Lambda,\mu)$, $f_S$ and $N$. 
\end{lem}
\begin{proof} Suppose that $f$ is the characteristic function of $\bK_0(N) y_S\bK_0(N)$ for some $y=(y_p)_p \in \sG(\A_\fin)$ such that $y_p=1_4$ for $p\not\in S$. Then,  
\begin{align*}
|\Phi^{(s,\Lambda,\mu)}_{\phi,f}(k\eta)|
&\leq \Phi_{\phi_0, f}^{(\Re(s),{\bf 1}, {\bf 1})}(k\eta) \ll \prod_{p<\infty} I_p(N,y_p)
\end{align*}
with $I_p(N,y_p)$ being equal to 
\begin{align*}
\int_{\bK_p^\#} \int_{E_p^\times}\int_{\Q_p^\times}\int_{E_p}h_p(\iota(k^\#)\sm(I_\theta(\tau),\lambda \nr(\tau))^{-1}n(X_\beta)^{-1} k \eta_p)|\lambda|_p^{-\Re(s)+1}\,\d k^\#\d\beta \d^\times \lambda\d^\times\tau
\end{align*}
with $h_p:=\cchi_{\bK_0(p\Z_p) y_p \bK_0(p\Z_p)}$.
We have $I_p(N,f_S)\leq 1$ for $p\nmid N$, $p\not\in S$. If $p\in S$, then we may set $y_p=\varpi(\ud)\,(\ud\in D)$ and write $\bK_py_p\bK_p$ as in \eqref{Def-HeckeAlg-f1}. Thus, by \eqref{Def-HeckeAlg-f5}, the integral becomes  
\begin{align*}
&\sum_{i=1}^{n} \int_{E_p^\times}\int_{\Q_p^\times}\int_{E_p}\cchi_{\bK_p}(\sm(I_\theta(\tau),\lambda \nr(\tau))^{-1}n(X_\beta)^{-1} n(X_i)^{-1}\sm(\gamma_i;\lambda_i \det \gamma_i)^{-1}) |\lambda|_p^{-\Re(s)+1}\,\d k^\#\d\beta
\\
&\leq \sum_{i=1}^{n} |\lambda_i|_p^{\Re(s)-1}
\ll n \times (p^{2\kappa(\ud)})^{\Re(s)-1}=\|f_p\|_1\fa(f_p)^{\Re(s)-1}. 
\end{align*}
 If $p\mid N$, since $\eta_p$ normalizes $\bK_0(p\Z_p)$ and $\iota(\bK_p^\#(p\Z_p))\subset \bK_0(p\Z_p)$, for any $g\in \sG(\Q_p)$,
\begin{align*}
\int_{\bK_p^\#} \cchi_{\bK_0(p\Z_p)}(\iota(k^\#) g \eta_p)
\,\d k^\#
&=\int_{\bK_p^\#} \cchi_{\bK_0(p\Z_p)}(\eta_p\iota(k^\#) g)
\,\d k^\#
\\
&=[\bK_p^\#:\bK_0^\#(p\Z_p)]^{-1}\sum_{\ell \in \bK_0^\#(p\Z_p)\bsl \bK_p^\#} \cchi_{\bK_0(p\Z_p)}(\eta_p\iota(\ell) g).
\end{align*}
In the last sum, the summand is written as $\cchi_{\bK_0(p\Z_p)}(\iota(\ell)g \eta_p)$, which is non-zero for at most $1$ coset $\ell$ and is no greater than $\cchi_{\bK_p}(g\eta_p)$. Thus, 
\begin{align*}
I_p(N,y_p)\leq p^{-2}  \int_{E_p^\times}\int_{\Q_p^\times}\int_{E_p}\cchi_{\bK_p\eta_p\bK_p}(\sn(X_\beta)\sm(I_\theta(\tau),\lambda \nr(\tau)))|\lambda|_p^{-\Re(s)+1}\,\d\beta \d^\times \lambda\d^\times\tau.
\end{align*}
Since the integrand is non-zero only if $(\tau,\lambda,\beta)$ belongs to $p^{i}\cO_{E,p}^\times \times p^{1-2i}\Z_p^\times \times p^{-1}\cO_{E,p}$ with $i=0,1$, we get a bound $I_p(N,y_p)\leq p^{\Re(s)-1}(1+p^{-2(\Re(s)-1)})\leq 2 p^{\Re(s)-1}$. 
\end{proof}

For $g\in {\bf GL}_2(\A)$, set $y(g):=|a|_\A$ with $a,z\in \A^\times$, $b\in \A$ such that $g\in \left[\begin{smallmatrix}az & b \\ 0 & z \end{smallmatrix}\right]{\bf GL}_2(\widehat \Z)$.

\begin{lem} \label{JJw1-L1} For any compact interval $I\subset (-1,\infty)$, the integral \eqref{PreJJw1-f5} with $\su=w_1$ is uniformly bounded in $s\in \cT_{I}$; moreover, \eqref{PreJJw1-f5} with $\su=w_1$ is zero unless $\lambda \in \fa(f_S)^{-1}N\Z-\{0\}$. We have the bound
\begin{align*}
\cJ_{\phi,f}^{w_1,(s,\Lambda,\mu)}(\lambda,h)& \ll y(h)^{\Re(s)+1}\,2^{\omega(N)}N^{2\Re(s)}\fa(f_S)^{2\Re(s)+3}, 
\\
&h\in {\bf GL}_2(\A_\fin), \,\lambda \in \fa(f_S)^{-1}N\Z-\{0\},\,s\in \cT_{I}
\end{align*}
with the implied constant independent of $N$ and $f_S$. 
\end{lem}
\begin{proof} 
Set $g=w_1\sn(X)\sm(h,\lambda^{-1}\det h)$ with $X=\left[\begin{smallmatrix} x_1 & x_2 \\ x_2 & x_3 \end{smallmatrix}\right]$, $h\in {\bf GL}_2(\A_\fin)$ and $\lambda\in \Q^\times$. A direct computation 
shows the formulas
\begin{align}
\sr(g)^{-1}\,[T_\theta^\dagger]&=
\lambda^{-1}(-\ttN+\tt+\det X)\,\xi_1-[\left[\begin{smallmatrix} x_1 & x_2-\tt/2 \\ x_2-\tt/2 & x_3 \end{smallmatrix}\right]\ss({}^th^{-1})]+\lambda\,\xi_5, 
 \label{JJw1-L1-f1}
\\ 
\sr(g)^{-1}\xi_1&=\lambda^{-1}x_3\,\xi_1-[\left[\begin{smallmatrix} 1 & 0 \\ 0 & 0 \end{smallmatrix}\right]\ss({}^t h^{-1})],
 \label{JJw1-L1-f2}
\end{align}
where $\sr$ as in \S\ref{sec: ExIso}. By Lemma \ref{HeckeFtn-L1}, for \eqref{PreJJw1-f5} to be non zero, $g=\iota\left(\left[\begin{smallmatrix} \tau & 0 \\ 0 & \bar \tau a\end{smallmatrix}\right] \left[\begin{smallmatrix} 1& \beta \\ 0& 1 \end{smallmatrix} \right]\right)k\eta$ with $\tau\in \A_{E,\fin}^\times$, $a\in \A_\fin^\times$ and $\beta\in \A_{E}$ and $k=(k_p)_{p<\infty}\in \cK_\fin$, where $\cK_\fin=\prod_{p<\infty}\cK_p$ with $\cK_p=\bK_0(N\Z_p)$ for $p\not\in S\cup S(M)$ and $\cK_p={\rm supp}(f_p)$ for $p\in S\cup S(M)$.
Thus, 
\begin{align}
 \sr(g)^{-1}[T_\theta^\dagger]=\sr(k\eta)^{-1}[T_\theta^\dagger], \quad 
 \sr(g)^{-1}\xi_1=a^{-1}\sr(k\eta)^{-1}\xi_1.
 \label{JJw1-L1-f3}
\end{align}
By Lemma \ref{Def-HeckeAlg-L}, $\sr(\cK_p)\sX(\Z_p)\subset \fa(f_S)^{-1}\sX(\Z_p)$ for $p\nmid N$. Note that $[T_\theta^\dagger]\in \sX(\Z)$; then, by \eqref{JJw1-L1-f1} and the first formula in \eqref{JJw1-L1-f3}, we have $\lambda \in \fa(f_S)^{-1}\Z_p$ for $p\nmid N$, and
\begin{align}
\left[\begin{smallmatrix} x_1 & x_2-\tt/2 \\ x_2-\tt/2 & x_3 \end{smallmatrix}\right]\ss({}^th^{-1})\in \fa(f_S)^{-1} \sV(\widehat \Z).
 \label{JJw1-L1-f4}
\end{align}
If $q\mid N$, then $\eta_q=\iota\left(\left[\begin{smallmatrix} 0 & 1 \\ 1 & 0\end{smallmatrix}\right]\right)\sm(-\sigma, q)$; hence, by Lemma \ref{ExIso-L1} and \eqref{ExIso-f0}, we have $\sr(\eta_q)^{-1}[T_\theta^{\dagger}]=[-T_\theta^{\dagger}]$. Since $\eta_q$ normalizes $\bK_0(q\Z_q)$, the element $k_q':=\eta_q^{-1}k_q\eta_q$ is in $\bK_0(q\Z_q)$; thus $k'_q=\sm(A,c)\sn(B)\left[\begin{smallmatrix} 1_2 & 0 \\ NC & 1_2\end{smallmatrix}\right]$ with some $A\in {\bf GL}_2(\Z_q)$, $B\in \sV(\Z_q)$ and $C\in \sV(\Z_q)$. Using \eqref{ExIso-f0}, we easily have $\sr(k_q\eta_q)^{-1}[T_\theta^\dagger]=\sr(k_q')^{-1}[-T_\theta^\dagger]\in \sV(\Z_q)+N\Z_q\xi_5$. This, combined with the first formula in \eqref{JJw1-L1-f3} and \eqref{JJw1-L1-f1}, yields $\lambda \in N\Z_q$ for all $q\mid N$. Recall \S\ref{sec:Ht}. We have $\|\sr(k\eta)^{-1}\xi_1\|_\fin \asymp |N\fa(f_S)|_\fin^{-1}=N\fa(f_S)$ for $k\in \cK_\fin$ uniformly in $N$. By this, due to \eqref{JJw1-L1-f2} and the second formula in \eqref{JJw1-L1-f3}, 
\begin{align*}
N\fa(f_S)|a|_\fin^{-1}\gg \|\sr(g)^{-1}\xi_1\|_{\fin}\geq 
\|\xi\left(\left[\begin{smallmatrix} 1 & 0 \\ 0 & 0 \end{smallmatrix}\right]\ss({}^t h^{-1})\right)\|_\fin=y(h)^{-1}. 
\end{align*}
By this, combined with \eqref{HeckeFtn-L1-f0} and Lemma \ref{JJw1-L0}, $|\Phi_{f,\phi}^{(s,\Lambda,\mu)}(g)|=|a|_\fin^{\Re(s)+1}|\Phi_{f,\phi}^{(s,\Lambda,\mu)}(k\eta)|$ is bounded by the product of 
$2^{\omega(N)} \|f_S\|_1 (N\fa(f_S))^{2\Re(s)} y(h)^{\Re(s)+1}$ and the characteristic function of the set of $(\lambda,X)$ with $\lambda \in \fa(f_S)^{-1}N\Z$ and \eqref{JJw1-L1-f4}. Since $X\mapsto X\ss({}^t h^{-1})$ preserves the Haar measure on $\sV(\A_\fin)$, the $X$-integral yields $\fa(f_S)^3$.\end{proof}

Recall the vectors $\e_i$ $(i=1,2,3)$; see \eqref{BasisV}. We need the following lemma; for the $J$-Bessel function $J_{\nu}(z)$ $(\nu>0)$, we refer to \cite[Chap.III]{MOS}.
\begin{lem} \label{JJw1L4} 
For $q>0$, $A>0$, $B<0$,
$$
\int_{\Re(\zeta)=1} \zeta^{-q}\exp(-A\zeta^{-1}-B\zeta)\,\d \zeta=-2\pi i (A|B|^{-1})^{(1-q)/2}J_{q-1}(2\sqrt{|AB|}).
$$
\end{lem}
\begin{proof} Expand $e^{-A\zeta^{-1}}$ to the power series in $\zeta^{-1}$ and evaluate the integral termwise by \cite[3.382]{GR}. Then use the series expression of $J_{q-1}$. \end{proof}

\begin{lem} \label{JJw_1-L2}
For any compact interval $I\subset (1,l-3)$, the integral \eqref{PreJJw1-f6} is bounded by $\exp(\tfrac{\pi}{2}|\Im(s))$ uniformly in $s\in \cT_{I}$; moreover, $\cJ_{l}^{w_1,(s,\mu)}(\lambda,h)$ is zero unless $\lambda\,\tr(T_\theta\ss(h)^{-1}\varepsilon_1)>0$, in which case it equals
\begin{align*}
-i^{-l}\mu_\infty(\lambda^{-1}\det h)&\frac{2^{-3s+2l-3}(2\pi)^{-s+l}}{\Gamma(-s+l-1)}\, \frac{|\tr(T_\theta\ss(h)^{-1}\varepsilon_1)|^{-s-1}}{|\lambda|^{-s+l-5/2}}\\
&\times J_{l-\frac{3}{2}}(\pi^{1/2}|D||\lambda|^{-1})\,\exp(-2\pi i \lambda \tt \tr(T_\theta \ss(h)^{-1}\,\varepsilon_2)).\end{align*}
\end{lem}
\begin{proof} 
By \eqref{ConvAInt-P-f0} followed by \eqref{ConvAInt-P-f1} and \eqref{ConvAInt-P-f2}, \\the value $W_{w_1}^{Z}(b_\R^\theta)$ is $(\tfrac{2i}{\sqrt{|D|}})^{-l}\,
(\tfrac{\sqrt{|D|}}{8})^{s-l+1}\exp(-2\pi i \tt \tr(Z \varepsilon_2))$ times the integral
\begin{align} 
\int_{\sV(\R)}\Bigl(1-\tfrac{2}{\sqrt{|D|}}ix_3\Bigr)^{-(s+1)} 
\Bigl(1+\tfrac{|D|}{4}-\tfrac{1}{2} \langle X, X \rangle-\tfrac{2i}{\sqrt{|D}} \langle T_\theta^\dagger, X \rangle \Bigr)^{s-l+1}\,e^{-2\pi i \tr(ZX))}\,\d X. \label{JJw_1-L2-f0}
\end{align}
Let 
$$
\eta^{+}:=(\tfrac{2}{|D|})^{1/2}\,T_\theta^{\dagger}, \quad \eta^{-}:=(\tfrac{2}{|D|})^{1/2}(T_\theta^\dagger-\tfrac{|D|}{2}\varepsilon_1).
$$
Then, $\langle \eta^{\pm},\eta^{\pm}\rangle=\pm 1$ and $\langle \eta^{+},\eta^{-}\rangle=0$; choose a vector $\eta_0\in \sV(\R)$ so that $\langle \eta_0,\eta_0\rangle=-1$ and that $\{\eta^{+},\eta^{-},\eta_0\}$ is an orthogonal basis of $\sV(\R)$. Write 
$$
X=y_+\eta^{+}+y_{-}\eta^{-}+y_0\eta_0, \quad Z^{\dagger}=c_{+}\eta^{+}+c_{-}\eta^{-} +c_0\eta_0. 
$$
Then, the integral in \eqref{JJw_1-L2-f0} becomes
\begin{align*}
\int_{\R^{3}}& \left(1-\tfrac{y_{+}+y_{-}}{\sqrt{2}}i\right)^{-(s+1)}
\Bigl\{-\left(\tfrac{y_+}{\sqrt{2}}+i \right)^2+
\left(\tfrac{y_{-}}{\sqrt{2}}\right)^2+\Bigl(\tfrac{y_0}{\sqrt{2}}\Bigr)^2+\left(\tfrac{\sqrt{|D|}}{2}\right)^2
\Bigr\}^{s-l+1} \\
&\quad \times \exp(-2\pi i (c_+y_+-c_-y_--c_0y_0))\,\d y_{+}\d y_{-}\d y_0.
\end{align*}
By setting $\zeta=1-\frac{y_{+}+y_{-}}{\sqrt{2}}i$, this equals
\begin{align}
&\tfrac{-i}{\sqrt{2}}\exp(-2\sqrt{2} \pi c_{+})
 \int_{\Re(\zeta)=1} \exp(2\sqrt{2}\pi \zeta c_{+})\zeta^{-(s+1)}\ d \zeta \int_{\R} \exp(2\pi i c_0 y_0)\,\d y_0.
 \label{JJw_1-L2-f1}
\\
& \quad \times \int_{\R}\Bigl\{\zeta^2+\sqrt{2}i y_{-}\zeta+\left(\tfrac{y_0}{\sqrt{2}}\right)^2+\left(\tfrac{\sqrt{|D|}}{2}\right)^2\Bigr\}^{s+1-l} \exp(2\pi i y_{-}(c_{+}+c_{-})) \d y_{-}.
\notag
\end{align}
By \cite[3.382.6]{GR}, the $y_{-}$-integral is zero unless $c_{+}+c_{-}>0$, in which case it equals
\begin{align*}
\frac{-2\pi |2\pi(c_{+}+c_{-})|^{-s+l-2}}{\Gamma(-s+l-1)}\exp\left(-\tfrac{2\pi (c_{+}+c_{-})}{\sqrt{2}\zeta}(x_{0}^2+\zeta^2+\tfrac{|D|}{4})\right). 
\end{align*}
In the rest of the proof, we assume $c_{+}+c_{-}>0$. Thus, \eqref{JJw_1-L2-f1} becomes
{\allowdisplaybreaks\begin{align*}
&\frac{i\exp(-2\sqrt{2}\pi c_{+})}{\Gamma(-s+l-1)}(2\pi)^{-s+l-1}|c_{+}+c_{-}|^{-s+l-2}\,2^{(s-l+1)/2}\\
&\times \int_{\Re(\zeta)=1} \zeta^{-l}\exp\left(2\sqrt{2}\pi c_{+}\zeta-\sqrt{2}\pi(c_{+}+c_{-})(\zeta+\tfrac{|D|}{4\zeta})\right)\,\d \zeta \\
&\times \int_{\R} \exp(-2\sqrt{2}\pi i c_0y_0)\,\exp(-\sqrt{2}\pi \zeta^{-1}(c_{+}+c_{-})y_0^2)\,\d y_0.
\end{align*}}The $y_0$-integral is computed to be $-(\tfrac{\zeta}{\sqrt{2}|c_++c_-|})^{1/2} \exp\left(\tfrac{\sqrt{2}\pi c_0^2\zeta}{|c_{+}+c_{-}|}\right)$. Putting this back to the last displayed formula, after a simple computation, we obtain the following:
\begin{align*}
&\frac{-i 2^{\frac{s-l+1}{2}+\frac{3}{4}}(2\pi)^{-s+l-1}}{\Gamma(-s+l-1)} \exp(-2\sqrt{2}\pi c_{+})\,|c_{+}+c_{-}|^{-s+l-5/2} \\
&\times \int_{\Re(\zeta)=1}\zeta^{-l+1/2} \exp\left( -\tfrac{\sqrt{2}\pi\langle Z, Z\rangle}{|c_{+}+c_{-}|}\zeta- \sqrt{2}\pi \tfrac{|D|(c_{+}+c_{-})}{4\zeta} \right)\,\d \zeta. 
\end{align*}
By Lemma \ref{JJw1L4}, after a direct computation, we get the following as the value $W_{w_1}^{Z}(b_\R^\theta)$: 
\begin{align*}
&(\tfrac{2i}{\sqrt{|D|}})^{-l}\,
(\tfrac{\sqrt{|D|}}{8})^{s-l+1}\exp(-2\pi i \tt \tr(Z \varepsilon_1))\cdot
(-1)2^{(s+1)/2}(2\pi)^{-s+l} \exp(-2\sqrt{2}\pi c_{+})\tfrac{|c_{+}+c_{-}|^{-s-1}}{\Gamma(-s+l-1)} \\
&\quad \times \left(\tfrac{\langle Z,Z\rangle}{|D|}\right)^{\frac{1}{2}(l-\frac{3}{2})}J_{l-\frac{3}{2}}((2\pi|D|\langle Z, Z\rangle)^{1/2}). 
\end{align*}
Note that $c_{+}=\langle Z^{\dagger}, \eta_{+}\rangle=(\tfrac{2}{|D|})^{1/2}\tr(Z T_\theta^\dagger)$ and $c_{+}+c_{-}=(\frac{|D|}{2})^{1/2}\tr(Z \varepsilon_2)$. By \eqref{CompArchInt-f1} and \eqref{CompArchInt-f2}, we obtain the desired formula after a simple computation. 
\end{proof}

Let $I\subset (1,l-3)$ be any compact interval. Then, from Lemmas \ref{JJw_1-L2}, applying the uniform estimate $J_{q}(x)\ll (x/2)^{q}\Gamma(q+1)^{-1}$ ($x>0$, $q>1$) (proved by \cite[8.411 8]{GR}) and Lemma \ref{JJw1L5}, we get the estimate
\begin{align}
\cJ_{l}^{w_1,(s,\mu)}(\lambda,h) \ll _{I} \tfrac{(4\pi|D|)^{l}}{|\Gamma(-s+l+1)|\,\Gamma(l-1/2)}|\lambda|^{\Re(s)-2l+4}\,y(h)^{\Re(s)+1}, \quad h\in {\bf GL}_2(\R),\,s\in \cT_{I},
 \label{JJw1L2-f000}
\end{align}
which is uniform in $l$. This bound, combined with Lemma \ref{JJw1-L1}, yields the following majorization, which yields \eqref{JJu-f4} for $\su=w_1$:   
\begin{align*}
&\sum_{{\delta \in \sB(\Q)\bsl {\bf GL}_2(\Q)}}\sum_{\lambda \in \Q^\times}|\cJ_{\phi,f,l}^{w_1,(s,\Lambda,\mu)}(\lambda,\delta I_\theta(\tau))
| \\
&\ll_{I}  C 2^{\omega(N)}N^{2\Re(s)}\fa(f_S)^{2\Re(s)+3}
\Bigl(\sum_{\lambda \in \fa(f_S)^{-1}N\Z-\{0\}}|\lambda|^{-2l+\Re(s)+4}\Bigr)\,\Bigl(\sum_{\delta } y(\delta I_\theta(\tau))^{\Re(s)+1}\Bigr)
\end{align*}
for $\tau \in \cN_{\fin}\cN_\infty$, $s\in \cT_{I}$ and $l\in 2\Z_{>1}$, where $C:=
\frac{(4\pi|D|)^{l}}{|\Gamma(-s+l+1)|\,\Gamma(l-1/2)}\ll_{l} \exp(\frac{\pi}{2}|\Im(s)|)$ and $\delta$ is summed over $\sB(\Q)\bsl {\bf GL}_2(\Q)$. Note that, when $s\in \cT_{(1,l-3)}$, the first series in the majorant is convergent and amounts to $O((\fa(f_S)^{-1}N)^{-2l+\Re(s)+4})$. The second one is the Eisenstein series on ${\bf GL}_2(\A)$, which is convergent for $\Re(s)>0$. 
\begin{lem}\label{JJw1L5}
 There exists a constant $C>1$ such that 
$$C^{-1}y(h)^{-1}\leq |\tr(T_\theta\ss(h)^{-1}\,\varepsilon_1)|\leq C y(h)^{-1}, \qquad h\in {\bf GL}_2(\R).$$ 
\end{lem}
\begin{proof}
The function ${\fy}(h):=y(h)|\langle (T_\theta\ss(h)^{-1})^\dagger,\varepsilon_1\rangle|$ is continuous on the compact manifold $\sB(\R)\bsl {\bf GL}_2(\R)$. Moreover, $\fy(h)>0$ since the orthogonal of $(T_\theta\ss(h)^{-1})^\dagger$ is negative definite and $\varepsilon_1$ is isotropic. Hence, the image $\fy(\sB(\R)\bsl {\bf GL}_2(\R)$) is in a compact set of $\R_{>0}$. 
\end{proof}

\subsection{Estimation of $\JJ_{\sn(T_\theta^\dagger)w_2}(\Phi)$} 
\label{sec:JJTtw2}
 
%Recall the height function $\|\cdot\|_\fin=\prod_{p<\infty}\|\cdot\|_p$ on ${\mathsf X}^{\star}(\A_\fin)$ and on $\sV^\star(\A_\fin)$ (see \S\ref{sec:Ht}).

\begin{lem} \label{JJTtw2-L1} For any compact interval $I\subset (-1,\infty)$, the integral \eqref{PreJJw1-f5} with $\su=\sn(T_\theta^\dagger)w_2$ is uniformly bounded in $s\in \cT_{I}$; moreover, \eqref{PreJJw1-f5} with $\su=\sn(T_\theta^\dagger)w_2$ is zero unless $\lambda \in\frac{2}{D}\fa(f_S)^{-1}N\Z-\{0\}$. We have the bound
\begin{align*}
\cJ_{\phi,f}^{\sn(T_\theta^\dagger)w_2, (s,\Lambda,\mu)}(\lambda, h) 
\ll \|f_S\|_1 \fa(f_S)^{3\Re(s)+4}2^{\omega(N)}N^{2\Re(s)}
\Bigl(\|T_\theta^\dagger \ss({}^th)^{-1}\|_\fin
|\lambda|_\fin \Bigr)^{-(\Re(s)+1)}
\end{align*}
for $h\in {\bf GL}_2(\A_\fin)$ and $\lambda \in \tfrac{2}{D}\fa(f_S)^{-1}N\Z-\{0\}$ uniformly in $s\in \cT_{I}$ with the implied constant independent of $N$ and $f_S$. 
\end{lem}
\begin{proof} Set $g=\sn(T_\theta^\dagger)w_2\sn(X)\sm(h,\lambda^{-1}\det h)$ with $X\in \sV(\A_\fin)$, $\lambda\in \Q^\times$, and $h\in {\bf GL}_2(\A_\fin)$.
By Lemma \ref{HeckeFtn-L1}, $\Phi_{f,\phi}^{(s,\Lambda,\mu)}(g)\not=0$ implies $g\in \iota(\sB^\#(\A_\fin))\,\cK\eta$ with 
$\cK:=\prod_{p<\infty}\cK_p\subset \sG(\A_\fin)$, $\cK_p:={\rm supp}(f_p)$. Since $\sr(\iota(\sB^\#(\A_\fin)))$ fixes the vector $[T_\theta^\dagger]\in {\mathsf X}(\Q)$ (Lemma \ref{ExIso-L1}), the set $\cK$ yields a compact set $\cU$ of ${\mathsf X}(\A)$ such that $\Phi_{f,\phi}^{(s,\Lambda,\mu)}(g)\not=0$ implies $\sr(g)^{-1}[T_\theta^{\dagger}]\in \cU$. Enlarging $\cU$ if necessary, we may suppose $\cU=\prod_{p<\infty}\cU_p$ with $\cU_p=\fa(f_p)^{-1}\sX(\Z_p)$ for $p\in S\cup S(M)$ and $\cU_p=\Z_p\xi_1+\sV(\Z_p)+N\Z_p \xi_5$ for $p\not\in S\cup S(M)$. By formulas \eqref{ExIso-f0} and \eqref{ExIso-f1}, \begin{align*}
\sr(g)^{-1}[T_\theta^\dagger]&=\lambda^{-1}\left(\tfrac{D}{2}\det X-\langle X,T_\theta^\dagger\rangle\right)\,\xi_1+[(-\tfrac{D}{2}X+T_\theta^\dagger)\ss({}^t h^{-1})]+\tfrac{D}{2}\lambda\,\xi_5, \\
\sr(g)^{-1}\xi_1&=\lambda^{-1}\det(X)\,\xi_1-[X\ss({}^t h^{-1})]+\lambda\,\xi_5.\end{align*}
Hence, $\sr(g)^{-1}[T_\theta^\dagger]$ and  $\sr(g)^{-1}\xi_1\in \cU$ yield
\begin{align}
\left(\tfrac{-D}{2}X+T_\theta^\dagger\right)\ss({}^th^{-1})\in \fa(f_S)^{-1}\sX(\widehat \Z),\, \qquad \tfrac{D}{2}\lambda \in \fa(f_S)^{-1}N\widehat \Z-\{0\}. 
\label{JJTtw2-L1-f0}
\end{align}
The second condition yields $|\lambda|_p^{-1}\geq |\tfrac{D}{2}|_p|\fa(f_S)|_p$; using this on the way, We get   
\begin{align*}
\|\sr(g_p)^{-1}\xi_1\|_p &\geq \|-[X\ss({}^t h_p^{-1}]+\lambda\xi_5\|_p
\\
&=\max(\|X\ss({}^t h_p^{-1})\|_p,|\lambda|_p)
\\
&=|\lambda|_p\max(|\lambda|_p^{-1}\|X\ss({}^th_p^{-1})\|_p,1)
\geq |\lambda|_p\,\min(|\tfrac{D}{2}|_p|\fa(f_S)|_p, 1)\,\max(\|X\ss({}^t h_p^{-1})\|_p,1),
\end{align*}
 and
\begin{align*}
\|T_\theta^\dagger \ss({}^t h_p^{-1})\|_p&\leq \max(\|\tfrac{D}{2}X\ss({}^t h_p^{-1})\|_p, \|\left(\tfrac{-D}{2}X+T_\theta^\dagger\right)\ss({}^th^{-1})\|_p)
\\
&\leq \max(\|\tfrac{D}{2}X\ss({}^t h_p^{-1})\|_p, |\fa(f_S)^{-1}|_p) 
\leq \max(|\tfrac{D}{2}|_p,|\fa(f_S)^{-1}|_p)\,\max(\|X\ss({}^t h_p^{-1})\|_p, 1).
\end{align*}
Set $C_p':=\max(|\frac{D}{2}|_p,|\fa(f_S)^{-1}|_p)^{-1}\min(|\frac{D}{2}|_p|\fa(f_S),1)=|\fa(f_S)|_p^{-2} \min(|\frac{D}{2}|_p,|\frac{D}{2}|_p^{-1})$, because $S$ contains no ramified primes. Then, $C'_p$ is a positive constant being equal to $1$ for almost all $p$ such that $C':=\prod_{p<\infty}C_p'\,(\geq \frac{4}{|D|}\fa(f_S)^{-2})$ satisfies
\begin{align}
\|\sr(g)^{-1}\xi_1\|_\fin  \geq C' 
\|T_\theta^\dagger \ss({}^t h^{-1})\|_\fin\,|\lambda|_\fin
 \label{JJTtw2-L1-f1}
\end{align}
for all $g=(g_p)_{p}\in \sG(\A_\fin)$ with $\Phi_{f,\phi}^{(s,\Lambda,\mu)}(g)\not=0$. If we write $g=\iota\left(\left[\begin{smallmatrix} \tau & \beta \\ 0 &a \bar \tau
\end{smallmatrix}\right] \right)k \eta \in \iota(\sB^{\#}(\A_\fin))\cK \eta$, then by \eqref{Def-iota1}, \eqref{Def-iotaf2} and \eqref{ExIso-f0}, 
$$
\|\sr(g)^{-1}\xi_1\|_\fin\leq \|\sr(\eta^{-1}\iota(\left[\begin{smallmatrix} \tau & \beta \\ 0 &a \bar \tau\end{smallmatrix}\right] )^{-1})\xi_1\|_\fin
=|a|_\fin \|\sr(\eta)^{-1}\xi_1\|_{\fin}=|a
|_\fin\,N
$$
By this, combined with \eqref{JJTtw2-L1-f1} and Lemma \ref{JJw1-L0}, $ |\Phi_{\phi,f}^{(s,\Lambda,\mu)}(g)|\leq |a |_\fin^{-(\Re(s)+1)}\,|\Phi_{\phi_0,f}^{(\Re(s),{\bf 1}, {\bf 1})}(k\eta)|$ is bounded by the product of $(C'\, N^{-1}\|T_\theta^\dagger \ss({}^t h^{-1})\|_\fin\,|\lambda|_\fin))^{-(\Re(s)+1)}\times 2^{\omega(N)}\|f_S\|_1 (N\fa(f_S))^{\Re(s)-1}$ and the characteristic function of the set of $(\lambda,X)$ satisfying \eqref{JJTtw2-L1-f0} if $\Re(s)>-1$. The $X$-integral yields the factor $(\frac{|D|}{2}\fa(f_S))^{3}$. \end{proof}

\begin{lem}\label{JJTtw2-L2}
For any interval $I\subset (0,l-2)$, the integral \eqref{PreJJw1-f6} with $\su=\sn(T_\theta^\dagger)w_2$ is bounded by $\exp(\tfrac{\pi}{2}|\Im(s)|)$ for $s\in \cT_I$; we have $\cJ_{l}^{\sn(T_\theta^\dagger)w_2,(s,\mu)}(\lambda,h)$ equals
\begin{align}
-i^{l-s}\mu_\infty(\lambda^{-1}\det h)&2^{\frac{l}{2}+1}\pi^{l+\frac{3}{2}}|D|^{\frac{1}{2}(s+l-3)}|\lambda|^{\frac{3}{2}}\Gamma(s+1)^{-1}\Gamma(l-s-1)^{-1} 
 \notag
\\
&\quad \times \int_{0}^{1}(1-x)^{s}x^{-(s+\frac{1}{2})}J_{l-\frac{3}{2}}\left(\tfrac{\pi x}{|\lambda|}\right)\,\exp\left({-\pi i }\fq(h))
\right)\,\d x.
 \label{JJTtw2-L2-f100}
\end{align} 
\end{lem}
\begin{proof} Let us evaluate the integral $W_{\sn(T_\theta^\dagger)w_2}^{Z}(\sm(1_2,\lambda^{-1})b_\R^\theta)$ for large $\lambda>0$. Let $X\in \sV(\R)$ and write it as $X=x T_\theta^\dagger+X_{\beta}$ with $x\in \R$ and $\beta\in \C$. Set $z:=-{ix}+\frac{2\lambda}{\sqrt{|D|}}$ and $\beta_1=\frac{\sqrt{|D|}}{2}\beta$; then, by \eqref{ConvAInt-P-f0} followed by \eqref{ConvAInt-P-f1} and \eqref{ConvAInt-P-f2}, 
we get 
\begin{align*}
&\Phi_l^{(s,\mu_\infty)}(\sn(T_\theta^\dagger)w_2\,\sn(X)\,\sm(1_2,\lambda^{-1})b_\R^\theta)
\\
&=(2^{-1}\sqrt{|D|})^{s-3l+1}\lambda^{l}
\left(z^2+|\beta_1|^2\right)^{-l}
\left\{(-i)\left(1-\tfrac{{4}{|D|^{-1}}i\,z}{z^2+|\beta_1|^2}\right)\right\}^{s-l+1}.
\end{align*} Set $Z=cT_\theta^\dagger+X_\gamma$ ($c\in \R,\,\gamma\in \C)$. Then, \begin{align}
W_{\sn(T_\theta^\dagger)w_2)}^Z(b_\R^\theta)=(2^{-1}\sqrt{|D|})^{s-3l+1}\lambda^{l}\times ie^{-2\pi c \sqrt{|D|}\lambda}\,\int_{\Re(z)=\frac{2\lambda}{\sqrt{|D|}}} I(z) e^{\pi i D c z}\,\d z
\label{JJTtw2-L2-f3}
\end{align}
with 
\begin{align*}
I(z):=\tfrac{|D|}{4}\,\int_{\C} \left(z^2+|\beta_1|^2\right)^{-l}
\left\{(-i)\left(1-\tfrac{{4}{|D|^{-1}}i\,z}{z^2+|\beta_1|^2}\right)\right\}^{s-l+1}e^{-2\pi i \sqrt{|D|} \Re(\bar \beta_1 \gamma))}\,\d\eta_\C(\beta_1).
\end{align*}
Let $\varepsilon\in (0,1)$. It is easy to see that $w:=\frac{{-4}{|D|^{-1}}i\,z}{z^2+|\beta_1|^2}$ is in the disc $|w|<\varepsilon$ if $\Re(z)=\frac{2\lambda}{\sqrt{|D|}}>\frac{4}{\varepsilon |D|}$, or equivalently $\lambda>\frac{2}{\varepsilon\sqrt{ |D|}}$. Then, $((-i)(1+w))^{s-l+1}=(-i)^{s-l+1}(1+w)^{s-l+1}$ by \eqref{CplxPw-f2}, and we can expand $(1+w)^{s-l+1}$ to the Taylor series in $w$ to see that $I(z)$ equals 
\begin{align*}
\tfrac{|D|}{4}\,(-i)^{s-l+1}\,\sum_{n=0}^{\infty} \tbinom{s-l+1}{n}({-4}{|D|^{-1}}i\,z)^{n} 
\int_{\C}(z^2+|\beta_1|^2)^{-(n+l)}e^{-2\pi i \sqrt{|D|} \Re(\bar \beta_1 \gamma))}\,\d \eta_{\C}(\beta_1).
\end{align*}
Suppose $\gamma\not=0$. Then, by Lemma \ref{JJw2reg-L3}, we obtain
\begin{align*}
I(z)=\tfrac{|D|}{4}\,
(-i)^{s-l+1}\,2\pi^l(\sqrt{|D|}|\gamma|)^{l-1}z^{1-l}
\sum_{n=0}^{\infty} \tbinom{s-l+1}{n}\tfrac{(-2\pi i |D|^{-1/2}|\gamma|)^{n}}{\Gamma(l+n)} K_{n+l-1}(2\pi |D|^{1/2}|\gamma|z). 
\end{align*} 
By this and \eqref{JJTtw2-L2-f3}, $
W_{\sn(T_\theta^\dagger)w_2)}^Z(\sm(1_1,\lambda^{-1}) b_\R^\theta)$ with large $\lambda>0$ equals
\begin{align*}
&(2^{-1}\sqrt{|D|})^{s-3l+1}\lambda^{l}\times ie^{-2\pi c \sqrt{|D|}\lambda}
\times \tfrac{|D|}{4}\,(-i)^{s-l+1}\,2\pi^l(\sqrt{|D|}|\gamma|)^{l-1}\\
&\times \sum_{n=0}^{\infty} \tbinom{s-l+1}{n}\tfrac{(-2\pi i |D|^{-1/2}|\gamma|)^{n}}{\Gamma(l+n)} \int_{\Re(z)=\frac{2\lambda}{\sqrt{|D|}}} z^{1-l} K_{n+l-1}(2\pi |D|^{1/2}|\gamma|z)\,e^{\pi D c z}\,\d z.
\end{align*}
Note the the contour integral is independent of $\lambda$. Now, we use 
\begin{align*}
&\int_{\Re(z)=\frac{2\lambda}{\sqrt{|D|}}} z^{1-l} K_{s}(2\pi |D|^{1/2}|\gamma|z)\,e^{\pi D c z}\,\d z
\\
&=\tfrac{-\pi i}{\Gamma(l-1)}(\pi |D|^{1/2})^{l-2}|\gamma|^{-s}(\det Z)^{(s+l-2)/2}\,\int_{-1}^{1}(1-t^2)^{l-2}\left(t+\tfrac{-\sqrt{|D|}\,c}{2\sqrt{\det Z}}\right)^{s-l+1}\,\d t, 
\end{align*}
which is proved in the proof of Lemma \ref{JJw2reg-L2}. Thus, $W^{Z}_{\sn(T_\theta^\dagger)w_2}(b_\R^\theta)$ is the product of 
\begin{align}
&(2^{-1}\sqrt{|D|})^{s-3l+1}\lambda^{l}\times ie^{-2\pi c \sqrt{|D|}\lambda}
 \notag
\\
&\times \tfrac{|D|}{4}\,(-i)^{s-l+1}\,2\pi^l(\sqrt{|D|}|\gamma|)^{l-1}
\times \tfrac{-\pi i}{\Gamma(l-1)}(\pi |D|^{1/2})^{l-2}|\gamma|^{-(l-1)}(\det Z)^{l-3/2}
 \label{JJTtw2-L2-f4}
\end{align}
and 
\begin{align}
\sum_{n=0}^{\infty} \tbinom{s-l+1}{n}\tfrac{(-2\pi i |D|^{-1/2}|\gamma|)^{n}}{\Gamma(l+n)}(|\gamma|^{-1}\sqrt{\det Z}|)^n \int_{-1}^{1}(1-t^2)^{l-2}\left(t+\tfrac{-\sqrt{|D|}\,c}{2\sqrt{\det Z}}\right)^{n}\,\d t. 
 \label{JJTtw2-L2-f5}
\end{align}
The expression in \eqref{JJTtw2-L2-f4} is easily simplified as
\begin{align*}
\Gamma(l-1)^{-1}\lambda^{l}e^{-2\pi c\sqrt{|D|}\lambda}\,2^{-s+3l-2}\,\pi^{2l-1}\,|D|^{(s-l)/2}\,i^{l-s}\,(\det Z)^{l-3/2}. 
\end{align*}
By using the hypergeometric series ${}_1F_1(a,b;z)$, the expression in \eqref{JJTtw2-L2-f5} is written as  
\begin{align*}
-{\Gamma(l)}^{-1}
\int_{-1}^{1}(1-t^2)^{l-2}{}_1F_1\left(l-s+1,l;\,{2\pi i}\sqrt{\tfrac{{\det Z}}{|D|}}\left\{t+\tfrac{-c}{2} \sqrt{\tfrac{|D|}{\det Z}}\right\}\right)\,\d t,
\end{align*}
which in turn is computed by Lemma \ref{JJTtw2-L5}. 
Thus, we end up with the formula
\begin{align}
W_{\sm(T_\theta^\dagger)w_2}^{Z}(\sm(1_2,\lambda^{-1})b_\R^\theta)
&=-2^{-s+3l-2}\,\pi^{l+\frac{3}{2}}\,|D|^{\frac{1}{2}(s-\frac{3}{2})}\,i^{l-s}\,(\det Z)^{\frac{1}{2}(l-\frac{3}{2})}\,\lambda^{l}e^{-2\pi c\sqrt{|D|}\lambda}\, 
\label{JJTtw2-L2-f6} 
\\
&\times \tfrac{1}{\Gamma(s+1)\,\Gamma(l-s+1)}\int_{0}^{1}(1-x)^{s}x^{-(s+\frac{1}{2})}J_{l-\frac{3}{2}}\left(2\pi \sqrt{\tfrac{\det Z}{|D|}} x \right)\,e^{-\pi i c x}\,\d x.
 \notag
\end{align}
for large $\lambda>0$ and for $Z=c T_\theta^\dagger+X_\gamma$ such that $\gamma\not=0$, $\det Z>0$. Since the both sides of \eqref{JJTtw2-L2-f6} are continuous in $Z$, the equality is extended to all $Z\in \sV(\R)$ such that $\det Z>0$. By \eqref{CompArchInt-fff}, the factor $\lambda^{l}e^{-2\pi c\sqrt{|D|}\lambda}$ equals ${\mathcal B}_{l}^{Z}( \sm(1_2,\lambda^{-1})b_\R^\theta)$. Hence by \eqref{CompArchInt-f0}, we obtain the formula of $W_{\sm(T_\theta^\dagger)w_2}^{Z}(b_\R^\theta)$ as on the right-hand side of \eqref{JJTtw2-L2-f6} with $\lambda=1$ being substituted. Applying \eqref{CompArchInt-f1} and \eqref{CompArchInt-f2}, we are done.  
\end{proof}

\begin{lem}\label{JJTtw2-L5} For $l\in \Z_{>1}$, $s\in \cT_{(-1,l-1)}$ and $a\in \R^\times$, $b\in \R$, 
\begin{align}
&\int_{-1}^{1}(1-t^2)^{l-2}{}_1F_1(l-s+1,l;2\pi i (at+b))\,\d t
\label{JJTtw2-L5-f0}
\\
&=\tfrac{\pi \Gamma(l)\Gamma(l-1)}{\Gamma(s+1)\Gamma(l-s+1)}(\pi|a|)^{\frac{3}{2}-l} \int_{0}^{1}(1-x)^{s}x^{-(s+\frac{1}{2})}J_{l-\frac{3}{2}}(2\pi |a|x)\,e^{2\pi i b x}\,\d x.
 \notag
\end{align}
\end{lem}
\begin{proof} By the integral representation 
\begin{align*}
&{}_1F_1\left(-s+l-1,l;\,{2\pi i }(at+b)\right) \\
&=\tfrac{\Gamma(l)}{\Gamma(-s+l-1)\,\Gamma(s+1)}\,\int_{0}^{1} 
\exp\left({2\pi i }(at+b)\,x \right)\,x^{-s+l-2}(1-x)^{s}\,\d x, 
\end{align*}valid on the region $\cT_{(-1, l-1)}$ (\cite[p.274]{MOS}), the left-hand side of \eqref{JJTtw2-L5-f0} equals 
\begin{align*}
\tfrac{\Gamma(l)}{\Gamma(-s+l-1)\,\Gamma(s+1)}\times 
\int_{0}^{1} 
\left\{\int_{-1}^{+1}(1-t^2)^{l-2}\,\exp\left({2\pi i }a\,t\,x\right)\d t\right\}\,e^{2\pi i bx}\,x^{-s+l-2}(1-x)^{s}\,\d x.
\end{align*}
The $t$-integral, which is absolutely convergent for $l>1$, is evaluated by \cite[3.384.1]{GR} as
\begin{align*}
2^{2(l-1)-1}\,B(l-1,l-1)\,e^{{2\pi i }|a|x}\,
{}_1F_1\left(l-1,2l-2;\,{4\pi i}|a| x\right),
\end{align*}
where $B(l-1,l-1)=\Gamma(l-1)^2\Gamma(2l-2)^{-1}$ is the beta function. The hypergeometric function in the formula equals $\sqrt{\pi}\,\Gamma(l-\rho+1/2)\,\left({\pi |a|x}\right)^{-l+3/2}\,J_{l-3/2}\left(2{\pi x}|a|\right)$ by the fifth displayed formula on \cite[p.283]{MOS}. \end{proof}
By \eqref{JJTtw2-L2-f100} and Lemma \ref{P6-L9}, noting $|i^{-s}|=O(\exp(\frac{\pi}{2}|\Im(s)|)$, we get the bound 
\begin{align}
\cJ_l^{\sn(T_\theta^\dagger)w_2,(s,\mu)}(\lambda,h) \ll \tfrac{\exp(\tfrac{\pi}{2}|\Im(s)|)}{|\Gamma(s+1)\Gamma(l-s-1)|}\tfrac{(8\pi\sqrt{|D|})^{l}}{|\Gamma(l-s-1)|}\,|\lambda|^{-l+5}\fq(h)^{-2}
 \label{JJTtw2-L6}
\end{align}
for $(\lambda,h)\in \R^{\times} \times {\bf GL}_2(\R)$ uniform in $l\in 2\Z_{\geq 3}$ and $s\in \cT_{(1,l-3)}$. Combining \eqref{JJTtw2-L6} with Lemma \ref{JJTtw2-L1}, we obtain the following quantity as a uniform majorant of the series in \eqref{JJu-f4} with $\su=\sn(T_\theta^\dagger)w_2$ for $s\in \cT_{(1,l-3)}$ and $l\in \Z_{\geq l_1}$ and $\tau \in \cN_\fin \cN_\infty$:  
\begin{align*}
&\tfrac{\exp\left(\tfrac{\pi}{2}|\Im (s)|\right)(1+|s|)^{2}}{|\Gamma(s+1)|}\tfrac{l^{2}\,(8\pi\sqrt{|D|})^{l}}{|\Gamma(l-s-1)|\Gamma(l-1/2)} \|f_S\|_1 2^{\omega(N)}N^{2\Re(s)}\fa(f_S)^{3\Re(s)+4}\\
&\times\Bigl(\sum_{\lambda \in \frac{2}{D}\fa(f_S)^{-1}N\Z_{>0}} \lambda^{\Re(s)-l+6}\Bigr)\times \Bigl\{\sum_{\delta \in I_\theta(E^\times)\bsl {\bf GL}_2(\Q)} \|T_\theta\ss(\delta^{-1})\|_\fin^{-\Re(s)-1}\fq(\delta_\infty)^{-2}\Bigr\}. 
\end{align*} 
Note that $\tfrac{\exp\left(\tfrac{\pi}{2}|\Im (s)|\right)}{|\Gamma(s+1)|}\ll 1$ by Stirling's formula. The $\lambda$-series is convergent and is bounded by $O((\fa(f_S)^{-1}N)^
{-l+\Re(s)+6})$ if $\Re(s)<l-7$. The $\delta$-sum is convergent by Lemma \ref{P6-L12} if $\Re(s)>1$, because $\|Z\|_\infty \asymp \{|c(Z)|^2+1\}^{1/2}$ for $Z\in \Omega_{D}(\R)$ by \eqref{JJw2regPre-f0}.

\subsection{An integral involving Bessel function}  
For $(l,a,b)\in \Z_{>0} \times \R^\times\times \R$, set 
$${\mathcal I}_l^{(s)}(a,b):=\int_0^{1}(1-x)^{s}x^{-(s+\frac{1}{2})}J_{l-\frac{3}{2}}(2\pi a x)\exp(2\pi i b x)\,\d x, \quad \Re(s)>0. 
$$

\begin{lem} \label{P6-L7}
Let $\Re(s)<l-2$ and $b\not=0$. For any $q\in \Z_{>0}$ such that $q\leq l-\Re(s)-1$ and $q-1\leq \Re(s)$, 
\begin{align*}
\Ical_l^{(s)}(a,b)&=(2\pi i b)^{-q}\,\int_{0}^{1} 
\tfrac{\d^{q}}{\d x^q}\bigl\{x^{-(s+1/2)}(1-x)^{s}\,
J_{l-3/2}\left(2\pi ax \right)\bigr\} \,\exp\left(2\pi i b x\right)\,\d x.
\end{align*}
\end{lem}
\begin{proof} Note $J_\nu(x) \sim (x/2)^{\nu}$ as $x\rightarrow +0$. By using the formula 
\begin{align}
J'_\nu(x)=2^{-1}(J_{\nu-1}(x)-J_{\nu+1}(x))
 \label{P6-L7-f0}
\end{align} (\cite[\S 3.1.1(p.67)]{MOS}), it is not hard to confirm that for $q \leq l-\Re(s)-2$ and $q\leq \Re(s)$ all the derivatives $(\frac{\d}{\d x})^{j} x^{-(s+1/2)}(1-x)^{s}J_{l-3/2}(2\pi ax)$ for $j<q$ vanish at $x=0$. Thus, a successive application of integration by parts finishes the proof. \end{proof}

By the formula in \cite[8.411 8]{GR}, we have
the uniform bound 
\begin{align}
|J_{\nu}(x)| \leq 2\Gamma(\nu+1)^{-1}(x/2)^{\nu} \quad (x>0,\,\nu>1).
\label{JBesselEst3}
\end{align}

\begin{lem} \label{P6-L8}
Let $q\in \Z_{>0}$. Then,  
\begin{align*}
&\left|\tfrac{\d^{q}}{\d x^q}\bigl\{x^{-(s+1/2)}(1-x)^{s}\,J_{l-3/2}(2\pi ax)\bigr\}\right| 
\ll_{q} (1+|s|)^{q} \tfrac{l^{q}\,a^{l-3/2}(a^{2q}+a^{-q})}{\Gamma(l-1/2)}
\end{align*}
holds uniformly for $x\in (0,1)$, $a>0$, $l> 2q$, $s\in \cT_{(q-1, 
l-q-1)}$ and $a>0$. 
\end{lem}
\begin{proof} By a successive application of the formula \eqref{P6-L7-f0}, using \eqref{JBesselEst3}, we have
\begin{align*}
\left|\left(\tfrac{\d}{\d x}\right)^{k}J_{l-3/2}(2\pi ax) \right| 
&\ll_{k} \max(1,a^{k})\sum_{u=-k}^{k} {a^{l-3/2+u}}\Gamma(l-1/2+u) \\
&\ll_{k}a^{l-3/2}(a^{2q}+a^{-q})\,l^{q}
\,\Gamma(l-1/2)^{-1}.
\end{align*}
Estimating all the non-negative powers $x$ and $1-x$ on $(0,1)$ bounded trivially by $1$, we get the desired bound by Leibnitz' rule. 
\end{proof}

From Lemmas~\ref{P6-L7} and \ref{P6-L8}, we obtain the following uniform bound of ${\Ical}_l^{(s)}(a,b)$. 

\begin{lem} \label{P6-L9}
Let $q\in \N$. Then, 
\begin{align*}
|{\mathcal I}_l^{(s)}(a,b)|&\ll_{q} 
{(1+|\Im(s)|)^{q}}\Gamma(l-1/2)^{-1}\,\left|b\right|^{-q}{l^{q}a^{l-3/2}(a^{2q}+a^{-q})}
\end{align*}
holds for $l>2q,\,s \in \cT_{(q-1,l-q-1)},\,a\in \R^\times_+$ and $b\in \R^\times $.
\end{lem}

\section{The $p$-adic local orbital integrals for the main terms}\label{sec:pLOIMTM}
In this section, we abbreviate superscripts in notation for most of the local objects, e.g., write $E$, $\cO$, $\psi$, $\Lambda$, $\bK^\#$, $\bK$ in place of $E_p$, $\cO_{E,p}$, $\psi_{p}$, $\Lambda_p$, $\bK_p^\#$, $\bK_p$. Let $\mu:\Q_p^\times\rightarrow \C^\times$ be a character with conductor $p^{e}\Z_p$, and $W_{\Q_p}(\mu)$ its root number ({\it cf}. \cite[\S6.3]{KugaTsuzuki}). Recall the integral $\cW^{(T_\theta, \Lambda)}(f;g)$ defined by \eqref{MainT-f1} for $f \in \cH(\sG(\Q_p))$ and $g\in \sG(\Q_p)$. 
Let $\delta\in {\bf GO}_{T_\theta}(\Q_p)$. 
By \eqref{GOpoints}, the set $I_\theta(E^\times)\bsl {\bf GO}_{T_\theta}(\Q_p)/I_\theta(E^\times)$ is represented by $1_2$ and $\sigma=\left[\begin{smallmatrix}1 & \tt \\ 0 & -1 \end{smallmatrix}\right]$. 
For $\delta\in \{1_2,\sigma\}$ and $(\phi,f)\in \cS(E^2)\times \cH(\sG(\Q_p))$, define
\begin{align}
{\mathcal I}_{\phi,f}^{(s,\Lambda,\mu)}(\delta):=\int_{\Q_p^\times}
\int_{\bK^\#} & \mu^{-1}(\lambda)|\lambda|_p^{-s+1}\sf_{\phi}^{(s,\Lambda,\mu)}(k^\#) {\cW^{(T_\theta, \Lambda^{\delta})}}\left(f; \sm(\delta,\lambda^{-1}\det \delta)^{-1}\iota(k^\#)\right)\,\d^\times\lambda\,\d k^\#,
\label{Def-OBIIdCtp}
\end{align}
where $\Lambda^{\delta}:=\Lambda^{-1}$ if $\delta=1_2$ and $\Lambda^{\delta}:=\Lambda$ if $\delta=\sigma$. Here is the main result of this section: 
\begin{prop} 
\label{OBIIdCt} 
Let $E/\Q_p$ be an unramified extension, $\mu$ a ramified character of $\Q_p^\times$ of conductor $p^{e}\Z_p$, and 
$$ \phi(x,y)=\cchi_{p^{e}\cO}(x)\cchi_{1+p^e\cO}(y), \quad f=\cchi_{b_p\bK}, \quad \Lambda={\bf 1} \quad \text{with $b_{p}:=\left[\begin{smallmatrix} p^e 1_2 & T_\theta^\dagger \\ 0 & 1_2 \end{smallmatrix} \right]\in \sG(\Q_p)$}.
$$ Then, 
\begin{align}
{\mathcal I}_{\phi,f}^{(s,\Lambda,\mu)}(\delta)&=
 p^{(s-11/2)e}(1+p^{-2})^{-1}\mu\left(\tfrac{D\det \delta }{2}\right)\,\zeta_{\Q_p}(1)\zeta_{E}(1) W_{\Q_p}(\mu), 
\label{OBIIdCt-f0}
\\ 
{\mathcal I}_{\widehat \phi,f}^{(-s,\Lambda^{-1},\mu^{-1})}(\delta)
&=p^{e(-3s-11/2)}(1+p^{-2})^{-1}\mu^{-1}\left(\tfrac{D\det \delta}{2}\right)\,\zeta_{\Q_p}(1)\zeta_{E}(1)\,W_{\Q_p}(\mu^{-1})^{3}.
\label{OBIIdCt-f1}
\end{align}
\end{prop}

\begin{prop} \label{LocalSingw2}
Let $E/\Q_p$ be an unramified extension and $$\phi=\cchi_{\cO\oplus \cO}, \quad f=R(\eta)\cchi_{\bK_0(p\Z_p)}, \quad \Lambda={\bf 1},\quad \mu|\Z_p^\times={\bf 1} \quad \text{with $\eta:=
\left[\begin{smallmatrix} {} & {} & {} & {-1} \\
 {} & {} & {1} & {} \\
{} & {p} & {} & {} \\
{-p} & {} & {} & {}
\end{smallmatrix} \right]\in \sG(\Q_p)$}.$$
 Then, ${\mathcal I}_{\phi,f}^{(s,\Lambda,\mu)}(\delta)=\mu(p)\,p^{s-1}.$ 
\end{prop}

\subsection{Proof of Proposition \ref{OBIIdCt}}
We prove formula \eqref{OBIIdCt-f0}. Note that $p\,(>2)$, is a prime element of $E$, and $e>0$. Set
\begin{align*}
\cU(p^{e}\Z_p)&:=\bK \cap (1_4+p^{e}{\bf M}_4(\Z_p)), \quad \cU^\#(p^{e}\Z_p):=\bK^\#\cap (1_2+p^{e}{\bf M}_2(\cO)). 
\end{align*}
Then we easily have the following. 
\begin{lem}\label{LocalSingIdM-L1}
We have $\iota_{\theta}(\cU^\#(p^e\Z_p))\subset \cU(p^e\Z_p)$, $b_{p}^{-1}\cU(p^{e}\Z_p)b_{p}\subset \bK$, and
\begin{align}
\sB^\#(\Z_p)\cU^\#(p^e\Z_p)=\bK_0^\#(p^e\Z_p). 
\label{BU=K0}
\end{align}
\end{lem}

\begin{lem}\label{LocalSingIdM-L2} 
Set $$
F(k^\#):=\mu(\det k^\#)\,\,\mu(d\bar d)^{-1}
\int_{\Q_p^\times}\mu^{-1}(\lambda)|\lambda|_p^{-s+1}\,\cW^{T_\theta,\Lambda}(f; \sm(\delta, \lambda^{-1}\det \delta)\iota(k^\#))\,\d^\times \lambda,
$$  
viewed as a function in $k^\#=\left[\begin{smallmatrix} a & b \\ c & d \end{smallmatrix}\right]\in \bK_p^\#$.
\begin{itemize}
\item[(i)] $F(k^\#)$ is constant on $\sB^\#(\Z_p)\cU^\#(p^{e}\Z_p)$. \item[(ii)] The value $F(1_2)$ equals 
$$
\mu(p^{-e})\,p^{e(s-1)}\,\mu\left(\tfrac{D\det \delta}{2}\right)^{-1}\int_{\Z_p^\times}\mu(u)\psi(p^{-e}u)\,\d^\times u
.$$
\end{itemize}
\end{lem}
\begin{proof} 
(i) Write $k^\#=x^\# u^\#$ with $u^{\#}\in \cU^\#(p^e\Z_p)$ and $x^\#\in \sB^\#(\Z_p)$; then, $F(k^\#)$ is the product of $\mu(\det k^\#)\,\,\mu(d\bar d)^{-1}$ and the integral 
\begin{align}
\int\mu(\lambda)^{-1}|\lambda|_p^{-s+1}\cchi_{\bK}(b_p^{-1}\iota(u^\#)^{-1}\iota(x^\#)^{-1}\sm(\delta, \lambda^{-1}\det \delta)\sm_\theta(\tau)\sn(X))\,\psi(-\tr(T_\theta X))\,\d X\,\,\d^\times \tau\,\d^\times \lambda
\label{LocalSingIdM-L1-f0}
\end{align}
with integration domain being $\Q_p^\times\times E^\times \times \sV(\Q_p)$. By Lemma \ref{LocalSingIdM-L1}, this is independent of $u^\#$. Write $(x^\#)^{-1}=\left[\begin{smallmatrix} \tau & 0 \\0 & \bar \tau v\end{smallmatrix}\right] \left[\begin{smallmatrix} 1 & \beta \\ 0 & 1 \end{smallmatrix}\right]$ with $\tau \in \cO^\times$, $\beta\in \cO$ and $v\in \Z_p^\times$. Since 
\begin{align*}
\iota(x^\#)^{-1}\sm(\delta, \lambda^{-1}\det \delta)\sm_\theta(\tau_1 \tau)\sn(X)^{-1}
=\sm(\delta, (v^{-1}\lambda)^{-1}\det \delta)\sm_\theta(\tau)\sn(X+X_{\beta_1})^{-1}
\end{align*}
with some $\tau_1\in \cO^\times$ and $\beta_1\in \cO$, by the change of variables $X\rightarrow X-X_{\beta'}$, $\tau\rightarrow \tau_1^{-1}\tau$ and $\lambda\rightarrow v\lambda$, the integral \eqref{LocalSingIdM-L1-f0} as a function in $x^\#$ equals its value at $x^\#=1_2$ multiplied with $\mu(v)^{-1}$, because $\tr(T_\theta X_{\beta_1})=0$. Since $\det k^{\#}\in (1+p^{e}\Z_p)(\tau\bar \tau v)^{-1}$ and $d \bar d\in (1+p^e\Z_p)(\tau\bar\tau v^2)^{-1}$, we have $\mu(\det k^\#)\,\,\mu(d\bar d)^{-1}=\mu(v)$.

(ii) The condition $b_p^{-1}\sm(\delta, \lambda^{-1}\det \delta)\sm_\theta(\tau)\sn(X)\in \bK$ is equivalent to 
$$
\tau\in p^{e}\cO^\times, \quad \lambda^{-1} \in p^{-e}\Z_p^\times,\,X-\lambda^{-1} T_\theta^\dagger \ss({}^t\delta^{-1})\in {\bf M}_2(\Z_p).
$$
Thus, $\psi(-\tr(T_\theta X))=\psi(-\lambda^{-1}\tr(T_\theta \ss(\delta^{-1})\, T_\theta^\dagger))$, which is equal to $\psi(\lambda^{-1}D/2)$ if $\delta=1_2$ and to $\psi(-\lambda^{-1}D/2)$ if $\delta=\sigma$ due to $T_\theta\ss(\sigma)=-T_\theta$. Thus, $F(1_2)$ equals 
\begin{align*}
&\int_{p^{e}\Z_p^\times} \mu^{-1}(\lambda)|\lambda|_p^{-s+1}\,\psi(\det \delta\,\lambda^{-1}D/2)\,\d^\times \lambda
=p^{e(s-1)}\,\mu(p^{-e})\mu(\tfrac{D\det \delta}{2})\int_{\Z_p^\times}\mu(u)\,\psi\left(\tfrac{u}{p^{e}}\right)\,\d^\times u. 
\end{align*}
\end{proof}

\begin{prop} We have $[\bK^\#:\bK_0^\#(p^{e}\Z_p)]=p^{2e}(1+p^{-2})$, 
\begin{align*}
\mathcal I_{\phi,f}^{(s,\Lambda,\mu)}(\delta)=[\bK^\#:\bK_0^\#(p^{e}\Z_p)]^{-1}\,\mu\left(\tfrac{D\det \delta }{2}\right)\,p^{(s-7/2)e}\,\zeta_{\Q_p}(1)\zeta_{E}(1) W_{\Q_p}(\mu). 
\end{align*}
\end{prop}
\begin{proof}
By \eqref{SBftn-f0}, \eqref{DefJaqS}, and \eqref{BU=K0}, for $k^{\#}$ as above, \begin{align*}
\sf_{\phi}^{(s,\Lambda,\mu)}(k^{\#})=\vol(1+p^{e}\cO)\,\mu(\det k^\#)\,\mu(d\bar d)\,\cchi_{\sB^\#(\Z_p)\cU^\#(p^{e})}(k^\#).
\end{align*}
Hence, by Lemma \ref{LocalSingIdM-L2}, we have $
{\mathcal I}_{\phi,f}^{(s,\Lambda,\mu)}(\delta)
=\vol(1+p^{e}\cO)\,\vol(\sB^\#(\Z_p)\cU^\#(p^{e}\Z_p))\,F(1_2)$. Note that $\vol(\sB^\#(\Z_p)\cU^\#(p^{e}\Z_p))=[\bK^\#:\bK_0^\#(p^{e})]^{-1}$ and $\vol(1+p^{e}\cO)=\zeta_{E}(1)p^{-2e}$. Moreover, we use \cite[Lemma 6.1]{KugaTsuzuki}. \end{proof} 
Next, we move to the proof of \eqref{OBIIdCt-f1}. Recall $w_0=\left[\begin{smallmatrix} 0 & -1 \\ 1 & 0 \end{smallmatrix}\right]
$.  

\begin{lem} \label{LocalSingw2M-L0}
 The value $\sf_{\widehat {\phi}}^{(-s,\Lambda^{-1},\mu^{-1})}(k^\#w_0)$ for $ k^\#=\left[\begin{smallmatrix} a & b \\  c & d \end{smallmatrix}\right]\in \bK^\#_p$ equals
\begin{align*}
\mu^{-1}(D)\mu(p^{2e})p^{2e(-s-1)}\,\mu^{-1}(-\det k^\#)\,\mu(d\bar d)\,\cchi_{\cO^\times}(d)\int_{\cO^\times}\psi_{E}(p^{-e}u))\,\mu^{-1}(u\bar u)\,\d^\times u. 
\end{align*}
\end{lem}
\begin{proof} The value $\widehat{\phi}((x,y)w_0)$ is known by \eqref{FrTr-phie}. Since $\Lambda$ is the trivial character, by \eqref{DefJaqS}, the value $\sf_{\widehat {\phi}}^{(-s,\Lambda^{-1},\mu^{-1})}(k^\#w_0)$ equals 
\begin{align*}
\mu^{-1}(-\det k^\#)p^{-4e} \int_{E^\times} \cchi_{p^{-e}\cO}(c\bar \tau)\cchi_{p^{-e}\cO}(d\bar \tau)\,\psi_{E}({d\bar \tau}{\sqrt{D}}^{-1})\mu_{E}^{-1}(\tau)|\tau\bar \tau|_{p}^{-s+1}\d^\times \tau.
\end{align*}
Since $c$ and $d$ are relatively prime in $\cO$, the condition $c\bar t\in p^{-e}\cO_{E,p}$ and $d\bar \tau \in p^{-e}\cO$ is equivalent to $\tau \in p^{-e}\cO$. Set $d=p^{k}d_0$ with $k\in \Z_{\geq 0}$ and $d_0\in \cO^\times$. Then, since the conductor of $\mu_{E}$ is $p^{e}\cO$ $(e={\rm ord}_p(M)>0)$, 
{\allowdisplaybreaks\begin{align*}
&\int_{p^{-e}\cO}\psi_{E}({d\bar \tau}{\sqrt{D}}^{-1})\mu_{E}^{-1}(\tau)|\tau\bar \tau|_{p}^{-s+1}\d^\times \tau
\\
&=\sum_{l\geq -e}\mu_{E}^{-1}(-\sqrt{D}\,p^l)p^{-2l(-s+1)} \int_{\cO^\times} \psi_{E} \left(p^{k+l}d_0\bar u\right)\mu_{E}^{-1}(u)\,\d^\times u
\\
&=\mu^{-1}(D)\sum_{l\geq -e} \mu^{-1}(p^{2l})p^{-2l(-s+1)}\times \delta_{-l-k,e}\,
\mu_{E}^{-1}(\bar d_0^{-1})\int_{\cO^\times}\psi_{E} (p^{-e}u)\mu_{E}^{-1}(\bar u)\,\d^\times u
\\
&=\mu^{-1}(D)\mu(p^{2e})p^{2e(-s+1)}\mu(d \bar d
)\cchi_{\cO^\times}(d)\, \int_{\cO^\times} \psi_{E}(p^{-e}u)\mu
^{-1}(u\bar u)\,\d^\times u. \qedhere
\end{align*}}
\end{proof}
Define 
$$F^*(k^\#):=\mu^{-1}(\det k^\#)\,\,\mu(d\bar d)
\int_{\Q_p^\times}\mu(\lambda)|\lambda|_p^{s+1}\,\cW^{T_\theta,\Lambda^{-1}}(f; \sm(\delta, \lambda^{-1}\det \delta)^{-1}\iota(k^\# w_0))\,\d^\times \lambda,
$$
Then, by Lemma~\ref{LocalSingw2M-L0}, ${\mathcal I}_{\widehat{\phi},f}^{(-s,\Lambda^{-1},\mu^{-1})}(\delta)$ equals 
\begin{align}
\mu^{-1}(-D)\,p^{2e(-s-3/2)}\,\zeta_{E}(1)\,W_{E}(\mu_E^{-1})
\label{LocalSingw2M-f00}
\end{align}
times the integral
\begin{align}
\int_{\substack {k^\#\in \bK^\# \\ d\in \cO^\times}}F^{*}(k^\#)\,\d k^\#.
\label{LocalSingw2M-f0}
\end{align}
The function $F^{*}(k^\#)$ is left $\sB^\#(\Z_p)$-invariant and right $\cU^\#(p^e\Z_p)$-invariant ({\it cf}. the proof of Lemma \ref{LocalSingIdM-L2}). Since $\bK_0^{\#}(p^e\Z_p)\bsl \bK^\#=\sB^\#(\Z_p)\bsl \bK_p^\#/\cU^\#(p^\e\Z_p)$ is represented by the elements $\left[\begin{smallmatrix}  1 & 0 \\ \bar \xi & 1 \end{smallmatrix}\right]$ $(\xi \in \cO/p^e\cO)$ and $\left[\begin{smallmatrix} 0 & -1 \\ 1 & p^j \eta \end{smallmatrix}\right]$ $(i\in [1,e],\,\eta \in (\cO/p^{e-j}\cO)^\times)$, the integral in \eqref{LocalSingw2M-f0} becomes the sum $
[\bK^\#:\bK_0^{\#}(p^e\Z_p)]^{-1}
\sum_{\xi \in \cO/p\cO} F^{*}\left(\left[\begin{smallmatrix}  1 & 0 \\ \bar \xi & 1 \end{smallmatrix}\right]w_0\right)$.
 
\begin{lem} \label{LocalSingw2M-L1}
If $\tfrac{D}{4}+\nr_{E/\Q_p}(\xi)\in \Z_p^\times$, then 
$$
F^{*}
\left(\left[\begin{smallmatrix}  1 & 0 \\ \bar \xi & 1 \end{smallmatrix}\right]w_0\right)=\mu(p^{e})p^{-e(s+1)}\,\mu^{-1}\left(\tfrac{D}{4}+\nr_{E/\Q_p}(\xi)\right)\,\mu\left(\tfrac{D\det \delta}{2}\right)\, \int_{\Z_p^\times}\mu^{-1}(v)\psi(p^{-e}v)\,\d^\times v.
$$
If $\tfrac{D}{4}+\nr_{E/\Q_p}(\xi)\in p\Z_p$, then $
F^{*}\left(\left[\begin{smallmatrix}  1 & 0 \\ \bar \xi & 1 \end{smallmatrix}\right]w_0\right)=0$.  
\end{lem}
\begin{proof}
By definition, $F^{*}
\left(\left[\begin{smallmatrix}  1 & 0 \\ \bar \xi & 1 \end{smallmatrix}\right]w_0\right)$ equals
\begin{align*}
\int \mu(\lambda) |\lambda|_p^{s+1}\cchi_{\bK}(b_p^{-1}\iota(k^\# w_0)^{-1}\sm(\delta, \lambda^{-1}\det \delta)\sm_\theta(\tau)\sn(X))\,\psi(-\tr(T_\theta X)\,\d X\,\,\d^\times \tau\,\d^\times \lambda.
\end{align*}
If $g_1,g_2\in \sG(\Q_p)$ satisfies $g_2=g_1 k\,(\exists k\in \bK)$, we write $g_1\sim_{\bK}g_2$. Then,  
\begin{align}
\iota\left(\left[\begin{smallmatrix}  1 & 0 \\ \bar \xi & 1 \end{smallmatrix}\right]w_0\right)\,b_p
\sim_{\bK_p} \left[\begin{smallmatrix} 1_2 & 0 \\ T_\theta+X_\xi^{\dagger} & p^e 1_2 \end{smallmatrix} \right].
 \label{LocalSingw2M-L1-f00}
\end{align}
Note that $\det(T_\theta+X_\xi^\dagger)=-(\tfrac{D}{4}+\nr_{E/\Q_p}(\xi))$. 
If $\tfrac{D}{4}+\nr_{E/\Q_p}(\xi)\in \Z_p^\times$, then $A:=(T_\theta+X_\xi^\dagger)^{-1}\in {\bf GL}_2(\Z_p)$ and $B:=-p^eA\in \sV(\Z_p)$; hence, 
$$ 
\iota(\left[\begin{smallmatrix}  1 & 0 \\ \bar \xi & 1 \end{smallmatrix}\right]w_0)b_p
\sim_{\bK} \left[\begin{smallmatrix} 1_2 & 0 \\ T_\theta+X_\xi^{\dagger} & p^e 1_2 \end{smallmatrix} \right]\left[\begin{smallmatrix} 1_2 & B \\ 0 & 1 \end{smallmatrix}\right]\left[\begin{smallmatrix} 0 & 1_2 \\ -1_2 & 0 \end{smallmatrix}\right]\left[\begin{smallmatrix} A & 0 \\ 0 & {}^t A^{-1}\end{smallmatrix}\right]=\left[\begin{smallmatrix} p^e 1_2 & (T_\theta+X_\xi^\dagger)^{-1} \\ 0 & 1_2
\end{smallmatrix}\right]. 
$$
By this, 
\begin{align*}
&\cchi_{\bK}(b_p^{-1}\iota(k^\# w_0)^{-1}\sm(\delta, \lambda^{-1}\det \delta)\sm_\theta(\tau)\sn(X))
\\
&=\cchi_{\bK}(\left[\begin{smallmatrix} p^e 1_2 & (T_\theta+X_\xi^\dagger)^{-1} \\ 0 & 1_2
\end{smallmatrix}\right]^{-1}
\left[\begin{smallmatrix} \delta I(\tau)  & 0 \\ 0 & \lambda^{-1} \det \delta \nr_{E/\Q_p}(\tau) {}^t \delta^{-1} {}^t I(\tau)^{-1}
 \end{smallmatrix}\right] \left[\begin{smallmatrix} 1_2 & X \\ 0 & 1_2 
\end{smallmatrix}\right])
\end{align*}
is non-zero if and only if
\begin{align*}
&\tau\in p^{e}\cO^\times, \quad \lambda \in p^{e}\Z_p^\times,\quad 
X+\lambda^{-1}\{(\tfrac{D}{4}+X^{\dagger}_{\xi\frac{\bar \tau'}{\tau'}})\ss(\delta)\}^{-1} \in \sV(\Z_p) 
\end{align*}
with $\tau'$ denoting $\tau$ or its conjugate for $\delta=1_2$ or for $\sigma$, respectively. For $X$ as above, 
\begin{align*}
\psi(-\tr(T_\theta X))&=\psi(-\lambda^{-1} \tr(T_\theta\{(\tfrac{D}{4}+X^{\dagger}_{\xi\frac{\bar \tau'}{\tau'}})\ss(\delta)\}^{-1}) 
\\
&=\psi(-\det \delta\cdot \lambda^{-1}\left(\tfrac{D}{4}+\nr_{E/\Q_p}(\xi)\right)^{-1}\tr(T_\theta(T_\theta^\dagger +X_{\xi\bar\tau'/\tau'}))) \\
&=\psi(-\det\delta \cdot \lambda^{-1}\left(\tfrac{D}{4}+\nr_{E/\Q_p}(\xi)\right)^{-1}\tfrac{-D}{2}).
\end{align*}
Hence, $F^{*}
\left(\left[\begin{smallmatrix}  1 & 0 \\ \bar \xi & 1 \end{smallmatrix}\right]w_0\right)$ equals
\begin{align*}
\int_{\lambda \in p^{e}\Z_p^\times} \int_{\tau \in p^{e}\cO} 
 \mu(\lambda) |\lambda|_p^{s+1}\psi(\lambda^{-1}\left(\tfrac{D}{4}+\nr_{E/\Q_p}(\xi)\right)^{-1}\tfrac{D\det\delta }{2}
)\,\d^\times \tau\,\d^\times \lambda.
\end{align*}
By a simple change of variable, we obtain the desired expression. 

Suppose that $\tfrac{D}{4}+\nr_{E/\Q_p}(\xi)\in p\Z_p$. We can write $\tfrac{D}{4}+\nr_{E/\Q_p}(\xi)=p^m v$ with $m\in \Z_{>0}$ and $v\in \Z_p^\times$. Then, $\nr_{E/\Q_p}(\tfrac{2}{\sqrt{D}}\xi)\in 1+4p^{m}\Z_p^{\times}$, which yields $\tfrac{2}{\sqrt{D}}\xi \in E^1(1+4p^{m}\cO)-E^1(1+4p^{m+1}\cO)$ because $E/\Q_p$ is unramified.  Since $F^{*}\left(\left[\begin{smallmatrix} 1 & 0 \\ \bar \xi & 1 \end{smallmatrix}\right]\right)$ as a function in $\xi$ is $E^1$-invariant, we may assume $\frac{2}{\sqrt{D}}\xi=1+4p^{m}y$ with $2\sqrt{D}\,y=y_2+\theta y_3\,(y_2,y_3\in \Z_p)$. Then, 
\begin{align*}
T_\theta+X^\dagger_{\xi}=\left[\begin{smallmatrix} p^{m} & 0 \\ 0 & 1 \end{smallmatrix}\right]Y, \quad Y:=\left[\begin{smallmatrix} y_3 & y_2 \\ p^{m}y_2& -D/2+p^{m}y_1\end{smallmatrix} \right] \quad (y_1=-\tt y_2-\ttN y_3).
\end{align*}  
Since $\det(T_\theta+X_\xi^\dagger)=-(\frac{D}{4}+\nr_{E/\Q_p}(\xi))\in p^{m}\Z_p^\times$, we have $\det Y\in \Z_p^\times$ and $Y\in {\bf M}_2(\Z_p)$, i.e., $Y\in {\bf GL}_2(\Z_p)$. The last element in \eqref{LocalSingw2M-L1-f00} equals 
\begin{align*}
\left[\begin{smallmatrix} p{-m} & {} & {} & {} 
\\ {} & 1 & {} & {} 
\\ {} & {} & {p^m} & {} \\
{} & {} & {} & {1}
\end{smallmatrix} \right] 
\left[\begin{smallmatrix} p^{m} & {0} & {0} & {0} \\
{0} & {1} & {0} & {0} \\
y_3 & y_2 & -p^{e-m} & 0 \\ 
p^m y_2 & -D/2+p^m y_1 & 0 & -p^m
\end{smallmatrix}\right],
\end{align*}
which turns out to be $\bK$-equivalent to $
\left[\begin{smallmatrix}p^{e}(T_\theta+X^\dagger_\xi)^{-1}  
& 1_2 \\ 0 & T_\theta+X^\dagger_\xi
\end{smallmatrix} \right].$ Thus, 
\begin{align}
&\cchi_{\bK}(b_p^{-1}\iota(k^\# w_0)^{-1}\sm(\delta, \lambda^{-1}\det \delta)\sm_\theta(\tau)\sn(X))
 \label{LocalSingw2M-L1-f11}
\\
&=\cchi_{\bK}(
\left[\begin{smallmatrix}p^{e}(T_\theta+X^\dagger_\xi)^{-1}  
& 1_2 \\ 0 & T_\theta+X^\dagger_\xi
\end{smallmatrix} \right]^{-1} 
\left[\begin{smallmatrix} \delta I(\tau)  & 0 \\ 0 & \lambda^{-1} \det \delta \nr_{E/\Q_p}(\tau) {}^t \delta^{-1} {}^t I(\tau)^{-1}
 \end{smallmatrix}\right] \left[\begin{smallmatrix} 1_2 & X \\ 0 & 1_2 
\end{smallmatrix}\right])
\notag
\end{align}
and this value is zero unless $
p^{-e} (T_\theta+X^\dagger_\xi)\delta I(\tau)\in {\bf GL}_2(\Z_p)$. Since $T_\theta+X^\dagger=\left[\begin{smallmatrix} p^{m} & 1 \\ 0  & 1 \end{smallmatrix}\right] Y$ and $Y\in {\bf GL}_2(\Z_p)$, the last condition implies $
p^{-e}\left[\begin{smallmatrix} p^{m} & 1 \\ 0  & 1 \end{smallmatrix}\right] I(\tau')=k$ for some $k=\left[\begin{smallmatrix} k_{11} & k_{12} \\ k_{21} & k_{22}\end{smallmatrix} \right]\in {\bf GL}_2(\Z_p)$. Set $p^{-e}I(\tau')=\left[\begin{smallmatrix} a & -b\ttN \\ b & a+\tt b \end{smallmatrix}\right]$ with $a,b\in \Q_p$. Then, 
\begin{align*}
a=p^{-m}k_{11}, \quad b=k_{21}, \quad b\ttN=-p^{-m}k_{12}, \quad a+\tt b=k_{22},\end{align*}
which together with $(k_{11},k_{12})_{\cO}=\cO$ and $(k_{21}, k_{22})_{\cO}=\cO$ yields
$$
(a,b)_{\cO}=(k_{21},k_{22})_{\cO}=\cO \subset p^{-m}\cO=(p^{-m}k_{11},p^{-m}k_{12})_{\cO}=(a,b\ttN)_{\cO}\subset (a,b)_{\cO},
$$
a contradiction for $m>0$. Hence, the value \eqref{LocalSingw2M-L1-f11} is zero, which in turn yields the vanishing of $F^{*}\left(\left[\begin{smallmatrix}  1 & 0 \\ \bar \xi & 1 \end{smallmatrix}\right]w_0\right)$. 
\end{proof}

\begin{lem} $(${\it cf.} \cite[Lemmas 6.2 and 6.4]{KugaTsuzuki}$)$
\label{LocalSingw2M-L2} We have 
 $W_{E}(\mu^{-1}_{E})=(-1)^{e}W_{\Q_p}(\mu^{-1})^2$, $W_{\Q_p}(\mu^{-1})W_{\Q_p}(\mu)=\mu(-1)$, and 
$$\sum_{\substack{\xi \in \cO/p\cO \\ \tfrac{D}{4}+\nr_{E/\Q_p}(\xi)\in \Z_p^\times}} 
\mu^{-1}\left(\tfrac{D}{4}+\nr_{E/\Q_p}(\xi)\right)
=(-1)^{e}\mu^{-1}\left(\tfrac{D}{4}\right)\,p^{e}. 
$$
\end{lem}

\begin{prop} We have that ${\mathcal I}_{\widehat \phi}^{(-s,\Lambda^{-1},\mu^{-1})}(\delta)$ equals
\begin{align*}
[\bK^\#:\bK^\#_0(p^e\Z_p)]^{-1}\mu^{-1}\left(\tfrac{D\det \delta}{2}\right)\,\zeta_{\Q_p}(1)\zeta_{E}(1)p^{e(-3s-7/2)}W_{\Q_p}(\mu^{-1})^{3}.
\end{align*}
\end{prop}
\begin{proof}
By Lemma \ref{LocalSingw2M-L1}, the integral \eqref{LocalSingw2M-f0} becomes
\begin{align*}
[\bK^\#:\bK_0^{\#}(p^e\Z_p)]^{-1}\,p^{-e(s+3/2)}\,
\mu\left(\tfrac{D\det\delta}{2}\right)\,\zeta_{\Q_p}(1)\,W_{\Q_p}(\mu^{-1})
\,\sum_{\substack{\xi \in \cO/p\cO \\ \tfrac{D}{4}+\nr_{E/\Q_p}(\xi)\in \Z_p^\times}} 
\mu^{-1}\left(\tfrac{D}{4}+\nr_{E/\Q_p}(\xi)\right). 
\end{align*}
As a product of this expression and \eqref{LocalSingw2M-f00}, we obtain an evaluation of ${\mathcal I}_{\widehat \phi}^{(-s,\Lambda^{-1},\mu^{-1})}(\delta)$. Then, we use Lemma \ref{LocalSingw2M-L2} to get the desired formula. 
\end{proof}

\subsection{The Proof of Proposition \ref{LocalSingw2}}
The value 
\begin{align}
\int_{\bK^\#} \cW^{(T_\theta,\Lambda)}(f; \sm(\delta,\lambda^{-1}\det \delta)^{-1} 
\iota(k^\#))\,\d k^\#
\label{LocalSingIdq-f0}
\end{align}
is given by the integral of 
$$
\cchi_{\bK_0(p\Z_p)}\left(\iota(k^{\#})^{-1} \sm(\lambda \delta I_\theta(\tau),\det \delta \nr_{E/\Q_p}(\tau))\,\sn(X)\eta\right)\,\psi(-\tr(T_\theta X))
$$
over $(\tau,X)\in E^\times \times \sV(\Q_p)$. Since $\left[\begin{smallmatrix} 0 & 1 \\ 1 & 0 \end{smallmatrix}\right]$ and $\left[\begin{smallmatrix} 1 & 0 \\ -\bar \xi & 1 \end{smallmatrix}\right]\,(\xi \in \cO/q\cO)$ yield a complete set of representatives of $\bK^\#/\bK^\#_0(p\Z_p)$, noting that $\vol(\bK_0^\#(p\Z_p))=\frac{1}{p^2+1}$, we have that \eqref{LocalSingIdq-f0} becomes the sum of the following terms: 
{\allowdisplaybreaks\begin{align}
&\tfrac{1}{p^2+1}\int_{\sV(\Q_p)}\d X\int_{E^\times}\,\d^\times \tau\, \sum_{\xi\in \cO/p\cO}
\cchi_{\bK_0(p\Z_p)}
\left(\iota(\left[\begin{smallmatrix} 1 & 0 \\ \bar \xi & 1 \end{smallmatrix}\right])\sm(\lambda \delta I_\theta(\tau),\det \delta \nr_{E/\Q_p}(\tau))\sn(X)\eta\right)\,\psi(-\tr(T_\theta X)),
\label{LocalSingIdq-f00} \\
&\tfrac{1}{p^2+1}\int_{\sV(\Q_p)}\d X\int_{E^\times}\,\d^\times \tau\,\cchi_{\bK_0(p\Z_p)}\left(\iota(\left[\begin{smallmatrix} 0 & 1 \\ 1 & 0 \end{smallmatrix}\right])\sm(\lambda I_\theta(\tau),\det \delta \nr_{E/\Q_p}(\tau))\sn(X)\eta\right)\,\psi(-\tr(T_\theta X)). \label{LocalSingIdq-f1} 
\end{align}} We have
{\allowdisplaybreaks\begin{align*}
\iota(\left[\begin{smallmatrix} 1 & 0 \\ \bar \xi & 1 \end{smallmatrix}\right])\sm(\lambda \delta I_\theta(\tau),\det \delta \nr_{E/\Q_p}(\tau))\sn(X)\eta
&=
\left[\begin{matrix} \lambda h & \lambda h X\\ -\lambda X_\xi^{\dagger}h &
 -\lambda X_\xi^\dagger h X+{}^t h^{-1}\det h \end{matrix}\right] 
\left[\begin{smallmatrix} {} & {} & {} & {-1} \\
 {} & {} & {1} & {} \\
{} & {p} & {} & {} \\
{-p} & {} & {} & {}
\end{smallmatrix} \right]
\end{align*}
with $h=\delta I_\theta(\tau)$ belongs to $\bK_0(p\Z_p)$ if and only if 
\begin{align}
&  \lambda h X \left[\begin{smallmatrix} 0 & p \\ -p & 0 \end{smallmatrix}\right] \in {\bf GL}_2(\Z_p) 
\label{LocalSingIdq-f2},
\\
& \lambda h \left[\begin{smallmatrix} 0 & 1 \\ -1 & 0 \end{smallmatrix}\right]  \in {\bf M}_2(\Z_p), \label{LocalSingIdq-f3}
\\
&(-\lambda X^\dagger_\xi h X+{}^t h \det h)\left[\begin{smallmatrix} 0 & p \\ -p & 0 \end{smallmatrix}\right]  \in p\,{\bf M}_2(\Z_p), \label{LocalSingIdq-f4}
\\
& \lambda X_{\xi}^\dagger h\left[\begin{smallmatrix} 0 & 1 \\ -1 & 0 \end{smallmatrix}\right]  \in {\bf GL}_2(\Z_p) \label{LocalSingIdq-f5},
\\
&\lambda p \det h \in \Z_p^\times.
\label{LocalSingIdq-f6}
\end{align}}By taking the determinant of \eqref{LocalSingIdq-f5}, we have $\lambda^2\det h\in \Z_p^\times$, which together with \eqref{LocalSingIdq-f6} yields $\lambda \in p\Z_p^\times$; then, by \eqref{LocalSingIdq-f5}, $h\in p^{-1}{\bf GL}_2(\Z_p)$, which is equivalent to $I_\theta(p\tau) \in {\bf GL}_2(\Z_p)$, and \eqref{LocalSingIdq-f3} is satisfied. From $I_\theta(p\tau)\in {\bf GL}_2(\Z_p)$, we get $p\tau \in \cO^\times$. Condition \eqref{LocalSingIdq-f2} is equivalent to $X\in p^{-1}{\bf GL}_2(\Z_p)$. Condition \eqref{LocalSingIdq-f4} is equivalent to $
X-\lambda^{-1}X_{\xi}\ss({}^t I_\theta(\tau))\in \sV(\Z_p)$, 
which yields $\psi(\tr(T_\theta X))=\psi(\lambda^{-1}\tr(T_\theta \ss(I_\theta(\tau))\,X_\xi)=\psi(\lambda^{-1}\tr(T_\theta X_\xi))=1$. Therefore, 
\begin{align*}
\eqref{LocalSingIdq-f00}=\tfrac{1}{p^2+1}\cchi_{\lambda \in p\Z_p^\times}\, \vol_{E^\times}(p^{-1}\cO^\times)\, \vol(\sV(\Z_p))\,
\cdot\#(\cO/p\cO)
=\tfrac{p^2}{p^2+1} \cchi_{\lambda \in p\Z_p^\times}.
\end{align*}
In the same manner, we obtain 
\begin{align*}
\eqref{LocalSingIdq-f1}=\tfrac{1}{p^2+1}\cchi_{\lambda \in p\Z_p^\times}\,\vol(p^{-1}\cO^\times)\,\vol(\sV(\Z_p))=\tfrac{1}{p^2+1}\cchi_{\lambda \in p\Z_p^\times}.
\end{align*}
Thus, \eqref{LocalSingIdq-f0} is equal to 
$\tfrac{p^2}{p^2+1} \cchi_{\lambda \in p\Z_p^\times}
+\tfrac{1}{p^2+1}\cchi_{\lambda \in p \Z_p^\times}=\cchi_{\lambda \in p \Z_p^\times}$. \qed

%%%%%%%%%%%%%%%%%%%%%%%%%%%%%%%%%%%%%%%%%

\end{document}